\documentclass[a4paper, 10pt]{amsart}
\usepackage{geometry}
\usepackage{indentfirst}  
\usepackage{times}
\usepackage{mathrsfs}
\usepackage{latexsym}
\usepackage[dvips]{graphics} 
\usepackage{enumitem} 
\usepackage{epsfig}
\usepackage{hyperref, amsmath, amsthm, amsfonts, amscd, flafter,epsf} 
\usepackage{amsmath,amsfonts,amsthm,amssymb,amscd}
\usepackage{color}
\usepackage[T1]{fontenc}
 
\usepackage{amssymb} 
\usepackage{stackrel}
\usepackage{mathtools} 
\usepackage{tikz} 
\usepackage[utf8]{inputenc}
\usepackage{cleveref}
\crefname{section}{§}{§§}
\Crefname{section}{§}{§§}
\usetikzlibrary{matrix,arrows,decorations.pathmorphing} 
\usetikzlibrary{chains} 

\tikzset{node distance=2em, ch/.style={circle,draw,on chain,inner sep=2pt},chj/.style={ch,join},every path/.style={shorten >=4pt,shorten <=4pt},line width=1pt,baseline=-1ex}

\newcommand\be{\begin{equation}}
\newcommand\ee{\end{equation}}
\newcommand\bea{\begin{eqnarray}}
\newcommand\eea{\end{eqnarray}}
\newcommand\bi{\begin{itemize}}
	\newcommand\ei{\end{itemize}}
\newcommand\ben{\begin{enumerate}}
	\newcommand\een{\end{enumerate}}

\newtheorem{theorem}{Theorem}[section]
\newtheorem{conjecture}[theorem]{Conjecture}
\newtheorem{corollary}[theorem]{Corollary}
\newtheorem{lemma}[theorem]{Lemma}
\newtheorem{proposition}[theorem]{Proposition}

\newtheorem{claim}[theorem]{Claim}
\theoremstyle{definition}
\newtheorem{definition}[theorem]{Definition}

\theoremstyle{remark}
\newtheorem{example}[theorem]{Example}
\newtheorem{remark}[theorem]{Remark}

\newcommand{\Z}{\ensuremath{\mathbb{Z}}}
\newcommand{\simrightarrow}{\stackrel{\sim}{\rightarrow}}
\newcommand{\Q}{\mathbb{Q}}

\newcommand{\N}{\mathbb{N}}
\newcommand{\fM}{\mathcal{M}}

\newcommand{\cF}{\mathcal{F}}
\newcommand{\cL}{\mathcal{L}}

\newcommand{\cO}{\mathcal{O}}
\newcommand{\fg}{\mathfrak{g}}
\newcommand{\ft}{\mathfrak{t}}
\newcommand{\fb}{\mathfrak{b}}

\newcommand{\fp}{\mathfrak{p}}
\newcommand{\fq}{\mathfrak{q}}
\newcommand{\fm}{\mathfrak{m}}
\newcommand{\fn}{\mathfrak{n}}
\newcommand{\fu}{\mathfrak{u}}

\newcommand{\cD}{\mathcal{D}}
\newcommand{\cN}{\mathcal{N}}

\newcommand{\cS}{\mathcal{S}}
\newcommand{\cR}{\mathcal{R}}
\newcommand{\cC}{\mathcal{C}}

\newcommand{\cM}{\mathcal{M}}
\newcommand{\Gal}{\mathrm{Gal}}
\newcommand{\cG}{\mathcal{G}}
\newcommand{\cT}{\mathcal{T}}
\newcommand{\Qp}{\mathbb{Q}_p}
\newcommand{\BdR}{\mathrm{B}_{\mathrm{dR}}}
\newcommand{\GL}{\mathrm{GL}}
\newcommand{\gr}{\mathrm{gr}}
\newcommand{\wt}{\mathrm{wt}}
\newcommand{\Ad}{\mathrm{Ad}}

\newcommand{\Hom}{\mathrm{Hom}}

\newcommand{\Spec}{\mathrm{Spec}}
\newcommand{\Spf}{\mathrm{Spf}}

\newcommand{\Sp}{\mathrm{Sp}}
\newcommand{\Fil}{\mathrm{Fil}}
\newcommand{\pdR}{\mathrm{pdR}}

\newcommand{\dw}{\dot{w}}

\newcommand{\fX}{\mathfrak{X}}
\newcommand{\trivar}{X_{\mathrm{tri}}(\overline{r})}
\newcommand{\tildefg}{\widetilde{\mathfrak{g}}}

\numberwithin{equation}{section}

\title[Local models for the trianguline variety and partially classical families]{Local models for the trianguline variety and partially classical families\\\small Modèles locaux pour la variété trianguline et familles partiellement classiques}

\author{Zhixiang Wu} 

\date{}
\begin{document}
\address{Mathematisches Institut, Universität Münster, Einsteinstrasse 62, 48149 Münster, Germany}
\email{zhixiang.wu@uni-muenster.de}
\begin{abstract}
	We generalize Breuil-Hellmann-Schraen's local model for the trianguline variety to certain points with non-regular Hodge-Tate weights. With the local models we are able to prove, under the Taylor-Wiles hypothesis, the existence of certain companion points on the global eigenvariety and the appearance of related companion constituents in the completed cohomology for non-regular crystalline Galois representations. The new ingredients in the proof of the global applications are results relating the partial classicality of locally analytic representations (the existence of non-zero locally algebraic vectors in the parabolic Emerton's Jacquet modules), the partially de Rham properties of Galois representations (the de Rhamness of graded pieces along the paraboline filtrations of the associated $(\varphi,\Gamma)$-modules over the Robba rings) and the relevant properties of cycles on the generalized Steinberg varieties. We prove that partial classicality implies partial de Rhamness in finite slope cases using Ding's partial eigenvarieties.\\
	\\
	\textbf{Résumé.} Nous généralisons le modèle local de Breuil-Hellmann-Schraen pour la variété trianguline à certains points à poids de Hodge-Tate non régulier. Avec les modèles locaux, nous prouvons, sous l'hypothèse de Taylor-Wiles, l'existence de certains points compagnons sur la variété de Hecke et l'apparition de constituants compagnons correspondants dans la cohomologie complétée pour les représentations galoisiennes cristallines non régulières. Les nouveaux ingrédients dans la preuve des applications globales sont des résultats mettant en relation la classicité partielle des représentations localement analytiques (l'existence de vecteurs localement algébriques non nuls dans les modules de Jacquet-Emerton paraboliques), les propriétés de De Rham partiel des représentations galoisiennes (propriété de De Rham des morceaux gradués des filtrations parabolines des $(\varphi,\Gamma)$-modules associés sur les anneaux de Robba) et les propriétés correspondantes des cycles sur les variétés de Steinberg généralisées. Nous prouvons que la classicité partielle implique les propriétés de De Rham partiel dans les cas de pente finie en utilisant les variétés de Hecke partielles de Ding.
\end{abstract}
\maketitle

\tableofcontents
\section{Introduction}
Let $p$ be a prime number. This paper concerns about $p$-adic automorphic forms of definite unitary groups and the locally analytic aspect of the $p$-adic local Langlands program. Its aim is to generalize several results of Breuil-Hellmann-Schraen in \cite{breuil2019local} (local model for the trianguline variety, existence of companion points on the eigenvariety, locally analytic socle conjecture, etc.) to the cases when the Hodge-Tate weights are non-regular (i.e., not pairwise distinct). 
\subsection{Companion points and main results}
Let $F^+$ be a totally real number field and $S_p$ be the set of places of $F^+$ above $p$. Let $F$ be a quadratic imaginary extension of $F^+$ such that every place in $S_p$ splits in $F$, $n\geq 2$ be an integer and $\mathbb{G}$ be a totally definite unitary group in $n$ variables over $F^+$ that is split over $F$. We fix an open compact subgroup $U^p=\prod_{v\nmid p} U_v$ of $\mathbb{G}(\mathbf{A}_{F^+}^{p\infty})$ and a finite extension $L$ of $\Q_p$ with residue field $k_L$. For all $v\in S_p$, let $\Sigma_v:=\{\tau:F_v^{+}\hookrightarrow L\}$ and we assume $|\Sigma_v|=[F_v^{+}:\Q_p]$. For each $v\in S_p$, we fix a place $\widetilde{v}$ of $F$ above $v$ and identify $F_v^+\simeq F_{\widetilde{v}}$. The space of $p$-adic automorphic forms on $\mathbb{G}$ of tame level $U^p$, denoted by $\widehat{S}(U^p,L)$, consists of continuous functions $\mathbb{G}(F^+)\setminus \mathbb{G}(\mathbf{A}^{\infty}_{F^+})/U^p\rightarrow L$. Let $G_p=\prod_{v\in S_p}G_v$ be the $p$-adic Lie group $\mathbb{G}(F^+\otimes_{\Q}\Q_p)=\prod_{v\in S_p}\mathbb{G}(F_v^+)$. Let $B_p=\prod_{v\in S_p}B_v$ (resp. $T_p=\prod_{v\in S_p}T_v$) be the Borel subgroup (resp. the maximal torus) of $G_p\simeq \prod_{v\in S_p}\GL_n(F^+_v)$ consisting of upper-triangular (resp. diagonal) matrices. Then $G_p$ acts on $\widehat{S}(U^p,L)$ via right translations. We assume furthermore that $p>2$ and $\mathbb{G}$ is quasi-split at all finite places of $F^{+}$. Let $\overline{F}$ be an algebraic closure of $F$. We fix a (modular) absolutely irreducible Galois representation $\overline{\rho}:\Gal(\overline{F}/F)\rightarrow \GL_n(k_L)$ so that $\overline{\rho}$ is associated with a maximal ideal of some usual Hecke algebra acting on $\widehat{S}(U^p,L)$. \par

After Emerton \cite{emerton2006interpolation}, one way to construct eigenvarieties, rigid analytic varieties parameterizing finite slope overconvergent $p$-adic eigenforms, is using Emerton's Jacquet module functor for locally analytic representations of $p$-adic Lie groups. There exists a rigid space over $L$ (our eigenvariety), denoted by $Y(U^p,\overline{\rho})$, on which a point is a pair $(\rho,\underline{\delta})$, where $\rho$ is a $p$-adic continuous $n$-dimensional representation of $\Gal(\overline{F}/F)$ and $\underline{\delta}$ is a continuous character of $T_p$ which appears in $J_{B_p}(\Pi(\rho)^{\mathrm{an}})$. Here $\Pi(\rho)$ is the sub-$G_p$-representation of $\widehat{S}(U^p,L)$ associated with $\rho$ cut out by a prime ideal of certain Hecke algebra, $\Pi(\rho)^{\mathrm{an}}$ is the subspace of $\Pi(\rho)$ consisting of locally analytic vectors which is an admissible locally analytic representation of $G_p$ and $J_{B_p}(-)$ denotes the Emerton's Jacquet module functor so that $J_{B_p}(\Pi(\rho)^{\mathrm{an}})$ is a locally analytic representation of the Levi subgroup $T_p$ of $B_p$. \par

Take a point $(\rho,\underline{\delta})$ on $Y(U^p,\overline{\rho})$. The problem of companion forms seeks to determine the set of characters $\underline{\delta}'$ of $T_p$, denoted by $W(\rho)$, such that pairs $(\rho,\underline{\delta}')$ appear on $Y(U^p,\overline{\rho})$. The existence of such \emph{companion points} $(\rho,\underline{\delta}')$ is closely related to the appearance of certain irreducible locally analytic representations of $G_p$ explicitly determined by $\underline{\delta}'$ and $\rho$, which we call \emph{companion constituents}, inside $\Pi(\rho)^{\mathrm{an}}$. The existence of such companion constituents is a special case of the locally analytic socle conjecture of Breuil \cite{breuil2016versI,breuil2015versII}. For $v\in S_p$, we let $\rho_v:=\rho\mid_{\Gal(\overline{F_{v}^{+}}/F_v^{+})}$. The general recipe for $W(\rho)$ has been conjectured by Hansen \cite{hansen2017universal} which depends only on those local Galois representations $\rho_v$ for $v\in S_p$ and the notion of \emph{trianguline representations} introduced by Colmez \cite{colmez2008trianguline}. One could view the problem of companion forms or locally analytic socles as a locally analytic analogue of the weight part of Serre's modularity conjecture.

Let $D_{\mathrm{rig}}(\rho_v)$ be the \'etale $(\varphi,\Gamma)$-module over the Robba ring associated with $\rho_v$ for $v\in S_p$. In the $p$-adic local Langlands program, locally analytic representations of $p$-adic Lie groups are expected to be related to $(\varphi,\Gamma)$-modules over the Robba rings which is the case for $\GL_2(\Q_p)$ by Colmez \cite[V]{colmez2010representations}. Beyond the foundational works of Kisin, Colmez and Emerton for $\GL_2(\Q_p)$ \cite{kisin2003overconvergent,colmez2008trianguline,emerton2011local}, we know in general and especially in our setting by the global triangulation results of Liu \cite{liu2015triangulation} or Kedlaya-Porttharst-Xiao \cite{kedlaya2014cohomology} that the non-triviality of the Borel Emerton's Jacquet module $J_{B_p}(\Pi(\rho)^{\mathrm{an}})$ (i.e. in the finite slope case) implies that $\rho_v$ is trianguline, i.e. $D_{\mathrm{rig}}(\rho_v)$ admits a full filtration 
\begin{equation}\label{equationintroductiontriangulinefiltration}
    \Fil^{\bullet}D_{\mathrm{rig}}(\rho_v): D_{\mathrm{rig}}(\rho_v)=\Fil^nD_{\mathrm{rig}}(\rho_v)\supsetneq \cdots \supsetneq \Fil^1D_{\mathrm{rig}}(\rho_v)\supsetneq \Fil^0D_{\mathrm{rig}}(\rho_v)=\{0\}
\end{equation}
of sub-$(\varphi,\Gamma)$-modules such that the graded pieces are rank one $(\varphi,\Gamma)$-modules.\par

Under the Taylor-Wiles hypothesis on $\overline{\rho}$, Breuil-Hellmann-Schraen proved in \cite{breuil2019local} the existence of all companion forms for regular generic crystalline points. In this paper, we generalize their results to non-regular generic crystalline points. To be precise, we take a point $(\rho,\underline{\delta})\in Y(U^p,\overline{\rho})$. We say $\rho$ (or the point $(\rho,\underline{\delta})$) is crystalline if for all $v\in S_p$, $\rho_v$ is crystalline. If $\rho$ is crystalline, let $(\varphi_{v,i})_{i=1,\cdots,n}$ be the eigenvalues of $\varphi^{f_v}$ where $\varphi$ is the crystalline Frobenius acting on $D_{\mathrm{cris}}(\rho_v)$ and $q_v=p^{f_v}$ is the cardinality of the residue field of $F_v^+$. Then we say $\rho$ (or the point $(\rho,\underline{\delta})$) is \emph{generic} if for any $v\in S_p$, $\varphi_{v,i}\varphi_{v,j}^{-1}\notin \{1,q_v\}$ for $i\neq j$. Assume that $\rho$ is generic crystalline. A refinement $\cR_v$ of $\rho_v$ is a choice of an ordering of the pairwise distinct eigenvalues $\varphi_{v,1},\cdots,\varphi_{v,n}$ and a refinement $\cR=(\cR_v)_{v\in S_p}$ of $\rho$ is a choice of a refinement $\cR_v$ for each $v\in S_p$. In fact the refinements $\cR_v$ correspond to triangulations of $D_{\mathrm{rig}}(\rho_v)$ as (\ref{equationintroductiontriangulinefiltration}) by \cite{berger2008equations}. Then the conjectural set of characters $W(\rho)$ admits a partition $W(\rho)=\coprod_{\cR}W_{\cR}(\rho)$ where $W_{\cR}(\rho)=\prod_{v\in S_p} W_{\cR_v}(\rho_v)$ and each $W_{\cR_v}(\rho_v)$ is a finite set which can be explicitly described by $\cR_v$ and $\rho$. Remark that the partition of $W(\rho)$ according to the refinements is also the partition under the equivalence relation that $\underline{\delta}\sim \underline{\delta}'$ if and only if $\underline{\delta}^{-1} \underline{\delta}'$ is a $\Q_p$-algebraic character of $T_p$. Our main theorem is the following.
\begin{theorem}[Theorem \ref{theoremmainautomorphicforms}]\label{theoremintroductionmain}
    Assume that $U^p$ is small enough and assume the Taylor-Wiles hypothesis (cf. \S\ref{sectionglobalsettings}): $F$ is unramified over $F^+$, $F$ doesn't contain non-trivial $p$-th root of unity, $U^p$ is hyperspecial at any finite place of $F^+$ that is inert in $F$ and $\overline{\rho}(\Gal(\overline{F}/F(\sqrt[p]{1})))$ is adequate. Let $(\rho,\underline{\delta})\in Y(U^p,\overline{\rho})$ be a point such that $\rho$ is generic crystalline. Then there exists a refinement $\cR$ of $\rho$ such that $\underline{\delta}\in W_{\cR}(\rho)$ and for any $\underline{\delta}'\in W_{\cR}(\rho)$, the point $(\rho,\underline{\delta}')$ exists on $Y(U^p,\overline{\rho})$. Moreover, all the companion constituents associated with $W_{\cR}(\rho)$ appear in $\Pi(\rho)$.
\end{theorem}
\begin{remark}
    In \cite{breuil2019local}, the above theorem was proved under the extra assumption that for each $v\in S_p$, the Hodge-Tate weights of $\rho_v$ are regular (pairwise distinct). But a stronger version was proved in \cite{breuil2019local}: in regular cases, $(\rho,\underline{\delta}')$ exists on $Y(U^p,\overline{\rho})$ for any refinement $\cR'$ of $\rho$ and $\underline{\delta}'\in W_{\cR'}(\rho)$. This stronger result is easy to get from Theorem \ref{theoremintroductionmain} in regular cases using locally algebraic vectors in $\Pi(\rho)$ and is not available in this paper for general crystalline points due to the non-existence of non-zero locally algebraic vectors in $\Pi(\rho)$ when $\rho$ is non-regular (the non-existence can be seen using the results on infinitesimal characters in \cite{dospinescu2020infinitesimal}). See Remark \ref{remarksmoothcompanionpoints} for a partial result. The existence of all companion points in generic non-regular crystalline cases will need other methods.
\end{remark}
The method in \cite{breuil2019local} was firstly replacing the eigenvariety $Y(U^p,\overline{\rho})$ by a larger \emph{patched eigenvariety} $X_p(\overline{\rho})$ in \cite{breuil2017interpretation,breuil2017smoothness} constructed from the patching module in \cite{caraiani2016patching}. The patching method allows us to reduce the study of the geometry of the patched eigenvariety to that of its local component, called \emph{trianguline variety}, which parameterizes local trianguline Galois representations. Then Breuil-Hellmann-Schraen used a local model to describe the local geometry of the trianguline variety at certain points. We prove Theorem \ref{theoremintroductionmain} by developing further the theory of local models. The major new inputs are the following two results.\par

Firstly, we construct local models of the trianguline variety for certain points with possibly non-regular Hodge-Tate weights and prove that the trianguline variety is irreducible at those points (Theorem \ref{theoremintroductionlocalmodel}). Those local models are algebraic varieties which are similar to the regular cases and reflect the phenomenon of the existence of companion points or companion constituents on the eigenvariety or in the space of $p$-adic automorphic forms. 

Secondly, we show that for a general point $(\rho,\underline{\delta})\in Y(U^p,\overline{\rho})$ where $\rho$ may not be de Rham above $p$, the existence of certain companion constituents, which we call \emph{partially classical} constituents, will force the local Galois representations $\rho_v,v\in S_p$ satisfy certain special properties for which we say $\rho_v$ are \emph{partially de Rham} (Theorem \ref{theoremintroductionpartiallyclassical}). The partially classical constituents are locally analytic representations of $G_p$ which will give rise to the existence of certain \emph{locally algebraic vectors} inside some non-Borel parabolic Emerton's Jacquet module $J_{Q_p}(\Pi(\rho)^{\mathrm{an}})$ of $\Pi(\rho)^{\mathrm{an}}$, where $Q_p$ is some parabolic subgroup of $G_p$ containing $B_p$ and is not equal to $B_p$. Let $M_{Q_p}$ be the Levi subgroup of $Q_p$ containing $T_p$. Then $J_{Q_p}(\Pi(\rho)^{\mathrm{an}})$ is a locally analytic representation of $M_{Q_p}$ which, in analogue with the case of Borel Emerton's Jacquet module, should correspond to some so called (after Chenevier \cite{chenevier2011infinite}, see also \cite{bergdall2017paraboline}) \emph{paraboline filtrations} of $D_{\mathrm{rig}}(\rho_v)$
\[\Fil_{Q_p}^{\bullet}D_{\mathrm{rig}}(\rho_v): D_{\mathrm{rig}}(\rho_v)=\Fil_{Q_p}^{t_v}D_{\mathrm{rig}}(\rho_v)\supsetneq \cdots \supsetneq \Fil_{Q_p}^{1}D_{\mathrm{rig}}(\rho_v)\supsetneq \Fil_{Q_p}^{0}D_{\mathrm{rig}}(\rho_v)=\{0\}, v\in S_p\]
where the ranks of the graded pieces of the above filtrations should be sizes of the blocks of the Levi subgroup $M_{Q_p}$. Since we are always in the finite slope cases, we only focus on those paraboline filtrations that are sub-filtrations of the trianguline filtrations (\ref{equationintroductiontriangulinefiltration}). This means that there exist integers $0=s_{v,0}<s_{v,1}<\cdots<s_{v,t_v-1}<s_{v,t_v}=n$ such that 
\[\Fil_{Q_p}^{i}D_{\mathrm{rig}}(\rho_v)=\Fil^{s_{v,i}}D_{\mathrm{rig}}(\rho_v).\]
From $(\varphi,\Gamma)$-modules over the Robba rings one can always obtain semi-linear Galois representations over Fontaine's ring $\BdR$ after Berger (e.g. \cite{berger2008construction}), thus we can define the de Rham property for $(\varphi,\Gamma)$-modules as for $p$-adic Galois representations. Our result then states that the appearance of certain locally algebraic vectors in $J_{Q_p}(\Pi(\rho)^{\mathrm{an}})$ implies that the graded pieces 
\[\Fil^{s_{v,i}}D_{\mathrm{rig}}(\rho_v)/\Fil^{s_{v,i-1}}D_{\mathrm{rig}}(\rho_v)\]
are de Rham $(\varphi,\Gamma)$-modules for $i=1,\cdots t_v,v\in S_p$. Recall that locally algebraic vectors in $\widehat{S}(U^p,L)$ with respect to the action of $G_p$ are $p$-adic avatars of algebraic regular automorphic forms which correspond to $p$-adic Galois representations that are de Rham over $p$ with regular Hodge-Tate weights. Hence the result on partially classical constituents can be viewed as a form of generalization with some functoriality of the classical correspondence, beyond ordinary cases (\cite{ding2019ordinary}, etc.).
\begin{remark}
    Partial de Rhamness as well as partial classicality was proposed by Ding in a narrow sense for $2$-dimensional Galois representations \cite{ding2017formes,ding2017partiallyderham,ding2019companion} and partial classicality was also mentioned by Ding for his partial eigenvariety for $\GL_n(\Q_p)$ \cite{ding2019some} which we will use. Our results combine and generalize both Ding's works.
\end{remark}
Our key step (Proposition \ref{propositionmaincycle}) to prove the existence of the companion constituents or companion points for a point $x=(\rho,\underline{\delta})$ as in Theorem \ref{theoremintroductionmain} goes roughly in the following way (see also \S\ref{sectionintroductionsectionexistence}, especially Example \ref{exampleintroduction}). It will also simplify the relavent arguments in \cite{breuil2019local} even for the regular case. As in \cite{breuil2019local}, there are cycles (closed subspaces) passing through $x$ on the eigenvariety which correspond to the appearance of companion constituents in $\Pi(\rho)^{\mathrm{an}}$. By the second result above, those cycles corresponding to partially classical constituents (with respect to some parabolic subgroups) are partially de Rham which means that the corresponding Galois representations are partially de Rham. On the other hand, the local model also gives rise to cycles near $x$ on the patched eigenvariety, which are expected to match those cycles corresponding to companion constituents. The point is that the partially de Rham properties are determined by the datum of local models, and it turns out that the partially de Rham cycles on the local models are exactly those cycles that should match partially classical constituents (Theorem \ref{theoremintroductionsteinberg}). Then a finer study of the local models tells that there exist non-partially de Rham cycles passing through $x$ which implies the existence of non-partially classical companion constituents inside $\Pi(\rho)$. In this way, we can obtain all companion constituents that can be seen by the local models (those constituents in Theorem \ref{theoremintroductionmain}). \par
In the remaining parts of this introduction, we give more details on the above results and their proofs.
\subsection{Local models for the trianguline variety}
We now explain our local results on the trianguline variety. For $v\in S_p$, the trianguline variety $X_{\mathrm{tri}}(\overline{\rho}_v)$ with respect to $\overline{\rho}_v:=\overline{\rho}\mid_{\Gal(\overline{F_{v}^{+}}/F_v^{+})}$ is a rigid analytic variety, a point of which is given by a pair $(r,\underline{\delta})$ where $r$ is a deformation of $\overline{\rho}_v$ and $\underline{\delta}=(\delta_i)_{1\leq i\leq n}$ is a character of $T_v=((F_v^{+})^{\times})^{n}$, such that the subset of points $(r,\underline{\delta})$, where $r$ is trianguline and $\underline{\delta}$ corresponds to the graded pieces of certain trianguline filtration of $r$ as (\ref{equationintroductiontriangulinefiltration}), is Zariski dense. \par
We take an $L$-point $x=(r,\underline{\delta})$ of $X_{\mathrm{tri}}(\overline{\rho}_v)$. The weight $\mathrm{wt}(\delta_i)$ of each character $\delta_i$ is a number in $F_v^{+}\otimes_{\Q_p}L\simeq \oplus_{\tau\in\Sigma_v} L$ and we write $\mathrm{wt}_{\tau}(\delta_i)\in L$ for the $\tau$-part of $\mathrm{wt}(\delta_i)$ for each $\tau\in \Sigma$. The multiset $\{\mathrm{wt}_{\tau}(\delta_i)\mid i\in \{1,\cdots,n\} ,\tau\in\Sigma_v\}$ is also the $\tau$-Sen weights of $r$ (the generalized Hodge-Tate weights, counted with multiplicities).
Then $\underline{\delta}$ is locally $\Q_p$-algebraic if for all $i=1,\cdots,n,\tau\in \Sigma_v$, $\mathrm{wt}_{\tau}(\delta_i)\in\Z$. \par
We say $\underline{\delta}$ is \emph{generic} if for any $i\neq j$, both $\delta_i^{-1}\delta_j$ and $\delta_i^{-1}\delta_j|\mathrm{Norm}_{F_v^{+}/\Q_p}|_p$, where $|p|_p=p^{-1}$, are not $\Q_p$-algebraic characters (i.e. not of the form $z\mapsto \prod_{\tau\in \Sigma_v}\tau(z)^{k_{\tau}}$ where $k_{\tau}\in \Z$ for every $\tau\in\Sigma_v$).\par
In the case when $\underline{\delta}$ is generic and locally $\Q_p$-algebraic, $r$ is almost de Rham in the sense of Fontaine \cite{fontaine2004arithmetique}. Fontaine's theory associates $r$ with a finite free $F_{v}^+\otimes_{\Q_p}L$-module $D_{\mathrm{pdR}}(r)$ of rank $n$ and a linear nilpotent operator $N$ acting on $D_{\mathrm{pdR}}(r)$. The space $D_{\mathrm{pdR}}(r)$ is equipped with two filtrations, both are stable under the action of $N$. One filtration is the Hodge filtration denoted by $\Fil_{\bullet}$. Another filtration $\cD_{\bullet}:=D_{\mathrm{pdR}}(\Fil^{\bullet}D_{\mathrm{rig}}(r))$ comes from a trianguline filtration $\Fil^{\bullet}D_{\mathrm{rig}}(r)$ on $D_{\mathrm{rig}}(r)$ determined by the point $x$ (we emphasize that the functor $D_{\mathrm{pdR}}(-)$ is also defined for these $(\varphi,\Gamma)$-modules). \par
For $\tau\in\Sigma_v$, define $D_{\mathrm{pdR},\tau}(r):=D_{\mathrm{pdR}}(r)\otimes_{L\otimes_{\Q_p} F_v^{+},1\otimes\tau}L$, $\Fil_{\tau,\bullet}:=\Fil_{\bullet}\otimes_{L\otimes_{\Q_p} F_v^{+},1\otimes\tau}L$ and $\cD_{\tau,\bullet}:=\cD_{\bullet}\otimes_{L\otimes_{\Q_p} F_v^{+},1\otimes\tau}L$.
Then for each $\tau\in \Sigma_v$, $\cD_{\tau, \bullet}:\cD_{\tau, 1}\subsetneq \cdots\subsetneq \cD_{\tau, n} $ is a complete flag of the $L$-space $D_{\mathrm{pdR},\tau}(r)$. The graded pieces of the Hodge filtration $\Fil_{\tau,\bullet}$ have $L$-dimensions that are equal to the multiplicities of the $\tau$-Sen weights. \par
Let $G:=\mathrm{Res}_{F_v^{+}/\Q_p}(\GL_{n/F_v^+})\otimes_{\Q_p}L=\prod_{\tau\in\Sigma_v} \GL_{n/L}$ be the algebraic group which acts on $ D_{\mathrm{pdR}}(r)\simeq (F_{v}^+\otimes_{\Q_p}L)^n\simeq \prod_{\tau\in \Sigma_v} L^n$, $P$ be the standard parabolic subgroup of block upper-triangular matrices in $G$ that is conjugate to the stabilizer subgroup of the Hodge filtration $\Fil_{\bullet}$ and $B$ the Borel subgroup of upper-triangular matrices of $G$. Let $\fg$ (resp. $\fb$, resp. $\fp$) be the Lie algebra of $G$ (resp. $B$, resp. $P$). The datum $(N,\cD_{\bullet},\Fil_{\bullet})$ associated with the point $x$ can define a point $x_{\mathrm{pdR}}$ of the following algebraic scheme 
\begin{equation}\label{equationintroductionXP}
    X_P:=\left\{\left(\nu,g_1B,g_2P\right)\in \fg\times G/B\times G/P\mid \Ad\left(g_1^{-1}\right)\nu\in\fb, \Ad\left(g_2^{-1}\right)\nu\in\fp\right\}
\end{equation}
where $G/B,G/P$ are flag varieties and $\Ad$ denotes the adjoint action. Let $W\simeq \prod_{\tau\in\Sigma_v}\mathcal{S}_n$ (resp. $W_P$) be the Weyl group of $G$ (resp. of the standard Levi subgroup of $P$) where $\mathcal{S}_n$ denotes the $n$-th symmetric group. Then $X_P$ is equidimensional and its irreducible components $X_{P,w}$ are parameterized by $w\in W/W_P$ (see Definition \ref{definitionirreduciblecomponents}). Let $w=(w_{\tau})_{\tau\in\Sigma_v}\in \prod_{\tau\in\Sigma_v}\mathcal{S}_n$ be an element such that $\mathrm{wt}_{\tau}(\delta_{w_{\tau}(1)})\leq \cdots \leq \mathrm{wt}_{\tau}(\delta_{w_{\tau}(n)})$ for all $\tau\in \Sigma_v$ and we use the same notation $w$ to denote the image of $w$ in $W/W_P$. The following theorem is proved in \cite{breuil2019local} when $P=B$, i.e. when $r$ has regular Hodge-Tate weights ($\wt_{\tau}(\delta_i)\neq \wt_{\tau}(\delta_j)$ for all $\tau\in \Sigma_v,i\neq j$).
\begin{theorem}[Theorem \ref{theoremirreducibletriangullinevariety}]\label{theoremintroductionlocalmodel}
    Let $x=(r,\underline{\delta})$ be an $L$-point of $X_{\mathrm{tri}}(\overline{\rho}_v)$ such that $\underline{\delta}$ is generic and locally $\Q_p$-algebraic. Then up to formally smooth morphisms of formal schemes, the completion $\widehat{X}_{\mathrm{tri}}(\overline{\rho}_v)_x$ of the trianguline variety $X_{\mathrm{tri}}(\overline{\rho}_v)$ at $x$ is isomorphic to the completion $\widehat{X}_{P,w,x_{\mathrm{pdR}}}$ of $X_{P,w}$ at $x_{\mathrm{pdR}}$. Moreover, the trianguline variety $X_{\mathrm{tri}}(\overline{\rho}_v)$ is irreducible at $x$. 
\end{theorem}
The proof of Theorem \ref{theoremintroductionlocalmodel} follows the strategy for regular cases in \cite{breuil2019local}. The difficulty in our situation is to show that $X_{P,w}$ is \emph{unibranch} at $x_{\mathrm{pdR}}$, i.e. the completion of the local ring of $X_{P,w}$ at $x_{\mathrm{pdR}}$ is irreducible (Theorem \ref{theoremunibranchYP}). When $P=B$, it was proved by Bezrukavnikov-Riche in \cite{bezrukavnikov2012affine} that for $w\in W$, $X_{B,w}$ is Cohen-Macaulay and based on the Cohen-Macaulay result, Breuil-Hellmann-Schraen proved that $X_{B,w}$ is normal in \cite{breuil2019local} which in particular implies that $X_{B,w}$ is unibranch. We prove that $X_{P,w}$ is unibranch at $x_{\mathrm{pdR}}$ based on the normality of $X_{B,w}$ (here $w\in W$). There is a natural birational proper map $f:X_{B,w}\rightarrow X_{P,w}$ of integral varieties. We can prove that the fiber $f^{-1}(x_{\mathrm{pdR}})$ is connected (Proposition \ref{propositionconnectedfibersnilpotent}). Since $X_{B,w}$ is normal, the connectedness of the fiber is enough to establish the unibranch property that we need (Proposition \ref{propfiberconnectedunibranch}). The important problem whether $X_{P,w}$ (or $X_{\mathrm{tri}}(\overline{\rho}_v)$ at $x$) is Cohen-Macaulay or normal remains unsolved.
\subsection{Partially classical families and partial de Rhamness}\label{sectionintroductionpartiallyclassical}
We need some more notation to state the result on partially classical companion constituents. For a point $(\rho,\underline{\delta})\in Y(U^p,\overline{\rho})$, $\underline{\delta}=(\underline{\delta}_v)_{v\in S_p}=((\delta_{v,i})_{i=1,\cdots,n})_{v\in S_p}$ is a character of $T_p=\prod_{v\in S_p}T_v\simeq \prod_{v\in S_p} ((F_v^{+})^{\times})^{n}$. Let 
\[\lambda=(\lambda_{\tau})_{\tau\in \Sigma_v,v\in S_p}=((\lambda_{\tau,i})_{i=1,\cdots,n})_{\tau\in \Sigma_v,v\in S_p}:=((\wt_{\tau}(\delta_i))_{i=1,\cdots,n})_{\tau\in \Sigma_v,v\in S_p}\] 
be the weight of $\underline{\delta}$. When $\underline{\delta}$ is locally algebraic for which we mean that $\lambda_{\tau,i}\in \Z$ for all $\tau\in \Sigma_v,v\in S_p$, Orlik-Strauch's theory \cite{orlik2015jordan} can construct a (generically irreducible) locally analytic representation $\cL(\lambda,\underline{\delta})$ of $G_p$ from $\lambda$ and (the smooth part of) $\underline{\delta}$ which is a subquotient of some locally analytic principal series representation (cf. \S\ref{sectionOrlikStrauch}). \par
We fix $v_0\in S_p$, $\tau_{0}\in \Sigma_{v_0}$ and a parabolic subalgebra $\fq_{\tau_0}$ of block upper-triangular matrices of $\mathfrak{gl}_n$, the Lie algebra of $\GL_{n/L}$. We assume that the standard Levi subalgebra $\fm_{\tau_0}$ of $\fq_{\tau_0}$ is isomorphic to $\mathfrak{gl}_{s_1-s_0}\times\cdots\times\mathfrak{gl}_{s_i-s_{i-1}}\cdots\times \mathfrak{gl}_{s_t-s_{t-1}}$ where $ 0=s_0<\cdots< s_i<\cdots< s_t=n$. Suppose that $\lambda_{\tau_0}$ is a dominant weight of $\fm_{\tau_0}$, or explicitly, $\lambda_{\tau_0,a}\geq \lambda_{\tau_0,b}$ for every $a\leq b,a,b\in [s_{i}+1,s_{i+1}]$ and $0\leq i\leq t-1$. Let $L_{\fm_{\tau_0}}(\lambda_{\tau_0})$ be the finite-dimensional irreducible $\fm_{\tau_0}$-representation of the highest weight $\lambda_{\tau_0}$. Let $Q_p=\prod_{v\in S_p}Q_{v}$ be a standard parabolic subgroup of $G_p$ such that its $\tau_0$-part Lie algebra $\mathrm{Lie}(Q_{v_0})\otimes_{F_{v_0}^{+},\tau_{0}}L$ is equal to $\fq_{\tau_0}$. Then the companion constituent $\cL(\lambda,\underline{\delta})$ is partially classical (with respect to $Q_p$ and the set $\{\tau_0\}$) in the sense that we have (by an adjunction formula, see \S\ref{sectionadjunctionformula})
\begin{equation}\label{equationintroductionpartialclassicality}
    \Hom_{G_p}\left( \cL(\lambda,\underline{\delta}),\Pi(\rho)^{\mathrm{an}}\right)\neq 0\quad\Rightarrow\quad\Hom_{\fm_{\tau_0}}\left(L_{\fm_{\tau_0}}(\lambda_{\tau_0}),J_{Q_p}(\Pi(\rho)^{\mathrm{an}})\right)\neq 0.
\end{equation}
\indent Assume that $(\rho,\underline{\delta})$ is a point on $Y(U^p,\overline{\rho})$ such that $\underline{\delta}$ is locally algebraic and satisfies certain generic condition. Then $D_{\mathrm{rig}}(\rho_{v_0})$ is associated via the point $(\rho,\underline{\delta})$ with a trianguline filtration $\Fil^{\bullet}D_{\mathrm{rig}}(\rho_{v_0})$ as (\ref{equationintroductiontriangulinefiltration}) which in turn leads to a filtration $\cD_{v_0,\bullet}=\prod_{\tau\in\Sigma_{v_0}}\cD_{\tau,\bullet}$ of $D_{\mathrm{pdR}}(\rho_{v_0})$ with a nilpotent operator $N_{v_0}=(N_{\tau})_{\tau\in \Sigma_{v_0}}$ where every $ N_{\tau}$ keeps the filtration $\cD_{\tau,\bullet}$. For $i=1,\cdots,t$, we let $\mathrm{gr}^{s_i}D_{\mathrm{rig}}(\rho_{v_0}):=\Fil^{s_i}D_{\mathrm{rig}}(\rho_{v_0})/\Fil^{s_{i-1}}D_{\mathrm{rig}}(\rho_{v_0})$ be the graded pieces of the paraboline sub-filtration $\Fil^{s_{\bullet}}D_{\mathrm{rig}}(\rho_{v_0})$ corresponding to $\fm_{\tau_0}$. In the case of Galois representations, an almost de Rham representation $\rho_v$ is de Rham if and only if the nilpotent operator on $D_{\mathrm{pdR}}(\rho_v)$ is zero. We can identify $D_{\mathrm{pdR},\tau_0}(\mathrm{gr}^{s_i}D_{\mathrm{rig}}(\rho_{v_0})):={D_{\mathrm{pdR}}(\mathrm{gr}^{s_i}D_{\mathrm{rig}}(\rho_{v_0}))\otimes_{L\otimes_{\Q_p}F_{v_0}^{+},1\otimes\tau_0}L}$ with $\cD_{\tau_0,s_{i}}/\cD_{\tau_0,s_{i-1}}$ equipped with the restriction of the action of $N_{\tau_0}$. We say that the $(\varphi,\Gamma)$-module $\mathrm{gr}^{s_i}D_{\mathrm{rig}}(\rho_{v_0})$ is $\{\tau_0\}$-partially de Rham if the restriction of the nilpotent operator $N_{\tau_0}$ on $\cD_{\tau_0,s_{i}}/\cD_{\tau_0,s_{i-1}}$ is zero.
\begin{theorem}[Theorem \ref{theoremQderham}]\label{theoremintroductionpartiallyclassical}
    Let $(\rho,\underline{\delta})\in Y(U^p,\overline{\rho})$ be a point such that $\underline{\delta}$ is locally algebraic and generic (Definition \ref{definitiongenericglobal}). If the $\tau_0$-part weight $\lambda_{\tau_0}$ of $\underline{\delta}$ is a dominant weight for $\fm_{\tau_0}$ and $\Hom_{G_p}\left( \cL(\lambda,\underline{\delta}),\Pi(\rho)^{\mathrm{an}}\right)\neq 0$, then for every $1\leq i\leq t$, the graded piece $\mathrm{gr}^{s_i}D_{\mathrm{rig}}(\rho_{v_0})$ is a $\{\tau_0\}$-partially de Rham $(\varphi,\Gamma)$-module.
\end{theorem}
Theorem \ref{theoremintroductionpartiallyclassical} is in fact a corollary of the global triangulation results and Ding's construction of partial eigenvarieties in \cite{ding2019some} which was based on the work of Hill-Loeffler \cite{hill2011emerton}. Assume that $\lambda_{\tau_0}$ is dominant for $\fm_{\tau_0}$. The partial eigenvariety of Ding, denoted by $Y(U^p,\overline{\rho})(\lambda_{\tau_0}')$, is a subvariety of $Y(U^p,\overline{\rho})$ roughly consisting of points $x=(\rho_x,\underline{\delta}_x)\in Y(U^p,\overline{\rho})$ such that $\Hom_{\fm_{\tau_0}'}\left(L_{\fm_{\tau_0}}(\lambda_{\tau_0}),J_{Q_p}(\Pi(\rho_x)^{\mathrm{an}})\right)\neq \emptyset$ where $\fm_{\tau_0}':=[\fm_{\tau_0},\fm_{\tau_0}]$ is the derived subalgebra of $\fm_{\tau_0}$. Such construction forces that for any point $(\rho_x,\underline{\delta}_x)\in Y(U^p,\overline{\rho})(\lambda_{\tau_0}')$, the $\tau_0$-part weight $\lambda_{x,\tau_0}=(\lambda_{x,\tau_0,i})_{i=1,\cdots,n}$ of $\underline{\delta}_x$ satisfies that $\lambda_{x,\tau_0,a}-\lambda_{x,\tau_0,b}=\lambda_{\tau_0,a}-\lambda_{\tau_0,b}$ are non-negative integers independent of $x$ for every $a\leq b,a,b\in [s_{i}+1,s_{i+1}]$ and $0\leq i\leq t-1$. Then the arguments of Berger-Colmez in \cite{berger2008familles} and the global triangulation show that the subset of points $x\in Y(U^p,\overline{\rho})(\lambda_{\tau_0}')$ such that a suitable twist of $\mathrm{gr}^{s_{i}}D_{\mathrm{rig}}(\rho_{x,v_0})$ is $\{\tau_0\}$-partially de Rham is Zariski closed. The feature of Ding's construction is that such constrain on the $\tau_0$-weights $\lambda_{x,\tau_0}$ still allows $\lambda_{x,\tau_0}$ to vary and to be dominant with respect to $\fg_{\tau_0}:=\mathfrak{gl}_n$ even if $\lambda_{\tau_0}$ is not. The usual eigenvariety arguments imply that the subset of classical points, where there exist non-zero locally algebraic vectors in $\Pi(\rho_x)$ and $\underline{\delta}_x$ admit dominant weights, is Zariski dense in $Y(U^p,\overline{\rho})(\lambda_{\tau_0}')$. It follows from the classical local-global compatibility when $\ell=p$ that classical points are de Rham ($\rho_x$ are de Rham) where $\mathrm{gr}^{s_{i}}D_{\mathrm{rig}}(\rho_{x,v_0})$ is automatically $\{\tau_0\}$-partially de Rham. Combining the Zariski dense and closed statements leads to Theorem \ref{theoremintroductionpartiallyclassical}.
\subsection{Existence of companion constituents}\label{sectionintroductionsectionexistence}
The key observation to prove Theorem \ref{theoremintroductionmain} is that Theorem \ref{theoremintroductionpartiallyclassical} is reflected by the local models of the trianguline variety. There is a closed embedding $Y(U^p,\overline{\rho})\hookrightarrow X_p(\overline{\rho})$. Here the patched eigenvariety $X_p(\overline{\rho})$ is equidimensional and can be identified as a union of irreducible components of $X_{\mathrm{tri}}(\overline{\rho}_p)\times\fX_{\overline{\rho}^p}\times \mathbb{U}^g$ where $X_{\mathrm{tri}}(\overline{\rho}_p):=\prod_{v\in S_p}X_{\mathrm{tri}}(\overline{\rho}_v)$, $\mathbb{U}^g$ is an open polydisk and $\fX_{\overline{\rho}^p}$ is certain tame part. \par
Let $x=(\rho,\underline{\delta})\in Y(U^p,\overline{\rho})\subset X_p(\overline{\rho})$ be a generic crystalline point as in Theorem \ref{theoremintroductionmain}. Let $G:=\prod_{v\in S_p}\mathrm{Res}_{F_v^{+}/\Q_p}(\GL_{n/F_v^+})\otimes_{\Q_p}L$ and let $W\simeq \prod_{v\in S_p,\tau\in\Sigma_v} \mathcal{S}_n$ be its Weyl group. A companion character in $W_{\cR}(\rho)$ is certain character $\underline{\delta}_{\cR,w}$ for some $w\in W$ (Definition \ref{definitioncharacterrefinementglobal}) with weight $ww_0\cdot \lambda$. Here $\lambda=(\lambda_{\tau})_{\tau\in \Sigma_v,v\in S_p}\in\prod_{v\in S_p,\tau\in \Sigma_v}\Z^{n}$ is a ``dominant'' weight in the sense that $\lambda_{\tau,i}-\lambda_{\tau,i+1}\geq -1$ for all $\tau$ and $1\leq i\leq n-1$ which is determined by the Hodge-Tate weights of $\rho$, $w_0$ is the longest element in $W$ and $ww_0\cdot\lambda$ denotes the usual dot action. In \cite{breuil2019local}, for each companion constituent $\cL(ww_{0}\cdot \lambda,\underline{\delta}_{\cR,w})$, there is an associated cycle $[\mathcal{L}(ww_0\cdot \lambda)]$ on $X_p(\overline{\rho})$ in the infinitesimal neighbourhood of $x$ such that $ [\mathcal{L}(ww_0\cdot \lambda)]\neq \emptyset$ if and only if $\Hom_{G_p}\left( \cL(ww_{0}\cdot \lambda,\underline{\delta}_{\cR,w}),\Pi(\rho)^{\mathrm{an}}\right)\neq 0$. \par
The idea of \cite{breuil2019local} is to compare the cycles $[\mathcal{L}(ww_0\cdot \lambda)]$ with the cycles pulled back from \emph{Steinberg varieties} via the theory of local models for $X_{\mathrm{tri}}(\overline{\rho}_p)$. Let $P=\prod_{v\in S_p}P_v$ be the standard parabolic subgroup of $G$ where each $P_v$ is the parabolic subgroup determined by the Hodge filtration of $\rho_v$ as in 
Theorem \ref{theoremintroductionlocalmodel} and let $B$ be the standard Borel subgroup of $G$. Let $\fg,\fp,\fb$ be the Lie algebras as before and let $\fu$ be the nilpotent radical of $\fb$. The (generalized) Steinberg variety
\[Z_P:=\left\{\left(\nu,g_1B,g_2P\right)\in \fg\times G/B\times G/P\mid \Ad\left(g_1^{-1}\right)\nu\in\fu, \Ad\left(g_2^{-1}\right)\nu\in\fp\right\}\]
is a subvariety of $X_P$. Let $W_P$ be the Weyl group of the standard Levi subgroup of $P$. Then any $w\in W_P$ fixes $w_0\cdot\lambda$ under the dot action. The irreducible components $Z_{P,w}$ of $Z_{P}$ are also parameterized by cosets $w\in W/W_P$ (see \S\ref{sectionsteinbergvarieties} for details). Pulling back each $Z_{P,w}$ defines a cycle $\mathfrak{Z}_{P,w}$ on $X_p(\overline{\rho})$. The spirit of \cite{breuil2019local} expects that $\mathfrak{Z}_{P,w}\subset [\mathcal{L}(ww_0\cdot\lambda)]$.\par 
Let $\fq=\fm_Q+\fn_Q\subset \fg$ be the Lie algebra of a standard parabolic subgroup $Q$ of upper-triangular block matrices of $G$ where $\fn_{Q}$ is the nilpotent radical and $\fm_Q$ is the Levi factor of diagonal block matrices. Recall that partial de Rhamness means the vanishing of the nilpotent operator on the graded pieces of the paraboline sub-filtration which, in the notion of local models (\ref{equationintroductionXP}), is translated to that the entries of the upper-triangular matrix $\Ad(g_1^{-1})\nu$ in certain Levi diagonal blocks are zero. Hence if $ww_0\cdot\lambda$ is a dominant weight for $\fm_Q$, then Theorem \ref{theoremintroductionpartiallyclassical} implies that the cycle $[\mathcal{L}(ww_0\cdot\lambda)]$ is contained in the locus pulled back from the subspace of $Z_P$ cut out by the condition $\Ad(g_1^{-1})\nu\in\fn_Q$. The following elementary result for which we state as a theorem is the counterpart on the local models.
\begin{theorem}[Theorem \ref{theoremvarietyweight}]\label{theoremintroductionsteinberg}
    For each $w\in W/W_P$, the irreducible component $Z_{P,w}$ is contained in the subspace of $Z_P$ where $\Ad(g_1^{-1})\nu\in \fn_Q$ if and only if $ww_0\cdot\lambda$ is a dominant weight for $\fm_Q$.
\end{theorem}
\begin{remark}
    When $P=B$, we can replace the irreducible component $Z_{P,w}$ in Theorem \ref{theoremintroductionsteinberg} by the characteristic variety associated with the $G$-equivariant $\mathcal{D}$-module of the localization of the irreducible $U(\fg)$-module $L(ww_0\cdot 0)$ of the highest weight $ww_0\cdot 0$ (Proposition \ref{propositioncharacteristiccycle}). This will give a more conceptual proof of the ``if'' part of the theorem. However, we do not need characteristic cycles in contrast to \cite{breuil2019local} (our new argument will be simpler than that in \textit{loc. cit.}, even for regular cases). Moreover, it is the ``only if'' part that will play a role. 
\end{remark}
We illustrate how Theorem \ref{theoremintroductionsteinberg} works in the proof of Theorem \ref{theoremintroductionmain} and the difference between regular and non-regular cases by the following basic example. 
\begin{example}\label{exampleintroduction}
    We assume $n=3$, $x=(\rho,\underline{\delta})\in Y(U^p,\overline{\rho})$ and that $\underline{\delta}=\underline{\delta}_{\cR,w_0}$ has weight $\lambda$ which is ``dominant''. Take $\tau_0\in\Sigma_{v_0}, v_0\in S_p$. Assume $\lambda_{\tau_0,1}\geq\lambda_{\tau_0,2}$ and that for any $\tau\neq \tau_0$, $\lambda_{\tau,1}\geq\lambda_{\tau,2}\geq \lambda_{\tau,3}$. Suppose that we are in the case when $W_{\cR}(\rho)=\{\underline{\delta}_{\cR,w_0},\underline{\delta}_{\cR,sw_0}\}$ where $s=(s_{\tau})_{\tau\in\Sigma_v,v\in S_p}$ is a simple reflection such that $s_{\tau}\cdot\lambda_{\tau}=\lambda_{\tau}$ if $\tau\neq \tau_0$ and $s_{\tau_0}\cdot (\lambda_{\tau_0,1},\lambda_{\tau_0,2},\lambda_{\tau_0,3})=(\lambda_{\tau_0,2}-1,\lambda_{\tau_0,1}+1,\lambda_{\tau_0,3})$. Then there is an equality of the underlying closed subspaces of cycles near $x$:
\begin{equation}\label{equationintroductioncycles}
    [\mathcal{L}(w_0w_0\cdot \lambda)]\cup [\mathcal{L}(sw_0w_0\cdot\lambda)]= \mathfrak{Z}_{P,w_0}\cup \mathfrak{Z}_{P,sw_0}
\end{equation}
where both sides describe the fibers of the infinitesimal neighbourhood of $x$ over the weight $w_0w_0\cdot \lambda$.  The left-hand side of (\ref{equationintroductioncycles}) comes from the construction of the eigenvarietiy using $J_{B_p}(-)$, and the knowledge of possible companion constituents for $\rho$ in the situation.  The right-hand side is provided by the local model where both $\mathfrak{Z}_{P,w_0}$ and $\mathfrak{Z}_{P,sw_0}$ are non-empty. By methods in \cite{breuil2019local} and Theorem \ref{theoremintroductionlocalmodel} we know $[\mathcal{L}(w_0w_0\cdot \lambda)]\neq \emptyset$ (Proposition \ref{propositionsocleappear}). We need to prove that $ [\mathcal{L}(sw_0w_0\cdot\lambda)]\neq \emptyset$ which will imply $(\rho,\underline{\delta}_{\cR,sw_0})\in Y(U^p,\overline{\rho})$ and $\Hom_{G_p}\left( \cL(sw_{0}w_{0}\cdot \lambda,\underline{\delta}_{\cR,sw_0}),\Pi(\rho)^{\mathrm{an}}\right)\neq 0$.\par
Firstly assume that the Hodge-Tate weights of $\rho$ are regular. In this case $\lambda_{\tau,1}\geq \lambda_{\tau,2}\geq \lambda_{\tau,3}$ for all $\tau$. Hence $\lambda$ is a dominant weight. The locally analytic representation of the form $\cL(\lambda,\delta_{\cR,w_0})$ is locally algebraic and the cycle $[\mathcal{L}(w_0w_0\cdot \lambda)]$ is then contained in the de Rham locus. However, $Z_{P,w_0}$ is equal to the locus where $\nu=0$ in $Z_P=Z_B$. From which we get $[\mathcal{L}(w_0w_0\cdot \lambda)]\subset \mathfrak{Z}_{P,w_0}$ and $\mathfrak{Z}_{P,sw_0}\nsubseteq [\mathcal{L}(w_0w_0\cdot\lambda)]$. Hence $[\mathcal{L}(sw_0w_0\cdot\lambda)]\neq \emptyset$ by (\ref{equationintroductioncycles}). This is the strategy used in \cite{breuil2019local} for such situation.\par
Now assume that the $\tau_0$-Hodge-Tate weights of $\rho_{v_0}$ are not regular and are equal to $(2,1,1)$ so that $\lambda_{\tau_0}=(2,2,3)$ and $s_{\tau_0}\cdot \lambda_{\tau_0}=(1,3,3)$. Theorem \ref{theoremintroductionpartiallyclassical} implies that the cycle $[\mathcal{L}(w_0w_0\cdot \lambda)]$ is $\{\tau_0\}$-partially de Rham with respect to the standard Levi subalgebra $\mathfrak{gl}_2\times\mathfrak{gl}_1$ of $\mathfrak{gl}_3$. Since $(1,3,3)$ is not dominant with respect to $\mathfrak{gl}_2\times\mathfrak{gl}_1$, Theorem \ref{theoremintroductionsteinberg} tells that the cycle $\mathfrak{Z}_{P,sw_0}$ is not fully contained in the $\{\tau_0\}$-partially de Rham locus (with respect to $\mathfrak{gl}_2\times\mathfrak{gl}_1$). Hence $\mathfrak{Z}_{P,sw_0}\nsubseteq[\mathcal{L}(w_0w_0\cdot \lambda)]$ which forces $[\mathcal{L}(sw_0w_0\cdot\lambda)]\neq \emptyset$ by (\ref{equationintroductioncycles}).
\end{example}
\begin{remark}
    The above strategy also allows obtaining certain companion points or constituents for non-de Rham trianguline representations. Theorem \ref{theoremintroductionsteinberg} suggests a partial classicality conjecture (a converse of Theorem \ref{theoremintroductionpartiallyclassical}) for almost de Rham representations with regular Hodge-Tate weights which is closely related to the locally analytic socle conjecture (Proposition \ref{propositiontwoconjecture}).
\end{remark}
\subsection{Outline of the paper}
We give a brief overview of the contents of the paper.\par
\S\ref{sectionunibranchness} studies the varieties appearing for the local models. The unibranch property of the local models is proved in \S\ref{sectionconnectedness} (Theorem \ref{theoremunibranchYP}). \S\ref{sectionsteinbergvarieties} contains the results on the generalized Steinberg varieties (Theorem \ref{theoremvarietyweight}). \S\ref{sectioncharacteristiccycles} is a complement of \S\ref{sectionsteinbergvarieties} to provide a point of view from geometric representation theory.\par
\S\ref{sectionlocalmodeltriangulinevariety} establishes the local models for the trianguline variety in the non-regular cases. This part follows closely with \cite{breuil2019local} in the regular cases. The first sections are devoted to recall and generalize the deformation theory of trianguline $(\varphi,\Gamma)$-modules. \S\ref{sectioncycles} transports the results on the Steinberg varieties in \S\ref{sectionsteinbergvarieties} to the trianguline variety via the local models.\par
\S\ref{sectionapplocations} contains our main results on companion points and constituents and their proofs. \S\ref{sectionlocalcompanionpoints} concerns the existence of local companion points on the trianguline variety. \S\ref{sectionglobalsettings} is to recall the global settings and \S\ref{sectionOrlikStrauch} recalls the theory of locally analytic representations and the definition of companion constituents. \S\ref{sectionthesocleconjection} is the core part where we prove the main theorems (Theorem \ref{theoremmaincrystalline} and Theorem \ref{theoremmainautomorphicforms}). The key induction step is Proposition \ref{propositionmaincycle} which uses the results on Steinberg varieties in \S\ref{sectionsteinbergvarieties} and Theorem \ref{theoremcyclepartialderham}. The proof of Theorem \ref{theoremcyclepartialderham} absorbs the results, postponed in \S\ref{sectionthepartialeigenvariety}, on the partially classical companion constituents. \par
\S\ref{sectionthepartialeigenvariety} concerns the partially classical families and the partial classicality. A large effort (\S\ref{sectionthepartialemertonjacquetmodule}-\S\ref{sectiondesityofclassicalpoints}) is for the construction of the partial eigenvarieties and studying their basic properties where most results have been obtained by Ding. We adapt his method in our setting to get the patched and more partial versions. The aim is to prove Theorem \ref{theoremQderham} in \S\ref{sectionpartiallyderhamtrianguline} on the partial de Rhamness of partially classical constituents which have been used in \S\ref{sectionthesocleconjection}. \S\ref{sectionconjectures} discusses the conjecture on partial classicality and several results for almost de Rham representations.\par
In Appendix \ref{sectionfamiliesofalmostdeRhamrepresetnations}, we generalize certain result of Berger-Colmez in \cite{berger2008familles} on de Rham families of Galois representations to almost de Rham families of $(\varphi,\Gamma)$-modules (Proposition \ref{propositionfamilydpdr}). The result we get is stronger than what we need (in the proof of Proposition \ref{propositionpartialvarietypartialderham} and Theorem \ref{theoremcyclepartialderham}) and is possibly known to experts. We include a proof as it might be of use in the future.
\subsection{Acknowledgments}
This is part of the author's PhD thesis. I would like to express my sincere gratitude to my advisor Benjamin Schraen for introducing me to the subject, for many helpful discussions and suggestions, and for his reading and comments on earlier drafts of this paper. I would like to thank Simon Riche for answering my questions and helpful suggestions and thank Yiwen Ding for answering my questions. Part of the work in this paper was presented at the Paris-London Number Theory Seminar (online) in November 2020, and I would like to thank the organizers for their invitation. This work was supported by Ecole Doctorale de Math\'ematique Hadamard (EDMH). I would like to thank the anonymous referee for helpful comments and corrections.
\subsection{Notation}\label{sectionintroductionnotation}
\subsubsection{Reductive groups}
Let $G$ denote a connected split reductive group over a field $k$ with a maximal torus $T$, a Borel subgroup $B$ containing $T$ and the Levi decomposition $B=TU$. Write $\overline{B}$ for the opposite Borel subgroup and $\overline{U}$ for the unipotent radical of $\overline{B}$. Write $R^+$ (resp. $R$, resp. $R^-$) for the set of all positive roots with respect to $B$ (resp. roots, resp. negative roots with respect to $B$) of $G$. Write $\Delta$ for the set of positive simple roots. For a root $\alpha$, denote by $\alpha^{\vee}$ the corresponding coroot. We have in particular $\langle \alpha^{\vee},\alpha\rangle=2$, where $\langle-,-\rangle$ denotes the pairing between the lattices of coweights $X_{*}(T)$ and weights $X^*(T)$ of $T$. Write $W$ for the Weyl group of $G$ and for $\alpha\in R$, denote by $s_{\alpha}\in W$ for the corresponding reflection. Let $S=\{s_{\alpha}\mid \alpha\in\Delta\}$ be the set of simple reflections. For every $w\in W$, we fix an element $\dw\in N_G(T)(k)$ that is sent to $w$ via the isomorphism $N_G(T)/T\simeq W$ where $N_G(T)$ denotes the normalizer of $T$ in $W$. We have $\langle \alpha,s_{\beta}(\mu) \rangle=\langle s_{\beta}(\alpha),\mu \rangle$ for $\alpha\in R,\beta \in \Delta,\mu\in X_{*}(T)$.\par

We use fraktur letters $\fg$ (resp. $\fb$, resp. $\fp$, resp. $\ft$, resp. $\mathfrak{u}$, resp. $\overline{\fu}$, etc.) for the Lie algebra of $G$ (resp. $B$, resp. $P$, resp. $T$, resp. $U$, resp. $\overline{U}$, etc.). Denote by $\Ad:G\rightarrow \mathrm{End}(\fg)$ the adjoint representation. For a Lie algebra $\fg$, denote by $U(\fg)$ the universal enveloping algebra.

If $P$ is a standard parabolic subgroup of $G$ containing $B$, let $P=M_PN_P$ be the standard  Levi decomposition, where $M_P$ is the standard Levi subgroup containing $T$. Let $B_{M_P}=B\cap M_P$ and $U_{M_P}=U\cap M_P$. Let $R_P\subset R$ be the set of roots of $M_P$ and let $R_P^+=R^+\cap R_P,R_P^-=R^-\cap R_P,\Delta_P=\Delta\cap R_P$. Let $\fm_P$ (resp. $\fn_P$, resp. $\fb_{M_P}$, resp. $\fu_{M_P}$) be the Lie algebra of $M_P$ (resp. $N_P$, resp. $B_{M_P}$, resp. $U_{M_P}$). In particular, $\fn_B=\mathfrak{u}$.

Write $W_P$ for the Weyl group of $M_P$. Let $\lg(-)$ denote the length of elements in $W$ with respect to the set of simple reflections in $S$. We use the symbols $\leq,\geq,<,>$ to denote the strong Bruhat order (resp. partial Bruhat order) on $W$ (resp. $W/W_P$) with respect to the Coxeter system $(W,S)$ \cite[\S2.1, \S2.5]{bjorner2006combinatorics}. Write $W^P$ for the set of elements $w\in W$ that are the unique shortest elements in the cosets $wW_P$ (cf. \cite[\S2.4]{bjorner2006combinatorics}). Then $W=W^PW_P$. If $w\in W$, let $w=w^{P}w_P$ be the unique decomposition such that $w^P\in W^{P},w_P\in W_P$ (\cite[Prop. 2.4.4]{bjorner2006combinatorics}). The map $W\rightarrow W^P: w\mapsto w^P$ is order preserving (\cite[Prop. 2.5.1]{bjorner2006combinatorics}) and the partial order on $W/W_P$ is induced by the order on $W$ via the bijection $W^P\leftrightarrow W/W_P$. For $w\in W/W_P$, let $\mathrm{lg}_P(w):=\mathrm{lg}(w^P)$ where $w^P\in W^P\cap wW_P$. When it is clear from the context, for $w\in W$, we use the same notation $w$ to denote the coset $wW_P\in W/W_P$. Write $w_0$ (resp. $w_{P,0}$) for the longest element in $W$ (resp. $W_P$). 

We write $\overline{BwP/P}$ for the Schubert variety in the flag variety $G/P$ corresponding to $w\in W/W_P$. It is the closure of the Schubert cell $BwP/P$ in $G/P$ (cf. \cite[II.13.8]{jantzen2007representations}).\par

A weight $\lambda\in X^{*}(T)$ which is also viewed as a weight of $\ft$ is said to be a dominant (resp. antidominant) weight for a standard Levi subgroup $M_P$ or its Lie algebra $\fm_P$ (with respect to $B_{M_P}$ or $\fb_{M_P}$) if $\langle \alpha^{\vee}, \lambda \rangle\geq 0$ (resp. $\langle \alpha^{\vee}, \lambda \rangle\leq 0$) for all $\alpha\in \Delta_P$.\par
The dot action is given by $w\cdot\lambda=w(\lambda+\rho)-\rho$ for all $w\in W$ and $\lambda\in X^{*}(T)$ where $\rho$ is the half sum of all positive roots.
\subsubsection{Local fields}
Let $K$ be a finite extension of $\Q_p$ with a uniformizer $\varpi_K$. Write $\cO_K$ for the ring of integers of $K$ and $k_K$ for the residue field. Let $\overline{K}$ be an algebraic closure of $K$ and $C$ be the completion of $\overline{K}$. Let $K_0$ be the maximal unramified subfield of $K$. Write $\cG_K:=\Gal(\overline{K}/K)$ for its Galois group. Let $\BdR^{+}=\BdR^+(C),\BdR=\BdR^{+}[\frac{1}{t}]$ be Fontaine's de Rham period rings, where $t$ is Fontaine's $2\pi i$. Let $K(\mu_{\infty})$ be the extension of $K$ by adding all $p$-th power roots of unity, and we define $\Gamma_K:=\Gal(K(\mu_{\infty})/K)$.\par
Take $L$ a finite extension of $\Q_p$ that splits $K$. Let $\cC_L$ denote the category of commutative local Artinian $L$-algebras with residue field $L$. If $A\in\cC_L$, let $\fm_A$ be its maximal ideal and the tensor product $-\otimes_AL$ is always with respect to the map modulo $\fm_A$. Let $\Sigma$ be the set of embeddings $\tau:K\hookrightarrow L$. \par
Write $\cT$ for the $\Q_p$-rigid analytic space parametrizing continuous characters of $K^{\times}$ (cf. \cite[Exam. 6.1.5]{kedlaya2014cohomology}) and $\cT_L=\cT\times_{\Q_p}L$. If $A$ is an affinoid $L$-algebra, and $\delta:K^{\times}\rightarrow A^{\times}$ is a continuous character, i.e. locally $\Q_p$-analytic, then define the weight $\mathrm{wt}(\delta)\in\mathrm{Hom}_{\Q_p}(K,A):x\mapsto \frac{d}{dt}\delta(\exp(tx))\mid_{t=0}$. We identify $\mathrm{Hom}_{\Q_p}(K,A)$ with $K\otimes_{\Q_p}A$ via the trace pairing of $K$ and let $\mathrm{wt}_{\tau}(\delta)$ be the $\tau$-part of $\mathrm{wt}(\delta)\in K\otimes_{\Q_p}A=\prod_{\tau\in\Sigma}A$. \par
Let $\epsilon$ be the cyclotomic character of $\cG_K$ and we still use $\epsilon$ to denote the character $\mathrm{N}_{K/\Q_p}|\mathrm{N}_{K/\Q_p}|_{\Q_p}$ of $K^{\times}$ where $\mathrm{N}_{K/\Q_p}$ is the norm map and $|-|_{\Q_p}$ is the standard valuation of $\Q_p$. For any $a\in L^{\times}$, let $\mathrm{unr}(a):K^{\times}\rightarrow L^{\times}$ be the unramified character sending $\varpi_K$ to $a$. If $\mathbf{k}=(k_{\tau})_{\tau\in\Sigma}\in\Z^{\Sigma}$, write $z^{\mathbf{k}}$ for the character $K^{\times}\rightarrow L^{\times}:z\mapsto \prod_{\tau\in\Sigma}\tau(z)^{k_{\tau}}$. If $\underline{\delta}=(\delta_i)_{i\in I}:(K^{\times})^{I}\rightarrow A^{\times}$ is a continuous character of $(K^{\times})^I$ for a finite set $I$, we write $\wt(\underline{\delta}):=(\wt(\delta_i))_{i\in I}\in (K\otimes_{\Q_p}A)^{I}$ and similarly for $\wt_{\tau}(\underline{\delta})$. If $A$ is a finite local $L$-algebra and $\delta:K^{\times}\rightarrow A^{\times}$, then we say $\delta$ is ($\Q_p$-)algebraic (resp. locally ($\Q_p$-)algebraic, resp. smooth) if $\delta=z^{\mathbf{k}}$ for some $\mathbf{k}\in\Z^{\Sigma}$ (resp. $\wt_{\tau}(\delta)\in\Z\subset A,\forall\tau\in\Sigma$, resp. $\wt(\delta)=0$). We say $\underline{\delta}:(K^{\times})^I\rightarrow A^{\times}$ is ($\Q_p$-)algebraic (resp. locally ($\Q_p$-)algebraic, resp. smooth) if $\delta_i$ is ($\Q_p$-)algebraic (resp. locally ($\Q_p$-)algebraic, resp. smooth) for all $i\in I$.\par
If $X$ is a rigid space, we write $\cR_{X,K}$ for the Robba ring of $K$ over $X$ (\cite[Def. 6.2.1]{kedlaya2014cohomology}, our notation follows \cite{breuil2019local}) and if $A$ is an affinoid algebra, write $\cR_{A,K}:=\cR_{\mathrm{Sp}(A),K}$. If $\delta:K^{\times}\rightarrow\Gamma(X, \cO_X)^{\times}$ is a continuous character, let $\cR_{X,K}(\delta)$ (or $\cR_{A,K}(\delta)$ if $X=\Sp(A)$) be the rank one $(\varphi,\Gamma_K)$-module over $\cR_{X,K}$ constructed in \cite[Cons. 6.2.4]{kedlaya2014cohomology}. If $D_X$ is a $(\varphi,\Gamma_K)$-module over $\cR_{X,K}$, set $D_X(\delta):=D_X\otimes_{\cR_{X,K}}\cR_{X,K}(\delta)$.
\subsubsection{Miscellaneous}
For a positive integer $n$, write $\mathcal{S}_n$ for the $n$-th symmetric group.\par
If $x$ is a point on $X$, a scheme locally of finite type over a field or a rigid analytic variety, then we denote by $k(x)$ the residue field at $x$. Write $X^{\mathrm{red}}$ for the underlying reduced subspace. 

If $X$ is a scheme locally of finite type over a finite extension $L$ of $\Q_p$, then we write $X^{\mathrm{rig}}$ for its rigid analytification (\cite[\S5.4]{bosch2014lectures}). If $R$ is a commutative Noetherian complete local ring over $\cO_L$ of residue field finite over $k_L$, then we denote by $\Spf(R)$ the formal scheme defined by $R$ with its maximal ideal and we write $\Spf(R)^{\mathrm{rig}}$ for its rigid generic fiber in the sense of Berthelot (\cite[\S7]{de1995crystalline}). 

If $R$ is a commutative ring, then write $R^{\mathrm{red}}:=R/J$ for its nilreduction where $J$ is the nilradical of $R$. If $R$ is a commutative Noetherian local ring, denote by $\widehat{R}$ the completion of $R$ with respect to the maximal ideal of $R$. 

If $Z$ is a topologically finitely generated abelian $p$-adic Lie group, then we write $\widehat{Z}$ for the rigid analytic space over $\Q_p$ parameterizing continuous characters of $Z$ (cf. \cite[Prop. 6.1.1]{kedlaya2014cohomology}).

If $\fg$ is a finite-dimensional Lie algebra over a field $k$, then we use the same notation $\fg$ to denote the affine scheme over $k$ such that $\fg(A)=A\otimes_k\fg$ for any commutative $k$-algebra $A$.

If $V$ is a module over a ring $R$ and $I$ is an ideal of $R$, then write $V[I]$ for the subset $\{v\in V\mid av=0,\forall a\in I \}$. If $V$ is a vector space over a field with a linear action of a group $G$, then write $V^G$ for the subspace $\{v\in V\mid gv=v,\forall g\in G\}$. 
\section{Unibranchness}\label{sectionunibranchness}
In this section, we study some generalized version of the varieties built from Grothendieck's simultaneous resolution in \cite{riche2008geometric}, \cite{bezrukavnikov2012affine} and \cite{breuil2019local}. \par
We fix a connected split reductive group $G$ over a field $k$ with characteristic that is very good for $G$ (\cite[Def. VI.1.6]{kiehl2013weil}, we will only need the case when $\mathrm{char}(k)=0$), with a maximal torus $T$, a Borel subgroup $B=TU$ containing $T$ and a standard parabolic subgroup $P=M_PN_P$. Let $\fg$ (resp. $\fb$, resp. $\fp$, resp. $\ft$, etc.) be the Lie algebra of $G$ (resp. $B$, resp. $P$, resp. $T$, etc.). Let $W$ be the Weyl group of $G$.
\subsection{The varieties}\label{sectionthevariety}
We shall define the varieties that we are going to study. Define the following schemes over $k$:
\begin{align*}
X_P&:=\left\{\left(\nu,g_1B,g_2P\right)\in \fg\times G/B\times G/P\mid \Ad\left(g_1^{-1}\right)\nu\in\fb, \Ad\left(g_2^{-1}\right)\nu\in\fp\right\}\\
Y_P&:=\left\{(\nu,gP)\in \fb\times G/P\mid \Ad\left(g^{-1}\right)\nu\in\fp\right\}.
\end{align*}
If $P=B$, then $X_P$ is defined in \cite{breuil2019local} and we denote by $X:=X_B,Y:=Y_B$. The scheme $X_P$ (resp. $Y_P$) is equipped with a left $G$-action (resp. $B$-action) given by $g\left(\nu,g_1B,g_2P\right)=(\Ad\left(g\right)\nu,gg_1B,gg_2P)$ for any $g\in G,\left(\nu,g_1B,g_2P\right)\in X_P$ (resp. $b(\nu,gP)=(\Ad(b)\nu,bgP)$ for any $b\in B,(\nu,gP)\in Y_P$). The morphism $G\times^BY_P\rightarrow X_P$ sending $(g,(\nu,g_1P))$ to $(\Ad\left(g\right)\nu,gB,gg_1P)$ is an isomorphism, where the notation of $G\times^BY_P$ is taken from \cite[I.5.14]{jantzen2007representations}. Let $\overline{U}$ be the opposite unipotent subgroup with respect to $B$. The projection $G\rightarrow G/B$ is locally trivial: $G$ is covered by open subsets of the form $g\overline{U}B,g\in G$ and $g\overline{U}B\simeq \overline{U}\times B$ as varieties. Hence $X_P\simeq G\times^BY_P$ is covered by open subschemes that are isomorphic to $\overline{U}\times Y_P$. Note that $\overline{U}$ is smooth. \par
Suppose $w\in W/W_P$. Let $U_{P,w}=\left\{(g_1B,g_2P)\in G/B\times G/P\mid g_1^{-1}g_2P\in BwP/P\right\}$ be the equivariant partial Schubert cell in $G/B\times G/P$. Then $U_{P,w}$ is a locally closed subscheme in $G/B\times G/P$ of dimension $\dim G-\dim B+ \mathrm{lg}_P(w)$ (\cite[II.13.8]{jantzen2007representations}). We let 
\begin{equation}\label{equationdefinitionVPw}
	V_{P,w}:=\left\{\left(\nu,g_1B,g_2P\right)\in X_P\mid g_1^{-1}g_2P\in BwP \right\}
\end{equation}
be the preimage of $U_{P,w}$ in $X_P$ under the natural projection $X_P\rightarrow G/B\times G/P$ and define $V^Y_{P,w}:=\left\{(\nu,gP)\in \fb\times BwP/P\mid \Ad\left(g^{-1}\right)\nu\in\fp\right\}$ similarly. 
\begin{definition}\label{definitionirreduciblecomponents}
	For every $w\in W/W_P$, let $X_{P,w}$ (resp. $Y_{P,w}$) be the Zariski closure of $V_{P,w}$ in $X_P$ (resp. the Zariski closure of $V^Y_{P,w}$ in $Y_P$) equipped with the reduced induced subscheme structure. When $P=B$, we write $X_{w}$ (resp. $Y_{w}$) for $X_{B,w}$ (resp. $Y_{B,w}$).	
\end{definition}
We define a variety $\widetilde{\fg}_P:=\left\{(\nu,gP)\in\fg\times G/P\mid \Ad\left(g^{-1}\right)\nu\in\fp\right\}\simeq G\times^P\fp$ and denote by $\widetilde{\fg}:=\widetilde{\fg}_B$. The projection to the first factor $q_P:\widetilde{\fg}_P\rightarrow \fg$ is the (partial) Grothendieck simultaneous resolution (\cite[VI.8]{kiehl2013weil}). The scheme $X_P$ is isomorphic to $\widetilde{\fg}\times_{\fg}\widetilde{\fg}_P$ as in \cite{breuil2019local} for $X$.\par
Define $\mathrm{pr}_{P}: X_P=\widetilde{\fg}\times_{\fg}\widetilde{\fg}_P\rightarrow \widetilde{\fg}$ to be the projection to the first factor and $\mathrm{pr}_{P,w}$ to be its restriction to $X_{P,w}$. Similarly, we can define morphisms $\mathrm{pr}^Y_P$ and $\mathrm{pr}^Y_{P,w}$ which send $(\nu,gP)\in Y_P\subset \fb\times G/P$ to $\nu\in \fb$. We let $\fg^{\mathrm{reg}}\subset \fg$ (resp. $\fg^{\mathrm{reg-ss}}$) be the open subscheme of $\fg$ consisting of regular elements which by definition are those elements in $\fg$ whose orbits under the adjoint action of $G$ have the maximal possible dimension (resp. regular semisimple elements). Let $\tildefg^{\mathrm{reg}}:=q_B^{-1}(\fg^{\mathrm{reg}})$ (resp. $\tildefg^{\mathrm{reg-ss}}:=q_B^{-1}(\fg^{\mathrm{reg-ss}})$).
\begin{proposition}\label{propositionXPbasic}
	\begin{enumerate}
		\item The scheme $X_P$ (resp. $Y_P$) is reduced, is a locally complete intersection, hence Cohen-Macaulay, and is equidimensional of dimension $\dim G$ (resp. $\dim B$). Its irreducible components are $X_{P,w}$ (resp. $Y_{P,w}$) for $w\in W/W_P$.
		\item For each $w\in W/W_P$, the morphism $\mathrm{pr}_{P,w}:X_{P,w}\rightarrow \tildefg$ (resp. $\mathrm{pr}^Y_{P,w}:Y_{P,w}\rightarrow \fb$) is proper birational surjective and is an isomorphism over $\tildefg^{\mathrm{reg}}$ (resp. $\fb^{\mathrm{reg}}:=\fb\cap \fg^{\mathrm{reg}}$). 
	\end{enumerate}
\end{proposition}
\begin{proof}
	The proof goes in the same way as that in \cite[\S2.2]{breuil2019local} and we only give a sketch here. The fiber of the projection $V_{P,w}\rightarrow U_{P,w}$ over a point $(g_1B,g_2P)\in U_{P,w}\subset  G/B\times G/P$ is 
	\[\left\{\left(\nu,g_1B,g_2P\right)\in X_{P}\mid \nu\in \Ad(g_1)\fb\cap \Ad(g_2)\fp\right\}\]
	which is isomorphic to $\fb\cap \Ad(\dw)\fp$ as schemes where $\dw\in W^P$ is the shortest element in $wW_P$. The variety $\fb\cap \Ad(\dw)\fp$ is an affine space of dimension $\dim B-\lg_P(w)$ (\cite[Prop. 3.9 (ii)]{borel1972complements} or Lemma \ref{lemmaweylgroupshortestelement} below). Using \cite[Lem. 2.2.2]{breuil2019local} we see that $V_{P,w}$ is a geometric vector bundle over $U_{P,w}$ of total dimension $\dim G$. Hence $X_{P,w}$ is irreducible of dimension $\dim G$ for every $w\in W/W_P$. The scheme $X_P$ is a union of the subsets $X_{P,w},w\in W/W_P$ and is locally cut out by $\dim G-\dim B+\dim G-\dim P$ equations from a smooth variety $\fg\times G/B\times G/P$. Thus $X_P$ is locally of complete intersection, hence Cohen-Macaulay and equidimensional of dimension $\dim G$. The reducedness of $X_P$ and (2) for $X_P$ can be proved by the same arguments in the proof of \cite[Thm. 2.2.6]{breuil2019local} using Lemma \ref{lemmagaloiscovering} below to argue that each $X_{P,w}$ contains one point in the fiber over any point of $\tildefg^{\mathrm{reg-ss}}$ of the map $\mathrm{pr}_P:X_P\rightarrow\widetilde{\fg}$. The proof of results for $Y_P$ is similar or using the results for $X_P$ together with the isomorphism $G\times^BY_P\simeq X_P$.  
\end{proof}
There is a natural proper surjective morphism of schemes $q_{B,P}:\widetilde{\fg}\rightarrow\widetilde{\fg}_P, (\nu,gB)\mapsto (\nu,gP)$. In fact, the surjectivity can be tested over an algebraically closed field and for closed points since the source and the target are both algebraic varieties. For any geometric point $(\nu,g)\in \fg\times G$ such that $\Ad\left(g^{-1}\right)\nu\in \fp$, one can always find an element $h\in M_P$ such that $\Ad(h^{-1}g^{-1})\nu\in\fb$, then the point $(\nu,ghB)\in X$ is sent to $(\nu,gP)\in X_P$ by $q_{B,P}$. Now we have a factorization of $q_B:\tildefg\stackrel{q_{B,P}}{\rightarrow}\widetilde{\fg}_P\stackrel{q_P}{\rightarrow} \fg$. The following lemma is a plain generalization of \cite[Prop. 2.1.1]{breuil2019local}.
\begin{lemma}\label{lemmagaloiscovering}
	The morphism $q_P:\tildefg_P\rightarrow \fg$ is proper and surjective. It is finite over $\fg^{\mathrm{reg}}$ and is \'etale of degree $|W/W_P|$ over $\fg^{\mathrm{reg-ss}}$.
\end{lemma}
\begin{proof}
		The properness of $q_P$ comes from the factorization $q_P: \tildefg_P \hookrightarrow\fg\times G/P\rightarrow \fg$ and the fact that the flag variety $G/P$ is proper. Since $q_B=q_P\circ q_{P,B}$ and $q_B$ is surjective by \cite[Prop. 2.1.1]{breuil2019local}, $q_P$ must be surjective. Since $q_B$ is quasi-finite over $ \fg^{\mathrm{reg}}$, for any point $s\in \fg^{\mathrm{reg}}$, the fiber $q_B^{-1}(s)$ is finite, hence the fiber $q_P^{-1}(s)=q_{B,P}\left(q_B^{-1}(s)\right)$ is finite using the fact that the map $q_{B,P}$ is surjective. Hence $q_{P}$ is quasi-finite over $\fg^{\mathrm{reg}}$ and thus is also finite over $\fg^{\mathrm{reg}}$ since $q_{P}$ is proper. Let $\ft^{\mathrm{reg}}$ be the open subscheme of $\ft$ consisting of regular elements in the Lie algebra $\ft$ of the torus $T$. By the proof of \cite[Thm. VI.9.1]{kiehl2013weil} and the assumption that the characteristic of $k$ is good for $G$, the morphism $\ft^{\mathrm{reg}}\times G/T\rightarrow \tildefg^{\mathrm{reg-ss}}: (t,gT)\mapsto (\Ad\left(g\right)t,gB)$ is an isomorphism. The Weyl group $W$ acts on the right on $\ft^{\mathrm{reg}}\times G/T$ by $\dw(t,gT)=(\Ad(\dw^{-1})t,g\dw T)$ for $w\in W$. Then the composite map $q'_B:\ft^{\mathrm{reg}}\times G/T\stackrel{\sim}{\rightarrow}\tildefg^{\mathrm{reg-ss}}\stackrel{q_B}{\rightarrow} \fg^{\mathrm{reg-ss}}$ is a Galois covering with Galois group $W$. Consider the morphism $q'_{B,P}:\ft^{\mathrm{reg}}\times G/T\stackrel{\sim}{\rightarrow}q^{-1}_B(\fg^{\mathrm{reg-ss}})\stackrel{q_{B,P}}{\rightarrow} q_{P}^{-1}(\fg^{\mathrm{reg-ss}})$. One check that $q'_{B,P}$ factors through $(\ft^{\mathrm{reg}}\times G/T)/W_P$, the \'etale sub-covering of $q'_B$ associated with the subgroup $W_P$. We only need to verify that the induced morphism of varieties $(\ft^{\mathrm{reg}}\times G/T)/W_P\rightarrow q_{P}^{-1}(\fg^{\mathrm{reg-ss}})$ is an isomorphism. We may assume $k$ is algebraically closed. If two $k$-points $(t_1,g_1),(t_2,g_2)\in (\ft^{\mathrm{reg}}\times G)(k)$ are sent to the same point in $\tildefg_P$, then $\Ad(g_1)t_1=\Ad(g_2)t_2$ and $g_1^{-1}g_2\in P(k)$. Since $t_1,t_2$ are regular, their centralizer in $G$ is $T$. Comparing the centralizer of $\Ad(g_1)t_1$ and $\Ad(g_2)t_2$ we get $g_2^{-1}g_1\in N_G(T)(k)\cap P(k)$. Thus the image of $g_2^{-1}g_1$ in the Weyl group lies in $W_P$. Hence the map on $k$-points $\left((\ft^{\mathrm{reg}}\times G/T)/W_P\right)(k)\rightarrow q_{P}^{-1}(\fg^{\mathrm{reg-ss}})(k)$ is a bijection. Now using an infinitesimal argument as in Step 2 and Step 5 of the proof of \cite[Thm. VI.9.1]{kiehl2013weil}, and notice that both $(\ft^{\mathrm{reg}}\times G/T)/W_P$ (being an \'etale covering of the smooth variety $\fg^{\mathrm{reg-ss}}$) and $q_{P}^{-1}(\fg^{\mathrm{reg-ss}})$ (being an open subscheme of the smooth variety $\tildefg_P$) are smooth varieties, we conclude that the map $(\ft^{\mathrm{reg}}\times G/T)/W_P\rightarrow q_{P}^{-1}(\fg^{\mathrm{reg-ss}})$ is an isomorphism.
\end{proof}
We have a surjective proper morphism $p_{P}:X=\widetilde{\fg}\times_{\fg}\widetilde{\fg}\rightarrow X_P=\widetilde{\fg}\times_{\fg}\widetilde{\fg}_P$ from $q_{B,P}:\tildefg\rightarrow\tildefg_P$ by base change. For $w\in W$, we will use the same notation $w$ to denote the image $wW_P$ in $W/W_P$ when it is clear from the context and write $X_{P,w},V_{P,w},U_{P,w}$, etc. for simplicity. The natural morphism $G/B\times G/B\rightarrow G/B\times G/P$ sends $U_{B,w}$ to $U_{P,w}$ for every $w\in W$. Thus the open dense subscheme $V_{B,w}$ of $X_w$ is sent into the open dense subscheme $V_{P,w}$ of $X_{P,w}$ by $p_P$. Since $X_w$ is a reduced closed subscheme of $X$ and $X_{P,w}$ is Zariski closed in $X_{P}$, $p_P$ sends each $X_w$ into $X_{P,w}$. We let $p_{P,w}$ be the restriction of $p_P$ on $X_w$. Then there is a factorization $\mathrm{pr}_{B,w}:X_w\stackrel{p_{P,w}}{\rightarrow} X_{P,w}\stackrel{\mathrm{pr}_{P,w}}{\rightarrow}\tildefg$. Since both $\mathrm{pr}_{B,w}$ and $\mathrm{pr}_{P,w}$ are proper birational morphisms by Proposition \ref{propositionXPbasic}, so is $p_{P,w}$. We have a similarly defined morphism $p_{P,w}^Y:Y_{w}\rightarrow Y_{P,w}$ for every $w\in W$ and we get a sequence of proper birational morphisms $Y_w\stackrel{p_{P,w}^Y}{\rightarrow}Y_{P,w}\stackrel{\mathrm{pr}^Y_{P,w}}{\rightarrow}\fb$. Since $X_{P,w}$ (resp. $Y_{P,w}$) is irreducible, we have
\begin{proposition}\label{propositionsequencebirational}
	For every $w\in W$, the morphism $p_{P,w}: X_w\rightarrow X_{P,w}$ (resp. $p_{P,w}^Y: Y_w\rightarrow Y_{P,w}$) is a proper birational surjection.
\end{proposition} 
\begin{remark}\label{remarkisoopensubvarieties}
	Assume $w\in W^P$, then the map $p_{P,w}:X_{w}\rightarrow X_{P,w}$ (reps. $p_{P,w}^Y:Y_{w}\rightarrow Y_{P,w}$) induces an isomorphism of open subvarieties $V_{B,w}\simrightarrow V_{P,w}$ (resp. $V_{B,w}^Y\simrightarrow V_{P,w}^Y$): if $\nu\in\fb$ and $g=b\dw\in BwP$ satisfy $\Ad\left(g^{-1}\right)\nu\in\fp$, then $\Ad\left(g^{-1}\right)\nu\in \fb$ by Lemma \ref{lemmaweylgroupshortestelement} below. 
\end{remark}
The following lemma is elementary (see \cite[Prop. 3.9 (ii)]{borel1972complements}). But it is the combinatorial reason for several results in this section, therefore we include a proof here.
\begin{lemma}\label{lemmaweylgroupshortestelement}
	Let $w\in W$, then the following statements are equivalent:
	\begin{enumerate}
		\item $w\in W^P$;
		\item If $\nu\in\fb$, then $\Ad(\dw)^{-1}\nu\in \fp$ if and only if $\Ad(\dw)^{-1}\nu\in \fb$;
		\item $\{\alpha\in R^+\mid w(\alpha)\in R^-\}\cap  R_P^+=\emptyset$;
		\item $w(R_P^{+})\subset R^+$;
		\item $\Ad(\dw)(\fm_{P}\cap\mathfrak{u})\subset \mathfrak{u}$.
	\end{enumerate}
	and for any $w\in W$, $\lg_P(w)=|\{\alpha\in R^+\setminus R_P^+\mid w(\alpha)\in R^-\}|$.
\end{lemma}
\begin{proof}
	We have $s_{\alpha}(R^+)=\{-\alpha\}\cup R^+\setminus \{\alpha\}$ for every $\alpha\in\Delta$ (\cite[Lem. 10.2.B]{humphreys2012introduction}). Hence for $\alpha\in\Delta$, $|\{\alpha'\in R^+\mid  ws_{\alpha}(\alpha')\in R^- \}|$ is equal to $|\{\alpha'\in R^+\mid w(\alpha')\in R^-\}|-1$ if and only if $w(\alpha)\in R^-$. Hence $\lg(ws_{\alpha})=\lg(w)-1$ if and only if $w(\alpha)\in R^-$, from which we deduce (3) $\Rightarrow$ (1). Conversely we assume (1). Then $w(\Delta_P)\subset R^+$. Hence (1) $\Rightarrow$ (3). The equivalence between (3), (4), and (5) is trivial. The assertion (2) is equivalent to that $w^{-1}(R^+)\cap R_P^{-}=\emptyset$ or $w(R_P^{-})\cap R^+=\emptyset$ which is just (4) with a minus sign. Now if $w\in W$ and we write $w=w^Pw_P,w_P\in W_P$. Since $w_P(R^+\setminus R_P^+)=R^+\setminus R_P^+$, we get $\lg(w^P)=\{\alpha\in R^+\setminus R_P^+\mid w^P(\alpha)\in R^-\}=\{\alpha\in R^+\setminus R_P^+\mid w(\alpha)\in R^-\}$. 
\end{proof}
We will also need the following lemma.
\begin{lemma}\label{lemmaXpwVpw}
	If $w,w'\in W/W_P$, then $X_{P,w}\cap V_{P,w'}\neq \emptyset$ only if $w\geq w'$ with respect to the Bruhat order on $W/W_P$. 
\end{lemma}
\begin{proof}
	For $w\in W/W_P$, let $\overline{U}_{P,w}$ be the closure of $U_{P,w}$ in $G/B\times G/P$. As $U_{P,w}\simeq G\times^BBwP/P$ under the isomorphism $G/B\times G/P\simrightarrow G\times^BG/P:(g_1B,g_2P)\mapsto (g_1,g_1^{-1}g_2P)$, by \cite[I.5.21 (2)]{jantzen2007representations}, $\overline{U}_{P,w}\simeq G\times^B\overline{BwP/P}$. For $w\in W/W_P$, write $w^P\in wW_P$ for the shortest representative. Then by definition, $w\geq w'$ in $W/W_P$ if and only if $w^P\geq (w')^P$ in $W$. Hence by \cite[II.13.8 (4)]{jantzen2007representations}, $\overline{Bw'P/P}\subset \overline{BwP/P}$ if and only if $w\geq w'$ in $W/W_P$. In particular, $Bw'P/P\subset \overline{BwP/P}$ if and only if $w\geq w'$ in $W/W_P$. Since $\overline{BwP/P}$ is $B$-invariant, we get $\overline{BwP/P}=\cup_{w'\leq w}Bw'P/P$. Hence $\overline{U}_{P,w}=\cup_{w'\leq w}U_{P,w'}$. Thus $X_{P,w}$ is contained in the closed subspace 
	\[\left\{\left(\nu,g_1B,g_2P\right)\in X_P\mid (g_1B,g_2P)\in \overline{U}_{P,w} \right\}=\cup_{w'\leq w} V_{P,w'}\]
	of $X_P$ by definition. As $V_{P,w'},w'\in W/W_P$ are pairwise disjoint, $X_{P,w}\cap V_{P,w'}=\emptyset$ if $w'$ is not $\leq w$ in $W/W_P$.
\end{proof}\subsection{Unibranchness}
We recall the notion of unibranchness. A local ring $R$ is called unibranch if the reduced reduction $R^{\mathrm{red}}$ is a domain and the integral closure $R'$ of $R^{\mathrm{red}}$ in its field of fractions is local. We say that a locally Noetherian scheme $S$ is unibranch at a point $s\in S$ if the local ring $\cO_{S,s}$ is unibranch.
\begin{lemma}
Let $R$ be a reduced excellent Noetherian local ring, then $R$ is unibranch if and only if its completion $\widehat{R}$ with respect to the maximal ideal is irreducible.
\end{lemma}
\begin{proof}
	This is \cite[Sch. 7.8.3(vii)]{grothendieck1965EGAIV2}.
\end{proof}
\begin{proposition}\label{propfiberconnectedunibranch}
	Let $f:Y\rightarrow X$ be a morphism of integral algebraic varieties over a field $k$. Assume that $Y$ is normal and that $f$ is proper and surjective. If $x\in X$ is a point such that the fiber $f^{-1}(x)$ of $f$ over $x$ is connected, then the local ring $\cO_{X,x}$ is unibranch and its completion $\widehat{\cO}_{X,x}$ with respect to the maximal ideal is irreducible.
\end{proposition}
\begin{proof}
	Let $\nu:X^{\nu}\rightarrow X$ be the normalization of $X$. Since $X$ is integral, $X^{\nu}$ is integral and the morphism $\nu$ is a finite surjection (see \cite[\href{https://stacks.math.columbia.edu/tag/035Q}{Tag 035Q}]{stacks-project} and \cite[\href{https://stacks.math.columbia.edu/tag/035S}{Tag 035S}]{stacks-project}). Since $Y$ is normal, there exists a unique factorization $f: Y\stackrel{f'}{\rightarrow}X^{\nu}\stackrel{\nu}{\rightarrow}X$ (see also \cite[\href{https://stacks.math.columbia.edu/tag/035Q}{Tag 035Q}]{stacks-project}). Since $f$ is proper and $\nu$ is finite, $f'$ is also proper (\cite[\href{https://stacks.math.columbia.edu/tag/01W6}{Tag 01W6}]{stacks-project}). The image of $f'$ is then a closed subset of $X^{\nu}$. If $f'$ is not dominant, the generic point of $X^{\nu}$ is not in the image of $f'$, then the generic point of $X$ is not in the image of $f$, which contradicts that $f$ is surjective. Hence $f'$ is surjective.
	We have morphisms $$f^{-1}(x)\stackrel{f'}{\rightarrow} \nu^{-1}(x)\rightarrow \Spec(k(x)),$$
	where $k(x)$ denotes the residue field at $x\in X$.
	Since the morphism $f^{-1}(x)\stackrel{f'}{\rightarrow} \nu^{-1}(x)$ is surjective and $f^{-1}(x)$ is connected, we get that $\nu^{-1}(x)$ is connected. Now assume that $x$ is contained in an affine open subset $\Spec(A)$ of $X$. Then $\nu^{-1}(\Spec(A))=\Spec(A')$ where $A'$ is the integral closure of $A$ in its field of fractions. Suppose that $x$ corresponds to a prime ideal $\fp$ of $A$. Then $\nu^{-1}(x)=\Spec(A'_{\fp}/\fp A'_{\fp})$ where $A'_{\fp}$ is also the normalization of $A_{\fp}$ in its field of fractions (\cite[\href{https://stacks.math.columbia.edu/tag/0307}{Tag 0307}]{stacks-project}). As $A'_{\fp}$ is finite over $A_{\fp}$, the fiber $\Spec(A'_{\fp}/\fp A'_{\fp})$ is finite over $\Spec(k(x))$ and thus is a finite union of discrete points as a topological space. The connectedness of the fiber means that $\Spec(A'_{\fp}/\fp A'_{\fp})$ consists of only one point. Hence there is only one prime ideal $\fp'$ of $A'_{\fp}$ which lies above $\fp$. Since finite morphisms are closed, $\fp'$ is the unique maximal ideal of $A'_{\fp}$ and thus $A'_{\fp}$ is local. Since $\cO_{X,x}$ is excellent, we get $\cO_{X,x}=A_{\fp}$ is unibranch and $\widehat{\cO}_{X,x}$ is irreducible by \cite[Prop. 7.6.1]{grothendieck1965EGAIV2} and \cite[Sch. 7.8.3(vii)]{grothendieck1965EGAIV2}.
\end{proof}
\subsection{Connectedness of fibers over nilpotent elements}\label{sectionconnectedness}
We establish the unibranch property for $X_{P,w}$ (or $Y_{P,w}$) at certain points using Proposition \ref{propfiberconnectedunibranch} and the normality of $X_w$ (\cite[Thm. 2.3.6]{breuil2019local}).\par
Recall by Proposition \ref{propositionsequencebirational}, we have a sequence of birational proper surjective morphisms: $$Y_w\stackrel{p_{P,w}^Y}{\rightarrow}Y_{P,w}\stackrel{\mathrm{pr}^Y_{P,w}}{\rightarrow}\fb.$$
Let $\cN$ be the nilpotent subvariety of $\fg$ consisting of nilpotent elements (cf. \cite[VI.3]{kiehl2013weil}). Then $\mathfrak{u}=\fb\cap \cN$. Denote by $\mathrm{pr}^Y_{w}:=\mathrm{pr}^Y_{B,w}:Y_w\rightarrow \fb$. 
\begin{proposition}\label{propositionsetsovernilpotent}
	If $\nu\in\mathfrak{u}\subset \fb$ is a closed point and $w\in W$, then the closed subset $(\mathrm{pr}_{w}^Y)^{-1}(\nu)$ of $Y_w$ is equal to 
	\[\left\{(\nu,gB)\in \fb\times \overline{BwB/B}\mid \Ad(g^{-1})\nu\in \fb\right\}.\]
\end{proposition}
\begin{proof} 
	The result is equivalent to that $(\mathrm{pr}_w^Y)^{-1}(\mathfrak{u})=\left\{(\nu,gB)\in \mathfrak{u}\times \overline{BwB/B} \mid \Ad(g^{-1})\nu\in \fb\right\}$ as closed subsets of $Y_w$. We have 
	\[G\times^B(\mathrm{pr}_w^Y)^{-1}(\mathfrak{u})=\left\{(\nu,g_1B,g_2B)\in \cN\times G/B\times G/B\mid (\nu,g_1B,g_2B)\in X_w\right\}=\overline{X}_w\] 
	as closed subsets of $X_w$ in the notation of \cite[\S2.4]{breuil2019local}. 
	The Steinberg variety 
	\[Z:=\left\{(\nu,g_1B,g_2B)\in \cN\times G/B\times G/B\mid \Ad\left(g_1^{-1}\right)\nu\in\fu, \Ad\left(g_2^{-1}\right)\nu\in\fu \right\}\] 
	has irreducible components $Z_{w'},w'\in W$ which are the Zariski closures of the following subsets (see \S\ref{sectionsteinbergvarieties} for more details) 
	\[\left\{(\nu,g_1B,g_2B)\in \cN\times G/B\times G/B\mid \Ad\left(g_1^{-1}\right)\nu\in\fu, \Ad\left(g_2^{-1}\right)\nu\in\fu, g_1^{-1}g_2\in Bw'B/B\right\}.\]
	The union of the irreducible components $Z_{w'}$ for $w'\leq w$ of the Steinberg variety is then the closed subset 
	\begin{align*}
		\left\{(\nu,g_1B,g_2B)\in\cN\times G/B\times G/B\mid \Ad\left(g_1^{-1}\right)\nu\in\fu, \Ad\left(g_2^{-1}\right)\nu\in\fu,g_1^{-1}g_2\in\overline{BwB/B}\right\}\\
		=G\times^B\left\{(\nu,gB)\in\mathfrak{u}\times \overline{BwB/B}\mid \Ad(g^{-1})\nu\in \mathfrak{u}\right\},
	\end{align*} 
	by the usual closure relation $\overline{BwB/B}=\cup_{w' \leq w}Bw'B/B$ (\cite[II.13.7]{jantzen2007representations}).
	By \cite[Thm. 2.4.7]{breuil2019local} and the discussion after it, we have $\overline{X}_w=\cup_{w'\leq w}Z_{w'}$. Then we get \[G\times^B(\mathrm{pr}_w^Y)^{-1}(\mathfrak{u})=G\times^B\left\{(\nu,gB)\in\mathfrak{u}\times \overline{BwB/B}\mid \Ad(g^{-1})\nu\in \mathfrak{u}\right\}\] 
	as closed subsets of $X_w$. Hence $(\mathrm{pr}_w^Y)^{-1}(\mathfrak{u})=\left\{(\nu,gB)\in \mathfrak{u}\times \overline{BwB/B} \mid \Ad(g^{-1})\nu\in \mathfrak{u}\right\}$ as closed subsets of $Y_w$. Finally, we have an equality
	\[\left\{(\nu,gB)\in \mathfrak{u}\times \overline{BwB/B} \mid \Ad(g^{-1})\nu\in \mathfrak{u}\right\}=\left\{(\nu,gB)\in \mathfrak{u}\times \overline{BwB/B} \mid \Ad(g^{-1})\nu\in \fb\right\}\]
	of closed subsets since the two sides are both closed subschemes in $Y_w$ and contain same closed points.
\end{proof}
Now if $w\in W$ is the \emph{longest} element in $wW_P$, then $\overline{BwB/B}$ is the preimage of $\overline{BwP/P}$ via the natural projection $G/B\rightarrow G/P$ (cf. \cite[II.13.8 (2)]{jantzen2007representations}). In particular, $\overline{BwB}P=\overline{BwB}$.
\begin{proposition}\label{propositionconnectedfibersnilpotent}
	Assume that $w\in W$ is the \emph{longest} element in $wW_P$. If $x=(\nu_x,g_xP)\in Y_{P,w}\subset \fb\times \overline{ BwP/P}$ is a closed point such that $\nu_x$ is nilpotent (i.e. $\nu_x\in\fu$), then the fiber $(p_{P,w}^Y)^{-1}(x)$ is connected.
\end{proposition}
\begin{proof}
	We have a commutative diagram
	\begin{center}
		\begin{tikzpicture}[scale=1.4]
			\node (A) at (0,0) {$x$};
			\node (B) at (2,0) {$(\mathrm{pr}_{P,w}^Y)^{-1}(\nu_x)$};
			\node (C) at (4,0) {$Y_{P,w}$};
			\node (D) at (6,0) {$\fb\times \overline{BwP/P}$};
			\node (E) at (0,1) {$(p_{P,w}^Y)^{-1}(x)$};
			\node (F) at (2,1) {$(\mathrm{pr}_{w}^Y)^{-1}(\nu_x)$};
			\node (G) at (4,1) {$Y_w$};
			\node (H) at (6,1) {$\fb\times \overline{BwB/B}$};
			\path[right hook ->,font=\scriptsize,>=angle 90]
			(A) edge node[above]{} (B)
			(B) edge node[above]{} (C)
			(C) edge node[above]{} (D)
			(E) edge node[above]{} (F)
			(F) edge node[above]{} (G)
			(G) edge node[above]{} (H);
			\path[->>,font=\scriptsize,>=angle 90]
			(E) edge node[above]{} (A)
			(F) edge node[above]{} (B)
			(G) edge node[above]{} (C)
			(H) edge node[above]{} (D)
			;
			\end{tikzpicture}
	 \end{center}
	where each horizontal arrow is a closed embedding and each vertical arrow is surjective and projective. One sees that the formation of the varieties $Y_{P,w}$ commutes with base change by fields: the formation of the varieties $V^Y_{P,w}$ and $Y_P$ commutes with base change by definition and after base change to a separable closure of $k$, the Zariski closure of $V^Y_{P,w}$ with the reduced structure is still irreducible and descends (cf. \cite[Cor. AG.14.6]{borel2012linear}). And the fiber $k(x)\times_{Y_{P,w}\times_k k(x)}(Y_w\times_k k(x))=k(x)\times_{Y_{P,w}\times_k k(x)}(Y_{P,w}\times_k k(x))\times_{Y_{P,w}}Y_w=k(x)\times_{Y_{P,w}}Y_w$. Thus we may assume $k(x)=k$ by base change.
	The composition $(\mathrm{pr}_{w}^Y)^{-1}(\nu_x)=\left\{(\nu_x,gB)\in \fb\times\overline{BwB/B}\mid \Ad(g^{-1})\nu_x\in \fb\right\}\rightarrow (\mathrm{pr}_{P,w}^Y)^{-1}(\nu_x)\hookrightarrow \nu_x\times \overline{BwP/P}$ can be identified with the morphism 
	\begin{align}\label{equationconnectedfibers}
		\left\{gB\in \overline{BwB/B}\mid \Ad(g^{-1})\nu_x\in \fb\right\}\rightarrow \overline{BwP/P}:gB\mapsto gP,
	\end{align} 
	where we only consider the underlying reduced varieties and have used Proposition \ref{propositionsetsovernilpotent}. To show that the fiber $(p_{P,w}^Y)^{-1}(x)$ is connected, we only need to show that the morphism (\ref{equationconnectedfibers})
	has connected fibers. We pick a closed point $g_xB$ in the fiber of $g_xP$. The fiber over $g_xP$ is 
	$$\left\{gB\in \overline{BwB/B}\mid \Ad(g^{-1})\nu_x\in\fb, g_x^{-1}g\in P/B\right\}\simeq \left\{gB\in P/B \mid \Ad(g^{-1})\left(\Ad (g_x^{-1})\nu_x\right)\in\fb\right\}$$
	since $g\in P/B$ implies that $g_xg\in \overline{BwB}P/B=\overline{BwB/B}$ by the assumption on $w$. 
	To show that the latter is connected, we can assume that $g_x$ is trivial and $\nu_x\in \mathfrak{u}$ by replacing $(\nu_x,g_xB)$ with $(\Ad(g_x^{-1})\nu_x,B)$.
	Assume that $P=M_PN_P$ is the standard Levi decomposition and $\fp=\fm_P+\mathfrak{n}_P$ where $\fm_P$ (resp. $\mathfrak{n}_P$) is the Lie algebra of $M_P$ (resp. $N_P$). Let $B_{M_P}=B\cap M_P$, $\fb_{M_P}$ be its Lie algebra and $\mathfrak{u}_{M_P}$ be the variety of nilpotent elements in $\fb_{M_P}$. We have $P/B\simeq M_P/B_{M_P}$ (\cite[II.1.8 (5)]{jantzen2007representations}). We can decompose $\nu_x=m_x+n_x$ where $m_x\in\mathfrak{u}_{M_P}$ and $n_x\in\mathfrak{n}_P$. Since $\Ad(P)\mathfrak{n}_P\subset \mathfrak{n}_P$, an element $gB_{M_P}\in M_P/B_{M_P}$ satisfies $\Ad(g^{-1})\nu_x\in\fb$ if and only if $\Ad(g^{-1})m_x\in \fb_{M_P}$. Hence there is an isomorphism 
	$$\left\{gB\in P/B \mid \Ad(g^{-1})(\nu_x)\in\fb\right\}\simeq\left\{gB_{M_P}\in M_P/B_{M_P} \mid \Ad(g^{-1})(m_x)\in\fb_{M_P}\right\}.$$
	As a closed subspace of $M_P/B_{M_P}$, this is the Springer fiber: the fiber of the Springer resolution $\widetilde{\cN}_{M_P}\rightarrow \cN_{M_P}$ over the point $m_x\in\cN_{M_P}$ where $\cN_{M_P}$ is the nilpotent variety of $\fm_P$\footnote{The Springer fiber is the reduced subvariety associated with the subscheme $\left\{gB_{M_P}\in M_P/B_{M_P} \mid \Ad(g^{-1})(m_x)\in\mathfrak{u}_{M_P}\right\}$ which shares the same underlying topological space with the subscheme $\left\{gB_{M_P}\in M_P/B_{M_P} \mid \Ad(g^{-1})(m_x)\in\fb_{M_P}\right\}$ since they have the same closed points (cf. \cite[\S1.2]{yun2016lectures}).}. The nilpotent variety is normal and, unlike the Grothendieck resolution, the Springer resolution is birational. Hence Springer fibers are connected by Zariski's main theorem (cf. \cite[Rem. 3.3.26]{Chriss1997RepresentationTA} or \cite[\S1.4.1]{yun2016lectures}). 
\end{proof}
\begin{theorem}\label{theoremunibranchYP}
	If $x=(\nu,g_1B,g_2P)\in X_{P,w}$ (resp. $x=(\nu,gP)\in Y_{P,w}$) is a closed point such that $\nu$ is nilpotent, then the local ring $\cO_{X_{P,w},x}$ (resp. $\cO_{Y_{P,w},x}$) is unibranch and the completion $\widehat{\cO}_{X_{P,w},x}$ (resp. $\widehat{\cO}_{Y_{P,w},x}$) is irreducible.
\end{theorem}
\begin{proof}
	Let $w\in W$ such that $w$ is the longest element in $wW_P$. Consider the surjective birational proper morphism $p_{P,w}:X_w\rightarrow X_{P,w}$ (resp. $p_{P,w}^Y:Y_w\rightarrow Y_{P,w}$) of integral varieties. The fiber $p_{P,w}^{-1}(x)$ (resp. $(p_{P,w}^Y)^{-1}(x)$) is connected by Proposition \ref{propositionconnectedfibersnilpotent} and $X_w$ (resp. $Y_w$) is normal by \cite[Thm. 2.3.6]{breuil2019local}. Hence $X_{P,w}$ (resp. $Y_{P,w}$) is unibranch at $x$ and the completion $\widehat{\cO}_{X_{P,w},x}$ (resp. $\widehat{\cO}_{Y_{P,w},x}$) is irreducible by Proposition \ref{propfiberconnectedunibranch}.
\end{proof}
\begin{remark}\label{remarkunibranchallpoints}
	The above results are true in generality. The assumptions in Proposition \ref{propositionconnectedfibersnilpotent} that $\nu_x$ is nilpotent and $w$ is the longest element in $wW_P$ can be removed. Using a Bott-Samelson-Demazure type resolution for $X_{P,w}$ in \cite[\S1.7]{riche2008geometric}, one can show that for any $w\in W$, the fiber of $p_{P,w}:X_w\rightarrow X_{P,w}$ over any point $x\in X_{P,w}$ is connected. Thus by Proposition \ref{propositionconnectedfibersnilpotent} and \cite[Thm. 2.3.6]{breuil2019local}, $X_{P,w}$ is unibranch at all points. One can see also from the proof of Proposition \ref{propositionconnectedfibersnilpotent} that $X_{P,w}$ is in fact geometrically unibranch (\cite[\S23.2.1]{grothendieck1960EGAIV1}). Moreover, Proposition \ref{propositionsetsovernilpotent} can be proved directly without using \cite[Thm. 2.4.7]{breuil2019local}.
\end{remark}
\begin{corollary}\label{corollaryirreduciblecomponentscompletion}
	If $x=(\nu,g_1B,g_2P)\in X_{P,w}$ is a closed point such that $\nu$ is nilpotent, then the irreducible components of $\Spec(\widehat{\cO}_{X_P,x})$ are $\Spec(\widehat{\cO}_{X_{P,w},x})$ for $w\in W/W_P$ such that $x\in X_{P,w}$.
\end{corollary}
\begin{proof}
	Suppose that $R$ is a reduced local excellent Noetherian ring with minimal prime ideals $\fp_1,\cdots,\fp_m$ such that every $R/\fp_{i}$ is unibranch. By definition, the normalization $R'$ of $R$ is a product of $(R/\fp_i)'$, the normalizations of $R/\fp_i$. Since $R/\fp_i$ is unibranch, $(R/\fp_i)'$ is local and thus the number of maximal ideals of $R'$ is $m$. By \cite[Sch. 7.8.3(vii)]{grothendieck1965EGAIV2}, minimal prime ideals of $\widehat{R}$ correspond bijectively to maximal ideals of $R'$. Hence there are exactly $m$ minimal ideals of $\widehat{R}$. Since the quotients $\widehat{R}/\fp_i\widehat{R}=\widehat{R/\fp_i}$ of $\widehat{R}$ are integral, they correspond to all irreducible components of $\Spec(\widehat{R})$.
\end{proof}
\subsection{The weight map}\label{subsectiontheweightmapthevariety}
We prove some results for the weight map of $X_{P,w}$. Since the characteristic of $k$ is very good for $G$, the ring morphisms $S(\fg^{*})\hookleftarrow S(\fg^{*})^{G}\simrightarrow S(\ft^{*})^W$ induce a morphism $\gamma_G:\fg\rightarrow \ft/W$ of $k$-schemes (\cite[VI.8]{kiehl2013weil}). Applying this fact to the standard Levi subgroup $M_P$ of $P$ we get a map $\gamma_{M_P}:\fm_P\rightarrow \ft/W_P$. We define a map $\kappa_{P}:\tildefg_P\rightarrow \ft/W_P:(\nu,gP)\mapsto \gamma_{M_P}\left(\overline{\Ad(g^{-1})\nu}\right)$ where $\overline{\Ad(g^{-1})\nu}$ denotes the image of $\Ad(g^{-1})\nu$ in $\fm_P$ under the projection $\fp\twoheadrightarrow\fp/\mathfrak{n}_P\simrightarrow \fm_P$. Let $\kappa_1: X_P\rightarrow \ft$ be the map sending $(\nu,g_1B,g_2P)$ to the image of $\Ad(g_1^{-1})\nu$ in $\ft=\fb/\mathfrak{u}$ and $\kappa_2$ be the composition $X_P=\tildefg\times_{\fg}\tildefg_P\rightarrow \tildefg_P\stackrel{\kappa_P}{\rightarrow}\ft/W_P$. We have the following commutative diagram
\begin{center}
	\begin{tikzpicture}[scale=1.3]
		\node (A) at (0,0) {$\ft$};
		\node (B) at (2,0) {$\ft/W_P$};
		\node (C) at (4,0) {$\ft/W$};
		\node (D) at (0,1) {$\tildefg$};
		\node (E) at (2,1) {$\tildefg_P$};
		\node (F) at (4,1) {$\fg$};
		\path[->,font=\scriptsize,>=angle 90]
		(A) edge node[above]{} (B)
		(B) edge node[above]{} (C)
		(D) edge node[above]{} (E)
		(E) edge node[above]{} (F)
		(D) edge node[left]{$\kappa_{B}$} (A)
		(E) edge node[left]{$\kappa_{P}$} (B)
		(F) edge node[left]{$\kappa_{G}=\gamma_G$} (C)
		;
		\end{tikzpicture}
\end{center}
where the horizontal arrows are natural projections.
For $i=1,2,w\in W/W_P$, let $\kappa_{i,w}$ be the restriction of $\kappa_i$ to the closed subscheme $X_{P,w}$.
\begin{lemma}\label{lemmaweightmap}
	We have the following commutative diagram 
	\begin{center}
		\begin{tikzpicture}[scale=1.3]
			\node (A) at (0,0) {$\ft$};
			\node (B) at (2,0) {$\ft/W$};
			\node (C) at (0,2) {$X_{P,w}$};
			\node (D) at (2,2) {$\ft/W_P$};
			\path[->,font=\scriptsize,>=angle 90]
			(A) edge node[above]{} (B)
			(D) edge node[above]{} (B)
			(C) edge node[left]{$\kappa_{1,w}$} (A)
			(C) edge node[above]{$\kappa_{2,w}$} (D)
			(A) edge node[left]{$\alpha$} (D)
			;
			\end{tikzpicture}
	 \end{center}
	where the map $\ft\rightarrow \ft/W$ and $\ft/W_P\rightarrow \ft/W$ are natural projections and the map $\alpha:\ft\rightarrow \ft/W_P$ is the composition of the map $\Ad(\dw^{-1}):\ft\rightarrow\ft$ with the projection $\ft\rightarrow \ft/W_P$ (thus $\alpha$ depends only on the class of $w$ in $W/W_P$).
\end{lemma}
\begin{proof}
	This is a generalization of \cite[Lem. 2.3.4]{breuil2019local}. We only need to show $\kappa_{2,w}=\alpha\circ\kappa_{1,w}$. Since $X_{P,w}$ is the closure of $V_{P,w}=\left\{(\nu,g_1B,g_2P)\in X_P\mid g_1^{-1}g_2\in BwP \right\}$ in $X_P$, we only need to verify $\kappa_{2,w}=\Ad(\dw^{-1})\kappa_{1,w}$ when restricted to $V_{P,w}$. Let $x=(\nu,g_1B,g_2P)\in X_P(S)$ for a $k$-algebra $S$ and by replacing $S$ by some fppf extension, we assume $g_2=g_1\dw\in G(S)$. Then $\Ad(g_1^{-1})\nu\in\fb(S)$ and $\Ad(\dw^{-1})\Ad(g_1^{-1})\nu\in\fp(S)$. We assume $w\in W^P$. Then we have $\Ad(\dw^{-1})\Ad(g_1)^{-1}\nu\in\fb(S)$ (cf. Lemma \ref{lemmaweylgroupshortestelement}). The image of $\Ad(g_2^{-1})\nu$ in $\fm_P(S)=\fp/\mathfrak{n}_P(S)$, denoted by $\overline{\Ad(g_2^{-1})\nu}$, lies in the subset $\fb_{M_P}(S)=\fm_P(S)\cap \fb(S)$. Let $t$ denote the image of $\overline{\Ad(g_2^{-1})\nu}$ in $\ft(S)=\ft_{M_P}(S)=\fb_{M_P}(S)/\mathfrak{u}_{M_P}(S)$. Then $\kappa_2(x)=\gamma_{M_P}\left(\overline{\Ad(g_2^{-1})\nu}\right)$ is the image of $t$ in $(\ft/W_P)(S)$ via the map $\ft\rightarrow \ft/W_P$ (cf. \cite[Thm. VI.8.3]{kiehl2013weil}). We have that $\kappa_{1,w}(x)$ is the image of $\Ad(g_1)^{-1}\nu$ in $\ft(S)=\fb/\mathfrak{u}(S)$, thus $t=\Ad(\dw^{-1})\kappa_{1,w}(x)$. Hence $\kappa_{2,w}=\alpha\circ\kappa_{1,w}$.
\end{proof}
Now let $T_P:=\ft\times_{\ft/W}\ft/W_P$ and for all $w\in W/W_P$, let $T_{P,w}=\left\{(z,\Ad(\dw^{-1})z)\mid z\in\ft\right\}\subset {\ft\times_{\ft/W}\ft/W_P}$ be closed subschemes of $T_P$. Then $T_{P,w}\simeq\ft$ is smooth for any $w\in W/W_P$. Similar to \cite[Lem. 2.5.1]{breuil2019local}, $T_P$ is equidimensional and $\left\{T_{P,w}\mid w\in W/W_P\right\}$ is the set of irreducible components of $T_P$. We have a map $(\kappa_1,\kappa_2):X_P\rightarrow T_P$ and $X_{P,w}$ is the unique irreducible component of $X_P$ that dominates $T_{P,w}$ by Lemma \ref{lemmaweightmap} (the dominance comes from the factorization $\kappa_1:X_{P,w}\stackrel{\mathrm{pr}_{P,w}}{\rightarrow}\widetilde{\fg}\stackrel{\kappa_B}{\rightarrow}\ft\simeq T_{P,w}$ and that $\mathrm{pr}_{P,w}$ is surjective by Proposition \ref{propositionXPbasic}). Suppose that $x=(\nu,g_1B,g_2P)\in X_{P}$ is a closed point such that $\nu$ is nilpotent and let $(0,0)=(\kappa_1(x),\kappa_2(x))\in T_P$. If $x\in X_{P,w}\subset X_P$ for some $w\in W/W_P$, we let $\widehat{X}_{P,x}$ (resp. $\widehat{X}_{P,w,x}$, resp. $\widehat{T}_{P,(0,0)}$, resp. $\widehat{T}_{P,w,(0,0)}$) be the completion of $X_P$ at $x$ (resp. $X_{P,w}$ at $x$, resp. $T_P$ at $(0,0)$, resp. $T_{P,w}$ at $(0,0)$). Since by Theorem \ref{theoremunibranchYP}, the structure ring of $\widehat{X}_{P,w,x}$ is irreducible, using the same argument for \cite[Lem. 2.5.2]{breuil2019local}, we get the following lemma.
\begin{lemma}\label{lemmafactorthroughXPw}
	The map $\widehat{X}_{P,w',x}\hookrightarrow \widehat{X}_{P,x}\rightarrow \widehat{T}_{P,(0,0)}$ induced by the completions of the closed embedding $X_{P,x}\hookrightarrow X_P$ and the map $(\kappa_1,\kappa_2)$ factors through $\widehat{T}_{P,w,(0,0)}\hookrightarrow \widehat{T}_{P,(0,0)}$ if and only if $w'=w$ in $W/W_P$.
\end{lemma}
\subsection{Generalized Steinberg varieties}\label{sectionsteinbergvarieties}
We shall study certain vanishing properties of irreducible components of generalized Steinberg varieties which might be well-known from the perspective of geometric representation theory (at least for the case when $P=B$, see \S\ref{sectioncharacteristiccycles}). These vanishing properties will be the major new ingredients in the global applications of the local models for the trianguline variety.\par
We pick a standard parabolic subgroup $Q=M_{Q}N_Q$ of $G$ with Lie algebra $\fq=\fm_{Q}+\fn_{Q}$ and Weyl group $W_Q$. Let ${^QW}$ be the set of elements $w\in W$ such that $w$ is the shortest element in the coset $W_Qw$. We consider the following scheme depending on the choice of the two parabolic subgroups $P$ and $Q$
\[Z_{Q,P}:=\left\{(\nu,g_1B,g_2P)\in \cN\times G/B\times G/P\mid \Ad(g_1^{-1})\nu\in \fn_{Q},\Ad(g_2^{-1})\nu\in \fp\right\}.\]
As there is an isomorphism $\widetilde{\cN}:=\left\{(\nu,g)\in\cN\times G/B \mid\Ad(g^{-1})\nu\in\fb  \right\}\simeq G\times^{B}\mathfrak{u}$, we can replace $\cN$ in the above definition by $\fg$ (cf. \cite[\S3.2]{Chriss1997RepresentationTA}).  
When $Q=B$, $Z_P:=Z_{B,P}$ is some generalized Steinberg variety considered in \cite{douglass2004geometry}. We have a natural closed embedding \[Z_{Q,P}\hookrightarrow Z_{P}\] 
and generally, $Z_{Q',P}\subset Z_{Q,P}$ if $Q\subset Q'$. For any $w\in  W_Q\backslash W/W_P$, we let $Z_{Q,P,w}$ be the Zariski closure of the subset $H_{Q,P,w}:=\left\{(\nu,g_1B,g_2P)\in Z_{Q,P}\mid g_1^{-1}g_2\in QwP/P \right\}$ in $Z_{Q,P}$ with the reduced induced scheme structure. We write $Z_{P,w}:=Z_{B,P,w}$ for every $w\in W/W_P$. There is a unique shortest element $w\in W$ in each double coset $W_QwW_P\in W_Q\backslash W/W_P$ and $w\in W$ is the shortest element in $W_QwW_P$ if and only if $w\in W^P\cap  {^QW}$ (\cite[Prop. 2]{tenner2020parabolic}).
\begin{proposition}\label{propositionsteinbergvarieties}
    \begin{enumerate}
        \item The scheme $Z_{Q,P,w}$ is irreducible and has dimension no more than $\dim G-\dim T$. 
        \item $Z_P$ is equidimensional of dimension $\dim G-\dim T$ with irreducible components $Z_{P,w},w\in W/W_P$.
        \item For any $w\in W_Q\backslash W/ W_P$, the following statements are equivalent:
        \begin{enumerate}[label=(\alph*)]
            \item $Z_{Q,P,w}=Z_{P,w'}$ for some $w'\in W^P$;
            \item $\Ad(\dw)\fm_P \cap \mathfrak{u}\subset \fn_Q$ (this condition is independent of the representative of $w$ in $W$);
            \item if we take a representative $w\in W^P\cap  {^QW}$, then $w_{Q,0}w\in W^P$.
        \end{enumerate}
            And if the above statements hold, $w'W_P=w_{Q,0}wW_P$ where $w\in W^P\cap  {^QW}$.
        \item $Z_{Q,B}$ is equidimensional of dimension $\dim G-\dim T$ with irreducible components $Z_{B,w_{Q,0}w},w\in {^QW}$.    
    \end{enumerate}
\end{proposition}
\begin{proof}
    Take a representative $w\in W$ for $w\in W_Q\setminus W/W_P$ and we write $w$ instead of $\dw$ for simplicity.\par
    (1) Let $\overline{Z}_{Q,P}:=\left\{(\nu,g_1Q,g_2P)\in \cN\times G/Q\times G/P\mid \Ad(g_1^{-1})\nu\in \fn_{Q},\Ad(g_2^{-1})\nu\in \fp\right\}$ be a generalized Steinberg variety in \cite{douglass2004geometry}. Then $Z_{Q,P}=\overline{Z}_{Q,P}\times_{G/Q}G/B$ and the natural morphism $Z_{Q,P}\rightarrow \overline{Z}_{Q,P}$ is a locally trivial fibration of relative dimension $ \dim Q-\dim B$. Let 
    \[\overline{H}_{Q,P,w}:=\left\{(\nu,g_1Q,g_2P)\in \cN\times G/Q\times G/P\mid \Ad(g_1^{-1})\nu\in \fn_{Q},\Ad(g_2^{-1})\nu\in \fp,g_1^{-1}g_2\in QwP/P\right\}\] 
    and let $\overline{Z}_{Q,P,w}$ be the Zariski closure of $\overline{H}_{Q,P,w}$ in $\overline{Z}_{Q,P}$. Then $H_{Q,P,w}=\overline{H}_{Q,P,w}\times_{G/Q}G/B$.\par 
    We work as in Proposition \ref{propositionXPbasic} or \cite[Prop. 2.2.5]{breuil2019local}. The projection $\overline{H}_{Q,P,w}\rightarrow G\cdot(Q, w P)\subset G/Q\times G/P$ is $G$-equivariant (with respect to the diagonal action of $G$ on the double flag variety). The fiber over the point $(Q, wP)$ is the affine space $\fn_{Q}\cap \Ad(w)\fp$. The $G$-orbit $G\cdot(Q, wP)$ is smooth, irreducible of dimension $\dim G-\dim Q\cap wPw^{-1}$. By Lemma \cite[Lem. 2.3]{douglass2004geometry}, $\dim G-\dim Q\cap wPw^{-1}=\dim \fn_{wPw^{-1}}+\dim \fn_{Q}-\dim (\fn_{Q}\cap \fn_{wPw^{-1}})=\dim \fn_P+\dim \fn_{Q}-\dim (\fn_{Q}\cap \Ad(w)\fn_P)$. Thus by \cite[Lem. 2.2.2]{breuil2019local}, $H_{Q,P,w}$ is a vector bundle over $G\cdot(Q,wP)\times_{G/Q}G/B=\left\{(g_1B,g_2P)\in G/B\times G/P\mid g_1^{-1}g_2\in QwP\right\}$ and is smooth of dimension $(\dim G-\dim T)-\dim \fn_Q-\dim \mathfrak{u}+\dim \fn_P+\dim \fn_{Q}+\dim \fn_{Q}\cap \Ad(w)\fm_P= (\dim G-\dim T)-(\dim \mathfrak{u}-\dim\fn_P-\dim \fn_{Q}\cap \Ad(w)\fm_P)=(\dim G-\dim T)-(\dim \mathfrak{u}\cap\Ad(w)\fm_P-\dim \fn_{Q}\cap \Ad(w)\fm_P)$ (the last equality can be deduced from Lemma \ref{lemmaweylgroupshortestelement} (4) if we take $w\in W^P$). Hence $Z_{Q,P,w}$ is also irreducible of dimension $(\dim G-\dim T)-(\dim \mathfrak{u}\cap\Ad(w)\fm_P-\dim \fn_{Q}\cap \Ad(w)\fm_P)\leq \dim G-\dim T$.\par
    (2) If $Q=B$, then $\dim \mathfrak{u}\cap\Ad(w)\fm_P-\dim \fn_{Q}\cap \Ad(w)\fm_P=0$, the result follows.\par
    (3) By the proof in (1), we see that the dimension of $Z_{Q,P,w}$ is equal to $\dim G-\dim T$ if and only if $\Ad(w)\fm_P\cap\fu\subset \fn_{Q}$. This proves $(a)\Leftrightarrow (b)$. In fact, for any $w'\in W_{Q}$, $\dim \Ad(w'w)\fm_P\cap\mathfrak{u}= \dim (\mathfrak{u}-\fn_P)=\dim \Ad(w)\fm_P\cap\mathfrak{u}$. Thus if $\Ad(w)\fm_P\cap\mathfrak{u}\subset \fn_Q$, then $\Ad(w'w)\fm_P\cap\mathfrak{u}\supset \Ad(w')(\Ad(w)\fm_P\cap\mathfrak{u})$ of the same dimension and hence $\Ad(w'w)\fm_P\cap\mathfrak{u}=\Ad(w')(\Ad(w)\fm_P\cap\mathfrak{u})\subset \fn_Q$. \par
    Now we take $w\in W^P$. Then $\Ad(w)(\fm_{P}\cap\mathfrak{u})\subset \mathfrak{u}$ by Lemma \ref{lemmaweylgroupshortestelement} and in this case, we have (similarly) $\Ad(w)(\fm_{P}\cap\overline{\fu})\subset\overline{\fu}$ and therefore $\Ad(w)\fm_{P}\cap\mathfrak{u}=\Ad(w)(\fm_{P}\cap\mathfrak{u})$. \par
    (b) $\Rightarrow$ (c): Since $w\in W^P$, we get $\Ad(w)(\fm_P\cap \mathfrak{u})= \Ad(w)\fm_P \cap \mathfrak{u}\subset \fn_Q$. As for any $w'\in W_{Q}$, $\Ad(w')\fn_Q=\fn_Q$, we have $\Ad(w'w)(\fm_P\cap \mathfrak{u})\subset\fn_Q\subset \fu$ and we conclude by Lemma \ref{lemmaweylgroupshortestelement} that $w'w\in W^P$ for any $w'\in W_Q$. \par
    (c) $\Rightarrow$ (b): We have $\Ad(w_{Q,0})(\Ad(w)\fm_P\cap\mathfrak{u})=\Ad(w_{Q,0}w)(\fm_{P}\cap\mathfrak{u})$ is contained in $\mathfrak{u}$ by Lemma \ref{lemmaweylgroupshortestelement}. Since $\Ad(w_{Q,0})\fn_Q=\fn_Q,\Ad(w_{Q,0})(\fm_Q\cap\mathfrak{u})=\fm_Q\cap\overline{\fu}$ and $\fu=\fm_Q\cap\fu+\fn_{Q}$, we get 
    \[\Ad(w)\fm_P\cap\mathfrak{u}=\Ad(w_{Q,0}w_{Q,0})(\Ad(w)\fm_P\cap\mathfrak{u})\subset \Ad(w_{Q,0})\fu\cap\fu =\fn_Q.\]
    \indent Assume above statements hold. We take $w\in W^P\cap {^QW}$ in (c), then $\lg(w_{Q,0}w)=\lg(w_{Q,0})+\lg(w)$. Hence $Bw_{Q,0}BwB=Bw_{Q,0}wB$ (cf. \cite[II.13.5 (7)]{jantzen2007representations}). As $Q=\overline{Bw_{Q,0}B}$ (\cite[II.13.2 (6)]{jantzen2007representations}), $Bw_{Q,0}wP=Bw_{Q,0}BwP\subset QwP$. Similarly, $QwP=\overline{Bw_{Q,0}B}BwBP\subset \overline{Bw_{Q,0}BwB}P=\overline{Bw_{Q,0}wP}$. Let 
    \[H_{Q,P,w}':=\left\{(\nu,g_1B,g_2P)\in Z_{Q,P}\mid g_1^{-1}g_2\in Bw_{Q,0}wP/P\right\}.\] 
    By the discussions in (1) and that $Bw_{Q,0}wP$ is open dense in $QwP$ (since $Bw_{Q,0}wP$ is open dense in $\overline{QwP}=\overline{Bw_{Q,0}wP}$ which contains $QwP$), the Zariski closure of $H'_{Q,P,w}$ is $Z_{Q,P,w}$. As $H'_{Q,P,w}\subset Z_{P,w_{Q,0}w}$, we get $Z_{Q,P,w}\subset  Z_{P,w_{Q,0}w}$. Since $\dim Z_{Q,P,w}=\dim Z_{P,w_{Q,0}w}$, we conclude that $Z_{Q,P,w}=Z_{P,w_{Q,0}w}$.\par
    (4) When $P=B$, we have $\Ad(w)\fm_B\cap\mathfrak{u}=\{0\}\subset \fn_Q$ for any $w\in W_Q\backslash W$, thus the result follows from (3).
\end{proof}
\begin{remark}
    The result that the scheme $\overline{Z}_{Q,B}$ in the proof of (1) of Proposition \ref{propositionsteinbergvarieties} (resp. $Z_{B,P}$) is equidimensional with the irreducible components parameterized by ${^QW}$ (resp. $W^P$) is already known by \cite[Thm. 4.1]{douglass2004geometry} (resp. \cite[Thm. 3.1]{douglass2004geometry})
\end{remark}
\begin{corollary}\label{corollarysteinbergvarieties}
    For any $w\in W/W_P$, $Z_{P,w}$ is contained in $Z_{Q,P}$ if and only if $wW_P=w_{Q,0}w_1W_P$ for some $w_1\in W^P\cap {^QW}$ such that $w_{Q,0}w_1\in W^P$.
\end{corollary}
\begin{proof}
    Assume $Z_{P,w}\subset Z_{Q,P}$. Since $Z_{Q,P}=\cup_{w'\in W_Q\backslash W/W_P}Z_{Q,P,w'}$ and each $Z_{Q,P,w'}$ is irreducible, we get $Z_{P,w}\subset Z_{Q,P,w'}$ for some $w'$. But $Z_{Q,P,w'}$ has dimension no more than $\dim Z_{P,w}$. Hence $Z_{P,w}=Z_{Q,P,w'}$. Now the result follows from (3) of Proposition \ref{propositionsteinbergvarieties}. 
\end{proof}
\begin{corollary}\label{corollarysteinbergvarietypartialclassicality}
    Let $x$ be a point of $Z_{Q,B}$, then there exists an irreducible component $Z_{B,w}$ of $Z_B$ such that $x\in Z_{B,w}$ and $Z_{B,w}\subset Z_{Q,B}$.
\end{corollary}
\begin{remark}
    The above result for points on $Z_{Q,B}$ doesn't hold in general for $P\neq B$. For example, if $Q=G$, then $Z_{G,P}=Z_{G,P,e}\subsetneq Z_{P,w_0}$ if $P\neq B$.
\end{remark}
\begin{definition}\label{definitionstrictlyQdominant}
    \begin{enumerate}
        \item Let $\mathbf{h}\in X_{*}(T)^{W_P}$ be an antidominant coweight (namely $\mathbf{h}\in X_{*}(T)$, $\langle \alpha, \mathbf{h} \rangle=0,\forall \alpha\in\Delta_P$ and $\langle \alpha, \mathbf{h} \rangle\leq 0, \forall \alpha\in \Delta$). We say $\mathbf{h}$ is $P$-regular if $\langle \alpha, \mathbf{h} \rangle<0, \forall \alpha\in \Delta\setminus \Delta_P$.
        \item For $\mathbf{h}\in X_{*}(T)$, we say $\mathbf{h}$ is strictly $Q$-dominant if $\langle \alpha, \mathbf{h} \rangle>0,\forall \alpha\in\Delta_Q$.
    \end{enumerate}
\end{definition}
\begin{lemma}
    If $\mathbf{h}\in X_{*}(T)^{W_P}$ is $P$-regular antidominant, then the set of $\alpha\in R$ such that $\langle \alpha ,\mathbf{h} \rangle<0$ (resp. $=0$, resp. $>0$) is $R^+\setminus R_{P}^+$ (resp. $R_P$, resp. $R^-\setminus R_{P}^-$).
\end{lemma}
\begin{theorem}\label{theoremvarietyweight}
    For any $w\in W/W_P$, $Z_{P,w}$ is contained in $Z_{Q,P}$ if and only if $w(\mathbf{h})$ is strictly $Q$-dominant for some (or every) $P$-regular antidominant coweight $\mathbf{h}\in X_*(T)^{W_P}$.
\end{theorem}
\begin{proof}
    We take an arbitrary $P$-regular antidominant coweight $\mathbf{h}\in X_*(T)^{W_P}$. Take the representative $w\in W^P$ for $w\in W/W_P$ and write $w$ for $\dw$. The statement $\Ad(w)\fm_P\cap \mathfrak{u}\subset \fn_{Q}$ (which is implied by that $Z_{P,w}$ is contained in $Z_{Q,P}$ by (3) of Proposition \ref{propositionsteinbergvarieties}) is equivalent to that $w(R_P)\cap R_Q^+=\emptyset$, or $w^{-1}(R_{Q}^+)\subset R \setminus R_P$. Since $\mathbf{h}$ is $P$-regular, for $\alpha\in R$, $\langle\alpha, \mathbf{h} \rangle\neq 0$ if and only if $\alpha\notin R_P$. Thus $\Ad(w)\fm_P\cap \mathfrak{u}\subset \fn_{Q}$ if and only if $\langle w^{-1}(\alpha), \mathbf{h} \rangle=\langle \alpha, w(\mathbf{h}) \rangle\neq 0,\forall \alpha\in R_Q^+$. On the other hand, if $w(\mathbf{h})$ is strictly $Q$-dominant, then $\langle\alpha, w(\mathbf{h}) \rangle> 0,\forall \alpha\in R_Q^+$. We now only need to prove that in the case when $\langle\alpha, w(\mathbf{h}) \rangle\neq 0$ for all $\alpha\in R_Q^+$, we have $w(\mathbf{h})$ is strictly $Q$-dominant if and only if $w_{Q,0}w\in W^P\cap {^QW}$ by (3) of Proposition \ref{propositionsteinbergvarieties}. \par
    Since now $\Ad(w)\fm_P\cap \mathfrak{u}\subset \fn_{Q}$ and $w\in W^P$, we have $w_{Q,0}w\in W^P$ as in the proof of (3.b) $\Rightarrow$ (3.c) of Proposition \ref{propositionsteinbergvarieties}. Let $w_1=w_{Q,0}w$. Then $w_1\in {^QW}$ if and only if $w_1^{-1}\in W^Q$. The latter is equivalent to $w_1^{-1}(R_{Q}^+)\subset R^+$ by Lemma \ref{lemmaweylgroupshortestelement}. We calculate that $\langle \alpha, w(\mathbf{h}) \rangle=\langle w_{Q,0}(\alpha), w_1(\mathbf{h}) \rangle= -\langle -w_{Q,0}(\alpha), w_1(\mathbf{h}) \rangle$ for every $\alpha\in R_Q^+$. As $-w_{Q,0}(R_{Q}^+)=-R_{Q}^{-}=R_{Q}^+$ (cf. \cite[II.1.5]{jantzen2007representations}), we get that $\langle \alpha, w(\mathbf{h}) \rangle\geq 0$ (resp. $>0$) for all $\alpha\in R_{Q}^+$ if and only if $\langle \alpha, w_1(\mathbf{h}) \rangle\leq 0$ (resp. $<0$) for all $\alpha\in R_{Q}^+$.\par 
    If $w_1^{-1}(R_{Q}^+)\subset R^+$, then $\langle \alpha, w_1(\mathbf{h}) \rangle=\langle w_1^{-1}(\alpha), \mathbf{h} \rangle \leq 0$ for all $\alpha\in R_Q^+$ since $\mathbf{h}$ is $P$-regular antidominant. Thus $\langle \alpha, w(\mathbf{h}) \rangle\geq 0$ for all $\alpha\in R_{Q}^+$. But we know $\langle \alpha, w(\mathbf{h}) \rangle\neq 0$, hence $\langle \alpha, w(\mathbf{h}) \rangle>0$ for any $\alpha\in R_Q^+$. Thus $w(\mathbf{h})$ is strictly $Q$-dominant.\par
    Conversely if $\langle \alpha, w(\mathbf{h}) \rangle>0$ for all $\alpha\in R_Q^+$, we get $\langle w_1^{-1}(\alpha), \mathbf{h} \rangle<0$ for all $\alpha\in R_Q^+$. Since $\mathbf{h}$ is $P$-regular and antidominant, we get $w_1^{-1}(\alpha)\in R^+\setminus R_P^{+},\forall\alpha\in R_Q^+$. In particular, $w_1^{-1}(R_{Q}^+)\subset R^+$. Thus $w_1\in W^P\cap{^QW}$. 
\end{proof}
The following lemma will be very important for us.
\begin{lemma}\label{lemmakeyinductionweylgroup}
    If $wW_P\neq w_0W_P$, then there exists $\alpha\in \Delta$ such that $s_{\alpha}w>w$ in $W/W_P$ and a standard parabolic subgroup $Q$ of $G$ satisfying both $Z_{P,w}\not\subset Z_{Q,P}$ and $Z_{P,s_{\alpha}w}\subset Z_{Q,P}$ (which implies that $s_{\alpha}w(\mathbf{h})$ is strictly $Q$-dominant by Theorem \ref{theoremvarietyweight} for every $P$-regular antidominant coweight $\mathbf{h}\in X_{*}(T)^{W_P}$).
\end{lemma}
\begin{proof}
    We assume $w\in W^P$. We claim that we can take a simple root $\alpha\in \Delta $ such that $\alpha\in w(R^+\setminus R_P^+)$ (or equivalently $\langle \alpha,w(\mathbf{h})\rangle<0$). If two roots $\alpha_1,\alpha_2\notin R^+\setminus R_P^+$ (equivalently $\langle \alpha_{i},\mathbf{h}\rangle\geq 0, i=1,2$), then $\alpha_1+\alpha_2\notin R^+\setminus R_P^+$. Thus $w^{-1}(R^+)\cap (R^+\setminus R_P^+)=\emptyset$ if and only if $w^{-1}(\Delta)\cap (R^+\setminus R_P^+)=\emptyset$. Assume that $R^+\cap w(R^+\setminus R^+_P)=\emptyset$. Then $R^-\cap w_0w(R^+\setminus R^+_P)=\emptyset$. This is only possible if $w_0w\in W_P$ (cf. Lemma \ref{lemmaweylgroupshortestelement}) which contradicts our assumption. Hence we can take $\alpha\in w(R^+\setminus R_P^+)\cap\Delta$. Since $s_{\alpha}(R^+\setminus \{\alpha\})=R^+\setminus \{\alpha\}$ and $s_{\alpha}(\alpha)=-\alpha$, we get $\left\{\alpha'\in R^+\setminus R_P^+\mid s_{\alpha} w(\alpha')\in R^- \right\}=\left\{\alpha'\in R^+\setminus R_P^+\mid w(\alpha')\in R^- \right\}\coprod \left\{w^{-1}(\alpha) \right\}$. Thus $\lg(s_{\alpha}w)=\lg(w)+1$ and $s_{\alpha}w\in W^P$ by Lemma \ref{lemmaweylgroupshortestelement}. Now take $Q=B(\alpha)=\overline{Bs_{\alpha}B}$ the standard parabolic subgroup with $R_Q=\{\alpha\}$. Then $w_{Q,0}=s_{\alpha}$, $w\in W^P\cap {^QW}$ and $s_{\alpha}w\notin {^QW}$. By Corollary \ref{corollarysteinbergvarieties}, $Z_{P,s_{\alpha}w}\subset Z_{Q,P}$ and $Z_{P,w}\not\subset Z_{Q,P}$ (and now $s_{\alpha}w(\mathbf{h})$ is strictly $Q$-dominant and $w(\mathbf{h})$ is not).
\end{proof}
The projection $p_P:X_{B}\rightarrow X_P$ induces a proper surjective morphism $Z_{B}\rightarrow Z_{P}$. Since $p(H_{B,B,w})\subset H_{B,P,w}$, we see $p_P$ sends $Z_{B,w}$ to $Z_{P,w}$ for any $w\in W$. When $w\in W^P$, the morphism $H_{B,B,w}\rightarrow H_{B,P,w}$ is an isomorphism (cf. Remark \ref{remarkisoopensubvarieties}) and $p_P$ induces a proper birational surjection $Z_{B,w}\twoheadrightarrow Z_{P,w}$ if and only if $w\in W^P$ (for the only if part, see \cite[Thm. 3.3]{douglass2004geometry}). 
For any $w\in W/W_P$, let $\overline{X}_{P,w}:=\kappa_{1,w}^{-1}(0)$ and $\overline{X}_{P}:=\kappa_{1}^{-1}(0)$ be the scheme-theoretic fiber over the zero weight. The underlying reduced space $\overline{X}_{P,w}^{\mathrm{red}}$ is contained in $Z_{P}$. It follows from the discussions after \cite[Thm. 2.4.7]{breuil2019local} that $\overline{X}_{B,w}=\cup_{w'\leq w}Z_{B,w'}$ which we have used in the proof of Proposition \ref{propositionsetsovernilpotent}. For $w\in W/W_P$, since $p_{P,w}:X_{B,w}\rightarrow X_{P,w}$ is surjective for any representative $w\in W$, we get $\overline{X}_{P,w}=\cup_{w'\leq w,w'\in W/W_P}Z_{P,w'}$.\par
We pick an arbitrary closed point $x\in \overline{X}_{P}\subset X_P$ and assume $x\in H_{B,P,w_x}$ for some $w_x\in W/W_P$ (or equivalently $x$ is in $V_{P,w_x}$ which is defined in \S\ref{sectionthevariety}). We have always $x\in Z_{P,w_x}\subset X_{P,w}$ for any $w\in W/W_P, w\geq w_x$. \par
Recall that if $A$ is an excellent Noetherian local ring and $\widehat{A}$ is the completion of $A$ at the maximal ideal, then the set of irreducible components of $\Spec(\widehat{A})$ is the disjoint union of the sets of irreducible components of $\Spec(\widehat{A}\otimes_AA/\fp_i),i\in I$ where $\{\fp_i,i\in I \}$ is the set of minimal prime ideals of $A$ (to see this, use \cite[Prop. 7.6.1, Sch. 7.8.3(vii)]{grothendieck1965EGAIV2} and that the normalization of $\Spec(A)$ is the disjoint union of normalizations of its irreducible components, see \cite[\href{https://stacks.math.columbia.edu/tag/035P}{Tag 035P}]{stacks-project}). Moreover, $\Spec(\widehat{A}\otimes_AA/\fp_i)=\Spec(\widehat{A/\fp_i})$ is equidimensional with the same dimension as $\Spec(A/\fp_i)$ (\cite[Sch. 7.8.3(x)]{grothendieck1965EGAIV2} and \cite[\href{https://stacks.math.columbia.edu/tag/07NV}{Tag 07NV}]{stacks-project}).\par
Since $\Spec(\cO_{\overline{X}_{P},x})$ is equidimensional with irreducible components $\Spec(\cO_{Z_{P,w},x})$ for $w$ such that $x\in Z_{P,w}$, $\Spec(\widehat{\cO}_{\overline{X}_{P},x})$ is equidimensional and its set of irreducible components is the disjoint union of the sets of irreducible components of $\Spec(\widehat{\cO}_{Z_{P,w},x})$ for $w$ such that $x\in Z_{P,w}$. Similarly, the subspace $\Spec(\widehat{\cO}_{\overline{X}_{P,w},x})$ is equidimensional and its set of irreducible components is the disjoint union the sets of irreducible components of $\Spec(\widehat{\cO}_{Z_{P,w'},x})$ for $w'$ such that $x\in Z_{P,w'}$ and $w\geq w'$ in $W/W_P$. In summary, we have, as topological spaces, for $w\geq w_x$:
\begin{equation}\label{formulairreduciblecomponents}
    \Spec(\widehat{\cO}_{\overline{X}_{P,w},x})=\bigcup_{x\in Z_{P,w'},w\geq w'} \Spec(\widehat{\cO}_{Z_{P,w'},x})
\end{equation}
where each term in the right hand side is non-empty. The closed immersion $Z_{Q,P}\hookrightarrow \overline{X}_P $ induces a closed immersion $\Spec(\widehat{\cO}_{Z_{Q,P},x})\hookrightarrow \Spec(\widehat{\cO}_{\overline{X}_{P},x})$ after completion at $x$. The dimensions of irreducible components of $\Spec(\widehat{\cO}_{Z_{Q,P},x})$ are no more than $\dim G-\dim T$ and the set of all irreducible components of dimension $\dim G-\dim T$ is the disjoint union of the sets of irreducible components of $\Spec(\widehat{\cO}_{Z_{P,w},x})$ with $w\in W/W_P$ such that $w(\mathbf{h})$ is strictly $Q$-dominant (Theorem \ref{theoremvarietyweight}). \par
Now we assume furthermore that the image of $x$ in $\fg$ is $0$. Since $Z_{P,w}$ is closed and contains the closed subset $\left\{(0,g_1B,g_2P)\in \fg\times G/B\times G/P\mid g_1^{-1}g_2\in BwP/P\right\}$, we get 
\[\left\{(0,g_1B,g_2P)\in \fg\times G/B\times G/P\mid g_1^{-1}g_2\in \overline{BwP/P}\right\}\subset Z_{P,w}.\] 
Hence in this case $x\in Z_{P,w}$ if and only if $w\geq w_x$ in $W/W_P$,  and for $w\geq w_x$,
\begin{equation}\label{formulairreduciblecomponents1}
    \Spec(\widehat{\cO}_{\overline{X}_{P,w},x})=\bigcup_{w\geq w'\geq w_x} \Spec(\widehat{\cO}_{Z_{P,w'},x})
\end{equation}
where each term in the right hand side is non-empty. A more practical form of Lemma \ref{lemmakeyinductionweylgroup} is the following.
\begin{lemma}\label{lemmakeyinductionweylgroup1}
    If the image of $x\in \overline{X}_P$ in $\fg$ is $0$ and $w_xW_P\neq w_0W_P$, then there exists $\alpha\in \Delta, s_{\alpha}w_x>w_x$ in $W/W_P$ and a standard parabolic subgroup $Q$ of $G$ such that $s_{\alpha}w_x(\mathbf{h})$ is strictly $Q$-dominant for every $P$-regular antidominant coweight $\mathbf{h}\in X_{*}(T)^{W_P}$ and the space 
    \[\Spec(\widehat{\cO}_{\overline{X}_{P,s_{\alpha}w_x},x})=\Spec(\widehat{\cO}_{Z_{P,s_{\alpha}w_x},x})\bigcup \Spec(\widehat{\cO}_{Z_{P,w_x},x}),\]
    where $\Spec(\widehat{\cO}_{Z_{P,s_{\alpha}w_x},x}),\Spec(\widehat{\cO}_{Z_{P,w_x},x})\neq\emptyset$, satisfies that $\Spec(\widehat{\cO}_{Z_{P,s_{\alpha}w_x},x})\subset \Spec(\widehat{\cO}_{Z_{Q,P},x})$ and $\Spec(\widehat{\cO}_{Z_{P,w_x},x})\not\subset \Spec(\widehat{\cO}_{Z_{Q,P},x})$ (where all spaces are viewed as subspaces of $\Spec(\widehat{\cO}_{\overline{X}_P,x})\subset \Spec(\widehat{\cO}_{X_P,x})$).
\end{lemma}
\subsection{Characteristic cycles}\label{sectioncharacteristiccycles}
Contents in this subsection will not be used subsequently. We assume that $k$ has characteristic $0$ and keep the notation in the last section. Theorem \ref{theoremvarietyweight} can be explained using geometric representation theory at least when $P=B$ and is true if we replace $Z_{B,w}$ by the \emph{Kazhdan-Lusztig cycles} denoted by $[\mathfrak{L}(ww_0\cdot 0)]$ in \cite[Thm. 2.4.7]{breuil2019local} (see below).\par
We assume $P=B$. For each weight $\mu$ of $\ft$, we let $M(\mu):=U(\fg)\otimes_{U(\fb)}\mu$ be the Verma module and let $L(\mu)$ be the irreducible quotient of $M(\mu)$. Then for any $w\in W$, the localization functor of Beilinson-Bernstein associates $M(ww_0\cdot 0)$ (resp. $L(ww_0\cdot 0)$) with a $G$-equivariant (regular holonomic) $\mathcal{D}$-module $\mathfrak{M}(ww_0\cdot 0)$ (resp. $\mathfrak{L}(ww_0\cdot 0)$) on $G/B\times G/B$ (\cite[Rem. 2.4.3]{breuil2019local}). Let $T^{*}(G/B\times G/B)\simeq \widetilde{\cN}\times\widetilde{\cN}$ be the cotangent bundle of $G/B\times G/B$ and identify the Steinberg variety $Z_{B}$ as a closed subscheme of $T^{*}(G/B\times G/B)$. Let $Z^0(Z_B)$ be the free abelian group generated by the codimension $0$ points in $Z_B$. The characteristic cycle $[\mathfrak{M}]$ of a coherent $\mathcal{D}$-module $\mathfrak{M}$ on $G/B\times G/B$ is the associated cycle in $Z^0(Z_B)$ of the characteristic variety $\mathrm{Ch}(\mathfrak{M})$, the scheme-theoretic support of some $\cO_{T^{*}(G/B\times G/B)}$-module $\mathrm{gr}(\mathfrak{M})$ constructed from $\mathfrak{M}$ with respect to some good filtration (\cite[\S2.4]{breuil2019local}). For $w\in W$, let $[Z_{B,w}]$ be the cycles associated with the irreducible components $Z_{B,w}$ which form a basis of $Z^0(Z_{B})$. It follows from \cite[Thm. 2.4.7(iii)]{breuil2019local} that the coefficient of $[Z_{B,w}]$ in $[\mathfrak{L}(ww_0\cdot 0)]$ is equal to $1$. Hence Theorem \ref{theoremvarietyweight} in this case ($P=B$) can be deduced from the same statement replacing $Z_{B,w}$ by $[\mathfrak{L}(ww_0\cdot 0)]$ (viewed as a union of irreducible components).\par
For a finitely generated $U(\fg)$-module $M$, there exists a good filtration $\{0\}=M_{-1}\subset M_0\subset M_1\subset\cdots$ of $M$ such that $\fg M_i\subset M_{i+1}$ and the $S(\fg)$-module $\mathrm{gr}(M):=\oplus_{i=0}^{\infty}M_{i}/M_{i-1}$ is finitely-generated, where $S(\fg)$ is the symmetric algebra of $\fg$ (cf. \cite[\S4.1]{borho1982differential}). The \emph{associated variety} $V(M)$ is the support of $\mathrm{gr}(M)$ in $\Spec(S(\fg))=\fg^{*}$, the dual space of $\fg$, and is independent of the choice of the good filtration (cf. \textit{loc. cit.}). We only consider $V(M)$ as an algebraic subset. Let $Q=M_{Q}N_Q$ be a standard parabolic subgroup of $G$ and $L_{\fm_{Q}}(ww_0\cdot 0)$ be a finite-dimensional irreducible representation of $\fm_{Q}$ of the highest weight $ww_0\cdot 0$ for some $w$ (which means that $ww_0\cdot 0$ is a dominant weight for $\fm_{Q}$) inflated to a representation of $\fq$. Then $M_{Q}(ww_0\cdot 0):=U(\fg)\otimes_{U(\fq)}L_{\fm_{Q}}(ww_0\cdot 0)$ is a parabolic Verma module in the category $\cO^{\fq}$ (\cite[\S9.4]{humphreys2008representations}) of the highest weight $ww_0\cdot 0$ and is in the principal block of the category $\cO$. Let $\fq^{\bot}$ be the subspace of $\fg^*$ consisting of elements that vanish on $\fq$.
\begin{lemma}\label{lemmaassociatedvariety}
    We have $V(M_{Q}(ww_0\cdot 0))\subset \fq^{\bot}$.
\end{lemma}
\begin{proof}
    We follow the proof of \cite[Thm. 4.6, Cor. 4.7]{borho1982differential}. If $w=w_0$, then $M_{Q}(0)=U(\fg)/U(\fq)$, the result is obvious. In general, $M_{Q}(ww_0\cdot 0)$ is a subquotient of $U(\fg)\otimes_{U([\fq,\fq])}W$ for some finite-dimensional $\fg$-module $W$ where $[\fq,\fq]$ denotes the commutator and the latter is equal to $U(\fg)/U([\fq,\fq])\otimes W$ by the tensor identity (\cite[Prop. 6.5]{knapp1988lie}). By \cite[Lem. 4.1]{borho1982differential}, we have $V(U(\fg)/U([\fq,\fq])\otimes W)=V(U(\fg)/U([\fq,\fq]))=[\fq,\fq]^{\bot}$. Thus $V(M_{Q}(ww_0\cdot 0))\subset [\fq,\fq]^{\bot}$. Moreover $V(M_{Q}(ww_0\cdot 0))\subset \fb^{\bot}$ since $M_Q(ww_0\cdot 0)$ is a subquotient of $M(0)$. Hence $V(M_{Q}(ww_0\cdot 0))\subset \fb^{\bot}\cap [\fq,\fq]^{\bot}=\fq^{\bot}$.
\end{proof}
Now we can prove a stronger version of Theorem \ref{theoremvarietyweight}. Remark that the statement ``$ww_0\cdot 0$ is a dominant weight for $\fm_Q$'' is equivalent to the statement ``$w(\mathbf{h})$ is strictly $Q$-dominant for some (or every) regular antidominant coweight $\mathbf{h}\in X_*(T)$''.
\begin{proposition}\label{propositioncharacteristiccycle}
    If $L(ww_0\cdot 0)\in \cO^{\fq}$, then the subset $\mathrm{Ch}(\mathfrak{L}(ww_0\cdot 0))$ of $Z_B$ is contained in $Z_{Q,B}$.
\end{proposition}
\begin{proof}
    By \cite[Prop. 9.3(e), \S9.4]{humphreys2008representations}, $L(ww_0\cdot 0)$ is a subquotient of $M_{Q}(ww_0\cdot 0)$. By \cite[Thm. 2.2.1(ii)]{hotta2007dmodules}, we only need to prove the same result replacing $L(ww_0\cdot 0)$ by $M_{Q}(ww_0\cdot 0)$. Now let $\mathfrak{M}_Q(ww_0\cdot 0)$ be the localization of $M_Q(ww_0\cdot 0)$ on $G/B\times G/B$ and $\mathfrak{M}_Q(ww_0\cdot 0)'$ be the corresponding $\mathcal{D}$-module on $G/B$ which is the usual localization of $M_Q(ww_0\cdot 0)$. Let $q:\widetilde{\cN}=G\times^B\fb^{\bot}\rightarrow \fg^{*}:(g,v)\mapsto \Ad(g)v$ be the moment map. By \cite[Prop. 8.1]{ginsburg1986g} and Lemma \ref{lemmaassociatedvariety}, $q(\mathrm{Ch}(\mathfrak{M}_Q(ww_0\cdot 0)'))=\fq^{\bot}$. As in the proof of \cite[Prop. 2.4.4]{breuil2019local}, we get $\mathrm{Ch}(\mathfrak{M}_Q(ww_0\cdot 0))=G\times^B\mathrm{Ch}(\mathfrak{M}_Q(ww_0\cdot 0)')$ is contained in $G\times^{B}q^{-1}(\fq^{\bot})$. Under the usual identification $\fg^*\simeq \fg$ given by the Killing form, $\fq^{\bot}$ is identified with $\fn_{Q}$. One can check under the isomorphism \cite[(2.15)]{breuil2019local}, we have $G\times^{B}q^{-1}(\fn_{Q})=Z_{Q,B}$. Hence $\mathrm{Ch}(\mathfrak{L}(ww_0\cdot 0))\subset Z_{Q,B}$.
\end{proof}
\begin{remark}\label{remarkcyclesKtheory}
    We discuss here some possible generalization of some results of cycles on $Z_B$ in \cite[\S2.13, \S2.14]{bezrukavnikov2012affine} and \cite[\S2.4]{breuil2019local} for the generalized Steinberg varieties. There exists already a theory of localization for singular blocks in characteristic $0$ (\cite{backelin2015singular}). However, the characteristic cycles on $Z_B$ can be produced via $K$-theory by \cite[Prop. 2.14.2]{bezrukavnikov2012affine} and \cite[Prop. 2.13.5]{bezrukavnikov2012affine}. The $K$-theory of generalized Steinberg variety $Z_P$ is well-behaved (\cite{douglass2014equivariant}) and can produce Kazhdan-Lusztig cycles on $Z_P$ corresponding to elements in the Weyl group by roughly pushing forward the cycles $[\mathfrak{M}(ww_0\cdot 0)]$ and $[\mathfrak{L}(ww_0\cdot 0)]$ on $Z_B$ via the map $Z_B\rightarrow Z_P$ (\cite[Thm. 2.1]{douglass2014equivariant}). We do not know for general $Z_P$ whether the similar formula as $[\overline{X}_{B,w}]=[\mathfrak{M}(ww_0\cdot 0)]$ (\cite[Prop. 2.4.6]{breuil2019local}) of cycles on $Z_B$, which was crucially used in \cite{breuil2019local} and proved in \cite[Prop. 2.14.2]{bezrukavnikov2012affine}, holds in general for the cycles on $Z_P$ from the $K$-theory. It is mentioned in \cite[Rem. 2.14.3]{bezrukavnikov2012affine} that the previous formula for $[\overline{X}_{B,w}]$ can be deduced by deformation arguments (\cite[\S6]{ginsburg1986g} or \cite[\S7.3]{Chriss1997RepresentationTA}). To get a generalized formula for $[\overline{X}_{P,w}]$, it seems that Cohen-Macaulayness of $X_{P,w}$ would be needed, which is unknown to the author for $P\neq B$.
\end{remark}
\section{Local models for the trianguline variety}\label{sectionlocalmodeltriangulinevariety}
We apply the results of \S\ref{sectionunibranchness} to study local geometry of the trianguline variety at certain points, generalizing the results in \cite[\S3]{breuil2019local}.

We fix a finite extension $K$ of $\Q_p$ with a uniformizer $\varpi_K$. Let $L$ be a finite extension of $\Q_p$ that splits $K$ with residue field $k_L$ and set $\Sigma=\Hom(K,L)$. 
\subsection{Almost de Rham trianguline $(\varphi,\Gamma_K)$-modules}\label{sectionalmostderham}
We recall some basic notions for the deformation theory of trianguline $(\varphi,\Gamma_K)$-modules. For details and notation, see \cite[\S3]{breuil2019local}. \par
Let $\mathrm{Rep}_{\BdR}(\cG_K)$ (resp. $\mathrm{Rep}_{\BdR^+}(\cG_K)$) be the category of $\BdR$-representations (resp. $\BdR^+$-representations) of $\cG_K$ (free of finite rank, continuous for the natural topology and semi-linear). We have rings $\mathrm{B}_{\mathrm{pdR}}^+=\mathrm{B}_{\mathrm{dR}}^+[\log(t)]$ and $\mathrm{B}_{\mathrm{pdR}}=\mathrm{B}_{\mathrm{dR}}[\log(t)]$ with the actions of $\cG_K$ satisfying that $g(\log(t))=\log(t)+\log(\epsilon(g))$ and $\mathrm{B}_{\mathrm{pdR}}$ admits a $\mathrm{B}_{\mathrm{dR}}$-derivative $\nu_{\mathrm{pdR}}$ such that $\nu_{\mathrm{pdR}}(\log(t))=-1$ which preserves $\mathrm{B}_{\mathrm{pdR}}^+$. \par
If $W$ is a $\BdR$-representation of $\cG_K$, then $D_{\mathrm{pdR}}(W):=\left(\mathrm{B}_{\mathrm{pdR}}\otimes_{\BdR}W\right)^{\cG_K}$ is a finite-dimensional $K$-vector space  of dimension no more than $\mathrm{dim}_{\BdR}W$ with a linear nilpotent endomorphism $\nu_W$. In other words, $D_{\mathrm{pdR}}(W)$ is a $K$-representation of the additive algebraic group $\mathbb{G}_a$. The $\mathrm{B}_{\mathrm{dR}}$-representation $W$ is called \emph{almost de Rham} if $\dim_KD_{\pdR}(W)=\dim_{\mathrm{B}_{\mathrm{dR}}}W$. A $\BdR^+$-representation $W^+$ is called almost de Rham if the $\BdR$-representation $W=W^+[\frac{1}{t}]$ is almost de Rham. The $\BdR^+$-lattices which are stable under $\cG_K$ in an almost de Rham $\BdR$-representation $W$ are in bijection with filtrations of $D_{\mathrm{pdR}}(W)$ as $K$-representations of $\mathbb{G}_a$ via $W^+\mapsto \mathrm{Fil}_{W^+}^{\bullet}\left(D_{\mathrm{pdR}}(W)\right)$ where $\mathrm{Fil}_{W^+}^{i}\left(D_{\mathrm{pdR}}(W)\right):=(t^i\mathrm{B}_{\mathrm{pdR}}^+\otimes_{\BdR^+}W^+)^{\cG_K}$ for $i\in\Z$ (\cite[Prop. 3.2.1]{breuil2019local}).\par
Let $A\in\cC_L$ be a local Artinian $L$-algebra with the maximal ideal $\fm_A$. Let $\mathrm{Rep}_{\mathrm{pdR},A}(\cG_K)$ be the category of almost de Rham $\mathrm{B}_{\mathrm{dR}}$-representations $W$ of $\cG_K$ together with a morphism of $\Q_p$-algebras $A\rightarrow \mathrm{End}_{\mathrm{Rep}_{\BdR}(\cG_K)}(W)$ such that $W$ is finite free over $\BdR\otimes_{\Q_p}A$. Let $\mathrm{Rep}_{A\otimes_{\Q_p}K}(\mathbb{G}_a)$ be the category of pairs $(V_A,\nu_A)$ where $\nu_A$ is a nilpotent endomorphism of a finite free $A\otimes_{\Q_p}K$-module $V_A$. The functor $D_{\mathrm{pdR}}$ induces an equivalence of categories between $\mathrm{Rep}_{\mathrm{pdR},A}(\cG_K)$ and $\mathrm{Rep}_{A\otimes_{\Q_p}K}(\mathbb{G}_a)$ (\cite[Lemma. 3.1.4]{breuil2019local}). \par
We have the Robba ring $\mathcal{R}_{A,K}$ of $K$ with $A$-coefficients (cf. \cite[Def. 6.2.1]{kedlaya2014cohomology}). A $(\varphi,\Gamma_K)$-module $\cM_A$ over $\cR_{A,K}[\frac{1}{t}]$ is defined to be a finite free $\cR_{A,K}[\frac{1}{t}]$-module equipped with commuting semilinear actions of $\varphi$ and $\Gamma_K$ such that $\cM_A$ admits a $(\varphi,\Gamma_K)$-stable $\cR_{K}$-lattice $D_A$ which is a $(\varphi,\Gamma_K)$-module over $\cR_K$ with the actions of $\varphi$ and $\Gamma_K$ given by those of $\cM_A$. Denote by $\Phi\Gamma_{A,K}^{+}$ (resp. $\Phi\Gamma_{A,K}$) the category of $(\varphi,\Gamma_K)$-modules over $\cR_{A,K}$ (resp. over $\cR_{A,K}[\frac{1}{t}]$).\par
A rank one $(\varphi,\Gamma_K)$-module $\cM_A$ over $\cR_{A,K}[\frac{1}{t}]$ is called of character type if $\cM_A$ is isomorphic to $\cR_{A,K}(\delta_A)[\frac{1}{t}]$ for some continuous character $\delta_A:K^{\times}\rightarrow A^{\times}$ (\cite[Cons. 6.2.4]{kedlaya2014cohomology}). A $(\varphi,\Gamma_K)$-module $\cM_A$ over $\cR_{A,K}[\frac{1}{t}]$ of rank $n$ is called trianguline if there exists an increasing filtration $\cM_{A,\bullet}:\left\{0\right\}=\cM_{A,0}\subset \cM_{A,1}\subset \cdots\subset \cM_{A,n}=\cM_A$ of $(\varphi,\Gamma_K)$-modules over $\cR_{A,K}[\frac{1}{t}]$ such that $\cM_{A,i}/\cM_{A,i-1}$ is a $(\varphi,\Gamma_K)$-module over $\cR_{A,K}[\frac{1}{t}]$ of character type for every $1\leq i\leq n$ and the filtration $\cM_{A,\bullet}$ is called a triangulation of $\cM_A$. Moreover, if $\cM_{A,i}/\cM_{A,i-1}\simeq\cR_{A,K}(\delta_{A,i})[\frac{1}{t}]$ for some characters $\delta_{A,i}: K^{\times}\rightarrow A^{\times},1\leq i\leq n$, then we say $\underline{\delta}_A=(\delta_{A,1},\cdots,\delta_{A,n})$ is a parameter of $\cM_A$. \par
We have an exact functor $W_{\mathrm{dR}}^+(-)$ (resp. $W_{\mathrm{dR}}(-)$) from $\Phi\Gamma^+_{A,K}$ (resp. $\Phi\Gamma_{A,K}$) to $\mathrm{Rep}_{\BdR^+,A}(\cG_K)$ (resp. $\mathrm{Rep}_{\BdR,A}(\cG_K)$) (\cite{berger2008construction} and \cite[Lem. 3.3.5]{breuil2019local}). If $\cM_A\in \Phi\Gamma_{A,K}$ is trianguline with a parameter $(\delta_{A,i})_{i=1,\cdots,n}$ and if the characters $\overline{\delta}_{A,i}:=\delta_{A,i} \mod \fm_A, 1\leq i\leq n$ are locally algebraic, then $W_{\mathrm{dR}}(\cM_A)\in \mathrm{Rep}_{\pdR,A}(\cG_K)$ and is a successive extension of rank one de Rham  $\BdR$-representations of $\cG_K$ (\cite[Lem. 3.3.6]{breuil2019local}). 
\subsection{Groupoids}\label{subsectiongroupoids}
We recall the definitions of some groupoids over $\cC_L$ defined in \cite[\S3]{breuil2019local}.\par 

Let $D$ be a fixed $(\varphi,\Gamma_K)$-module over $\cR_{L,K}$ of rank $n$. Let $\cM=D[\frac{1}{t}]$. We assume there exists and fix a triangulation $\cM_{\bullet}$ of $\cM$ of parameter $\underline{\delta}=(\delta_1,\cdots,\delta_n)$. We assume that $\delta_i$ is locally algebraic for any $i=1,\cdots,n$. \par
We let $\cT_0^n $ be the subset of $\cT_L^n$ that is the complement of characters $(\delta_1,\cdots,\delta_n)$ where $\delta_i/\delta_j$ or $\epsilon\delta_i/\delta_j$ is algebraic for some $i\neq j$. Let $\cT_0$ be the subset of $\cT_L$ which is the complement of the set of all $L$-points corresponding to characters of the form $z^{\mathbf{k}}$ or $\epsilon z^{\mathbf{k}}$ for some $\mathbf{k}\in \Z^{\Sigma}$. Remark that $\cT_0^n\neq (\cT_0)^n$.

We assume that the parameter $\underline{\delta}$ of $\cM$ (which we have assumed to be locally algebraic) lies in $\cT_0^n(L)$. \par

Let $W^{+}=W_{\mathrm{dR}}^{+}(D)\in \mathrm{Rep}_{\BdR^+,L}(\cG_K), W:=W_{\mathrm{dR}}(\cM)\in \mathrm{Rep}_{\BdR,L}(\cG_K).$ Then $ W\in \mathrm{Rep}_{\pdR,L}(\cG_K)$ and $W$ is filtered in $\mathrm{Rep}_{\pdR,L}(\cG_K)$ with a filtration $\cF_{\bullet}: \cF_1=W_{\mathrm{dR}}(\cM_1)\subset\cdots\subset \cF_n=W_{\mathrm{dR}}(\cM_n)$. We fix an isomorphism $\alpha:(L\otimes_{\Q_p}K)^n\simrightarrow D_{\mathrm{pdR}}(W)$.\par

The groupoid $X_W$ over $\cC_L$ consists of triples $(A,W_A,\iota_A)$ where $A\in\cC_L$, $W_A\in\mathrm{Rep}_{\mathrm{pdR},A}(\cG_K)$ and $\iota_A:W_A\otimes_A L\simrightarrow W$. A morphism $(A,W_A,\iota_A)\rightarrow (B,W_B,\iota_B)$ in $X_W$ is a morphism $A\rightarrow B$ in $\cC_L$ and an isomorphism $W_A\otimes_AB\simrightarrow W_B$ compatible with $\iota_A$ and $\iota_B$. The groupoid $X_W^{\square}$ consists of $(A,W_A,\iota_A,\alpha_A)$ where $(A,W_A,\iota_A)\in X_W$ and $\alpha_A:(A\otimes_{\Q_p}K)^{n}\simrightarrow D_{\mathrm{pdR}}(W_A)$ such that $\alpha_A$ modulo $\fm_A$ coincides with $\alpha$. Similarly we have $X_{W^{+}},X_{W^{+}}^{\square}$ by replacing $W,W_A$ with $W^{+},W_A^{+}$. We have a forgetful morphism $X_W^{\square}\rightarrow X_W$.\par

The groupoid $X_{W,\cF_{\bullet}} $ over $\cC_L$ consists of $(A,W_A,\cF_{A,\bullet},\iota_A)$ where $(A,W_A,\iota_A)\in X_W$ and $\cF_{A.\bullet}=(\cF_{A.i})_{i=1,\cdots,n}$ is a filtration of $W_A$ in $\mathrm{Rep}_{\BdR,A}(\cG_K)$ such that $\cF_{A,i}/\cF_{A,i-1}$ for $i=2,\cdots,n$ and $\cF_{A,1}$ are $A\otimes_{\Q_p}\BdR$-modules free of rank one and $\iota_A$ induces $ \cF_{A,\bullet}\otimes_AL\simrightarrow\cF_{\bullet}$. We let $X_{W,\cF_{\bullet}}^{\square}=X_{W,\cF_{\bullet}}\times_{X_W}X_W^{\square}$ where the morphism $X_{W,\cF_{\bullet}}\rightarrow X_W$ is the obvious one.\par

The groupoid $X_{\cM,\cM_{\bullet}}$ over $\cC_L$ consists of trianguline $(\varphi,\Gamma_K)$-modules $\cM_A$ over $\cR_{A,K}[\frac{1}{t}]$ for some $A\in\cC_L$ with a triangulation $\cM_{A,\bullet}$ of $\cM_A$ and an isomorphism $j_A:\cM_A\otimes_AL\simrightarrow \cM$ which is compatible with the filtrations. \par

The functor $W_{\mathrm{dR}}(-)$ induces a morphism $X_{\cM,\cM_{\bullet}}\rightarrow X_{W,\cF_{\bullet}}$. By our generic assumption on $\underline{\delta}$, the morphism is formally smooth by \cite[Cor. 3.5.6]{breuil2019local} and is relatively representable (\cite[Lem. 3.5.3]{breuil2019local}).\par

Let $X_D$ (resp. $X_{\cM}$) be the groupoid over $\cC_L$ of deformations of $D$ (resp. $\cM$). Then essentially due to Berger's equivalence of $B$-pairs and $(\varphi,\Gamma_K)$-modules (\cite{berger2008construction}), the morphism induced by inverting $t$ and the functors $W_{\mathrm{dR}}(-),W_{\mathrm{dR}}^+(-)$ 
$$X_D\rightarrow X_{\cM}\times_{X_{W}}X_{W^+}$$
is an equivalence of groupoids over $\cC_L$ (\cite[Prop. 3.5.1]{breuil2019local}). \par

Let $X_{D,\cM_{\bullet}}=X_D\times_{X_{\cM}}X_{\cM,\cM_{\bullet}}$, $X_D^{\square}=X_D\times_{X_W}X_W^{\square}$ and $X_{D,\cM_{\bullet}}^{\square}=X_{D,\cM_{\bullet}}\times_{X_{W}} X_{W}^{\square}$. We let $X_{W^+,\cF_{\bullet}}=X_{W^+}\times_{X_{W}}X_{W,\cF_{\bullet}}$ and $X_{W^+,\cF_{\bullet}}^{\square}=W_{W^+,\cF_{\bullet}}\times_{X_{W}}X_{W}^{\square}$. Then $W_{\mathrm{dR}}^{+}$ and $W_{\mathrm{dR}}$ induce morphisms
\begin{align*}
	X_{D,\cM_{\bullet}}&\rightarrow X_{W^{+},\cF_{\bullet}}\\
	X_{D,\cM_{\bullet}}^{\square}&\rightarrow X_{W^{+},\cF_{\bullet}}^{\square}
\end{align*}
which are formally smooth and relatively representable (as the base changes of $X_{\cM,\cM_{\bullet}}\rightarrow X_{W,\cF_{\bullet}}$ up to equivalence \cite[Cor. 3.5.4]{breuil2019local}). 

\subsection{Representability}\label{subsectionrepresentability}
We start with some slight generalization of some results in \cite[\S3.2]{breuil2019local} for cases of possibly non-regular Hodge-Tate weights.\par
We keep the notation in \S\ref{subsectiongroupoids}. If $A\in\cC_L$, $W_{A}\in \mathrm{Rep}_{\pdR,A}(\cG_K)$ and $W_{A}^+$ is a $\BdR^+\otimes_{\Q_p}A$-lattice of $W_A$, set 
\begin{align*}
	&D_{\pdR,\tau}(W_{A}):=D_{\pdR}(W_{A})\otimes_{A\otimes_{\Q_p}K,1\otimes \tau} A,\\
	&\Fil_{W_{A}^+}^{\bullet}\left(D_{\pdR,\tau}(W_{A})\right):=\Fil_{W_{A}^+}^{\bullet}(D_{\pdR}(W_{A}))\otimes_{A\otimes_{\Q_p}K,1\otimes \tau} A,\\
	&\gr_{\Fil_{W_{A}^+}^{\bullet}}^{i}\left(D_{\pdR,\tau}(W_{A})\right):=\Fil_{W_{A}^+}^{i}\left(D_{\pdR,\tau}(W_{A})\right)/\Fil_{W_{A}^+}^{i+1}\left(D_{\pdR,\tau}(W_{A})\right)	
\end{align*} 
for $i\in\Z,\tau\in\Sigma$.\par
Assume that for $\tau\in \Sigma$, the integers $i$ such that $\mathrm{gr}^i_{\mathrm{Fil}^{\bullet}_{W^+}}(D_{\mathrm{pdR},\tau}(W))\neq 0$ are
$$-k_{\tau,1}>\cdots >-k_{\tau,s_{\tau}}$$
for some positive integer $s_{\tau}$ and we set $m_{\tau,i}=\mathrm{dim}_{L}\mathrm{gr}^{-k_{\tau,i}}_{\mathrm{Fil}^{\bullet}_{W^+}}(D_{\mathrm{pdR},\tau}(W))$ for $1\leq i\leq s_{\tau}$. Then $m_{\tau,1}+\cdots+m_{\tau,s_\tau}=n$ for each $\tau\in\Sigma$. We get (partial) flags $D_{\mathrm{pdR},\tau}(W)=\Fil_{W^+}^{-k_{\tau,s_{\tau}}}(D_{\mathrm{pdR},\tau}(W))\supsetneq\cdots\supsetneq \Fil_{W^+}^{-k_{\tau,1}}(D_{\mathrm{pdR},\tau}(W))\supsetneq \left\{0\right\}$ inside $D_{\pdR,\tau}(W)$ for $\tau\in\Sigma$. \par

We set 
\[G:=\mathrm{Res}_{K/\Qp}(\GL_{n/K})\times_{\Q_p}L=\prod_{\tau\in\Sigma}\GL_{n/L}.\] 
Then $G$ acts on $D_{\mathrm{pdR}}(W)=\prod_{\tau\in\Sigma}D_{\mathrm{pdR},\tau}(W)$ via $\alpha:(L\otimes_{\Q_p}K)^n\simrightarrow D_{\mathrm{pdR}}(W)$. We let $P$ be the stabilizer of the filtration $\mathrm{Fil}_{W^{+}}^{\bullet}\left(D_{\mathrm{pdR}}(W)\right)$. Then $P=\prod_{\tau\in\Sigma} P_{\tau}$, where $P_{\tau}$ is the parabolic subgroup of $\GL_{n/L}$ which stabilizes the (partial) flag $\mathrm{Fil}^{-k_{\tau,\bullet}}_{W^+}(D_{\pdR,\tau}(W))$ via $\alpha$. \par

Recall we have the variety $\tildefg_P=\left\{(\nu,gP)\in\fg\times G/P\mid \Ad(g^{-1})\nu\in \fp\right\}$. For any $A\in\cC_L$, $A$-points of the (partial) flag variety $G/P=\prod_{\tau\in\Sigma}\GL_{n/L}/P_{\tau}$ correspond to (partial) flags $A^{n}=\Fil_{\tau,s_{\tau}}\supsetneq \cdots \supsetneq \Fil_{\tau,1}\supsetneq \Fil_{\tau,0}=\left\{0\right\}$ where for each $i=1,\cdots,s_{\tau}$, $\Fil_{\tau,i}/\Fil_{\tau,i-1}$ is a free $A$-module of rank $m_{\tau,i}$. An  $A$-point of $\tildefg_P$ then corresponds to a (partial) flag $(\Fil_{\tau,i})_{\tau\in\Sigma, 1\leq i\leq s_{\tau}}$ and a linear operator $\nu_A=\prod_{\tau\in \Sigma}\nu_{A,\tau}\in \prod_{\tau\in\Sigma}\mathrm{End}_A(A^n)$ which preserves the filtration. We thus have a point 
\[x_{W^{+}}:=\left(\alpha^{-1}\left(\left( \mathrm{Fil}_{W^{+}}^{-k_{\tau,\bullet}}\left(D_{\mathrm{pdR},\tau}(W)\right)\right)_{\tau\in\Sigma}\right),N_W:=\alpha^{-1}\circ \nu_{W}\circ \alpha\right)\in \tildefg_P(L)\]
where $\nu_W$ is the nilpotent operator acting on $D_{\pdR}(W)$. Given $\left(A,W_A^+,\iota_A,\alpha_A\right)\in X_{W^+}^{\square}$, the $A\otimes_{\Q_p}K$-module $D_{\pdR}(W_A)$ is equipped with a filtration $\Fil_{W_A^{+}}^{\bullet}(D_{\pdR}(W_A))$ together with an $A\otimes_{\Q_p}K$-linear nilpotent operator $\nu_{W_A}$ which preserves the filtration. Via the isomorphism $\alpha_A$, these datum give rise to an $A$-point 
\[\left(\alpha_A^{-1}\left(\left(\mathrm{Fil}_{W_A^+}^{-k_{\tau,\bullet}}(D_{\pdR,\tau}(W_A))\right)_{\tau\in \Sigma}\right),N_{W_A}:=\alpha_A^{-1}\circ \nu_{W_{A}}\circ \alpha_A\right)\] 
of $\widetilde{\fg}_P$ as in the proof of the following proposition.
\begin{proposition}\label{proprepresentablegP}
	The groupoid $X_{W^{+}}^{\square}$ is pro-representable and the functor 
	\[\left(A, W_A^{+},\iota_A,\alpha_A\right)\mapsto\left(\alpha_A^{-1}\left(\left(\mathrm{Fil}_{W_A^+}^{-k_{\tau,\bullet}}(D_{\pdR,\tau}(W_A))\right)_{\tau\in \Sigma}\right),N_{W_A}\right)\] 
	induces an isomorphism of functors between $|X_{W^+}^{\square}|$ and $\widehat{\tildefg}_P$, the completion of $\widetilde{\fg}_P$ at $x_{W^+}$.
\end{proposition}
\begin{proof}
	The proof is essentially the same as \cite[Thm. 3.2.5]{breuil2019local}. For any $A\in\cC_L$, by \cite[Lem. 3.2.2]{breuil2019local}, the functor $W^+_A\rightarrow \left(D_{\pdR}(W_A), \Fil^{\bullet}_{W_A^+}(D_{\pdR}(W_A)), \nu_{W_A}\right)$ induces an equivalence between the category of almost de Rham $A\otimes_{\Q_p} \BdR^+$-representations of $\cG_K$ and the category of filtered $A\otimes_{\Q_p}K$-representations of $\mathbb{G}_a$ (the definition before \cite[Lem. 3.2.2]{breuil2019local} should be that the graded pieces are projective $A\otimes_{\Q_p}K$-modules, see discussions below). If $\left(A,W_A^+,\iota_A,\alpha_A\right)\in X_{W^+}^{\square}$, by the proof of \textit{loc. cit.}, $\oplus_{i\in\Z}\gr^i_{\Fil_{W_A^+}^{\bullet}}(D_{\pdR}(W_A))$ is projective over $A\otimes_{\Qp}K$ (which is equivalent to the condition that $W_A^+$ is free over $A\otimes_{\Q_p}\BdR^+$) and moreover for any finite $A$-module $M$, there is an isomorphism 
	$$M\otimes_A\gr^i_{\Fil_{W_A^+}^{\bullet}}(D_{\pdR}(W_A))\simeq \gr^i_{\Fil_{M\otimes_AW_A^+}^{\bullet}}(D_{\pdR}(M\otimes_AW_A))$$
	for each $i\in\Z$. We have decompositions $D_{\pdR}(W_A)=\oplus_{\tau\in\Sigma} D_{\pdR,\tau}(W_A), \Fil^i_{W_A^+}(D_{\pdR}(W_A))=\oplus_{\tau\in\Sigma}\Fil^{i}_{W_A^+}(D_{\pdR,\tau}(W_A))$ and $\gr^i_{\Fil_{W_A^+}^{\bullet}}(D_{\pdR}(W_A))=\oplus_{\tau\in\Sigma}\gr^{i}_{\Fil_{W_A^+}^{\bullet}}(D_{\pdR,\tau}(W_A))$ for $i\in\Z$. Then $\gr^{i}_{\Fil_{W_A^+}^{\bullet}}(D_{\pdR,\tau}(W_A))$ is free over $A$ for each $\tau\in\Sigma$ and $i\in \Z$ since $A$ is local. And we have (\cite[Cor. 3.2.3]{breuil2019local}) $$\gr^{i}_{\Fil_{W_A^+}^{\bullet}}(D_{\pdR,\tau}(W_A))\otimes_A L\simeq \gr^{i}_{\Fil_{W^+}^{\bullet}}(D_{\pdR,\tau}(W)).$$ 
	Thus $\gr^{i}_{\Fil_{W_A^+}^{\bullet}}(D_{\pdR,\tau}(W_A))\neq 0$ if and only if $i\in \{-k_{\tau,1},\cdots,-k_{\tau,s_{\tau}}\}$ and $\gr^{-k_{\tau,i}}_{\Fil_{W_A^+}^{\bullet}}(D_{\pdR,\tau}(W_A))$ is free of rank $m_{\tau,i}$ for $i=1,\cdots,s_{\tau}$. The datum 
	\[\left(\left(\Fil^{-k_{\tau,\bullet}}_{W_A^+}\left(D_{\pdR,\tau}(W_A)\right)\right)_{\tau\in\Sigma},\nu_{W_A}\right)\]
	together with $\alpha_A,\iota_A$ up to isomorphisms is then equivalent to a morphism $\Spec(A)\rightarrow \tildefg_P$ whose image is in the infinitesimal neighbourhood of the $L$-point $x_{W^+}$ of $\tildefg_P$.
\end{proof}

Recall that from the split group $G$, the parabolic subgroup $P$ and a fixed Borel subgroup $B=TU$ contained in $P$, we have defined a scheme $X_P=\widetilde{\fg}\times_{\fg}\widetilde{\fg}_P$ in \S\ref{sectionthevariety} whose irreducible components are $X_{P,w}$ for $w\in W/W_P$ where $W$ denotes the Weyl group of $G$. We may assume that $B$ is the group of upper-triangular matrices and $T$ is the diagonal torus in $G=\prod_{\tau\in\Sigma}\GL_{n/L}$. Let $x$ be the $L$-point of the scheme $X_P$ corresponding to 
\[\left(\alpha^{-1}\left(D_{\mathrm{pdR}}\left(\cF_{\bullet}\right)\right),\alpha^{-1}\left(\Fil_{W^+}^{\bullet}(D_{\pdR}(W))\right), N_W\right)\]
and let $\widehat{X}_{P,x}=\mathrm{Spf}(\widehat{\cO}_{X_P,x})$ be the completion of $X_P$ at $x$. The groupoid $X_{W,\cF_{\bullet}}^{\square}$ is pro-represented by the completion $\widehat{\tildefg}$ of $\tildefg$ at the point $(\alpha^{-1}(D_{\pdR}(\cF_{\bullet})),N_{W})\in \tildefg(L)$ (\cite[Cor. 3.1.9]{breuil2019local}). Then by the same proof as for Proposition \ref{proprepresentablegP} and \cite[Cor. 3.1.9, Cor. 3.5.8]{breuil2019local}, we have the following generalization of \cite[Cor. 3.5.8]{breuil2019local}.
\begin{proposition}\label{propositionrepXPx}
(1) The groupoid $X_{W^+,\cF_{\bullet}}^{\square}$ is pro-representable and the functor $|X_{W^+,\cF_{\bullet}}^{\square}|$ is pro-represented by $\Spf(\widehat{\cO}_{X_P,x})$.\\
(2) The groupoid $X_{D,\cM_{\bullet}}^{\square}$ is pro-representable and the functor $|X_{D,\cM_{\bullet}}^{\square}|$ is pro-represented by a formal scheme which is formally smooth of relative dimension $[K:\Q_p]\frac{n(n+1)}{2}$ over $\widehat{X}_{P,x}$.
\end{proposition}
The scheme $\Spec(\widehat{\cO}_{X_P,x})$ is equidimensional of dimension $\dim \fg=[K:\Q_p]n^2$ with irreducible components $\Spec(\widehat{\cO}_{X_{P,w},x})$ for $w\in W/W_P$ such that $x\in X_{P,w}(L)$ (Corollary \ref{corollaryirreduciblecomponentscompletion}). Let $\widehat{X}_{P,w,x}=\Spf(\widehat{\cO}_{X_{P,w},x})$ for $w\in W/W_P$ such that $x\in X_{P,w}$.\par
For any $w\in W/W_P$, we define $X_{W^+,\cF_{\bullet}}^{\square,w}=X_{W^+,\cF_{\bullet}}^{\square}\times_{|X_{W^+,\cF_{\bullet}}^{\square}|}\widehat{X}_{P,w,x}$ and let $X_{W^+,\cF_{\bullet}}^{w}$ be the subgroupoid of $X_{W^+,\cF_{\bullet}}$ which is the image of $X_{W^+,\cF_{\bullet}}^{\square,w}$ under the forgetful map $X_{W^+,\cF_{\bullet}}^{\square}\rightarrow X_{W^+,\cF_{\bullet}}$. Let $X_{D,\cM_{\bullet}}^{w}=X_{D,\cM_{\bullet}}\times_{X_{W^+,\cF_{\bullet}}}X_{W^+,\cF_{\bullet}}^{w}$ and $X_{D,\cM_{\bullet}}^{\square,w}=X_{D,\cM_{\bullet}}^{\square}\times_{X_{W^+,\cF_{\bullet}}^{\square}}X_{W^+,\cF_{\bullet}}^{\square,w}$. Literally the same proof of \cite[Cor. 3.5.11]{breuil2019local} except that now we do not have the normalness result (but still have the irreducibility of $\Spec(\widehat{\cO}_{X_{P,w},x})$ by Theorem \ref{theoremunibranchYP}) shows that
\begin{corollary}\label{corollaryrepXPxw}
	For any $w\in W/W_P$ such that $x\in X_{P,w}(L)$, the functor $|X_{D,\cM_{\bullet}}^{\square,w}|$ is pro-represented by a complete Noetherian local domain of residue field $L$ and dimension $[K:\Q_p](n^2+\frac{n(n+1)}{2})$ and is formally smooth over $\widehat{X}_{P,w,x}$.
\end{corollary}
\subsection{The weight map}\label{subsectiontheweightmapthelocalmodel}
We keep the notation in \ref{subsectionrepresentability}. We view the parameter $\underline{\delta}$ of $\cM$ as an $L$-point of the rigid analytic space $\cT^n_L$. The completion of $\cT_L^n$ at the point $\underline{\delta}$ denoted by $\widehat{\cT^n_{\underline{\delta}}}$ pro-represents the functor from $\cC_L$ to deformations of the continuous character $\underline{\delta}: (K^{\times})^n\rightarrow L^{\times}$. Given an object $(A,\cM_A,\cM_{A,\bullet}, j_A)\in X_{\cM,\cM_{\bullet}}$, there exists a unique parameter $\underline{\delta}_{A}\in \cT^n_{L}(A)$ of $\cM_A$ such that $\delta_{A,i}\otimes_AL=\delta_i$ for $i=1,\cdots,n$ by \cite[Lem. 3.3.4]{breuil2019local}, which defines a morphism $\omega_{\underline{\delta}}:X_{\cM,\cM_{\bullet}}\rightarrow \widehat{\cT^n_{\underline{\delta}}}$ of groupoids over $\cC_L$. Recall that $\ft$ is the Lie algebra of $T$. Let $\widehat{\ft}$ be the completion of $\ft$ at $0\in\ft$. Composing the morphism $\mathrm{wt}-\mathrm{wt}(\underline{\delta}):\widehat{\cT^n_{\underline{\delta}}}\rightarrow \widehat{\ft}$ (cf. \cite[(3.16)]{breuil2019local}), we get a morphism $(\mathrm{wt}-\mathrm{wt}(\underline{\delta}))\circ \omega_{\underline{\delta}}:X_{\cM,\cM_{\bullet}}\rightarrow \widehat{\ft}$ of groupoids over $\cC_L$.\par
The map $\kappa=\kappa_B:\tildefg\rightarrow \ft$ (\S\ref{subsectiontheweightmapthevariety}) induces a map $\widehat{\tildefg}\rightarrow \widehat{\ft}$ by completion (since $N_W$ is nilpotent). The map $\kappa_{W,\cF_{\bullet}}:X_{W,\cF_{\bullet}}^{\square}\rightarrow |X_{W,\cF_{\bullet}}^{\square}|\simeq \widehat{\tildefg}\rightarrow\widehat{\ft}$ factors through a map $X_{W,\cF_{\bullet}}\rightarrow \widehat{\ft}$ which is also denoted by $\kappa_{W,\cF_{\bullet}}$. The composition $X_{\cM,\cM_{\bullet}}\stackrel{W_{\mathrm{dR}}(-)}{\rightarrow} X_{W,\cF_{\bullet}}\stackrel{\kappa_{W,\cF_{\bullet}}}{\rightarrow}\widehat{\ft}$ coincides with the morphism $(\mathrm{wt}-\mathrm{wt}(\underline{\delta}))\circ \omega_{\underline{\delta}}$ (\cite[Cor. 3.3.9]{breuil2019local}). Thus there is a morphism of groupoids over $\cC_L$
$$X_{\cM,\cM_{\bullet}}\rightarrow \widehat{\cT^n_{\underline{\delta}}}\times_{\widehat{\ft}}X_{W,\cF_{\bullet}}$$
which is formally smooth by \cite[Thm. 3.4.4]{breuil2019local}. 
\subsection{The trianguline variety}\label{sectiontriangulinevariety}
We now prove our main results on the local models for the trianguline variety.\par
We fix a continuous representation $\overline{r}:\cG_K\rightarrow \GL_{n}(k_L)$ of $\cG_K$. Let $R_{\overline{r}}$ be the usual \emph{framed} Galois deformation ring of $\overline{r}$ which is a complete Noetherian local ring over $\cO_L$ with residue field $k_L$ and let $\mathfrak{X}_{\overline{r}}$ be the rigid analytic space over $L$ associated with $\mathrm{Spf}(R_{\overline{r}})$ (the rigid generic fiber in the sense of Berthelot, cf. \cite[\S7]{de1995crystalline}). Let $\cT^n_{\mathrm{reg}}$ be the Zariski open subset of $\cT_{L}^n$ which consists of points $\underline{\delta}=(\delta_1,\cdots,\delta_n)$ such that $\delta_i/\delta_j\neq z^{-\mathbf{k}},\epsilon z^{\mathbf{k}}$ for all $i\neq j$ and $\mathbf{k}\in\Z_{\geq 0}^{\Sigma}$. The trianguline variety $X_{\mathrm{tri}}(\overline{r})$ defined to be the Zariski closure in $\mathfrak{X}_{\overline{r}}\times \cT_L^{n}$ of the subset $U_{\mathrm{tri}}(\overline{r})=\left\{(r,\underline{\delta})\in \mathfrak{X}_{\overline{r}}\times \cT_{\mathrm{reg}}^{n}\mid r\text{ is trianguline of parameter }\underline{\delta} \right\}\footnote{Here a representation $r:\cG_K\rightarrow \GL_n(L')$ for some finite extension $L'/\Q_p$ is said to be trianguline of character $\underline{\delta}$ means that the $(\varphi,\Gamma_K)$-module $D_{\mathrm{rig}}(r)$ over $\cR_{L',K}$ admits a filtration of sub-$(\varphi,\Gamma_K)$-modules whose graded pieces are isomorphic to $\cR_{L',K}(\delta_1),\cdots, \cR_{L',K}(\delta_n)$.}$ is a reduced rigid space over $L$, equidimensional of dimension $n^{2}+\frac{n(n+1)}{2}[K:\Q_p]$ with a Zariski open dense smooth subset $U_{\mathrm{tri}}(\overline{r})$ (\cite[Thm. 2.6]{breuil2017interpretation}).\par

Assume that $x=(r,\underline{\delta})\in X_{\mathrm{tri}}(\overline{r})\subset \mathfrak{X}_{\overline{r}}\times \cT_{L}^n$ is an $L$-point, then $r$ is trianguline of some parameter $\underline{\delta}'=(\delta_i')_{1\leq i\leq n}$ such that $\delta^{-1}_i\delta'_i$ is an algebraic character of $K^{\times}$ for all $i$ after possibly enlarging the coefficient field (\cite[Thm. 6.3.13]{kedlaya2014cohomology}). Let $D=D_{\mathrm{rig}}(r)$ be the \'etale $(\varphi,\Gamma_K)$-module over $\cR_{L,K}$ associated with $r$ and $\cM=D[\frac{1}{t}]$. We assume that $\underline{\delta}\in\cT_{0}^n$ is locally algebraic, then $\cM$ is equipped with a unique triangulation $\cM_{\bullet}$ of parameter $\underline{\delta}$ (\cite[Prop. 3.7.1]{breuil2019local}) and thus we are in the situation of \S\ref{subsectiongroupoids}. We let $W^{+}=W_{\mathrm{dR}}(D), W=W^{+}[\frac{1}{t}]$ and $\cF_{\bullet}=W_{\mathrm{dR}}(\cM_{\bullet})$ as before. Assume that the $\tau$-Sen weights of $r$ are integers $h_{\tau,n}\geq \cdots \geq h_{\tau,1}$ where each $k_{\tau,s_{\tau}}>\cdots>k_{\tau,1}$ appears in the sequence $(h_{\tau,n},\cdots,h_{\tau,1})$ with multiplicities $m_{\tau,s_{\tau}},\cdots,m_{\tau,1}$ and $m_{\tau,s_{\tau}},+\cdots+m_{\tau,1}=n$ for $\tau\in \Sigma$. The Hodge-Tate weights $(k_{\tau,i})_{1\leq i\leq s_{\tau},\tau\in\Sigma}$ coincide with what we have defined in \S\ref{subsectionrepresentability}. We fix an isomorphism $\alpha:(L\otimes_{\Q_p}K)^n\simrightarrow D_{\mathrm{pdR}}(W)$ and let $P$ be the parabolic subgroup of $G$ defined in \S\ref{subsectionrepresentability} determined by $\mathrm{Fil}_{W^+}^{\bullet}\left(D_{\mathrm{pdR}}(W)\right)$ and $\alpha$.\par
Let $V$ be the representation of $\cG_K$ associated with $r:\cG_K\rightarrow\GL_n(L)$ by forgetting the framing. Then there are groupoids $X_r$ (resp. $X_V$) over $\cC_L$ parameterizing liftings of $r$ (resp. deformations of $V$) (\cite[\S3.6]{breuil2019local}). The functor $D_{\mathrm{rig}}$ induces an equivalence $X_{V}\simrightarrow X_D$ and there is an isomorphism of formal schemes $X_{r}\simeq \widehat{\mathfrak{X}}_{\overline{r},r}$ where $\widehat{\mathfrak{X}}_{\overline{r},r}$ is the completion of $\mathfrak{X}_{\overline{r}}$ at the point $r$ (\cite[Lem. 2.3.3, Prop. 2.3.5]{kisin2009moduli}). Let $\widehat{\trivar}_x$ be the completion of the trianguline variety at the point $x$. Then we get a morphism $\widehat{\trivar}_x\rightarrow \widehat{\mathfrak{X}}_{\overline{r},r}\simeq X_r$ by projection. \par
Let $X_{r,\cM_{\bullet}}:=X_r\times_{X_V}\times X_D\times_{X_D} X_{D,\cM_{\bullet}}$ where the map $X_r\rightarrow X_V$ is forgetting the framing. Recall that the natural map $X_{\cM,\cM_{\bullet}}\rightarrow X_{\cM}$ as well as its base change $X_{r,\cM_{\bullet}}\rightarrow X_{r}$ is a closed immersion since we have assumed $\underline{\delta}\in \cT_0^n$ (\cite[Prop. 3.4.6]{breuil2019local}). 
\begin{proposition}\label{propositiontriangulinevarietyclosedimmersion}
	The morphism $\widehat{\trivar}_x\rightarrow X_r$ factors through a closed immersion $\widehat{\trivar}_x\rightarrow X_{r,\cM_{\bullet}}$.	
\end{proposition}
\begin{proof}
	This is \cite[Prop. 3.7.2]{breuil2019local} (based on \cite[Cor. 6.3.10]{kedlaya2014cohomology}) and \cite[Prop. 3.7.3]{breuil2019local}. The original proof for \cite[Prop. 3.7.2]{breuil2019local} is not complete and we write a proof here. The argument will be needed for Theorem \ref{theoremcyclepartialderham}.\par
	Pick an affinoid neighbourhood $U$ of $x$ in $X_{\mathrm{tri}}(\overline{r})$. Let $D_U$ be the $(\varphi,\Gamma_K)$-module over $\cR_{U,K}$ associated with the universal Galois representation of $\cG_{K}$ pulled back from $\fX_{\overline{r}}$ (cf. \cite[Thm. 2.2.17]{kedlaya2014cohomology}). By \cite[Cor. 6.3.10]{kedlaya2014cohomology}, there is a birational proper morphism $f:\widetilde{U}\rightarrow U$, a filtration of sub-$(\varphi,\Gamma_{K})$-modules $D_{\widetilde{U},\bullet}$ over $\cR_{\widetilde{U},K}$ of $D_{\widetilde{U}}:=f^*D_U$ and invertible sheaves $(\cL_{i})_{i=1,\cdots,n}$ such that $D_{\widetilde{U},0}=0,D_{\widetilde{U},n}=D_{\widetilde{U}}$ and there are inclusions $D_{\widetilde{U},i}/D_{\widetilde{U},i-1}\hookrightarrow \cR_{\widetilde{U},K}(\delta_{\widetilde{U},i})\otimes_{\cO_{\widetilde{U}}}\cL_i$ the cokernels of which are killed by $t$ for $i=1,\cdots,n$ where the characters $\underline{\delta}_{\widetilde{U}}$ is the pullback of the character on $\cT_L^n$ via $\widetilde{U}\rightarrow U\subset X_{\mathrm{tri}}(\overline{r})\rightarrow \cT_L^n$.\par
	Let $R_r$ be the completion of $R_{\overline{r}}[\frac{1}{p}]$ at the maximal ideal corresponding to $r$ so that $\widehat{\fX}_{\overline{r},r}\simeq \Spf(R_r)$ and let $R_{r,\cM_{\bullet}}$ be the quotient of $R_r$ such that $R_{r,\cM_{\bullet}}$ pro-represents $X_{r,\cM_{\bullet}}$. Take an arbitrary point $x'\in f^{-1}(x)$. We firstly prove that the map $R_{r}\rightarrow\widehat{\cO}_{U,x}\rightarrow \widehat{\cO}_{\widetilde{U},x'}$ induced by $\widetilde{U}\rightarrow U\subset X_{\mathrm{tri}}(\overline{r})\rightarrow \fX_{\overline{r}}$ factors through the quotient $R_{r,\cM_{\bullet}}$. Take $A$ a local Artin $L$-algebra with residue field $k(x')$ ($A$ is a $k(x')$-algebra, \cite[\href{https://stacks.math.columbia.edu/tag/0323}{Tag 0323}]{stacks-project}) with a composite $x=\Sp(k(x'))\rightarrow \Sp(A)\rightarrow V$ where $x\rightarrow \Sp(A)$ corresponds to the reduction map $A\rightarrow k(x')$. Then the pullback along the map $\Sp(A)\rightarrow \widetilde{U}$ of $D_{\widetilde{U},\bullet}[\frac{1}{t}]$ gives a triangulation $\cM_{A,\bullet}$ of $D_A[\frac{1}{t}]=D_{\mathrm{rig}}(r_A)[\frac{1}{t}]$ of parameter $\underline{\delta}_{A}=\underline{\delta}_{\widetilde{U}}\otimes_{\cO_{\widetilde{U}}}A$ where $r_A$ is the pullback of the universal Galois representation to $A$ via the map $R_r\rightarrow \widehat{\cO}_{\widetilde{U},x'}\rightarrow A$. Remark that the triangulation $\cM_{A,\bullet}\otimes_{A}k(x')$ of $\cM_A\otimes_Ak(x')$ coincides with ${\cM_{\bullet}\otimes_Lk(x')}$ where $\cM_{\bullet}$ is the unique triangulation on $\cM=D_{\mathrm{rig}}\left(r\right)[\frac{1}{t}]$ of parameter $\underline{\delta}$ obtained by Lemma \cite[Prop. 3.7.1]{breuil2019local} (and by \cite[Thm. 4.4.3]{kedlaya2014cohomology} and \cite[Lem. 3.4.3]{breuil2019local}). Let $\widetilde{A}:=A\times_{k(x')}L\in \cC_L$ be the subring of $A$ consisting of elements whose reduction modulo the maximal ideal $\fm_A$ of $A$ lie in $L$ (cf. \cite[\href{https://stacks.math.columbia.edu/tag/08KG}{Tag 08KG}]{stacks-project}). By writing out some $L$-bases of $\widetilde{A}$ and $A$, we see that the preimage of $\cR_{L,K}[\frac{1}{t}]\subset \cR_{k(x'),K}[\frac{1}{t}]$ under the map $\cR_{A,K}[\frac{1}{t}]\rightarrow\cR_{k(x'),K}[\frac{1}{t}]$ is $\cR_{\widetilde{A},K}[\frac{1}{t}]$. As the reduction of $r_A$ is $r\otimes_L k(x')$ (resp. the reduction of $\delta_{A,i}:\delta_{A}\otimes_Ak(x')$ is $\delta_{i}\otimes_Lk(x')$), $r_A$ (resp. $\delta_{A,i}$) can be defined over $\widetilde{A}$ and we denote the model by $r_{\widetilde{A}}$ (resp. $\delta_{\widetilde{A},i}$) whose reduction modulo $\fm_{\widetilde{A}}$ is $r$ (resp. $\delta_{i}$). Then $D_{\mathrm{rig}}\left(r_A\right)[\frac{1}{t}]=D_{\mathrm{rig}}(r_{\widetilde{A}})[\frac{1}{t}]\otimes_{\widetilde{A}}A$.\par 
	We need to show that the triangulation $\cM_{A,\bullet}$ also has a model $\cM_{\widetilde{A},\bullet}$ over $\cR_{\widetilde{A},K}$. We extend the injection 
	\[\cM_1=\cR_{L,K}(\delta_{1})[\frac{1}{t}]\hookrightarrow D_{\mathrm{rig}}\left(r\right)[\frac{1}{t}]\]
	to an isomorphism $\oplus_{i=1}^n\cR_{L,K}[\frac{1}{t}]\bar{e}_{i}\simrightarrow D_{\mathrm{rig}}\left(r\right)[\frac{1}{t}]$ of $\cR_{L,K}[\frac{1}{t}]$-modules where $\bar{e}_1$ is the image of a generator of the rank one free module $\cR_{L,K}(\delta_{1})[\frac{1}{t}]$. Since the reduction of the injection $\cR_{A,K}(\delta_{A,1})[\frac{1}{t}]\hookrightarrow D_{\mathrm{rig}}\left(r_A\right)[\frac{1}{t}]$ is identified via pull back with ${\cR_{L,K}(\delta_{1})[\frac{1}{t}]\otimes_Lk(x')}\hookrightarrow {D_{\mathrm{rig}}\left(r\right)[\frac{1}{t}]\otimes_Lk(x')}$, we can pick a lift $e_1\in D_{\mathrm{rig}}\left(r_A\right)[\frac{1}{t}]$ of $\bar{e}_1\otimes_L 1$ generating the image of the first injection. We can extend the injection to an isomorphism $\cR_{A,K}[\frac{1}{t}]^{n}\simeq D_{\mathrm{rig}}\left(r_A\right)[\frac{1}{t}]$ of $\cR_{A,K}[\frac{1}{t}]$-modules and we may assume, after changing basis given by a matrix in $\GL_{n-1}(\cR_{A,K}[\frac{1}{t}])$, the reduction of the extended basis $e_2,\cdots, e_n$ is equal to $\bar{e}_2\otimes_L 1, \cdots, \bar{e}_n\otimes_L 1$ since we have a surjection $\GL_{n-1}\left(\cR_{A,K}[\frac{1}{t}]\right)\twoheadrightarrow \GL_{n-1}\left(\cR_{k(x'),K}[\frac{1}{t}]\right)$. We also take a basis $\widetilde{e}_1,\cdots,\widetilde{e}_n$ of $D_{\mathrm{rig}}(r_{\widetilde{A}})[\frac{1}{t}]$ such that the reduction modulo $\fm_{\widetilde{A}}$ of $\widetilde{e}_1,\cdots,\widetilde{e}_n$ is equal to $\overline{e}_1,\cdots,\overline{e}_n$. The translation matrix $M$ between the basis $\widetilde{e}_1\otimes_{\widetilde{A}}1,\cdots,\widetilde{e}_n\otimes_{\widetilde{A}}1$ and $e_1,\cdots,e_n$ of ${D_{\mathrm{rig}}\left(r_A\right)[\frac{1}{t}]}$ has trivial reduction modulo $\fm_A$. In particular, $M\in\GL_n(\cR_{\widetilde{A},K}[\frac{1}{t}])$. This means that we can choose $\widetilde{e}_i$ such that $e_i=\widetilde{e}_i\otimes_{\widetilde{A}}1$. Then we see the element $\widetilde{e}_1$ defines an injection $\cR_{\widetilde{A},K}(\delta_{\widetilde{A},1})[\frac{1}{t}]\hookrightarrow {D_{\mathrm{rig}}(r_{\widetilde{A}})[\frac{1}{t}]}$ with the quotient a $(\varphi,\Gamma_{K})$-module over $\cR_{\widetilde{A},K}[\frac{1}{t}]$. Applying the same argument on the quotient and by induction, we see that ${D_{\mathrm{rig}}(r_{\widetilde{A}})[\frac{1}{t}]}$ admits a filtration $\cM_{\widetilde{A},\bullet}$ such that $\cM_{\widetilde{A},\bullet}\otimes_{\widetilde{A}}A=\cM_{A,\bullet}$ and $\cM_{\widetilde{A},\bullet}\otimes_{\widetilde{A}}L=\cM_{\bullet}$. Let $\widetilde{\cO}_{\widetilde{U},x'}$ be the subring $\widehat{\cO}_{\widetilde{U},x'}\times_{k(x')}L$ of $\widehat{\cO}_{\widetilde{U},x'}$. The composite map $R_{r}\rightarrow \widehat{\cO}_{U,x}\rightarrow \widehat{\cO}_{\widetilde{U},x'}\rightarrow A$ factors through $R_{r}\rightarrow \widetilde{\cO}_{\widetilde{U},x'}\rightarrow \widetilde{A}$ which gives rise the deformation $r_{\widetilde{A}}$ of $r$, i.e. an object $(\widetilde{A},r_{\widetilde{A}})\in X_r$. The discussion above shows that $r_{\widetilde{A}}$ admits a triangulation $\cM_{\widetilde{A},\bullet}$ on $D_{\mathrm{rig}}(r_{\widetilde{A}})[\frac{1}{t}]$ whose reduction modulo $\fm_{\widetilde{A}}$ is $\cM_{\bullet}$ and defines an object in $X_{r,\cM_{\bullet}}$. Hence the morphism $R_{r}\rightarrow \widetilde{A}$, as well as $R_{r}\rightarrow \widetilde{A}\rightarrow A$, factors through $R_{r,\cM_{\bullet}}$. This implies that the morphism $R_{r}\rightarrow \widehat{\cO}_{\widetilde{U},x'}$ factors through the quotient $R_{r,\cM_{\bullet}}$.\par
	We now prove that the map $R_{r}\rightarrow \widehat{\cO}_{U,x}$ also factors though $R_{r,\cM_{\bullet}}$. Otherwise assume there exists a non-zero element $a$ in the kernel of $R_{r}\rightarrow R_{r,\cM_{\bullet}}$ such that the image of $a$ in $\widehat{\cO}_{U,x}$ is not zero. As the morphism $f:\widetilde{U}\rightarrow U$ is proper and surjective, there is the Stein decomposition $f:\widetilde{U}\stackrel{f'}{\rightarrow} Z\stackrel{g}{\rightarrow} U$ such that $g$ is a finite surjective morphism, $f'$ is a surjective proper morphism and $\cO_Z\simrightarrow f'_*\cO_{\widetilde{U}}$ (\cite[Prop. 9.6.3/5]{bosch1984non}). Then if we write $U=\Sp(A),Z=\Sp(B)$, the map $A\rightarrow B$ induced by $g$ is a finite injection (since $U$ is reduced and $g$ is a surjection). Let $\fm$ be the maximal ideal of $A$ corresponding to the point $x$ and let $\fn_1,\cdots,\fn_m$ be the maximal ideals of $B$ above $\fm$. Then there is an injection $\widehat{A}_{\fm}\hookrightarrow B\otimes_A\widehat{A}_{\fm}\simeq \oplus_{i=1}^m \widehat{B}_{\fn_i}$ (\cite[\href{https://stacks.math.columbia.edu/tag/07N9}{Tag 07N9}]{stacks-project}) where $\widehat{A}_{\fm}$ (resp. $\widehat{B}_{\fn_i}$) denotes the completion of $A$ (resp. $B$) with respect to the ideal $\fm$ (resp. $\fn_i$). Since the image of $a$ in $\widehat{A}_{\fm}$ is not zero, the image of $a$ in one of $\widehat{B}_{\fn_i}$ is not zero, which we may assume to be $\widehat{B}_{\fn_1}$. Let $z$ be the point on $Z$ corresponding to $\fn_1$. By the theorem on formal functions (\cite[Thm. 9.6.3/2]{bosch1984non}), $(f'_*\cO_{\widetilde{U}})_{z}^{\wedge}\simrightarrow \varprojlim_s (\cO_{\widetilde{U}}/\fn_1^s)(\widetilde{U})$. Hence $\widehat{B}_{\fn_1}\simrightarrow \varprojlim_s (\cO_{\widetilde{U}}/\fn_1^s)(\widetilde{U})$. Thus the image of $a$ in $(\cO_{\widetilde{U}}/\fn_1^s)(\widetilde{U})$ is not zero for some $s$. It turns out that there exists $x'\in (f')^{-1} (z)$ such that the image of $a$ in $\cO_{\widetilde{U},x'}/\fn_1^s$ is not zero (\cite[Cor. 9.4.2/7]{bosch1984non} and \cite[Prop. 9.5.3/1(iii)]{bosch1984non}) and hence the image of $a$ in $\widehat{C}_{\fn_1}:=\varprojlim_s \cO_{\widetilde{U},x'}/\fn_1^s$ is not zero where $C:=\cO_{\widetilde{U},x'}$. The completion of $\widehat{C}_{\fn_1}$ with respect to the maximal ideal $\fm_{x'}\widehat{C}_{\fn_1}$ is then equal to (\cite[Prop. 10.12, 10.13]{atiyah1969introduction} and that $\cO_{\widetilde{U},x'}/\fm_{x'}^s\cO_{\widetilde{U},x'}$ is $\fn_1$-adically complete for any $s$) 
	\[ \varprojlim_s\widehat{C}_{\fn_1}/\fm_{x'}^s\widehat{C}_{\fn_1}=\varprojlim_sC/\fm_{x'}^sC=\varprojlim_s \cO_{\widetilde{U},x'}/\fm_{x'}^s\cO_{\widetilde{U},x'}=\widehat{\cO}_{\widetilde{U},x'}.\]
	Hence the morphism $\widehat{C}_{\fn_1}\rightarrow \widehat{\cO}_{\widetilde{U},x'}$ is the completion of the Noetherian local ring $\widehat{C}_{\fn_1}$ (\cite[Thm. 10.26]{atiyah1969introduction}) with respect to the maximal ideal, thus is injective (\cite[Cor. 10.19]{atiyah1969introduction}). We conclude that $a$ is sent to a non-zero element in $\widehat{\cO}_{\widetilde{U},x'}$. This contradicts that the morphism $R_{r}\rightarrow \widehat{\cO}_{\widetilde{U},x'}$ factors through the quotient $R_{r,\cM_{\bullet}}$ which we just proved! Thus we get the conclusion.
\end{proof}
We fix an isomorphism $W\simeq (\mathcal{S}_n)^{\Sigma}$ by identifying $n$-tuples $(a_{\tau,1},\cdots,a_{\tau,n})_{\tau\in\Sigma}\in (\Z^{n})^{\Sigma}$ with the diagonal matrix $\prod_{\tau\in\Sigma}\mathrm{diag}(a_{\tau,1},\cdots,a_{\tau,n})\in\fg$. By \cite[Prop. 2.9]{breuil2017interpretation} (and the proof of \cite[Lem. 3.7.6]{breuil2019local}), the multisets of Sen weights of $V$ are exactly the multisets $\left\{\mathrm{wt}_{\tau}(\delta_{i})\mid i=1,\cdots,n\right\},\tau\in\Sigma$. Hence there exists a unique $w=(w_{\tau})_{\tau\in\Sigma}\in W/W_P$ such that \[\left(h_{\tau,1},\cdots,h_{\tau,n}\right)=w_{\tau}^{-1}\left(\mathrm{wt}_{\tau}(\delta_1),\cdots,\mathrm{wt}_{\tau}(\delta_n)\right)=\left(\mathrm{wt}_{\tau}(\delta_{w_{\tau}(1)}),\cdots,\mathrm{wt}_{\tau}(\delta_{w_{\tau}(n)})\right)\] 
for any $\tau\in\Sigma$. We denote by $w\in W/W_P$ the element associated with the point $x$ in this way.\par

Recall that there is a morphism $(\kappa_{1},\kappa_{2}):X_{P}\rightarrow T_P=\ft\times_{\ft/W}\ft/W_P$ in \S\ref{subsectiontheweightmapthevariety} which induces a morphism $X_{D,\cM_{\bullet}}^{\square}\rightarrow X_{W^+,\cF_{\bullet}}^{\square}\rightarrow |X_{W^+,\cF_{\bullet}}^{\square}|\simrightarrow \widehat{X}_{P,x_{\mathrm{pdR}}}\rightarrow \widehat{T}_{P,(0,0)}$ where $x_{\mathrm{pdR}}$ is the point in $X_{P}(L)$ associated with the point $x\in \trivar(L)$ together with the fixed framing $\alpha$ as in Proposition \ref{propositionrepXPx}. The morphism $X_{D,\cM_{\bullet}}^{\square}\rightarrow \widehat{T}_{P,(0,0)}$ factors through a morphism $X_{D,\cM_{\bullet}}\rightarrow \widehat{T}_{P,(0,0)}$ (see the end of \cite[\S3.5]{breuil2019local}).
We let $\Theta_x:\widehat{X_{\mathrm{tri}}}(\overline{r})_x\rightarrow \widehat{T}_{P,(0,0)}$ be the composite map $\widehat{\trivar}_x\hookrightarrow X_{r,\cM_{\bullet}}\rightarrow X_{V,\cM_{\bullet}}\simeq X_{D,\cM_{\bullet}}\rightarrow \widehat{T}_{P,(0,0)}$. Recall that there are closed formal subschemes $\widehat{T}_{P,w',(0,0)}$ of $\widehat{T}_{P,(0,0)}$ for $w'\in W/W_P$.
\begin{proposition}\label{propositionfactorsthroughTPw}
	The morphism $\Theta_{x}$ factors through $\widehat{T}_{P,w,(0,0)}\hookrightarrow \widehat{T}_{P,(0,0)}$.
\end{proposition}
\begin{proof}
	Assume that $x_A:\mathrm{Spf}(A)\rightarrow \widehat{\trivar}_x$ is an $A$-point for some $A\in\cC_L$. Via the morphism  $\widehat{\trivar}_x\rightarrow X_{D,\cM_{\bullet}}\rightarrow X_{W^+,\cF_{\bullet}}$, the point $x_A$ is associated with a $\BdR^+\otimes_{\Q_p} A$-representation $W_A^+$ of $\cG_K$, a full filtration $\cF_{A,\bullet}$ of $W_A=W_A^+[\frac{1}{t}]$ and a parameter $\underline{\delta}_A=(\delta_{A,1},\cdots,\delta_{A,n})$ of the associated $(\varphi,\Gamma_K)$-module $\cM_A$ over $\cR_{A,K}[\frac{1}{t}]$, which is a deformation of the datum $(W^+,\cF_{\bullet},\underline{\delta})$. We can choose a framing $\alpha_A:(A\otimes_{\Q_p}K)^{n}\simrightarrow D_{\mathrm{pdR}}(W_A)$ such that $\alpha_A$ modulo $\fm_A$ coincides with $\alpha$. Let $x_{A,\mathrm{pdR}}$ be the point
	\[\left(\alpha_A^{-1}\left(D_{\mathrm{pdR}}(\cF_{A,\bullet})\right),\alpha_A^{-1}\left(\left(\mathrm{Fil}_{W_A^+}^{-k_{\tau,\bullet}}(D_{\pdR,\tau}(W_A))\right)_{\tau\in \Sigma}\right),N_{W_A}=\alpha^{-1}_A\circ \nu_{W_A}\circ\alpha_A\right)\in \widehat{X}_{P,x_{\mathrm{pdR}}}(A)\] 
	corresponding to the element in $X_{W^+,\cF_{\bullet}}(A)$ given by $x_A$. Then $\Theta_x$ sends $x_A$ to 
	\[\left(\kappa_1(x_{A,\mathrm{pdR}}),\kappa_2(x_{A,\mathrm{pdR}})\right)\in (\ft\times _{\ft/W}\ft/W_P) (A)\] 
	whose reduction module $\mathfrak{m}_A$ is $(0,0)$ where $\kappa_1,\kappa_2$ are defined in \S\ref{subsectiontheweightmapthevariety}. \par
	Explicitly, $\kappa_1(x_{A,\mathrm{pdR}})\in\ft(A)$ equals to
	\begin{equation}\label{equationfactorskappa1}
		(\nu_{W_{\mathrm{dR}}\left(\cR_{A,K}(\delta_{A,1})[\frac{1}{t}]\right)},\cdots, \nu_{W_{\mathrm{dR}}\left(\cR_{A,K}(\delta_{A,n})[\frac{1}{t}]\right)})=\left(\mathrm{wt}(\delta_{A,1})-\mathrm{wt}(\delta_1),\cdots,\mathrm{wt}(\delta_{A,n})-\mathrm{wt}(\delta_n)\right)
	\end{equation}
	by \cite[Cor. 3.3.9]{breuil2019local} ($\kappa_1$ and $\kappa_{W,\cF_{\bullet}}$ in \S\ref{subsectiontheweightmapthelocalmodel} are both defined using $\kappa_B:\tildefg\rightarrow\ft$).\par
	For each $\tau\in\Sigma, i\in\Z$, we let $ \mathrm{gr}^{\tau,i}(\nu_{W_A})$ be the restriction of $\nu_{W_A}$ on $\gr^{i}_{\Fil_{W_A^+}^{\bullet}}(D_{\pdR,\tau}(W_A))$ (see \S\ref{subsectionrepresentability} for notation). Then \[\kappa_2(x_{A,\mathrm{pdR}})\in\ft/W_P(A)=\prod_{\tau\in\Sigma}\ft_{\tau}/W_{P_{\tau}}(A)=\prod_{\tau\in\Sigma}\prod_{i=1}^{s_{\tau}} \ft_{\tau,i}/W_{M_{\tau,i}}(A)\] 
	is (to simplify the notation we omit the identification $\alpha_A$)
	\begin{equation}\label{equationfactorskappa2}
		\left(\gamma_{M_{\tau,1}} \left(\mathrm{gr}^{\tau,-k_{\tau,1}}(\nu_{W_A})\right),\cdots,\gamma_{M_{\tau,s_{\tau}}} \left(\mathrm{gr}^{\tau,-k_{\tau,s_{\tau}}}(\nu_{W_A})\right)\right)_{\tau\in\Sigma}
	\end{equation}
	where we decompose the Levi subalgebra $\fm_{P_{\tau}}=\oplus_{i=1}^{s_{\tau}}\fm_{\tau,i}$ of $\fp_{\tau}$, where $\fp_{\tau}$ denotes the Lie algebra of $P_{\tau}$, according to the projections to the graded pieces of the filtrations, $M_{\tau,i}$ is the subgroup associated with $\fm_{\tau,i}$, $\ft_{\tau,i}=\ft_{\tau}\cap \fm_{\tau,i}$ and we identify every $\mathrm{gr}^{\tau,-k_{\tau,i}}(\nu_{W_A})$ as an element in $\fm_{\tau,i}$. The map $\gamma_{M_{\tau,i}}:\fm_{\tau,i}(A)\rightarrow (\ft_{\tau,i}/W_{M_{\tau,i}})(A)$ defined in \S\ref{subsectiontheweightmapthevariety} is no more than sending a matrix to (the coefficients of) its characteristic polynomial. \par
	We need to prove that the $\Theta_x(x_A)\in \widehat{T}_{P,w,(0,0)}(A)$. By the definition of $ T_{P,w}$ in \S\ref{subsectiontheweightmapthevariety}, we only need to verify the following equality
	\begin{equation}\label{equationkappa1kappa2}
		\kappa_2(x_{A,x_{\mathrm{pdR}}})=w^{-1}(\kappa_{1}(x_{A,x_{\mathrm{pdR}}}))
	\end{equation}
	in $\ft/W_P(A)$. The strategy is as in \cite[Lem. 3.7.4]{breuil2019local}. We will compare two factorizations of the Sen polynomial of $W_A^+/tW_A^+$. The first factorization (\ref{equationfactorsthroughsen1}) will be related to the Hodge filtrations and the second one (\ref{equationfactorsthroughsen2}) will be given by the trianguline filtrations.\par
	First, by \cite[Lem. 3.7.5]{breuil2019local}, the Sen polynomial of $W_A^+/tW_A^+$ in 
	\[A\otimes_{\Q_p}K[Y]=\prod_{\tau\in\Sigma}(A\otimes_{\Q_p}K)\otimes_{A\otimes_{\Q_p}K, 1\otimes \tau} A[Y]=\prod_{\tau\in\Sigma} A[Y]\]
	is equal to 
	\begin{align}\label{equationfactorsthroughsen1}
		&\prod_{i\in\Z}\mathrm{det}\left( Y\mathrm{Id}+i\mathrm{Id}-\mathrm{gr}^{i}(\nu_{W_A}) \mid \mathrm{gr}^{i}_{\Fil^{\bullet}_{W_A^+}}\left(D_{\pdR}(W_A)\right) \right) \\
		=&\prod_{\tau\in\Sigma}\prod_{i=1}^{s_{\tau}}\mathrm{det}\left(Y\mathrm{Id}-k_{\tau,i}\mathrm{Id}-\mathrm{gr}^{\tau,-k_{\tau,i}}(\nu_{W_A})\mid \mathrm{gr}^{-k_{\tau,i}}_{\Fil^{\bullet}_{W_A^+}}(D_{\pdR,\tau}(W_A))\right).\nonumber
	\end{align}
	On the other hand, by \cite[Lem. 3.7.6]{breuil2019local}, the Sen polynomial is 
	\begin{align}\label{equationfactorsthroughsen2}
		\prod_{i=1}^n\left(Y-\mathrm{wt}(\delta_{A,i})\right)&=\prod_{\tau\in\Sigma}\prod_{i=1}^n\left(Y-\mathrm{wt}_{\tau}(\delta_i)-(\mathrm{wt}_{\tau}(\delta_{A,i})-\mathrm{wt}_{\tau}(\delta_i))\right).
	\end{align}
	Comparing (\ref{equationfactorsthroughsen1}) and (\ref{equationfactorsthroughsen2}), we get an equality in $A[Y]$ of the $\tau$-Sen polynomial of $W_A^+/tW_A^+$ for each $\tau\in\Sigma$:
	\begin{align}\label{equationfactorskappa}
		&\prod_{i=1}^n(Y-\mathrm{wt}_{\tau}(\delta_i)-(\mathrm{wt}_{\tau}(\delta_{A,i})-\mathrm{wt}_{\tau}(\delta_i)))\\
		=&\prod_{i=1}^{s_{\tau}}\mathrm{det}\left(Y\mathrm{Id}-k_{\tau,i}\mathrm{Id}-\mathrm{gr}^{\tau,-k_{\tau,i}}(\nu_{W_A})\mid \mathrm{gr}^{-k_{\tau,i}}_{\Fil^{\bullet}_{W_A^+}}\left(D_{\pdR,\tau}(W_A)\right)\right).\nonumber
	\end{align}
	Modulo $\fm_A$, the right hand side of (\ref{equationfactorskappa}) calculates the $\tau$-Sen polynomial of $W^+/tW^+$ in $L[Y]$ in the usual way using the Hodge-Tate weights:
	\begin{align*}
		&\prod_{i=1}^{s_{\tau}}\mathrm{det}\left(Y\mathrm{Id}-k_{\tau,i}\mathrm{Id}-\mathrm{gr}^{\tau,-k_{\tau,i}}(\nu_{W_A})\mid \mathrm{gr}^{-k_{\tau,i}}_{\Fil^{\bullet}_{W_A^+}}(D_{\pdR,\tau}(W_A))\right)\\
		\equiv &\prod_{i=1}^{s_{\tau}}\mathrm{det}\left(Y\mathrm{Id}-k_{\tau,i}\mathrm{Id}-\mathrm{gr}^{\tau,-k_{\tau,i}}(\nu_{W})\mid \mathrm{gr}^{-k_{\tau,i}}_{\Fil^{\bullet}_{W^+}}(D_{\pdR,\tau}(W))\right) \nonumber\\
		=&\prod_{i=1}^{s_{\tau}}(Y-k_{\tau,i})^{m_{\tau,i}}\nonumber
	\end{align*}
	where the last equality uses the fact that $\mathrm{gr}^{\tau,-k_{\tau,i}}(\nu_{W})\in \fm_{\tau,i}(L)$ is nilpotent (since $\nu_W$ is nilpotent). Actually, for each $i$,
	\begin{equation}\label{equationfactorsmoduo}
		\mathrm{det}\left(Y\mathrm{Id}-k_{\tau,i}\mathrm{Id}-\mathrm{gr}^{\tau,-k_{\tau,i}}(\nu_{W_A})\mid \mathrm{gr}^{-k_{\tau,i}}_{\Fil^{\bullet}_{W_A^+}}(D_{\pdR,\tau}(W_A))\right)\equiv (Y-k_{\tau,i})^{m_{\tau,i}} \mod \fm_A.
	\end{equation}
	As we have assumed that $k_{\tau,i}\neq k_{\tau,j}$ if $i\neq j$ and each $\mathrm{wt}_{\tau}(\delta_{A,j})-\mathrm{wt}_{\tau}(\delta_j)$ is in $\fm_A$, we get that $\mathrm{wt}_{\tau}(\delta_j)+(\mathrm{wt}_{\tau}(\delta_{A,j})-\mathrm{wt}_{\tau}(\delta_j))-k_{\tau,i}\in\fm_A$ if and only if $\mathrm{wt}_{\tau}(\delta_{j})=k_{\tau,i}$. Apply Lemma \ref{lemmapolynomial} below (where we take $k_i=k_{\tau,i},m_{i}=m_{\tau,i}$ and $a_{i,1},\cdots,a_{i,m_{i}}$ are those $\mathrm{wt}_{\tau}(\delta_j)+(\mathrm{wt}_{\tau}(\delta_{A,j})-\mathrm{wt}_{\tau}(\delta_j))$ such that $\mathrm{wt}_{\tau}(\delta_j)=k_{\tau,i}$) for (\ref{equationfactorskappa}) and using (\ref{equationfactorsmoduo}), we get 
	\begin{align*}
		&\prod_{\mathrm{wt}_{\tau}(\delta_{j})=k_{\tau,i}}\left(Y-\mathrm{wt}_{\tau}(\delta_j)-(\mathrm{wt}_{\tau}(\delta_{A,j})-\mathrm{wt}_{\tau}(\delta_j))\right)\\
		=&\mathrm{det}\left(Y\mathrm{Id}-k_{\tau,i}\mathrm{Id}-\mathrm{gr}^{\tau,-k_{\tau,i}}(\nu_{W_A})\mid \mathrm{gr}^{-k_{\tau,i}}_{\Fil^{\bullet}_{W_A^+}}(D_{\pdR,\tau}(W_A))\right).
	\end{align*}
	Replace $Y$ with $Y+k_{\tau,i}$ in the above identity, we get 
	\begin{equation}\label{equationfactoridentity}
		\prod_{\mathrm{wt}_{\tau}(\delta_{j})=k_{\tau,i}}(Y-(\mathrm{wt}_{\tau}(\delta_{A,j})-\mathrm{wt}_{\tau}(\delta_j)))=\mathrm{det}\left(Y\mathrm{Id}-\mathrm{gr}^{\tau,-k_{\tau,i}}(\nu_{W_A})\mid \mathrm{gr}^{-k_{\tau,i}}_{\Fil^{\bullet}_{W_A^+}}(D_{\pdR,\tau}(W_A))\right)
	\end{equation}
	for any $\tau\in\Sigma, i=1,\cdots,s_{\tau}$. \par
	In the following, we verify that (\ref{equationfactoridentity}) above for all $\tau$ and $i$ implies (and is equivalent to) the equality (\ref{equationkappa1kappa2}) which we want to prove.\par 
	Fix an arbitrary lift $w=(w_{\tau})_{\tau\in\Sigma}$ in $W$ with the same notation for $w\in W/W_P$. The $\tau$-part of $w^{-1}(\kappa_{1}(x_{A,x_{\mathrm{pdR}}}))$ is (by (\ref{equationfactorskappa1}))
	\begin{align*}
		&w_{\tau}^{-1}(\mathrm{wt}\left(\delta_{A,1})-\mathrm{wt}(\delta_1),\cdots,\mathrm{wt}(\delta_{A,n})-\mathrm{wt}(\delta_n)\right)\\
		=&\left(\mathrm{wt}(\delta_{A,w_{\tau}(1)})-\mathrm{wt}(\delta_{w_{\tau}(1)}),\cdots,\mathrm{wt}(\delta_{A,w_{\tau}(n)})-\mathrm{wt}(\delta_{w_{\tau}(n)})\right)	
	\end{align*} 
	whose image in $\ft_{\tau}/W_{P_{\tau}}(A)=\ft_{\tau,1}/W_{M_{\tau,1}}(A)\times\cdots \ft_{\tau,s_{\tau}}/W_{M_{\tau,s_{\tau}}}(A)$ is (we use characteristic polynomials to denote the image of an element of $\ft_{\tau,i}(A)$ in $\ft_{\tau,i}/W_{M_{\tau,i}}(A)$)
	\begin{equation*}
		\left(\prod_{i=1}^{m_{\tau,1}}\left(Y-\left(\mathrm{wt}_{\tau}(\delta_{A,w_{\tau}(i)})-\mathrm{wt}_{\tau}(\delta_{w_{\tau}(i)})\right)\right),\cdots, \prod_{i=n-m_{\tau,s_{\tau}}+1}^n\left(Y-\left(\mathrm{wt}_{\tau}(\delta_{A,w_{\tau}(i)})-\mathrm{wt}_{\tau}(\delta_{w_{\tau}(i)})\right)\right) \right).
	\end{equation*}
	Recall $w$ is chosen so that
	\[\left(\mathrm{wt}_{\tau}(\delta_{w_{\tau}(1)}),\cdots,\mathrm{wt}_{\tau}(\delta_{w_{\tau}(n)})\right)=(h_{\tau,1},\cdots,h_{\tau,n})=(\underbrace{k_{\tau,1},\cdots,k_{\tau,1}}_{m_{\tau,1}},\cdots,\underbrace{k_{\tau,s_{\tau}},\cdots,k_{\tau,s_{\tau}}}_{m_{\tau,s_{\tau}}})\] 
	for all $\tau\in\Sigma$. Hence the $\tau$-part of $w^{-1}(\kappa_{1}(x_{A,x_{\mathrm{pdR}}}))$ can be furthermore rewritten as 
	\begin{equation*}
		\left(\prod_{\mathrm{wt}_{\tau}(\delta_{j})=k_{\tau,1}}\left(Y-\left(\mathrm{wt}_{\tau}(\delta_{A,j})-\mathrm{wt}_{\tau}(\delta_j)\right)\right),\cdots, \prod_{\mathrm{wt}_{\tau}(\delta_{j})=k_{\tau,s_{\tau}}}\left(Y-\left(\mathrm{wt}_{\tau}(\delta_{A,j})-\mathrm{wt}_{\tau}(\delta_j)\right)\right) \right).
	\end{equation*}
	Using (\ref{equationfactoridentity}), the above element is equal to 
	\begin{align*}
		&\left(\mathrm{det}\left(Y\mathrm{Id}-\mathrm{gr}^{\tau,-k_{\tau,1}}(\nu_{W_A})\right),\cdots,\mathrm{det}\left(Y\mathrm{Id}-\mathrm{gr}^{\tau,-k_{\tau,s_{\tau}}}(\nu_{W_A})\right)\right)\\
		=&\left( \gamma_{M_{\tau,1}}\left(\mathrm{gr}^{\tau,-k_{\tau,1}}(\nu_{W_A})\right),\cdots,\gamma_{M_{\tau,s_{\tau}}}\left(\mathrm{gr}^{\tau,-k_{\tau,s_{\tau}}}(\nu_{W_A})\right)\right)
	\end{align*}
	which is exactly the $\tau$-part of $\kappa_2(x_{A,x_{\mathrm{pdR}}})$ by (\ref{equationfactorskappa2}). We conclude that \[\kappa_2(x_{A,x_{\mathrm{pdR}}})=w^{-1}(\kappa_{1}(x_{A,x_{\mathrm{pdR}}}))\]
	in $\ft/W_P(A)$. Thus the image of $x_A$ in $\widehat{T}_{P,(0,0)}(A)$ is in $\widehat{T}_{P,w,(0,0)}(A)$. Hence the morphism $\Theta_{x}$ factors through $\widehat{T}_{P,w,(0,0)}\hookrightarrow \widehat{T}_{P,(0,0)}$.
\end{proof}
\begin{lemma}\label{lemmapolynomial}
	Let $k_{1},\cdots,k_{s}$ be pairwise different numbers in $L$. For each $i=1,\cdots,s$, assume that $a_{i,1},\cdots,a_{i,m_i}$ are elements in $A\in\cC_L$ such that $a_{i,j}-k_i\in\fm_A$ for any $j=1,\cdots,m_i$. For each $i=1,\cdots,s$, let $P_{i}(Y)\in A[Y]$ be a monic polynomial of degree $m_i$ such that $P_i(Y)\equiv (Y-k_i)^{m_i} \mod \fm_A$. Assume that $$\prod_{i=1}^{s}\prod_{j=1}^{m_i}(Y-a_{i,j})=\prod_{i=1}^s P_i(Y)$$ in $A[Y]$, then $\prod_{j=1}^{m_i}(Y-a_{i,j})=P_i(Y)$ for each $i=1,\cdots,s$.
\end{lemma}
\begin{proof}
	Let $F(Y)=\prod_{i=1}^{s}\prod_{j=1}^{m_i}(Y-a_{i,j})$. Take any $t\in\left\{1,\cdots,s\right\}$. We have $\prod_{i=1}^sP_i(a_{t,1})=F(a_{t,1})=0$. If $i\neq t$, then $P_i(a_{t,1})\equiv (k_t-k_i)^{m_i}\mod \fm_A$ is not in $\fm_A$ since $k_i\neq k_t$. Hence $P_i(a_{t,1})\in A^{\times}$ if $i\neq t$. We get $P_t(a_{t,1})=0$ in $A$. Hence $P_t(Y)=(Y-a_{t,1})\widetilde{P}_t(Y)$ where $\widetilde{P}_t(Y)$ is a monic polynomial and $\widetilde{P}_t(Y)\equiv (Y-k_t)^{m_{t}-1}\mod \fm_A$. Using the fact that
	if there is a monic polynomial $G(Y)\in A[Y]$ such that $G(Y)=(Y-a)G_1(Y)=(Y-a)G_2(Y)$ for some $a\in A$ and $G_1(Y),G_2(Y) \in A[Y]$, then $G_1(Y)=G_2(Y)$, we get
	\[\left(\prod_{i\neq t}\prod_{j=1}^{m_i}(Y-a_{i,j})\right)\prod_{h=2}^{m_t}(Y-a_{t,h})=\left(\prod_{i\neq t} P_i(Y)\right)\widetilde{P}_t(Y).\] 
	Repeat the argument we find $P_t(Y)=(Y-a_{t,1})\cdots(Y-a_{t,m_t})$.
\end{proof}
For $w'\in W/W_P$, we define $X_{r,\cM_{\bullet}}^{w'}=X_{r,\cM_{\bullet}}\times_{X_{D,\cM_{\bullet}}}X_{D,\cM_{\bullet}}^{w'}$. The functor $|X_{r,\cM_{\bullet}}|$ is pro-represented by a reduced equidimensional local ring $R_{r,\cM_{\bullet}}$ of dimension $n^{2}+[K:\Q_p]\frac{n(n+1)}{2}$ with minimal ideals $\fp_{w'},w'\in W/W_P$ such that $R_{r,\cM_{\bullet}}^{w'}:=R_{r,\cM_{\bullet}}/\fp_{w'}$ pro-represents $ |X_{r,\cM_{\bullet}}^{w'}|$ (cf. \cite[Thm. 3.6.2 (i)(ii)]{breuil2019local}, using Proposition \ref{propositionrepXPx}, Corollary \ref{corollaryrepXPxw}).  
\begin{corollary}\label{corollarylocalmodelisomorphism}
	The closed immersion $\widehat{\trivar}_x\hookrightarrow X_{r,\cM_{\bullet}}$ induces an isomorphism $\widehat{\trivar}_x\simrightarrow X_{r,\cM_{\bullet}}^{w}$.
\end{corollary}
\begin{proof}
	The proof is the same with that of \cite[Cor. 3.7.8]{breuil2019local} using Proposition \ref{propositionfactorsthroughTPw}. By \cite[\S2.2]{breuil2017interpretation} and discussions above, both $\Spec(\widehat{\cO}_{\trivar,x})$ and $\Spec(R_{r,\cM_{\bullet}})$ are reduced equidimensional of dimension $n^{2}+[K:\Q_p]\frac{n(n+1)}{2}$, thus the image of $\Spec(\widehat{\cO}_{\trivar,x})\hookrightarrow\Spec(R_{r,\cM_{\bullet}})$ is a union of irreducible components of $\Spec(R_{r,\cM_{\bullet}})$ of the form $\Spec(R_{r,\cM_{\bullet}}^{w'})$ for some $w'\in W/W_P$. For any such $w'$, the morphism $\Spec( R_{r,\cM_{\bullet}}^{w'})\hookrightarrow \Spec(\widehat{\cO}_{\trivar,x})$ induces a closed immersion $X_{r,\cM_{\bullet}}^{w'}\hookrightarrow \widehat{\trivar}_x$ of formal schemes. The morphism $\widehat{\trivar}_{x}\hookrightarrow X_{r,\cM_{\bullet}}\rightarrow \widehat{T}_{P,(0,0)}$ factors through $\widehat{T}_{P,w,(0,0)}\hookrightarrow \widehat{T}_{P,(0,0)}$ by Proposition \ref{propositionfactorsthroughTPw}. Hence $X_{r,\cM_{\bullet}}^{w'}\rightarrow \widehat{T}_{P,(0,0)}$ factors through $\widehat{T}_{P,w,(0,0)}\hookrightarrow \widehat{T}_{P,(0,0)}$. We get that $w'=w$ in $W/W_P$ by Lemma \ref{lemmafactorthroughXPw} (cf. \cite[Thm. 3.6.2 (iii)]{breuil2019local}).
\end{proof}
Define $X_{V,\cM_{\bullet}}^{\square,w}:=X_{V,\cM_{\bullet}}^w\times_{X_{D,
\cM_{\bullet}}^w}X_{D,
\cM_{\bullet}}^{\square,w},X_{r,\cM_{\bullet}}^{\square,w}:=X_{r,\cM_{\bullet}}^w\times_{X_{V,
\cM_{\bullet}}^w}X_{V,
\cM_{\bullet}}^{\square,w}$. There are morphisms 
\begin{equation}\label{equationlocalmodel}
	\widehat{\trivar}_x\simrightarrow X_{r,\cM_{\bullet}}^w\leftarrow X_{r,
	\cM_{\bullet}}^{\square,w}\rightarrow X_{V,
	\cM_{\bullet}}^{\square,w}\simeq  X_{D,
	\cM_{\bullet}}^{\square,w}\rightarrow X_{W^+,\cF_{\bullet}}^{\square,w}\simeq \widehat{X}_{P,w,x_\mathrm{pdR}}	
\end{equation} 
where all morphisms are either isomorphisms or formally smooth and all groupoids are pro-representable. Let $w_x\in W/W_P$ be the unique element such that $x_{\mathrm{pdR}}\in V_{P,w_x}$ (see (\ref{equationdefinitionVPw})). 
\begin{theorem}\label{theoremirreducibletriangullinevariety}
	Let $x=(r,\underline{\delta})\in\trivar(L)$ be a point such that $\underline{\delta}$ is locally algebraic and both $\delta_i/\delta_j$ and $\epsilon\delta_i/\delta_j$ are not algebraic for any $i\neq j$. Then the trianguline variety $\trivar$ is irreducible at $x$ and we have formally smooth morphisms $\widehat{\trivar}_x\leftarrow |X_{r,
	\cM_{\bullet}}^{\square,w}| \rightarrow \widehat{X}_{P,w,x_\mathrm{pdR}}$ of formal schemes. Moreover, $w_x\leq w$ in $W/W_P$.
\end{theorem}
\begin{proof}
	The first assertions follow from Corollary \ref{corollarylocalmodelisomorphism} and (\ref{equationlocalmodel}). We remain to prove $w_x\leq w$ in $W/W_P$. Since $X_{r,\cM_{\bullet}}^w\simeq\widehat{\trivar}_x\neq \emptyset$, we get $\widehat{X}_{P,w,x_\mathrm{pdR}}\neq \emptyset$. Hence $x_{\pdR}\in X_{P,w}(L)$. Since $X_{P,w}\cap V_{P,w'}\neq \emptyset$ only if $ w\geq w'$ by Lemma \ref{lemmaXpwVpw}, we get $w\geq w_x$ in $W/W_P$. 
\end{proof}
\subsection{Partially de Rham cycles}\label{sectioncycles}
We transport the results of cycles on the local models obtained in \S\ref{sectionsteinbergvarieties} to the trianguline variety. We continue to assume that $x=(r,\underline{\delta})\in X_{\mathrm{tri}}(\overline{r})(L)$ is the point fixed in \S\ref{sectiontriangulinevariety}. Then we have a commutative diagram as in \cite[\S4.3]{breuil2019local}:
\begin{center}
    \begin{tikzpicture}[scale=1.3]
        \node (A1) at (0,2) {$X_{r,\cM_{\bullet}}^w$};
        \node (A2) at (2,2) {$X_{r,\cM_{\bullet}}^{\square,w}$};
        \node (A3) at (4,2) {$\widehat{X}_{P,w,x_\mathrm{pdR}}$};
        \node (B1) at (0,1) {$X_{r,\cM_{\bullet}}$};
        \node (B2) at (2,1) {$X_{r,\cM_{\bullet}}^{\square}$};
		\node (B3) at (4,1) {$\widehat{X}_{P,x_\mathrm{pdR}}$};
		\node (C1) at (0,0) {$X_r$};
        \node (C2) at (2,0) {$\widehat{\cT^n_{\underline{\delta}}}$};
		\node (C3) at (4,0) {$\widehat{\ft}$};
        \path[->,font=\scriptsize,>=angle 90]
        (A2) edge node[above]{} (A1)
        (A2) edge node[above]{} (A3)
        (B2) edge node[above]{} (B1)
		(B2) edge node[above]{} (B3)
		(B1) edge node[above]{$\omega_{\underline{\delta}}$} (C2)
		(C2) edge node[above]{$\wt-\wt(\underline{\delta})$} (C3)
		(B3) edge node[right]{$\kappa_1$} (C3)
		;
		\path[right hook->,font=\scriptsize,>=angle 90]
        (A1) edge node[above]{} (B1)
        (A2) edge node[above]{} (B2)
		(A3) edge node[above]{} (B3)
		(B1) edge node[above]{} (C1)
        ;
        \end{tikzpicture}
 \end{center}
where $\omega_{\underline{\delta}}$ is the composite of the map $\omega_{\underline{\delta}}$ defined in \S\ref{subsectiontheweightmapthelocalmodel} for $X_{\cM,\cM_{\bullet}}$ with $X_{r,\cM}\rightarrow X_{\cM,\cM}$ and all horizontal arrows are formally smooth. It follows from the proof of \cite[Prop. 3.7.2]{breuil2019local} that the composite $\widehat{X_{\mathrm{tri}}(\overline{r})}_x\rightarrow X_{r,\cM_{\bullet}}\stackrel{w_{\underline{\delta}}}{\rightarrow} \widehat{\cT^n_{\underline{\delta}}}$ is the completion of the map $\omega':X_{\mathrm{tri}}(\overline{r})\hookrightarrow \mathfrak{X}_{\overline{r}}\times \cT^n_L\rightarrow \cT_{L}^{n}$ at the points $x$ and $\underline{\delta}$ (cf. \cite[(3.30)]{breuil2019local}). Let $X_{\mathrm{tri}}(\overline{r})_{\wt(\underline{\delta})}$ denote the fiber of the map $\wt\circ\omega':X_{\mathrm{tri}}(\overline{r})\rightarrow \widehat{\ft}$ over $\wt(\underline{\delta})$. Let $\cR_{r,\cM_{\bullet}},\cR_{r,\cM_{\bullet}}^{w}$ be the complete local rings that pro-represent the groupoids $X_{r,\cM_{\bullet}}, X_{r,\cM_{\bullet}}^{w}$ and let $\cR_{r,\cM_{\bullet}}^{\square},\cR_{r,\cM_{\bullet}}^{\square,w}$ be the square versions. The above diagram of pro-representable groupoids corresponds to a diagram of spectra of complete local rings. Now take the fibers over $0\in\Spec(\widehat{\cO}_{\ft,0})$, we get the corresponding morphisms
\begin{center}
    \begin{tikzpicture}[scale=1.3]
        \node (A1) at (0,2) {$\Spec(\overline{R}_{r,\cM_{\bullet}}^w )$};
        \node (A2) at (3,2) {$\Spec(\overline{R}_{r,\cM_{\bullet}}^{\square,w})$};
        \node (A3) at (6,2) {$\Spec(\widehat{\cO}_{\overline{X}_{P,w},x_\mathrm{pdR}})$};
        \node (B1) at (0,1) {$\Spec(\overline{R}_{r,\cM_{\bullet}})$};
        \node (B2) at (3,1) {$\Spec(\overline{R}_{r,\cM_{\bullet}}^{\square})$};
		\node (B3) at (6,1) {$\Spec(\widehat{\cO}_{\overline{X}_{P},x_\mathrm{pdR}})$};
		\node (C1) at (0,0) {$\Spec(\widehat{\cO}_{\fX_{\overline{r}},r})$};
        \path[->,font=\scriptsize,>=angle 90]
        (A2) edge node[above]{} (A1)
        (A2) edge node[above]{} (A3)
        (B2) edge node[above]{} (B1)
		(B2) edge node[above]{} (B3)
		;
		\path[right hook->,font=\scriptsize,>=angle 90]
        (A1) edge node[above]{} (B1)
        (A2) edge node[above]{} (B2)
		(A3) edge node[above]{} (B3)
		(B1) edge node[above]{} (C1)
        ;
        \end{tikzpicture}
\end{center}
where $\overline{X}_P$ and $\overline{X}_{P,w}$ are defined in the end of \S\ref{sectionsteinbergvarieties}. We know from \S\ref{sectionsteinbergvarieties} that the set of irreducible components of $\Spec(\widehat{\cO}_{\overline{X}_{P},x_\mathrm{pdR}})$ is the disjoint union of the sets of irreducible components of $\Spec(\widehat{\cO}_{Z_{P,w'},x_\mathrm{pdR}})$ for $w'\in W/W_P$ such that $x_{\mathrm{pdR}}\in Z_{P,w'}$ and that $\Spec(\widehat{\cO}_{\overline{X}_{P},x_\mathrm{pdR}})$ is equidimensional. By pullback and descent through formally smooth morphisms (cf. \cite[\href{https://stacks.math.columbia.edu/tag/06HL}{Tag 06HL}]{stacks-project}), we get that $\Spec(\overline{R}_{r,\cM_{\bullet}})$ is equidimensional and there is a bijection between the irreducible components of $\Spec(\overline{R}_{r,\cM_{\bullet}})$ and the irreducible components of $\Spec(\widehat{\cO}_{\overline{X}_{P},x_\mathrm{pdR}})$. For each $w'\in W/W_P$, we let $\mathfrak{Z}_{w'}$ denote the union of irreducible components of $\Spec(\overline{R}_{r,\cM_{\bullet}})$ that correspond to irreducible components of $\Spec(\widehat{\cO}_{Z_{P,w'},x_\mathrm{pdR}})$. By base change, we get that $\Spec(\overline{R}_{r,\cM_{\bullet}}^w )$ is a union of irreducible components of $\Spec(\overline{R}_{r,\cM_{\bullet}})$ according to the formula (\ref{formulairreduciblecomponents}).
\begin{remark}
	In \cite[\S4.3]{breuil2019local}, $\mathfrak{Z}_{w'}$ is defined to be a formal sum of irreducible components (cycles). Here we only consider $\mathfrak{Z}_{w'}$ as a set-theoretic union of irreducible components or the underlying reduced subscheme which suffices for our applications. We also do not consider the Kazhdan-Lusztig cycles on $\Spec(\overline{R}_{r,\cM_{\bullet}})$ defined as \cite[(4.7)]{breuil2019local}. See \S\ref{sectioncharacteristiccycles}, especially Remark \ref{remarkcyclesKtheory}.
\end{remark}
Take a standard parabolic subgroup $Q=\prod_{\tau\in\Sigma}Q_{\tau}$ of $G=\prod_{\tau\in\Sigma}\GL_{n/L}$. Suppose that for each $\tau\in \Sigma$, the standard Levi subgroup $M_{Q_{\tau}}$ consists of diagonal block matrices of the form $\GL_{q_{\tau,1}/L}\times\cdots\times \GL_{q_{\tau,t_{\tau}}/L}$ where $q_{\tau,1}+\cdots+q_{\tau,t_{\tau}}=n$. For $i=1,\cdots,t_{\tau}$, let $\widetilde{q}_{i}=q_{\tau,1}+\cdots+q_{\tau,i}$. Let $\widetilde{q}_{\tau,0}=0$.
\begin{definition}\label{definitionQderham}
	A pair $(r_A, \cM_{A,\bullet})$ where $r_A$ is a continuous representation of $\cG_K$ of rank $n$ over a finite-dimensional local $L$-algebra $A$ and $\cM_{A,\bullet}$ is a triangulation of $D_{\mathrm{rig}}(r_A)[\frac{1}{t}]$ such that $W_{\mathrm{dR}}\left(D_{\mathrm{rig}}(r_A)[\frac{1}{t}]\right)$ is almost de Rham is said to be $Q_{\tau}$-de Rham if the nilpotent operator $\nu_A$ on $D_{\mathrm{pdR}, \tau}\left(W_{\mathrm{dR}}\left(D_{\mathrm{rig}}(r_A)[\frac{1}{t}]\right)\right)$ vanishes when restricted to the graded pieces
	\[D_{\mathrm{pdR},\tau}\left(W_{\mathrm{dR}}(\cM_{A,\widetilde{q}_{\tau,i}})\right)/D_{\mathrm{pdR},\tau}\left(W_{\mathrm{dR}}(\cM_{A,\widetilde{q}_{\tau,i-1}})\right),i=1,\cdots,t_{\tau}\] 
	of the sub-filtration of the full filtration $D_{\mathrm{pdR},\tau}\left(W_{\mathrm{dR}}(\cM_{A,\bullet})\right)$. A such pair $(r_A, \cM_{A,\bullet})$ is said to be $Q$-de Rham if it is $Q_{\tau}$-de Rham for all $\tau\in\Sigma$.
\end{definition}
We have defined a closed subscheme $Z_{Q,P}$ of $\overline{X}_{P}$ in \S\ref{sectionsteinbergvarieties}. Moreover, there is a closed immersion $\Spec(\widehat{\cO}_{Z_{Q,P},x_\mathrm{pdR}})\hookrightarrow\Spec(\widehat{\cO}_{\overline{X}_{P},x_\mathrm{pdR}})$ (we do not assume that $\Spec(\widehat{\cO}_{Z_{Q,P},x_\mathrm{pdR}})$ is not empty). We define $R_{r,\cM_{\bullet}}^{\square,Q}:=\overline{R}_{r,\cM_{\bullet}}^{\square}\otimes_{\widehat{\cO}_{\overline{X}_{P},x_\mathrm{pdR}}}\widehat{\cO}_{Z_{Q,P},x_\mathrm{pdR}}$. Let $|X_{r,\cM_{\bullet}}^{\square,Q}|:=\Spf(R_{r,\cM_{\bullet}}^{\square,Q})$ and $X_{r,\cM_{\bullet}}^{\square,Q}:=X_{r,\cM_{\bullet}}^{\square}\times_{|X_{r,\cM_{\bullet}}^{\square}|}|X_{r,\cM_{\bullet}}^{\square,Q}|$. Let $X_{r,\cM_{\bullet}}^Q$ be the image of the subgroupoid $X_{r,\cM_{\bullet}}^{\square,Q}\subset X_{r,\cM_{\bullet}}^{\square}$ in $X_{r,\cM_{\bullet}}$.
\begin{lemma}\label{lemmagroupoidspartiallydeRham}
	\begin{enumerate}
	\item The full subgroupoid $X_{r,\cM_{\bullet}}^{\square,Q}$ of $X_{r,\cM_{\bullet}}^{\square}$ consists of objects 
	\[(A,r_A, \cM_{A,\bullet},j_A,\alpha_A)\] 
	where $A\in\cC_L$, $r_A$ lifts $r$ and $(D_{\mathrm{rig}}(r_A)[\frac{1}{t}],\cM_{A,\bullet},j_A,\alpha_A )\in X_{\cM,\cM_{\bullet}}^{\square}(A)$ such that $j_A$ is the natural one and the pair $(r_A,\cM_{A,\bullet})$ is $Q$-de Rham.
	\item The full subgroupoid $X_{r,\cM_{\bullet}}^{Q}$ of $X_{r,\cM}$ consists of objects $(A,r_A, \cM_{A,\bullet},j_A)\in X_{r,\cM_{\bullet}}$ such that the pair $(r_A,\cM_{A,\bullet})$ is $Q$-de Rham. The inclusion $X_{r,\cM_{\bullet}}^Q\hookrightarrow X_{r,\cM_{\bullet}}$ is relatively representable and is a closed immersion. Moreover, $X_{r,\cM_{\bullet}}^{\square,Q}=X_{r,\cM_{\bullet}}^Q\times_{X_{r,\cM_{\bullet}}}X_{r,\cM_{\bullet}}^{\square}$.
\end{enumerate}
\end{lemma}
\begin{proof}
	(1) Recall $\alpha_A:(A\otimes_{\Q_p}K)^{n}\simrightarrow D_{\mathrm{pdR}}\left(W_{\mathrm{dR}}\left(D_{\mathrm{rig}}(r_A)[\frac{1}{t}]\right)\right)$. Let $(\nu, g_1B)\in \widehat{\widetilde{\fg}}(A)$ be the $A$-point corresponding to $\left(A,W_{\mathrm{dR}}\left(D_{\mathrm{rig}}(r_A)[\frac{1}{t}]\right),W_{\mathrm{dR}}(\cM_{A,\bullet}),\iota_A,\alpha_A\right)\in X_{W,\cF_{\bullet}}^{\square}$ via $|X_{W,\cF_{\bullet}}^{\square}|\simeq \widehat{\tildefg}$ in \cite[Cor. 3.1.9]{breuil2019local} induced by $D_{\mathrm{pdR}}$. It follows from the definition that $(r_A,\cM_{A,\bullet})$ is $Q$-de Rham if and only if $\Ad(g_1^{-1})\nu\in \fn_{Q}(A)$. Hence $\left(W_{\mathrm{dR}}^+\left(D_{\mathrm{rig}}(r_A)\right),W_{\mathrm{dR}}(\cM_{A,\bullet}),\iota_A,\alpha_A\right)\in |X_{W^+,\cF_{\bullet}}^{\square}|(A)$ satisfies that $(r_A,\cM_{A,\bullet})$ is $Q$-de Rham if and only if the corresponding object $(\nu,g_1B,g_2P)\in \widehat{X}_{P,x_{\mathrm{pdR}}}(A)$ under the isomorphism $|X_{W^+,\cF_{\bullet}}^{\square}|\simeq \widehat{X}_{P,x_{\mathrm{pdR}}}$ in (1) of Proposition \ref{propositionrepXPx} satisfies $\Ad(g_1^{-1})\nu\in\fn_{Q}(A)$, in other words, if and only if $(\nu,g_1B,g_2P)\in \widehat{Z}_{Q,P,x_{\mathrm{pdR}}}(A)$. Hence by definition, $|X_{r,\cM_{\bullet}}^{\square,Q}|$ is the subfunctor of $|X_{r,\cM_{\bullet}}^{\square}|$ sending $A\in\cC_L$ to the isomorphism classes of $(r_A, \cM_{A,\bullet},j_A,\alpha_A)$ such that $(r_A,\cM_{A,\bullet})$ is $Q$-de Rham. The description for $X_{r,\cM_{\bullet}}^{\square,Q}=X_{r,\cM_{\bullet}}^{\square}\times_{|X_{r,\cM_{\bullet}}^{\square}|}|X_{r,\cM_{\bullet}}^{\square,Q}|$ follows from the definition of the fiber product (see \cite[(A.4)]{kisin2009moduli}).\par
	(2) The condition for $(A,r_A, \cM_{A,\bullet},j_A,\alpha_A)\in X_{r,\cM_{\bullet}}^{\square}$ that $(r_A,\cM_{A,\bullet})$ is $Q$-de Rham in (1) is independent of the framing $\alpha_A$. Hence $X_{r,\cM_{\bullet}}^{\square,Q}=X_{r,\cM_{\bullet}}^Q\times_{X_{r,\cM_{\bullet}}}X_{r,\cM_{\bullet}}^{\square}$. The other statements for $X_{r,\cM_{\bullet}}^{Q}$ are now obvious.
\end{proof}
Hence $|X_{r,\cM_{\bullet}}^Q|$ is pro-represented by a formal scheme $\Spf(R_{r,\cM_{\bullet}}^Q)$ with a formally smooth morphism $\Spf(R_{r,\cM_{\bullet}}^{\square,Q})\rightarrow \Spf(R_{r,\cM_{\bullet}}^Q)$.
There is a closed immersion $\Spec(R_{r,\cM_{\bullet}}^Q)\hookrightarrow \Spec(\overline{R}_{r,\cM_{\bullet}})$. By Theorem \ref{theoremvarietyweight} and above discussions, $\mathfrak{Z}_{w'}$ is contained in $\Spec(R_{r,\cM_{\bullet}}^Q)$ if and only if $w'(\mathbf{h})$ is strictly $Q$-dominant where $\mathbf{h}=\left(h_{\tau,1},\cdots,h_{\tau,n}\right)_{\tau\in\Sigma}$ is the Hodge-Tate-Sen weights of $r$ in \S\ref{sectiontriangulinevariety}.
\section{Applications on companion points}\label{sectionapplocations}
In this section, we prove our main theorem (Theorem \ref{theoremmaincrystalline}) on the existence of certain companion points on the eigenvariety as well as the appearance of related companion constituents in the space of $p$-adic automorphic forms (Proposition \ref{propositionsocleappear}).
\subsection{Local companion points} \label{sectionlocalcompanionpoints}
We firstly generalize the result in \cite[\S4.2]{breuil2019local} on the existence of all local companion points on the trianguline variety for generic crystalline representations. We continue to assume that $K$ is a finite extension of $\Q_p$ and $L/\Q_p$ is a sufficiently large coefficient field as in \S\ref{sectionlocalmodeltriangulinevariety}.\par
Let $r:\cG_K\rightarrow \GL_n(L)$ be a crystalline representation of $\cG_K$ which is a deformation of $\overline{r}:\cG_K\rightarrow \GL_n(k_L)$ corresponding to an $L$-point on $\fX_{\overline{r}}$.  Let $D_{\mathrm{cris}}(r)$ be the associated $\varphi$-module of rank $n$ over $L\otimes_{\Q_p} K_0$ equipped with a filtration $\Fil^{\bullet}D_{\mathrm{dR}}(r)$ on $D_{\mathrm{dR}}(r)=(L\otimes_{\Q_p}K)\otimes_{L\otimes_{\Q_p}K_0}D_{\mathrm{cris}}(r)$. We assume that the Hodge-Tate weights of $r$ are $\mathbf{h}=(h_{\tau,1},\cdots,h_{\tau,n})_{\tau\in\Sigma}$ where each $k_{\tau,s_{\tau}}>\cdots>k_{\tau,1}$ appears in the sequence $h_{\tau,n}\geq\cdots\geq h_{\tau,1}$ with multiplicities $m_{\tau,s_{\tau}},\cdots,m_{\tau,1}$ and $m_{\tau,s_{\tau}},+\cdots+m_{\tau,1}=n$ for $\tau\in \Sigma$. Let $W_P$ be the stabilizer subgroup of $\mathbf{h}$ under the action of $W\simeq (\mathcal{S}_n)^{\Sigma}$ as in \S\ref{subsectionrepresentability} or \S\ref{sectiontriangulinevariety}. 
We fix an arbitrary embedding $\tau_0:K_0\hookrightarrow L$. After possibly enlarging $L$, we assume that the eigenvalues $\varphi_1,\cdots,\varphi_n$ of $\varphi^{[K_0:\Q_p]}$ on $L\otimes_{1\otimes \tau_0, L\otimes_{\Q_p}K_0}D_{\mathrm{cris}}(r)$ are all in $L^{\times}$. We say $r$ is \emph{generic} if the eigenvalues satisfy that $\varphi_i\varphi_j^{-1}\notin \{1,p^{[K_0:\Q_p]}\}$ for all $i\neq j$. This generic assumption is independent of the choice of $\tau_0$.\par
We assume that $r$ is generic and fix an ordering $\underline{\varphi}:=(\varphi_1,\cdots,\varphi_n)$ of the eigenvalues, which is called a \emph{refinement} of $r$, denoted by $\cR$. For any $w\in W/W_P$, denote by $z^{w(\mathbf{h})}\mathrm{unr}(\underline{\varphi})$ the character 
\[\left(z^{w(\mathbf{h})_1}\mathrm{unr}(\varphi_1),\cdots,z^{w(\mathbf{h})_n}\mathrm{unr}(\varphi_n)\right)\in \cT_L^n\]
of $(K^{\times})^n$ which lies in $\cT_0^n$ for any $w\in W/W_{P}$ by our generic assumption on $\underline{\varphi}$ (recall that $\cT_0^n$ is defined in the beginning of \S\ref{subsectiongroupoids}). The ordering $\underline{\varphi}$ defines a filtration $\Fil_{\bullet}D_{\mathrm{cris}}(r)$ on $D_{\mathrm{cris}}(r)$, which, under Berger's dictionary (\cite{berger2008equations}), corresponds to a triangulation $D_{\mathrm{rig}}(r)_{\bullet}:D_{\mathrm{rig}}(r)_{1}\subset\cdots\subset D_{\mathrm{rig}}(r)_{n}$ of $D_{\mathrm{rig}}(r)\in\Phi\Gamma_{L,K}^{+}$ (cf. \cite[\S 2.2]{breuil2017smoothness}) such that 
\[D_{\mathrm{rig}}(r)_{i}/D_{\mathrm{rig}}(r)_{i-1}\simeq \cR_{L,K}(z^{w_{\cR}(\mathbf{h})_i}\mathrm{unr}(\varphi_i)),\] 
for all $i=1,\cdots,n$ and certain $w_{\cR}\in W/W_P$ determined by $\cR$. Moreover, the relative position of the filtration $\Fil_{\bullet}D_{\mathrm{cris}}(r)\otimes_{K_0}K$ and the Hodge filtration $\Fil^{\bullet}D_{\mathrm{dR}}(r)$ is parameterized by $w_{\cR}\in W/W_P$ (that is, if we choose a basis of $D_{\mathrm{dR}}(r)$, the $L$-point in $G/B\times G/P$ associated with the two filtrations lies in the Schubert cell $U_{P,w_{\cR}}$ defined in \S\ref{sectionthevariety}).\par
For $w\in W/W_P$, let $x_w\in\mathfrak{X}_{\overline{r}}\times \cT_L^n$ be the point corresponding to the pair $(r, z^{w(\mathbf{h})}\mathrm{unr}(\underline{\varphi}))$. We let $x=x_{w_0}$ be the dominant point and let $w_x=w_{\cR}$ (to make the notation agree with \cite[\S4.2]{breuil2019local}, the element $w_x$ will coincide with the one in Theorem \ref{theoremirreducibletriangullinevariety}). Then $x_{w_x}\in U_{\mathrm{tri}}(\overline{r})\subset \trivar$ by the definition of $U_{\mathrm{tri}}(\overline{r})$. If $x_w\in \trivar$, then $x_w$ satisfies the assumption in Theorem \ref{theoremirreducibletriangullinevariety}. Hence by Theorem \ref{theoremirreducibletriangullinevariety}, $w\geq w_x$ in $W/W_P$. The converse is also true and the following theorem is a plain generalization of \cite[Thm. 4.2.3]{breuil2019local} which asserts that all expected local companion points of $x_{w_0}$ exist on the trianguline variety (cf. \cite[Def. 4.2.1]{breuil2019local}). We repeat the proof here to introduce the notation that will be needed in the proof of Theorem \ref{theoremmaincrystalline}. 
\begin{theorem}\label{theoremlocalcompanionpoint}
	The point $x_w\in \mathfrak{X}_{\overline{r}}\times \cT_L^n$ is in $\trivar$ if and only if $w\geq w_x$ in $W/W_P$.\par
	Moreover, the set of points $x'=(r,\underline{\delta}')\in \mathfrak{X}_{\overline{r}}\times \cT_L^n$ such that $x'$ is in $\trivar$ is equal to
	\[\bigcup_{\cR=\underline{\varphi}}\{ (r, z^{w(\mathbf{h})}\mathrm{unr}(\underline{\varphi})),w\geq w_{\cR}\}\]
	where $\cR$ runs over all possible refinements of $r$.
\end{theorem}
\begin{proof}
We need a variant of Kisin's crystalline deformation space for irregular Hodge-Tate weights that is embedded into the trianguline variety as in \cite[\S2.2]{breuil2017smoothness}. \par
Let $R_{\overline{r}}^{\mathbf{h}-\mathrm{cr}}$ be the \emph{framed} crystalline deformation ring of $p$-adic Hodge type determined by the Hodge-Tate weights $\mathbf{h}$ in the sense of \cite{kisin2008potentially} over $\cO_L$ (reduced and $\Z_p$-flat) and let $\mathfrak{X}_{\overline{r}}^{\mathbf{h}-\mathrm{cr}}$ be the rigid analytic generic fiber of $\mathrm{Spf}(R_{\overline{r}}^{\mathbf{h}-\mathrm{cr}})$. By \cite[Thm. 3.3.8]{kisin2008potentially}, $\mathfrak{X}_{\overline{r}}^{\mathbf{h}-\mathrm{cr}}$ is smooth, equidimensional of dimension $n^{2}+\sum_{\tau\in\Sigma}\sum_{1\leq i< j\leq s_{\tau}}m_{\tau,i}m_{\tau,j}$ over $L$. The beginning part of \cite[\S2.2]{breuil2017smoothness} produces \textit{mutatis mutandis} a rigid analytic space $\widetilde{\mathfrak{X}}_{\overline{r}}^{\mathbf{h}-\mathrm{cr}}=\mathfrak{X}_{\overline{r}}^{\mathbf{h}-\mathrm{cr}}\times_{T_L^{\mathrm{rig}}/\mathcal{S}_n}T_L^{\mathrm{rig}}$ where $T_L^{\mathrm{rig}}$ is the rigid split torus over $L$ of rank $n$, $\mathcal{S}_n$ is the symmetric group so that the quotient ${T_L^{\mathrm{rig}}/\mathcal{S}_n}$ parameterizes characteristic polynomials of the Frobenius on the Weil-Deligne representations associated with the crystalline deformations and $\widetilde{\mathfrak{X}}_{\overline{r}}^{\mathbf{h}-\mathrm{cr}}$ parameterizes pairs $\left(r,(\varphi_1,\cdots,\varphi_n)\right)$ where $r\in \mathfrak{X}_{\overline{r}}^{\mathbf{h}-\mathrm{cr}}$ and $(\varphi_1,\cdots,\varphi_n)$ is an ordering of the eigenvalues of the Frobenius on the Weil-Deligne representation associated with $r$. The same proof of \cite[Lem. 2.2]{breuil2017smoothness} shows that $\widetilde{\mathfrak{X}}_{\overline{r}}^{\mathbf{h}-\mathrm{cr}}$ is reduced. \par
Let $\mathfrak{Y}_{\overline{r}}^{\mathbf{h}-\mathrm{cr}}\rightarrow\mathfrak{X}_{\overline{r}}^{\mathbf{h}-\mathrm{cr}}$ be the ${\mathrm{Res}_{K_0/\Q_p}(\GL_{n/K_0})\times_{\Q_p}L}$-torsor of the trivialization of the underlying coherent sheaf of the universal filtered $\varphi$-modules over $K_0\times_{\Q_p}\cO_{\mathfrak{X}_{\overline{r}}^{\mathbf{h}-\mathrm{cr}}}$ on $\mathfrak{X}_{\overline{r}}^{\mathbf{h}-\mathrm{cr}}$. Then sending a crystalline representation with a trivialization of $D_{\mathrm{cris}}$ to its crystalline Frobenius and Hodge filtration defines a morphism $f:\mathfrak{Y}_{\overline{r}}^{\mathbf{h}-\mathrm{cr}}\rightarrow \left((\mathrm{Res}_{K_0/\Q_p}(\GL_{n/K_0})\times_{\Q_p}L)\times_L G/P\right)^{\mathrm{rig}}$ which is smooth. In fact, by \cite[Prop. 8.17]{hartl2020universal}, $\fX_{\overline{r}}^{\mathbf{h}-\mathrm{cr}}$ is isomorphic to an open subspace $\widetilde{\mathscr{D}}_{\varphi,\mu}^{\mathrm{ad},\mathrm{adm},+}(\overline{r})$ of $\widetilde{\mathscr{D}}_{\varphi,\mu}^{\mathrm{ad},\mathrm{adm}}$ where $\mathscr{D}_{\varphi,\mu}^{\mathrm{ad}}$ is the quotient stack of the adic space associated with $(\mathrm{Res}_{K_0/\Q_p}(\GL_{n/K_0})\times_{\Q_p}L)\times_L G/P$ by the action of ${\mathrm{Res}_{K_0/\Q_p}(\GL_{n/K_0})\times_{\Q_p}L}$, $\mathscr{D}_{\varphi,\mu}^{\mathrm{ad},\mathrm{adm}}$ is an open subspace of $\mathscr{D}_{\varphi,\mu}^{\mathrm{ad}}$ where there is a universal representation of $\cG_K$ on a vector bundle $\mathcal{V}$ on $\mathscr{D}_{\varphi,\mu}^{\mathrm{ad},\mathrm{adm}}$ and $\widetilde{\mathscr{D}}_{\varphi,\mu}^{\mathrm{ad},\mathrm{adm}}$ is the stack over $\mathscr{D}_{\varphi,\mu}^{\mathrm{ad},\mathrm{adm}}$ trivializing $\mathcal{V}$. Then the morphism $f$ induces a smooth morphism: 
\[\widetilde{f}:\widetilde{\mathfrak{X}}_{\overline{r}}^{\mathbf{h}-\mathrm{cr}}\times_{\mathfrak{X}_{\overline{r}}^{\mathbf{h}-\mathrm{cr}}}\mathfrak{Y}_{\overline{r}}^{\mathbf{h}-\mathrm{cr}} \rightarrow \left((\mathrm{Res}_{K_0/\Q_p}(\GL_{n/K_0})\times_{\Q_p}L)^{\mathrm{rig}}\times_{T_{L}^{\mathrm{rig}}/\mathcal{S}_n}T_{L}^{\mathrm{rig}}\right)\times_{L}(G/P)^{\mathrm{rig}},\]
where the map $(\mathrm{Res}_{K_0/\Q_p}(\GL_{n/K_0})\times_{\Q_p}L)^{\mathrm{rig}}\rightarrow T_{L}^{\mathrm{rig}}/\mathcal{S}_n$ is defined by \cite[(9-1)]{hartl2020universal}. 
The condition $\varphi_i\varphi_j^{-1}\notin \{ 1, p^{[K_0,\Q_p]}\}$ for $i\neq j$ cuts out a Zariski open subspace 
\[\mathscr{D}^{\mathrm{gen}}:=\left((\mathrm{Res}_{K_0/\Q_p}(\GL_{n,K_0})\times_{\Q_p}L)^{\mathrm{rig}}\times_{T_{L}^{\mathrm{rig}}/\mathcal{S}_n}T_{L}^{\mathrm{rig}}\right)^{\mathrm{gen}}\times_{L}(G/P)^{\mathrm{rig}}\]
in the target of $\widetilde{f}$ and thus the inverse image of $\mathscr{D}^{\mathrm{gen}}$ under $\widetilde{f}$, denoted by $\widetilde{Z}_{\overline{r}}^{\mathbf{h}-\mathrm{cr}}$, is Zariski open dense in $\widetilde{\mathfrak{X}}_{\overline{r}}^{\mathbf{h}-\mathrm{cr}}\times_{\mathfrak{X}_{\overline{r}}^{\mathbf{h}-\mathrm{cr}}}\mathfrak{Y}_{\overline{r}}^{\mathbf{h}-\mathrm{cr}}$ (smooth morphisms are open). The subspace $\widetilde{Z}_{\overline{r}}^{\mathbf{h}-\mathrm{cr}}$ is invariant under the action of ${\mathrm{Res}_{K_0/\Q_p}(\GL_{n/K_0})\times_{\Q_p}L}$ (change of bases) and thus descends along $\widetilde{\mathfrak{X}}_{\overline{r}}^{\mathbf{h}-\mathrm{cr}}\times_{\mathfrak{X}_{\overline{r}}^{\mathbf{h}-\mathrm{cr}}}\mathfrak{Y}_{\overline{r}}^{\mathbf{h}-\mathrm{cr}}\rightarrow \widetilde{\mathfrak{X}}_{\overline{r}}^{\mathbf{h}-\mathrm{cr}}$ to a Zariski open dense subspace $\widetilde{W}_{\overline{r}}^{\mathbf{h}-\mathrm{cr}}$ of $\widetilde{\fX}_{\overline{r}}^{\mathbf{h}-\mathrm{cr}}$.\par
For any $\varphi$-module $D$ of rank $n$ over $A\otimes_{\Q_p}K_0$ where $A$ is an $L$-algebra, let \[D_{\tau}:=D\otimes_{L\otimes_{\Q_p}K_0, \mathrm{id}\otimes \tau}L\] 
for $\tau\in \Hom(K_0,L)$. Let $\sigma$ be the Frobenius automorphism of $K_0$. Then $\varphi:D_{\tau_0\circ \sigma^i}\rightarrow D_{\tau_0\circ\sigma^{i+1}}$ where $\tau_0\circ\sigma^f=\tau_0\circ \sigma^0=\tau_0$. Given a basis $e_1,\cdots,e_n$ of the $A\otimes_{\Q_p}K_0$-module, equivalent bases $(e_i\otimes_{L\otimes_{\Q_p}K_0, \mathrm{id}\otimes \tau}L)_{i=1,\cdots, n}$ of $D_{\tau}$ for each $\tau\in \Hom(K_0,L)$, the matrix of $\varphi$ is given by $M=(M_{\tau})_{\tau\in \Hom(K_0,L)}\in (\mathrm{Res}_{K_0/\Q_p}(\GL_{n,K_0})\times_{\Q_p}L)(A)$ where $M_{\tau}$ is the matrix of the morphism $\varphi:D_{\tau}\rightarrow D_{\tau\circ\sigma}$ under the given bases. On $\mathscr{D}^{\mathrm{gen}}$, the condition that $\varphi_i\neq \varphi_j$ for $i\neq j$ allows to define a Zariski closed subspace $\mathscr{T}\subset \mathscr{D}^{\mathrm{gen}}$ cut out by the condition that $M_{\tau}=\mathrm{diag}(a_{\tau,1},\cdots,a_{\tau,n})$ are diagonal matrices where $\prod_{\tau\in\Hom(K_0,L)}a_{\tau,i}=\varphi_i$ for all $i=1,\cdots, n$ (this corresponds to the choices of bases $e_1,\cdots,e_n$ of $D$ such that $\varphi^{[K_0:\Q_p]}e_i= \varphi_ie_{i}$). It is easy to see that $\mathscr{T}$ is smooth over $(G/P)^{\mathrm{rig}}$. Let $\widetilde{\mathfrak{T}}_{\overline{r}}^{\mathbf{h}-\mathrm{cr}}:=\widetilde{f}^{-1}(\mathscr{T})$ be the inverse image of $\mathscr{T}$ in $ \widetilde{\mathfrak{X}}_{\overline{r}}^{\mathbf{h}-\mathrm{cr}}\times_{\mathfrak{X}_{\overline{r}}^{\mathbf{h}-\mathrm{cr}}}\mathfrak{Y}_{\overline{r}}^{\mathbf{h}-\mathrm{cr}}$. Then $\widetilde{\mathfrak{T}}_{\overline{r}}^{\mathbf{h}-\mathrm{cr}}\rightarrow \widetilde{W}_{\overline{r}}^{\mathbf{h}-\mathrm{cr}}$ is a $(\mathrm{Res}_{K_0/\Q_p}(\mathbb{G}_{m})\otimes_{\Q_p}L)^n$-torsor corresponding to the trivialization of the $\varphi$-modules with the bases given by eigenvectors of $\varphi^{[K_0,\Q_p]}$ as above (such bases exist locally because the morphism $\varphi^{[K_0,\Q_p]}-\varphi_i$ between projective modules has cokernel and kernel of constant ranks over a reduced base). The map $\widetilde{\mathfrak{T}}_{\overline{r}}^{\mathbf{h}-\mathrm{cr}}\rightarrow \mathscr{T} \rightarrow (G/P)^{\mathrm{rig}}$ is also smooth. For any $w\in W/W_P$, descending along the map $\widetilde{\mathfrak{T}}_{\overline{r}}^{\mathbf{h}-\mathrm{cr}}\rightarrow \widetilde{W}_{\overline{r}}^{\mathbf{h}-\mathrm{cr}}$ for the inverse image in $\widetilde{\mathfrak{T}}_{\overline{r}}^{\mathbf{h}-\mathrm{cr}}$ of the Bruhat cell $(BwP/P)^{\mathrm{rig}}$, where the inverse image is invariant under the action of $(\mathrm{Res}_{K_0/\Q_p}(\mathbb{G}_{m})\otimes_{\Q_p}L)^n$, is a Zariski locally closed subset in $\widetilde{W}_{\overline{r}}^{\mathbf{h}-\mathrm{cr}}$, denoted by $\widetilde{W}_{\overline{r},w}^{\mathbf{h}-\mathrm{cr}}$. Let $\overline{\widetilde{W}_{\overline{r},w}^{\mathbf{h}-\mathrm{cr}}}$ be the Zariski closure of $\widetilde{W}_{\overline{r},w}^{\mathbf{h}-\mathrm{cr}}$ in $\widetilde{W}_{\overline{r}}^{\mathbf{h}-\mathrm{cr}}$. Then we have the usual closure relations $\overline{\widetilde{W}_{\overline{r},w}^{\mathbf{h}-\mathrm{cr}}}=\coprod_{w'\leq w}\widetilde{W}_{\overline{r},w'}^{\mathbf{h}-\mathrm{cr}}$ by a similar argument as in the proof of \cite[Thm. 4.2.3]{breuil2019local} (using that smooth morphisms are open) and descent.\par 
Hence a point $x=(r,\underline{\varphi})$ with the refinement denoted by $\cR$ in $\widetilde{W}_{\overline{r}}^{\mathbf{h}-\mathrm{cr}}$ lies in $\widetilde{W}_{\overline{r},w}^{\mathbf{h}-\mathrm{cr}}$ (resp. $\overline{\widetilde{W}_{\overline{r},w}^{\mathbf{h}-\mathrm{cr}}}$) if and only $w_{\cR}=w$ (resp. $w_{\cR}\leq w$). For any $w\in W/W_P$, there is a morphism $\iota_{\mathbf{h},w}:\mathfrak{X}_{\overline{r}}^{\mathbf{h}-\mathrm{cr}}\times_L T_L^{\mathrm{rig}}\rightarrow \mathfrak{X}_{\overline{r}}\times_L\mathcal{T}_L^n$ that sends $(r,\underline{\varphi})$ to $\left(r,z^{w(\mathbf{h})}\mathrm{unr}(\underline{\varphi})\right)$. We have $\widetilde{W}_{\overline{r},w}^{\mathbf{h}-\mathrm{cr}}\subset \iota_{\mathbf{h},w}^{-1}(U_{\mathrm{tri}}(\overline{r}))$ by \cite{berger2008equations}. Hence $\overline{\widetilde{W}_{\overline{r},w}^{\mathbf{h}-\mathrm{cr}}}\subset \iota_{\mathbf{h},w}^{-1}(\trivar)$ for any $w\in W/W_P$. Thus $\iota_{\mathbf{h},w}(\widetilde{W}_{\overline{r},w_{\cR}}^{\mathbf{h}-\mathrm{cr}})\subset \trivar$ for any $w\geq w_{\cR}$ which finishes the proof of the first statement of the theorem.\par
To prove the last statement of the theorem, by the first statement, we only need to prove that if $x'=(r,\underline{\delta}')\in \trivar$, then $x'=x_w$ for some refinement $\cR$ of $r$ and $w\in W/W_P$. This follows from the bijection between the triangulation of $D_{\mathrm{rig}}(r)$ and the refinements of $r$, \cite[Thm. 6.3.13]{kedlaya2014cohomology} and \cite[Prop. 2.9]{breuil2017interpretation}.
\end{proof}
\subsection{$p$-adic automorphic forms and eigenvarieties}\label{sectionglobalsettings}
We now turn to the global settings in \cite[\S5]{breuil2019local} (see also \cite[\S2.4 \& \S3]{breuil2017interpretation} or \cite[\S3]{breuil2017smoothness}). We recall the basic notation and constructions. The reader should refer to \textit{loc. cit.} for details. We assume from now on $p>2$. \par

Let $F^+$ be a totally real field and $F/F^+$ be a CM quadratic extension such that any prime in $S_p$, the set of places of $F^+$ above $p$, splits in $F$. 

Let $G$ be a unitary group in $n\geq 2$ variables over $F^+$ which is definite at all real places, split over $F$ and quasi-split at all finite places. 

We fix a tame level $U^p=\prod_{v\notin S_p}U_v$ where for any finite place $v\notin S_p$, $U_v$ is an open compact subgroup of $G(F^+_v)$. We assume that $U^p$ is small enough in the sense of \cite[(3.9)]{breuil2017smoothness}

Let $S\supset S_p$ be a finite set of finite places of $F^+$ that split in $F$ and we require that every place of $F^+$ that splits in $F$ such that $U_v$ is not maximal is contained in $S$. For each $v\in S$, we fix a place $\widetilde{v}$ of $F$ above $F^+$.

We fix an isomorphism $G\times_{F^+} F\simeq \GL_{n,F}$ which induces $i_{\tilde{v}}:G(F^+_v)\simrightarrow \GL_n(F_{\tilde{v}})$ for any $v\in S_p$. 

Recall that $L$ denotes the coefficient field, a finite extension of $\Q_p$. We assume that $L$ is large enough such that $|\Hom(F_{\widetilde{v}},L)|=[F_{\widetilde{v}}:\Q_p]$ for all $v\in S_p$.\par

We define the set $\Sigma_v:=\Hom(F_{\widetilde{v}},L)$ for any $v\in S_p$ and let $\Sigma_p=\cup_{v\in S_p}\Sigma_v$. 

Denote by $B_{v}$ (resp. $\overline{B}_v$, resp. $T_v$) the subgroup of $G_v:=G(F^+_v)$ which is the preimage of the group of upper-triangular matrices (resp. lower-triangular matrices, resp. diagonal matrices) of $\GL_n(F_{\tilde{v}})$ under $i_{\tilde{v}}$ and let $B_p=\prod_{v\in S_p} B_v$, $\overline{B}_p=\prod_{v\in S_p} \overline{B}_v$ and $T_p=\prod_{v\in S_p} T_v$. Set $G_p=\prod_{v\in S_p} G_v$. \par

For $v\in S_p$, let $\fg_v$ (resp. $\fb_v$, resp. $\overline{\fb}_v$, resp. $\ft_v$) be the base change to $L$ of the $\Q_p$-Lie algebra of $G_v$ (resp. $B_v$, resp. $\overline{B}_v$, resp. $T_v$). We define $\fg=\prod_{v\in S_p}\fg_v, \ft=\prod_{v\in S_p}\ft_v$, etc. and for $v\in S_p, \tau\in \Sigma_v$, set $\fg_{\tau}=\fg_v\otimes_{F_{\widetilde{v}}\otimes_{\Q_p}L,\tau\otimes\mathrm{id}}L$, $\ft_{\tau}=\ft_v\otimes_{F_{\widetilde{v}}\otimes_{\Q_p}L,\tau\otimes\mathrm{id}}L$, etc..\par

Let $\widehat{S}(U^p,L)$ (resp. $\widehat{S}(U^p,\cO_L)$) be the space of continuous functions $G(F^+)\backslash G(\mathbf{A}_{F^+}^{\infty})/U^p\rightarrow L$ (resp. $G(F^+)\backslash G(\mathbf{A}_{F^+}^{\infty})/U^p\rightarrow \cO_L$) on which $G_p$ acts via right translations. Then there is a Hecke algebra $\mathbb{T}^S$ (a commutative $\cO_L$-algebra, see \cite[\S2.4]{breuil2017interpretation} for details) that acts on $\widehat{S}(U^p,\cO_L)$. We fix a maximal ideal $\mathfrak{m}^S$ of $\mathbb{T}^S$ with residue field $k_L$ (otherwise enlarging $L$) such that 
\[\widehat{S}(U^p,L)_{\mathfrak{m}^S}\neq 0\]
and that the associated Galois representation $\overline{\rho}:\cG_F\rightarrow\GL_n(k_L)$ is absolutely irreducible (i.e. $\mathfrak{m}^S$ is non-Eisenstein) where $\cG_F:=\Gal(\overline{F}/F)$. \par

We assume furthermore the ``standard Taylor-Wiles hypothesis'', that is we require that $F$ is unramified over $F^+$, $F$ contains no non-trivial $p$-th root of unity, $U_v$ is hyperspecial if the place $v$ of $F^+$ is inert in $F$, and $\overline{\rho}\left(\Gal(\overline{F}/F(\sqrt[p]{1}))\right)$ is adequate (cf. \cite[Rem. 1.1]{breuil2019local}). 

The Galois deformation ring $R_{\overline{\rho},S}$, which parameterizes polarized deformations of $\overline{\rho}$ unramified outside $S$, acts on $\widehat{S}(U^p,L)_{\mathfrak{m}^S}$. The subspace of locally $\Q_p$-analytic vectors of $\widehat{S}(U^p,L)_{\mathfrak{m}^S}$ for the action of $G_p$, denoted by $\widehat{S}(U^p,L)_{\mathfrak{m}^S}^{\mathrm{an}}$, is a very strongly admissible locally $\Q_p$-analytic representation of $G_p$ (\cite[Def. 0.12]{emerton2007jacquetII}). The eigenvariety $Y(U^p,\overline{\rho})$ is defined to be the support of the coherent sheaf $\left(J_{B_p}(\widehat{S}(U^p,L)_{\mathfrak{m}^S}^{\mathrm{an}})\right)'$, applying Emerton's Jacquet functor (with respect to the parabolic subgroup $B_p$ of $G_p$) on $\widehat{S}(U^p,L)_{\mathfrak{m}^S}^{\mathrm{an}}$ and then taking the continuous dual, on $\mathrm{Spf}(R_{\overline{\rho},S})^{\mathrm{rig}}\times \widehat{T}_{p,L}$, where $\widehat{T}_{p,L}=\prod_{v\in S_p}\widehat{T}_{v,L}$ denotes the base change to $L$ of the $\Q_p$-rigid space parameterizing continuous characters of $T_p=\prod_{v\in S_p}T_v$. \par

Let $R_{\infty}=\widehat{\bigotimes}_{v\in S}R_{\overline{\rho}_{\widetilde{v}}}'[[x_1,\cdots,x_g]]$ where for a place $\tilde{v}$ of $F$, $R_{\overline{\rho}_{\widetilde{v}}}'$ is the maximal reduced $\Z_p$-flat quotient of the local framed deformation ring of $\overline{\rho}_{\widetilde{v}}:=\overline{\rho}|_{\cG_{F_{\widetilde{v}}}}$ over $\cO_L$ and $g$ is certain determined integer. Then there is an $\cO_L$-module $M_{\infty}$ constructed in \cite{caraiani2016patching} (see \cite[Thm. 3.5]{breuil2017interpretation} and \cite[\S6]{breuil2017smoothness}) equipped with actions of $R_{\infty}$ and $G_p$ so that $\Pi_{\infty}:=M_{\infty}'[\frac{1}{p}]$ is a $R_{\infty}$-admissible Banach representation of $G_p$ (\cite[Def. 3.1]{breuil2017interpretation}). The \emph{patched eigenvariety} $X_p(\overline{\rho})$ is defined to be the support of the coherent sheaf $\fM_{\infty}:=\left(J_{B_p}(\Pi_{\infty}^{R_{\infty}-\mathrm{an}})\right)'$, applying Emerton's Jacquet functor on the subspace of locally $R_{\infty}$-analytic vectors of $\Pi_{\infty}$ (\cite[Def. 3.2]{breuil2017interpretation}) and then taking the continuous dual, on $\mathrm{Spf}(R_{\infty})^{\mathrm{rig}}\times \widehat{T}_{p,L}$. Then a point $x=(r_x,\underline{\delta})\in \mathrm{Spf}(R_{\infty})^{\mathrm{rig}}\times \widehat{T}_{p,L}$ lies in $X_p(\overline{\rho})$ if and only if $\mathrm{Hom}_{T_p}\left(\underline{\delta}, J_{B_p}\left(\Pi_{\infty}^{R_{\infty}-\mathrm{an}}[\mathfrak{m}_{r_x}]\otimes_{k(r_x)}k(x)\right)\right)\neq 0$ where $\fm_{r_x}$ denotes the maximal ideal of $R_{\infty}[\frac{1}{p}]$ corresponding to the point $r_x\in \mathrm{Spf}(R_{\infty})^{\mathrm{rig}}$ (cf. \cite[Prop. 3.7]{breuil2017interpretation}). Recall that $X_p(\overline{\rho})$ is reduced (\cite[Cor. 3.20]{breuil2017interpretation}). The eigenvariety $Y(U^p,\overline{\rho})$ is identified with a Zariski closed subspace of the patched one $X_p(\overline{\rho})$.
\par
We denote by 
\[\mathfrak{X}_{\overline{\rho}_{p}}:=\mathrm{Spf}\left(\widehat{\bigotimes}_{v\in S_p}R'_{\overline{\rho}_{\widetilde{v}}}\right)^{\mathrm{rig}}, \mathfrak{X}_{\overline{\rho}^{p}}:=\mathrm{Spf}\left(\widehat{\bigotimes}_{v\in S\setminus S_p}R'_{\overline{\rho}_{\widetilde{v}}}\right)^{\mathrm{rig}}, \mathbb{U}^g=\mathrm{Spf}(\cO_L[[x_1,\cdots,x_g]])^{\mathrm{rig}}\] and 
$\fX_{\infty}:=\mathrm{Spf}(R_{\infty})^{\mathrm{rig}}\simeq\fX_{\overline{\rho}_{p}}\times \fX_{\overline{\rho}^{p}}\times \mathbb{U}^g$. \par

Let $X_{\mathrm{tri}}(\overline{\rho}_p):=\prod_{v\in S_p} X_{\mathrm{tri}}(\overline{\rho}_{\widetilde{v}})$ which by definition is a Zariski closed subspace of $\fX_{\overline{\rho}_p}\times \widehat{T}_{p,L}$. Denote by $\delta_{B_v}:=|\cdot|_{F_{\widetilde{v}}}^{n-1}\otimes\cdots\otimes|\cdot|_{F_{\widetilde{v}}}^{n-2i+1}\otimes\cdots \otimes |\cdot|^{1-n}_{F_{\widetilde{v}}}$ the smooth modulus character of $B_v$ and $\delta_{B_p}:=\otimes_{v\in S_p}\delta_{B_v}$. Let $\iota_v$ be the automorphism
\[(\delta_{v,1},\cdots,\delta_{v,n})\mapsto \delta_{B_v}\cdot(\delta_{v,1},\cdots,\delta_{v,i}\epsilon^{i-1}, \cdots\delta_{v,n}\epsilon^{n-1})\] 
of $\widehat{T}_{v,L}$ and let $\iota=\prod_{v\in S_p}\iota_v: \widehat{T}_{p,L}\simrightarrow \widehat{T}_{p,L}$. We also use $\iota$ to denote the automorphism of $\fX_{\overline{\rho}_p}\times \widehat{T}_{p,L}:(x,\underline{\delta})\mapsto \left(x,\iota(\underline{\delta})\right)$.   
Then the reduced closed subvariety $X_p(\overline{\rho})$ of $\fX_{\overline{\rho}_{p}}\times \fX_{\overline{\rho}^{p}}\times \mathbb{U}^g\times \widehat{T}_{p,L}$ lies in the Zariski closed subspace $\iota\left(X_{\mathrm{tri}}(\overline{\rho}_p)\right)\times\fX_{\overline{\rho}^{p}}\times \mathbb{U}^g$ and is identified with a union of irreducible components of $\iota\left(X_{\mathrm{tri}}(\overline{\rho}_p)\right)\times\fX_{\overline{\rho}^{p}}\times \mathbb{U}^g$ (\cite[Thm. 3.21]{breuil2017interpretation}) any of which is of the form $\iota(X)\times \mathfrak{X}^p\times \mathbb{U}^g$ where $X$ (resp. $\mathfrak{X}^p$) is an irreducible component of $X_{\mathrm{tri}}(\overline{\rho}_p)$ (resp. $\mathfrak{X}_{\overline{\rho}^p}$). An irreducible component $X$ of $X_{\mathrm{tri}}(\overline{\rho}_p)$ is said to be $\fX^p$-automorphic for an irreducible component $\fX^p$ of $\mathfrak{X}_{\overline{\rho}^p}$ if $\iota(X)\times \mathfrak{X}^p\times \mathbb{U}^g$ is contained in $X_p(\overline{\rho})$. \par
\begin{definition}\label{definitiongenericglobal}
	\begin{enumerate}
		\item A character $\underline{\delta}=(\underline{\delta}_v)_{v\in S_p}\in \widehat{T}_{p,L}$ is called \emph{generic} if for each $v\in S_p$, $\iota_v^{-1}\left(\underline{\delta}_v\right)\in \cT_{v,0}^{n}$, where $\cT_{v,0}^{n}$ denotes the subset $\cT_{0}^{n}$ in \S\ref{subsectiongroupoids} for characters of $(F_{\widetilde{v}}^{\times})^n$. Explicitly, we say $\underline{\delta}$ is generic if ${\delta_{v,i}\delta_{v,j}^{-1}|\cdot|_{F_{\widetilde{v}}}^{2i-2j}\epsilon^{j-i}}\neq z^{\mathbf{k}},\epsilon z^{\mathbf{k}}$ for any $v\in S_p,i\neq j,\mathbf{k}\in \Z^{\Sigma_v}$.
		\item A point $x=(r_x,\underline{\delta})\in \fX_{\infty}\times \widehat{T}_{p,L}$ (or $y=(\rho,\underline{\delta})\in \Spf(R_{\overline{\rho},S})^{\mathrm{rig}}\times \widehat{T}_{p,L}$) is said to be generic if $\underline{\delta}$ is generic.
	\end{enumerate}	
\end{definition}
\subsection{Orlik-Strauch theory}\label{sectionOrlikStrauch}
We recall the theory of Orlik-Strauch on Jordan-Hölder factors of locally analytic principal series which will be the companion constituents in the locally analytic socles.\par
Let $\cO$ (resp. $\overline{\cO}$) be the BGG category of $U(\fg)$-modules attached to the Borel subalgebra $\fb$ (resp. $\overline{\fb}$) (\cite[\S1.1]{humphreys2008representations}). If $M$ is in $\overline{\cO}_{\mathrm{alg}}$ (\cite[\S2]{breuil2016versI}) and $V$ is a smooth representation of $T_p$ over $L$, then Orlik-Strauch constructs a locally $\Q_p$-analytic representation $\cF_{\overline{B}_p}^{G_p}(M,V)$ of $G_p$ (\cite{orlik2015jordan}, see \cite[\S2]{breuil2016versI}, \cite[\S2]{breuil2015versII} and \cite[Rem. 5.1.2]{breuil2019local}). The functor $\cF_{\overline{B}_p}^{G_p}(-,-)$ is exact and contravariant (resp. covariant) in the first (resp. second) arguments (cf. \cite[Thm. 2.2]{breuil2015versII}). If $\lambda=(\lambda_v)_{v\in S_p}=\left(\lambda_{\tau,1},\cdots,\lambda_{\tau,n}\right)_{\tau\in \Sigma_v,v\in S_p}\in \prod_{v\in S_p}(\Z^n)^{\Sigma_v}$, we let $z^{\lambda}:=\prod_{v\in S_p}z^{\lambda_v}$ be the algebraic character of $T_p=\prod_{v\in S_p}T_v$ which satisfies that $\mathrm{wt}_{\tau}\left((z^{\lambda_{v}})_i\right)=\lambda_{\tau,i}$ for every $\tau\in \Sigma_v$. Thus we may view $\lambda$ as a weight of $\ft$. Assume that $\underline{\delta}=(\underline{\delta}_v)_{v\in S_p}\in \widehat{T}_{p,L}(L)$ is locally algebraic of weight $\lambda$ (i.e. $\wt(\underline{\delta}):=(\wt_{\tau}(\delta_{v,i}))_{i=1,\cdots,n, \tau\in \Sigma_v,v\in S_p}\in \prod_{v\in S_p}(\Z^n)^{\Sigma_v}$ and $\lambda=\wt(\underline{\delta})$) and we write $\underline{\delta}=z^{\lambda}\underline{\delta}_{\mathrm{sm}}$ so that $\underline{\delta}_{\mathrm{sm}}$ is a smooth character of $T_p$. We define
\[\cF_{\overline{B}_p}^{G_p}(\underline{\delta}):=\cF_{\overline{B}_p}^{G_p}\left(\left(U(\fg)\otimes_{U(\overline{\fb})}(-\lambda)\right)^{\vee},\underline{\delta}_{\mathrm{sm}}\delta_{B_p}^{-1}\right)\]
where $-\lambda$ is viewed as a weight of $\overline{\fb}$ and $(-)^{\vee}$ is the dual in $\overline{\cO}$ (cf. \cite[\S3.2]{humphreys2008representations}). By \cite[Thm. 4.3, Rem. 4.4]{breuil2015versII}, 
\begin{equation}\label{equationorlikstrauchadjunction}
	\Hom_{G_p}(\cF_{\overline{B}_p}^{G_p}(\underline{\delta}),\Pi^{\mathrm{an}})\simeq \Hom_{T_p}(\underline{\delta},J_{B_p}(\Pi^{\mathrm{an}}))
\end{equation}
for any very strongly admissible locally analytic representation $\Pi^{\mathrm{an}}$ of $G_p$ over $L$. \par
For such $\lambda$, let $\overline{L}(\lambda)$ (resp. $L(\lambda)$) be the irreducible $U(\fg)$-module of the highest weight $\lambda$ in $\overline{\cO}$ (resp. in $\cO$). Let $W_{G_p}:=\prod_{v\in S_p} (\mathcal{S}_n)^{\Sigma_v}$ be the Weyl group of $\fg=\prod_{v}\fg_v=\prod_{\tau\in\Sigma_p}\fg_{\tau}$ acting naturally on $\prod_{v\in S_p}(\Z^n)^{\Sigma_v}$ and we identify this action with the usual action of the Weyl group on weights of $\ft$. Let 
\[\rho=(\frac{n-1}{2},\cdots,\frac{n-2i+1}{2},\cdots,\frac{1-n}{2})_{\tau\in\Sigma_v,v\in S_p}\]
be the half sum of positive roots (with respect to $\fb$) (the notation $\rho$ will also be used to denote Galois representations when it will not confuse the reader according to the context). The dot action is given by $w\cdot\mu=w(\mu+\rho)-\rho$ for any $w\in W_{G_p}$ and $\mu\in \prod_{v\in S_p} (\Z^n)^{\Sigma_v}$. \par
We say $\lambda\in \prod_{v\in S_p} (\Z^n)^{\Sigma_v}$ is \emph{dominant} (resp. \emph{anti-dominant}) (with respect to $\fb$) if $\lambda_{\tau,i}\geq \lambda_{\tau,i+1},\forall \tau\in\Sigma_p,i=1,\cdots,n-1$ (resp. $-\lambda$ is dominant).
Now assume that $\lambda\in \prod_{v\in S_p} (\Z^n)^{\Sigma_v}$ such that $\lambda+\rho$ is dominant (so that $\lambda$ is dominant in the sense of \cite[\S3.5]{humphreys2008representations} with respect to $\fb$). Let $w_0=(w_{v,0})_{v\in S_p}=\left((w_{\tau,0}\right)_{\tau\in \Sigma_v})_{v\in S_p}$ be the longest element in $W_{G_p}$ and let $W_{P_p}=\prod_{v\in S_p} W_{P_v}$ be the parabolic subgroup of $W_{G_p}$ consisting of elements that fix $w_0\cdot\lambda$ under the dot action where $P_p=\prod_{v\in S_v} P_v$ denotes the parabolic subgroup of $\prod_{v\in S_p}(\mathrm{Res}_{F_{\widetilde{v}}/\Q_p}\GL_{n/F_{\widetilde{v}}})\times_{\Q_p}L$ containing the Borel subgroup of upper-triangular matrices associated with $W_{P_p}$. Now $-\lambda$ is dominant with respect to $\overline{\fb}$ in the sense of \cite[\S3.5]{humphreys2008representations} and an irreducible module $\overline{L}(-\mu)$ is a subquotient of $U(\fg)\otimes_{U(\overline{\fb})}(-ww_0\cdot\lambda)$ if and only if $\mu\uparrow ww_0\cdot\lambda$ (cf. \cite[\S5.1]{humphreys2008representations}, the linkage relation $\uparrow$ here is defined with respect to $\fb$). One can prove that $\mu\uparrow ww_0\cdot\lambda$ if and only if $-\mu=-w'w_0\cdot\lambda$ for some $w'\in W_{G_p}/W_{P_p}$ such that $w'\leq w$ in $W_{G_p}/W_{P_p}$. Hence we conclude that the Jordan-Hölder factors of $U(\fg)\otimes_{U(\overline{\fb})}(-ww_0\cdot\lambda)$ are those $\overline{L}(-w'w_0\cdot\lambda)$ for $w'\leq w$ in $W_{G_p}/W_{P_p}$ (one can also use the fact that the translation functor $T^{w_0\cdot \lambda}_{w_0\cdot 0}$ is exact (\cite[\S7.1]{humphreys2008representations}) and $T^{w_0\cdot \lambda}_{w_0\cdot 0}M(ww_0\cdot 0)=M(ww_0\cdot \lambda)$ for all $w\in W_{G_p}$ where $M(-)$ denotes the Verma modules with respect to $\fb$, $ T^{w_0\cdot \lambda}_{w_0\cdot 0}L(ww_0\cdot 0)=L(ww_0\cdot \lambda)$ if $w\in (W_{G_p})^{P_p}$ and $T^{w_0\cdot \lambda}_{w_0\cdot 0}L(ww_0\cdot 0)=0$ if $w\notin (W_{G_p})^{P_p}$ to reduce to regular cases, cf. \cite[Thm. 7.6, Thm. 7.9]{humphreys2008representations} or \cite[Prop. 2.1.1]{irving1990singular}).\par
For a locally algebraic character $\underline{\delta}\in \widehat{T}_{p,L}$ of weight $\lambda$, we define characters $\underline{\delta}_{w}:=z^{ww_0\cdot \lambda}\underline{\delta}_{\mathrm{sm}}$ for $w\in W_{G_p}/W_{P_p}$. By the Orlik-Strauch theory (\cite[Thm. 2.3 \& (2.6)]{breuil2016versI}), if the smooth representation $\mathrm{Ind}_{\overline{B}_p}^{G_p}\underline{\delta}_{\mathrm{sm}}\delta_{B_p}^{-1}$ is irreducible (which will be the case in our later discussions), the locally $\Q_p$-analytic representation $\mathcal{F}_{\overline{B}_p}^{G_p}\left(\overline{L}(-ww_0\cdot\lambda),\underline{\delta}_{\mathrm{sm}}\delta_{B_p}^{-1}\right)$ is irreducible. The Jordan-Hölder factors of $\cF_{\overline{B}_p}^{G_p}(\underline{\delta}_{w})$ are those $\mathcal{F}_{\overline{B}_p}^{G_p}\left(\overline{L}(-w'w_0\cdot\lambda),\underline{\delta}_{\mathrm{sm}}\delta_{B_p}^{-1}\right)$ where $w'\leq w$ in $W_{G_p}/W_{P_p}$ and $\mathcal{F}_{\overline{B}_p}^{G_p}\left(\overline{L}(-ww_0\cdot\lambda),\underline{\delta}_{\mathrm{sm}}\delta_{B_p}^{-1}\right)$ is the unique irreducible quotient of $\cF_{\overline{B}_p}^{G_p}(\underline{\delta}_{w})$. 
\subsection{The locally analytic socle conjecture}\label{sectionthesocleconjection}
We prove our main results concerning the appearance of companion constituents in the completed cohomology and the existence of companion points on the eigenvariety in the situation of non-regular Hodge-Tate weights. The proofs of Proposition \ref{propositionsocleappear} and Theorem \ref{theoremmaincrystalline} follow essentially part of the proof of \cite[Thm. 5.3.3]{breuil2019local} with weakened assumptions. The new ingredients are in Theorem \ref{theoremcyclepartialderham} and Proposition \ref{propositionmaincycle}.\par
Let $\lambda=\left(\lambda_{\tau,1},\cdots,\lambda_{\tau,n}\right)_{\tau\in \Sigma_v,v\in S_p}\in\prod_{v\in S_p}(\Z^n)^{\Sigma_v}$ such that $\lambda+\rho$ is dominant. Then we have the parabolic subgroup $P_p$ of $\prod_{v\in S_p}(\mathrm{Res}_{F_{\widetilde{v}}/\Q_p}\GL_{n/F_{\widetilde{v}}})\times_{\Q_p}L$ associated with $\lambda$ as in \S\ref{sectionOrlikStrauch}.\par 
We firstly recall some representation theoretic results in \cite[\S5.2]{breuil2019local} and the construction of certain ``family of companion constituents'' in the proof of \cite[Thm. 5.3.3]{breuil2019local}. \par
We write $\Pi_{\infty}^{\mathrm{an}}$ for $\Pi_{\infty}^{R_{\infty}-\mathrm{an}}$. Let $U_p$ be the unipotent radical of $B_p$ and $U_0$ be a compact open subgroup of $U_p$. Suppose that $\Pi^{\mathrm{an}}$ is a very strongly admissible locally $\Q_p$-analytic representation of $G_p$ over $L$. If $M\in\cO_{\mathrm{alg}}$, then $M$ is equipped with an action of $B_p$ and $\Hom_{U(\fg)}\left(M, \Pi^{\mathrm{an}}\right)$ is equipped with a smooth action of $B_p$. Its space of $U_{0}$-invariants $\Hom_{U(\fg)}\left(M, \Pi^{\mathrm{an}}\right)^{U_0}$ is then equipped with a Hecke action of $T_p^{+}:=\{t\in T_p\mid tU_0t^{-1}\subset U_0\} $ given by (\cite[(5.9)]{breuil2019local})
\[f\mapsto t\cdot f:=\delta_{B_p}(t)\sum_{u_0\in U_{0}/tU_{0}t^{-1}}u_0tf.\] 
The starting point is the following adjunction formula. For any finite-dimensional smooth representation $V$ of $T_p$ over $L$, by \cite[Lem. 5.2.1]{breuil2019local} (which itself follows from \cite[Prop. 4.2]{breuil2015versII}), we have
\begin{align*}
	\Hom_{G_p}\left(\mathcal{F}_{\overline{B}_p}^{G_p}\left( \Hom(M,L)^{\overline{\mathfrak{u}}^{\infty}},V(\delta_{B_p}^{-1})\right),\Pi^{\mathrm{an}}\right)&=\Hom_{T_p^{+}}\left(V,\Hom_{U(\fg)}\left(M, \Pi^{\mathrm{an}}\right)^{U_0}\right)\\
	&=\Hom_{T_p}\left(V,\left(\Hom_{U(\fg)}\left(M, \Pi^{\mathrm{an}}\right)^{U_0}\right)_{\mathrm{fs}}\right).
\end{align*}
where $(-)_{\mathrm{fs}}$ denotes Emerton's finite slope part functor (\cite[Def. 3.2.1]{emerton2006jacquet}) and $\Hom(M,L)^{\overline{\mathfrak{u}}^{\infty}}$ is the subspace of $\Hom(M,L)$ consisting of elements annihilated by a power of $\overline{\fu}$, base change to $L$ of the $\Q_p$-Lie algebra of the unipotent radical of $\overline{B}_p$ (see \cite[\S3]{breuil2015versII}).\par 
In particular, for any point $y=(r_y, \underline{\delta})\in\fX_{\infty}\times\widehat{T}_{p,L}$ with the corresponding maximal ideal $\fm_{r_y}$ for $r_y$ such that $
\underline{\delta}$ is locally algebraic of weight $ww_0\cdot \lambda$ and has smooth part $\underline{\delta}_{\mathrm{sm}}$, we have
\begin{align}\label{equationadjointliealgebra}
	&\Hom_{G_p}\left(\mathcal{F}_{\overline{B}_p}^{G_p}\left( \overline{L}(-ww_0\cdot\lambda),\underline{\delta}_{\mathrm{sm}}\delta_{B_p}^{-1}\right),\Pi_{\infty}^{\mathrm{an}}[\fm_{r_y}]\otimes_{k(r_y)}k(y)\right)\\
	&=\Hom_{T_p}\left(\underline{\delta}_{\mathrm{sm}},\left(\Hom_{U(\fg)}\left(L(ww_0\cdot\lambda), \Pi_{\infty}^{\mathrm{an}}[\fm_{r_y}]\otimes_{k(r_y)}k(y)\right)^{U_0}\right)_{\mathrm{fs}}\right).\nonumber
\end{align}
Recall by (\ref{equationorlikstrauchadjunction}), $y\in X_p(\overline{\rho})$ if and only if $\Hom_{G_p}\left(\mathcal{F}_{\overline{B}_p}^{G_p}(\underline{\delta}_{w}),\Pi_{\infty}^{\mathrm{an}}[\fm_{r_y}]\otimes_{k(r_y)}k(y)\right)\neq 0$, and there is an injection
\begin{align}\label{equationsocleinjection}
	\Hom_{G_p}\left(\mathcal{F}_{\overline{B}_p}^{G_p}\left(\overline{L}(-ww_0\cdot\lambda),\underline{\delta}_{\mathrm{sm}}\delta_{B_p}^{-1}\right),\Pi_{\infty}^{\mathrm{an}}[\fm_{r_y}]\otimes_{k(r_y)}k(y)\right)\nonumber \\\hookrightarrow \Hom_{G_p}\left(\mathcal{F}_{\overline{B}_p}^{G_p}(\underline{\delta}_{w}),\Pi_{\infty}^{\mathrm{an}}[\fm_{r_y}]\otimes_{k(r_y)}k(y)\right)
\end{align} 
induced by the quotient $U(\fg)\otimes_{U(\overline{\fb})}(-ww_0\cdot \lambda)\twoheadrightarrow \overline{L}(-ww_0\cdot\lambda)$ for any $w\in W_{G_p}$.\par
For any $w\in W_{G_p}$, let $X_p(\overline{\rho})_{ww_0\cdot \lambda}$ be the fiber of the composite map $X_p(\overline{\rho})\rightarrow \widehat{T}_{p,L}\stackrel{\mathrm{wt}}{\rightarrow}\ft^{*}$ over $ww_0\cdot\lambda\in \ft^{*}$. Here $\ft^{*}$ denotes the rigid space associated with $\Hom_L(\ft,L)$ and the map $\mathrm{wt}$ sends a character of $T_p$ to its weight. Since $J_{B_p}(\Pi_{\infty}^{\mathrm{an}})=((\Pi^{\mathrm{an}}_{\infty})^{U_0})_{\mathrm{fs}}$, the quotient $U(\fg)\otimes_{U(\fb)}ww_0\cdot\lambda\twoheadrightarrow L(ww_0\cdot\lambda)$ induces a closed immersion as (\ref{equationsocleinjection})
\begin{align}\label{equationinjectionglobalsectionfamillyconstituents}
	\left(\Hom_{U(\fg)}\left(L(ww_0\cdot\lambda),\Pi_{\infty}^{\mathrm{an}}\right)^{U_0}\right)_{\mathrm{fs}}&\hookrightarrow \left(\Hom_{U(\ft)}\left(ww_0\cdot\lambda, \Pi_{\infty}^{\mathrm{an}}[\fu]\right)^{U_0}\right)_{\mathrm{fs}} \nonumber\\ 
	&=\Hom_{U(\ft)}\left(ww_0\cdot\lambda, J_{B_p}(\Pi_{\infty}^{\mathrm{an}})\right)
\end{align}
which is compatible with actions of $R_{\infty}[\frac{1}{t}]$ and $T_p$ (see \textbf{Step 8} of the proof of \cite[Thm. 5.3.3]{breuil2019local} for more details on the topology). The continuous dual $\Hom_{U(\ft)}\left(ww_0\cdot\lambda, J_{B_p}(\Pi_{\infty}^{\mathrm{an}})\right)'$ of the target is the global section of the coherent sheaf $\cM_{\infty}\otimes_{\cO_{X_p(\overline{\rho})}}\cO_{X_p(\overline{\rho})_{ww_0\cdot\lambda}}$ over the quasi-Stein space $X_p(\overline{\rho})_{ww_0\cdot\lambda}$ (as a closed subspace of the quasi-Stein space $\Spf(R_{\infty})^{\mathrm{rig}}\times \widehat{T}_{p,L}$, cf. \cite[Def. 2.1.17]{emerton2017locally} and \cite[\S3]{schneider2003algebras}). Then the continuous dual $\left(\Hom_{U(\fg)}\left(L(ww_0\cdot\lambda),\Pi_{\infty}^{\mathrm{an}}\right)^{U_0}\right)_{\mathrm{fs}}'$ of the closed subspace corresponds to a coherent sheaf, denoted by $\mathcal{L}_{ww_0\cdot\lambda}$, on $X_p(\overline{\rho})_{ww_0\cdot\lambda}$ and the continuous dual of (\ref{equationinjectionglobalsectionfamillyconstituents}) gives a surjection of coherent sheaves 
\[\cM_{\infty}\otimes_{\cO_{X_p(\overline{\rho})}}\cO_{X_p(\overline{\rho})_{ww_0\cdot\lambda}}\twoheadrightarrow \mathcal{L}_{ww_0\cdot\lambda}\]
on $X_p(\overline{\rho})_{ww_0\cdot\lambda}$. Let $Y_p(\overline{\rho})_{ww_0\cdot\lambda}$ be the schematic support of $\mathcal{L}_{ww_0\cdot \lambda}$ which is a Zariski-closed subspace of $X_p(\overline{\rho})_{ww_0\cdot\lambda}$. Let $Y_p(\overline{\rho})_{ww_0\cdot\lambda}^{\mathrm{red}}$ be the underlying reduced analytic subvariety of $Y_p(\overline{\rho})_{ww_0\cdot\lambda}$ (cf. \cite[\S9.5.3]{bosch1984non}).
\begin{remark}
	We warn the reader that the subspace $Y_p(\overline{\rho})_{ww_0\cdot\lambda}$ of $X_p(\overline{\rho})$ should not be confused with any subspace of the non-patched eigenvariety $Y(\overline{\rho},U^p)$. Note that when $L(ww_0\cdot\lambda)$ is a finite-dimensional representation of $\fg$, $Y_p(\overline{\rho})_{ww_0\cdot\lambda}$ is similar to a ``partial eigenvariety'' that will be defined in \S\ref{sectionpartialeigenvarietygeometry} (by taking $Q_p=G_p, J=\Sigma_p$), but with more restriction on the weights of characters. And in this case one can prove that $Y_p(\overline{\rho})_{ww_0\cdot\lambda}$ is equidimensional by the usual arguments for the eigenvarieties, cf. Lemma \ref{lemmaorthonormalmodules}. We don't know whether one should expect that $Y_p(\overline{\rho})_{ww_0\cdot\lambda}$ is equidimensional when $ww_0\cdot\lambda$ is no longer dominant.
\end{remark}
Let $\fm_{\underline{\delta}_{\mathrm{sm}}}$ be the kernel of the morphism $L[T_p]\rightarrow L$ of $L$-algebras given by a smooth character $\underline{\delta}_{\mathrm{sm}}$. Then for any $L$-point $y=\left(r_y, z^{ww_0\cdot\lambda}\underline{\delta}_{\mathrm{sm}}\right)\in X_p(\overline{\rho})_{ww_0\cdot\lambda}\subset \fX_{\infty}\times\widehat{T}_{p,L}$, 
\begin{align}\label{equationsoclesheafstalklimit}
	\mathcal{L}_{ww_0\cdot\lambda}\otimes_{\cO_{X_p(\overline{\rho})_{ww_0\cdot\lambda}}}\widehat{\cO}_{X_p(\overline{\rho})_{ww_0\cdot\lambda},y}&=\varprojlim_{s,t\in \N}\left(\Hom_{U(\fg)}\left(L(ww_0\cdot\lambda),\Pi_{\infty}^{\mathrm{an}}\right)^{U_0}\right)_{\mathrm{fs}}'/\left(\fm_{r_y}^s,\fm_{\underline{\delta}_{\mathrm{sm}}}^t\right)\\
	&=\varprojlim_{s,t\in \N}\left(\Hom_{U(\fg)}\left(L(ww_0\cdot\lambda),\Pi_{\infty}^{\mathrm{an}}\right)^{U_0}[\fm_{r_y}^s][\fm_{\underline{\delta}_{\mathrm{sm}}}^t]\right)'\nonumber\\
	&=\left(\varinjlim_{s,t\in \N}\Hom_{U(\fg)}\left(L(ww_0\cdot\lambda),\Pi_{\infty}^{\mathrm{an}}\right)^{U_0}[\fm_{r_y}^s][\fm_{\underline{\delta}_{\mathrm{sm}}}^t]\right)'\nonumber
\end{align}
where each 
\[\Hom_{U(\fg)}\left(L(ww_0\cdot\lambda),\Pi_{\infty}^{\mathrm{an}}\right)^{U_0}[\fm_{r_y}^s][\fm_{\underline{\delta}_{\mathrm{sm}}}^t]=\Hom_{U(\fg)}(L(ww_0\cdot\lambda),\Pi_{\infty}[\fm_{r_y}^s]^{\mathrm{an}})^{U_0}[\fm_{\underline{\delta}_{\mathrm{sm}}}^t]\] 
is finite-dimensional (using that $\Pi^{\mathrm{an}}_{\infty}$ is $R_{\infty}$-admissible, the finiteness comes from the related property of Emerton's Jacquet module, see \cite[Lem. 5.2.4]{breuil2019local}). \par
We denote by $\Hom_{U(\fg)}\left(L(ww_0\cdot\lambda),\Pi_{\infty}^{\mathrm{an}}\right)^{U_0}[\fm_{r_y}^{\infty}][\fm_{\underline{\delta}_{\mathrm{sm}}}^{\infty}]$ for the last term in the bracket in (\ref{equationsoclesheafstalklimit}). Thus 
\[\Hom_{G_p}\left(\mathcal{F}_{\overline{B}_p}^{G_p}\left(\overline{L}(-ww_0\cdot\lambda),\underline{\delta}_{\cR,\mathrm{sm}}\delta_{B_p}^{-1}\right),\Pi_{\infty}^{\mathrm{an}}[\fm_{r_y}]\right)\neq 0\]
if and only if (by (\ref{equationadjointliealgebra}))
\[\Hom_{U(\fg)}\left(L(ww_0\cdot\lambda),\Pi_{\infty}^{\mathrm{an}}\right)^{U_0}[\fm_{r_y}][\fm_{\underline{\delta}_{\mathrm{sm}}}]\neq 0\] 
if and only if $y\in Y_p(\overline{\rho})_{ww_0\cdot\lambda}$ which is equivalent to that
\[\Hom_{U(\fg)}\left(L(ww_0\cdot\lambda),\Pi_{\infty}^{\mathrm{an}}\right)^{U_0}[\fm_{r_y}^{\infty}][\fm_{\underline{\delta}_{\mathrm{sm}}}^{\infty}]\neq 0.\]
\indent 
Next, we take $y=\left((\rho_p,\underline{\delta}),z\right) \in Y_p(\overline{\rho})_{ww_0\cdot\lambda}\subset X_p(\overline{\rho})\subset\iota\left(X_{\mathrm{tri}}(\overline{\rho}_p)\right)\times (\mathfrak{X}_{\overline{\rho}^p}\times\mathbb{U}^g)$ a generic $L$-point (Definition \ref{definitiongenericglobal}) and $w=(w_v)_{v\in S_p}=(w_{\tau})_{\tau\in \Sigma_v,v\in S_p}\in W_{G_p}/W_{P_p}=\prod_{v\in S_p}(\mathcal{S}_n)^{\Sigma_v}/W_{P_v}$ such that $\wt(\underline{\delta})=ww_0\cdot \lambda$. Write $\rho_p=(\rho_{\widetilde{v}})_{v\in S_p}\in\prod_{v\in S_p}\Spf(R_{\overline{\rho}_{\widetilde{v}}})^{\mathrm{rig}}$. Let $\mathbf{h}=(\mathbf{h}_{\widetilde{v}})_{v\in S_p}=(\mathbf{h}_{\tau})_{\tau\in\Sigma_p}=(h_{\tau,1},\cdots,h_{\tau,n})_{\tau\in\Sigma_p}$ be the anti-dominant Hodge-Tate-Sen weights of $\rho_p$ where we view $\mathbf{h}$ as a coweight of $\ft$ and ``anti-dominant'' means that $h_{\tau,1}\leq \cdots\leq h_{\tau,n},\forall \tau\in \Sigma_p$. Then $w_0\cdot\lambda=(h_{\tau,1},\cdots,h_{\tau,i}+i-1,\cdots,h_{\tau,n}+n-1)_{\tau\in \Sigma_p} \in \prod_{v\in S_p}(\Z^n)^{\Sigma_v}$. The stabilizer subgroup of $\mathbf{h}$ under the usual action of $W_{G_p}=(\mathcal{S}_n)^{\Sigma_p}$ is $W_{P_p}$.\par
By Theorem \ref{theoremirreducibletriangullinevariety} and the generic assumption on $\underline{\delta}$, $(\rho_p,\iota^{-1}\left(\underline{\delta}\right))$ lies in a unique irreducible component $X$ of $X_{\mathrm{tri}}(\overline{\rho}_p)$. The union of irreducible components of $X_p(\overline{\rho})$ that pass through $y$ are of the form $\iota(X)\times \left( \cup_{i\in I} \fX_i^p\right)\times \mathbb{U}^g$ where $\fX_i^p\times \mathbb{U}^g, i\in I$ are some irreducible components of $\mathfrak{X}_{\overline{\rho}^p}\times\mathbb{U}^g$. Then $\iota^{-1}$ induces a closed immersion $Y_p(\overline{\rho})_{ww_0\cdot\lambda}^{\mathrm{red}}\hookrightarrow X_{\mathrm{tri}}(\overline{\rho}_p)_{w(\mathbf{h})}\times \mathfrak{X}_{\overline{\rho}^p}\times\mathbb{U}^g$ where $X_{\mathrm{tri}}(\overline{\rho}_p)_{w(\mathbf{h})}$ denotes the fiber of the composite $X_{\mathrm{tri}}(\overline{\rho}_p)\rightarrow \widehat{T}_{p,L}\stackrel{\mathrm{wt}}{\rightarrow} \ft^{*}$ over $w(\mathbf{h})$. Hence $\iota$ induces a surjection 
\[\widehat{\cO}_{X_{\mathrm{tri}}(\overline{\rho}_p)_{w(\mathbf{h})},\iota^{-1}(\rho_p,\underline{\delta})}\widehat{\otimes}_L \widehat{\cO}_{\mathfrak{X}_{\overline{\rho}^p}\times\mathbb{U}^g,z}\twoheadrightarrow \widehat{\cO}_{Y_p(\overline{\rho})_{ww_0\cdot\lambda}^{\mathrm{red}},y}.\]
Recall in \S\ref{sectiontriangulinevariety}, for each $v\in S_p$, $D_{\mathrm{rig}}(\rho_{\widetilde{v}})[\frac{1}{t}]$ is equipped with a unique triangulation $\fM_{\widetilde{v},\bullet}$ associated with the point $\left(\rho_{\widetilde{v}},\iota^{-1}_v(\underline{\delta}_v)\right)\in X_{\mathrm{tri}}(\overline{\rho}_{\widetilde{v}})$ and we have an isomorphism $\widehat{X_{\mathrm{tri}}(\overline{\rho}_{\widetilde{v}})}_{\left(\rho_{\widetilde{v}},\iota^{-1}_v(\underline{\delta}_v)\right)}\simeq X_{\rho_{\widetilde{v}},\cM_{\widetilde{v},\bullet}}^{w_v}$ by Corollary \ref{corollarylocalmodelisomorphism}. In \S\ref{sectioncycles}, we have $\widehat{\cO}_{X_{\mathrm{tri}}(\overline{\rho}_{\widetilde{v}})_{w_v(\mathbf{h}_{\widetilde{v}})},\iota^{-1}(\rho_{\widetilde{v}},\underline{\delta}_v)}\simeq\overline{R}_{\rho_{\widetilde{v}},\cM_{\widetilde{v},\bullet}}^{w_v}$. Hence $\widehat{\cO}_{X_{\mathrm{tri}}(\overline{\rho}_p)_{w(\mathbf{h})},\iota^{-1}(\rho_p,\underline{\delta})}\simeq \overline{R}_{\rho_p,\cM_{\bullet}}^w:=\widehat{\otimes}_{v\in S_p}\overline{R}_{\rho_{\widetilde{v}},\cM_{\widetilde{v},\bullet}}^{w_v}$. Define $R_{\rho_p}:=\widehat{\otimes}_{v\in S_p}R_{\rho_{\widetilde{v}}}$. We get a composite map 
\begin{equation}\label{equationcompositepartialderhamcycle}
	R_{\rho_p}\rightarrow \overline{R}_{\rho_p,\cM_{\bullet}}^{w}\simeq \widehat{\cO}_{X_{\mathrm{tri}}(\overline{\rho}_p)_{w(\mathbf{h})},\iota^{-1}(\rho_p,\underline{\delta})}\rightarrow \widehat{\cO}_{Y_p(\overline{\rho})_{ww_0\cdot\lambda}^{\mathrm{red}},y}.
\end{equation}
Let $Q_p=\prod_{v\in S_p}\prod_{\tau\in \Sigma_v}Q_{\tau}$ be a standard parabolic subgroup of $\prod_{v\in S_p}(\mathrm{Res}_{F_{\widetilde{v}}/\Q_p}\GL_{n/F_{\widetilde{v}}})\times_{\Q_p}L=\prod_{\tau\in\Sigma_p}\GL_{n/L}$ (with the Borel subgroup of upper-triangular matrices and the maximal torus the diagonal matrices). For $v\in S_p$, the statement  ``$w_v(\mathbf{h}_{\widetilde{v}})$ is strictly $Q_v$-dominant'' (Definition \ref{definitionstrictlyQdominant}) is equivalent to ``$w_vw_{v,0}\cdot \lambda_v$ is a dominant weight for the standard Levi subgroup of $Q_v$''. Let $R_{\rho_p,\cM_{\bullet}}^{Q_p}:=\widehat{\otimes}_{v\in S_p} R_{\rho_{\widetilde{v}},\cM_{\widetilde{v},\bullet}}^{Q_v}$ where the latter is defined in the end of \S\ref{sectioncycles}, rougly speaking, parametrizing trianguline deformations of $\rho_{\widetilde{v}}$ that are $Q_{v}$-de Rham. There are closed immersions 
\[\Spec( R_{\rho_p,\cM_{\bullet}}^{Q_p})\hookrightarrow \Spec( \overline{R}_{\rho_p,\cM_{\bullet}})\hookrightarrow\Spec( R_{\rho_p}).\]
\begin{theorem}\label{theoremcyclepartialderham}
	If $ww_{0}\cdot \lambda$ is a dominant weight for the standard Levi subgroup of $Q_p$, then the morphism (\ref{equationcompositepartialderhamcycle}): $R_{\rho_p}\rightarrow \widehat{\cO}_{Y_p(\overline{\rho})_{ww_0\cdot\lambda}^{\mathrm{red}},y}$ factors through $R_{\rho_p,\cM_{\bullet}}^{Q_p}$.
\end{theorem}
\begin{proof}
	The proof is based on Theorem \ref{theoremQderham}.\par
	We have a closed immersion $Y_p(\overline{\rho})_{ww_0\cdot\lambda}^{\mathrm{red}}\hookrightarrow X_{p}(\overline{\rho})$. We argue for a fixed $v\in S_p$. Let $D:=D_{\mathrm{rig}}(\rho_{\widetilde{v}}^{\mathrm{univ}})$ be the $(\varphi,\Gamma_{F_{\widetilde{v}}})$-module over $\cR_{X_p(\overline{\rho}),F_{\overline{v}}}$ associated with the universal Galois representation $\rho_{\widetilde{v}}^{\mathrm{univ}}$ of $\cG_{F_{\widetilde{v}}}$ pulled back from $\fX_{\overline{\rho}_{\widetilde{v}}}$ (cf. \cite[Def. 2.12]{liu2015triangulation}). By \cite[Thm. 3.19]{breuil2017interpretation} and \cite[Cor. 6.3.10]{kedlaya2014cohomology}, there is a birational proper surjective morphism $f:X'\rightarrow X_p(\overline{\rho})$, a filtration of sub-$(\varphi,\Gamma_{F_{\widetilde{v}}})$-modules $D_{X',\bullet}$ over $\cR_{X',F_{\widetilde{v}}}$ of $D_{X'}:=f^*D$ such that for any point $y'\in X'$, the base change $D_{y',\bullet}[\frac{1}{t}]$ is a triangulation of $D_{y'}[\frac{1}{t}]=D_{\mathrm{rig}}(\rho_{y',\widetilde{v}})[\frac{1}{t}]$ of parameter $\delta_{y',v,1},\cdots,\delta_{y',v,n}$ where we use the same notations with different subscripts to denote the pullback of the representation $\rho_{\widetilde{v}}^{\mathrm{univ}}$, characters from $\widehat{T}_v$, etc.. Let $Y'$ be the underlying reduced analytic subspace of $f^{-1}\left(Y_p(\overline{\rho})_{ww_0\cdot\lambda}^{\mathrm{red}}\right)$. We pick an arbitrary affinoid neighbourhood $V$ of an arbitrary point $y'\in f^{-1}(y)$ in $Y'$. By Theorem \ref{theoremQderham} below, the definition of $Y_p(\overline{\rho})_{ww_0\cdot\lambda}$, (\ref{equationadjointliealgebra}), and the assumption that $ww_{0}\cdot\lambda$ is dominant with respect to the standard Levi of $Q_p$, we have that for any point $y''\in V$, $D_{y''}[\frac{1}{t}]$ with the triangulation $D_{y'',\bullet}[\frac{1}{t}]$ is $Q_v$-de Rham (Definition \ref{definitionQderham}). \par
	We firstly prove that the map $R_{\rho_{\widetilde{v}}}\rightarrow \widehat{\cO}_{Y',y'}$ factors through the quotient $R_{\rho_{\widetilde{v}},\cM_{\widetilde{v},\bullet}}^{Q_v}$. The proof is similar to that of Proposition \ref{propositiontriangulinevarietyclosedimmersion}. Take $A$ a local Artin $L$-algebra with residue field $k(y')$ and a composite $y'=\Sp(k(y'))\rightarrow \Sp(A)\rightarrow V$. The pullback $D_{A,\bullet}$ along the map $\Sp(A)\rightarrow V$ of the global triangulation $D_{V,\bullet}={D_{X',\bullet}\times_{X'}V}$ gives a triangulation $\cM_{A,\bullet}:=D_{A,\bullet}[\frac{1}{t}]$ of $D_{\mathrm{rig}}\left(\rho_{A,\widetilde{v}}\right)[\frac{1}{t}]$ of parameter $\delta_{A,v,1},\cdots,\delta_{A,v,n}$. Since the filtration $D_{X',\bullet}$ is a strictly trianguline filtration on a Zariski open dense subset of $X'$, and the Sen polynomials vary analyticly, the Sen weights of $D_{V,i}$ are fixed integers (the weights of $\delta_{V,v,1},\cdots, \delta_{V,v,i}$). Hence we can apply Theorem \ref{propositionfamilydpdr} below for each $D_{V,i}$, and we get finite projective $\cO_V$-modules written by $D_{\mathrm{pdR}}(D_{V,i})$ (in short for $D_{\mathrm{pdR}}(W_{\mathrm{dR}}(D_{V,i}))=\prod_{\tau\in\Sigma_v}D_{\mathrm{pdR},\tau}(W_{\mathrm{dR}}(D_{V,i}))$) with $\cO_V$-linear operators $\nu_V$. Check the proof of Theorem \ref{propositionfamilydpdr}, the filtration $D_{V,\bullet}$ induces natural maps $D_{\mathrm{pdR}}(D_{V,i})\rightarrow D_{\mathrm{pdR}}(D_{V,i+1})$. By Theorem \ref{propositionfamilydpdr} again, $D_{\mathrm{pdR}}(D_{V,\bullet})$ are specialized to the $k(y'')$-filtration $D_{\mathrm{pdR}}(W_{\mathrm{dR}}(D_{y'',\bullet}))$ for any $y''\in V$. Since $V$ is reduced, the sheaves $D_{\mathrm{pdR}}(D_{V,i})$ form a saturated filtration of $D_{\mathrm{pdR}}(D_{V})$, i.e., $D_{\mathrm{pdR}}(D_{V,i})$ is mapped injectively into $D_{\mathrm{pdR}}(D_{V,i+1})$ and the graded pieces $D_{\mathrm{pdR}}(D_{V,i+1})/D_{\mathrm{pdR}}(D_{V,i})$ are locally free of rank one. Moreover, for any $\tau \in \Sigma_v,1\leq i<i'\leq n$ and $y''\in V$, 
	\[(D_{\mathrm{pdR},\tau}(D_{V,i'})/D_{\mathrm{pdR},\tau}(D_{V,i}))\otimes_{\cO_V}k(y'')\simeq D_{\mathrm{pdR},\tau}(W_{\mathrm{dR}}(\cM_{y'',i'}/\cM_{y'',i})).\] 
	Let $0=s_{\tau,0}<\cdots<s_{\tau,i}<\cdots<s_{\tau,t_{\tau}}=n$ be integers such that the Levi subgroup of $Q_{\tau}$ is $\GL_{s_{\tau,1}-s_{\tau,0}}\times\cdots\GL_{s_{\tau,i}-s_{\tau,i-1}}\times\cdots\times\GL_{s_{\tau,t_{\tau}}-s_{\tau,t_{\tau}-1}}$
	Since $D_{y''}[\frac{1}{t}]$ is $Q_v$-de Rham for any $y''\in V$, $\nu_{y''}={\nu_V\otimes_{\cO_V}k(y'')}$ acts as zero on $D_{\mathrm{pdR},\tau}(W_{\mathrm{dR}}(\cM_{y'',s_{\tau,i}}/\cM_{y'',s_{\tau, i-1}}))$ for $1\leq i\leq t_{\tau}$ and $\tau\in \Sigma_v$. As $V$ is reduced, it follows that $\nu_V$ itself is zero on the graded pieces $D_{\mathrm{pdR},\tau}(D_{V,s_{\tau,i}})/D_{\mathrm{pdR},\tau}(D_{V,s_{\tau, i-1}})$. By Theorem \ref{propositionfamilydpdr}, the same assertion holds for the base change $\nu_A$ of $\nu_V$ on $D_{\mathrm{pdR},\tau}(W_{\mathrm{dR}}(\cM_{A,s_{\tau,i}}/\cM_{A,s_{\tau,i-1}}))$. By Definition \ref{definitionQderham}, $D_{\mathrm{rig}}\left(\rho_{A,\widetilde{v}}\right)[\frac{1}{t}]$ with the triangulation is $Q_v$-de Rham (Definition \ref{definitionQderham}). \par
	We prove that the composite $R_{\rho_{\widetilde{v}}}\rightarrow \widehat{\cO}_{Y',y'}\rightarrow A$ factors through a map $R_{\rho_{\widetilde{v}}}\rightarrow R_{\rho_{\widetilde{v}},\cM_{\widetilde{v},\bullet}}^{Q_v}\rightarrow A$. Let $\widetilde{A}:=A\times_{k(y')}L$ be the subring of $A$ consisting of elements whose reduction modulo the maximal ideal $\fm_A$ of $A$ lie in $L$. Then the map $R_{\rho_{\widetilde{v}}}\rightarrow A$ factors through $R_{\rho_{\widetilde{v}}}\rightarrow \widetilde{A}\subset A$ automatically. By similar arguments as in the proof of Proposition \ref{propositiontriangulinevarietyclosedimmersion}, there exists a model $\rho_{\widetilde{A},\widetilde{v}}$ (resp. $\delta_{A,v,i}$, resp. $\cM_{\widetilde{A},\bullet}$) of $\rho_{A,\widetilde{v}}$ (resp. $\delta_{\widetilde{A},v,i}$, resp. $\cM_{A,\bullet}$) over $\widetilde{A}$ whose reduction modulo $\fm_{\widetilde{A}}$ is $\rho_{y,\widetilde{v}}$ (resp. $\delta_{y,v,i}$, resp. $\cM_{y,\widetilde{v},\bullet}$). This implies that the map $R_{\rho_{\widetilde{v}}}\rightarrow \widehat{\cO}_{Y',y'}\rightarrow A$ factors through $R_{\rho_{\widetilde{v}}}\rightarrow R_{\rho_{\widetilde{v}},\cM_{\widetilde{v},\bullet}} \rightarrow \widetilde{A}\subset A$. Moreover, $D_{\mathrm{rig}}\left(\rho_{A,\widetilde{v}}\right)[\frac{1}{t}]=D_{\mathrm{rig}}(\rho_{\widetilde{A},\widetilde{v}})[\frac{1}{t}]\otimes_{\widetilde{A}}A$ and $\cM_{\widetilde{A},\bullet}\otimes_{\widetilde{A}}A=\cM_{A,\bullet}$. It follows from the exactness and the functoriality of $D_{\mathrm{pdR}}(W_{\mathrm{dR}}(-))$ that there exist isomorphisms $D_{\mathrm{pdR}}(W_{\mathrm{dR}}(\cM_{\widetilde{A},\bullet}))\otimes_{\widetilde{A}}A\simeq D_{\mathrm{pdR}}\left(W_{\mathrm{dR}}\left(\cM_{A,\bullet}\right)\right)$ of finite free $A$-modules with $A$-linear nilpotent operators (cf. the proof of \cite[Lem. 3.1.4]{breuil2019local}, writing $A$ as the cokernel of finite free $\widetilde{A}$-modules). Since $\widetilde{A}\rightarrow A$ is injective, the vanishing of $\nu_{A}=\nu_{\widetilde{A}}\otimes_{\widetilde{A}}A$ on the finite free $A$-modules $D_{\mathrm{pdR},\tau}(W_{\mathrm{dR}}(\cM_{A,s_{\tau,i}}/\cM_{A,s_{\tau,i-1}}))=D_{\mathrm{pdR},\tau}(W_{\mathrm{dR}}(\cM_{\widetilde{A},s_{\tau,i}}/\cM_{\widetilde{A},s_{\tau,i-1}}))\otimes_{\widetilde{A}}A$ implies the vanishing of $\nu_{\widetilde{A}}$ on $D_{\mathrm{pdR},\tau}(W_{\mathrm{dR}}(\cM_{\widetilde{A},s_{\tau,i}}/\cM_{\widetilde{A},s_{\tau,i-1}}))$. This means that $\rho_{\widetilde{A},\widetilde{v}}$ with the filtration $\cM_{\widetilde{A},\bullet}$ is also $Q_v$-de Rham. The definition of $R_{\rho_{\widetilde{v}},\cM_{\widetilde{v},\bullet}}^{Q_v}$ and Lemma \ref{lemmagroupoidspartiallydeRham} implies that the map $R_{\rho_{\widetilde{v}}}\rightarrow \widetilde{A}$ factors through $R_{\rho_{\widetilde{v}},\cM_{\widetilde{v},\bullet}}^{Q_v}\rightarrow \widetilde{A}$. \par
	By taking $A=\cO_{Y',y'}/\fm_{\cO_{Y',y'}}^j$ for all $j\in\N$, we conclude that the map $R_{\rho_{\widetilde{v}}}\rightarrow \widehat{\cO}_{Y',y'}$ factors through the quotient $R_{\rho_{\widetilde{v}},\cM_{\widetilde{v},\bullet}}^{Q_v}$.\par
	Now the argument in the last part of the proof of Proposition \ref{propositiontriangulinevarietyclosedimmersion}, using the surjectivity and properness of $f:Y'\rightarrow Y_p(\overline{\rho})_{ww_0\cdot\lambda}^{\mathrm{red}}$ and the reducedness of $Y_p(\overline{\rho})_{ww_0\cdot\lambda}^{\mathrm{red}}$, shows that the map $R_{\rho_{\widetilde{v}}}\rightarrow \widehat{\cO}_{Y_p(\overline{\rho})_{ww_0\cdot\lambda}^{\mathrm{red}},y}$ also factors though $R_{\rho_{\widetilde{v}},\cM_{\widetilde{v},\bullet}}^{Q_v}$. 
\end{proof}
\begin{remark}
	We expect that the morphism $R_{\rho_p}\rightarrow \widehat{\cO}_{Y_p(\overline{\rho})_{ww_0\cdot\lambda},y}$ as in Theorem \ref{theoremcyclepartialderham} factors through $R_{\rho_p,\cM_{\bullet}}^{Q_p}$. However, we don't know how to prove this stronger result.
\end{remark}
From now on we focus on generic crystalline points. A point $\rho_p=(\rho_{\widetilde{v}})_{v\in S_p}\in \fX_{\overline{\rho}_p}(L)$ is said to be \emph{generic crystalline} if for each $v\in S_p$, $\rho_{\widetilde{v}}:\cG_{F_{\widetilde{v}}}\rightarrow \GL_n(L)$ is generic crystalline (in the sense in \S\ref{sectionlocalcompanionpoints}). If $\rho_p$ is generic, a refinement $\cR=(\cR_{\widetilde{v}})_{v\in S_p}$ of $\rho_p$ is a choice of a refinement $\cR_{\widetilde{v}}$ of $\rho_{\widetilde{v}}$ for each $v\in S_p$. \par
Suppose that $\rho_p\in \fX_{\overline{\rho}_p}(L)$ is a generic crystalline point with refinements $\cR_{\widetilde{v}}$ given by orderings $\underline{\varphi}_{\widetilde{v}}\in (L^{\times})^n$ and that $\mathbf{h}_{\widetilde{v}}$ is the anti-dominant Hodge-Tate weights of $\rho_{\widetilde{v}}$ for $v\in S_p$. Recall for each $v\in S_p$, there exists $w_{\cR_{\widetilde{v}}}\in (\mathcal{S}_n)^{\Sigma_v}/W_{P_v}$ such that $z^{w_{\cR_{\widetilde{v}}}(\mathbf{h}_{\widetilde{v}})}\mathrm{unr}(\underline{\varphi}_{\widetilde{v}})$ is a parameter of $\rho_{\widetilde{v}}$ and then the point $\left(\rho_{\widetilde{v}},z^{w_{\cR_{\widetilde{v}}}(\mathbf{h}_{\widetilde{v}})}\mathrm{unr}(\underline{\varphi}_{\widetilde{v}})\right)$ is in $U_{\mathrm{tri}}(\overline{\rho}_{\widetilde{v}})\subset X_{\mathrm{tri}}(\overline{\rho}_{\widetilde{v}})$ (see \S\ref{sectionlocalcompanionpoints}). We continue to assume $\lambda_{\tau,i}=h_{\tau,n+1-i}+i-1,\forall \tau\in \Sigma_p, i=1,\cdots,n$.
\begin{definition}\label{definitioncharacterrefinementglobal}
	For $w=(w_v)_{v\in S_p}\in W_{G_p}/W_{P_p}$, let $\underline{\delta}_{\cR,w}$ be the character $\iota\left(z^{w(\mathbf{h})}\mathrm{unr}(\underline{\varphi})\right)$ of $T_p$ where $z^{w(\mathbf{h})}\mathrm{unr}(\underline{\varphi})$ is the abbreviation for the character $\prod_{v\in S_p}z^{w_v(\mathbf{h}_{\widetilde{v}})}\mathrm{unr}(\underline{\varphi}_{\widetilde{v}})$ of $T_p=\prod_{v\in S_p}T_v$. 
\end{definition}
Then for each $w\in W_{G_p}/W_{P_p}$, $\underline{\delta}_{\mathcal{R},w}$ is generic and $ww_0\cdot \lambda$ is the weight of $\underline{\delta}_{\mathcal{R},w}$. The smooth part $\underline{\delta}_{\mathcal{R},\mathrm{sm}}= z^{-ww_0\cdot \lambda}\underline{\delta}_{\mathcal{R},w}$ is independent of $w$. We write $w_{\cR}:=(w_{\cR_{\widetilde{v}}})_{v\in S_p}\in W_{G_p}/W_{P_p}$. \par
We already know that the point $y_{w_{\cR}}:=\left(\rho_p,\iota^{-1}(\underline{\delta}_{\cR,w_{\cR}})\right)$ is in $U_{\mathrm{tri}}(\overline{\rho}_p)$. If we assume Conjecture 3.23 of \cite{breuil2017interpretation} (which follows from a general automorphy lifting conjecture by \cite[Prop. 3.27]{breuil2017interpretation}), we could get that the irreducible component of $X_{\mathrm{tri}}(\overline{\rho}_p)$ passing through $y_{w_{\cR}}$ is $\fX^p$-automorphic for any irreducible component $\fX^p$ of $\fX_{\overline{\rho}^p}$. Hence there should exist a point $\left((\rho_p,\underline{\delta}_{\cR,w_{\cR}}),z\right)\in X_p(\overline{\rho})\subset\iota(X_{\mathrm{tri}}(\overline{\rho}_p))\times (\fX_{\overline{\rho}^p}\times \mathbb{U}^g)$. Then by \cite[Thm. 5.5]{breuil2017smoothness}, together with the discussion on local companion points in \S\ref{sectionlocalcompanionpoints}, we could expect $x_w:=\left((\rho_p, \underline{\delta}_{\cR,w}),z\right)\in X_p(\overline{\rho})$ if and only if $w\geq w_{\cR}$ in $W_{G_p}/W_{P_p}$.\par 
We will not consider the automorphy lifting anymore in this paper. Rather, our aim (Theorem \ref{theoremmaincrystalline}) is to prove that all companion points $x_{w},w\geq w_{\cR}$ are in $X_p(\overline{\rho})$ under the assumption that there exists $w'\in W_{G_p}/W_{P_p}$ such that $x_{w'}$ is in $X_p(\overline{\rho})$. The assumption will always imply $x_{w_0}\in X_p(\overline{\rho})$ by \cite[Thm. 5.5]{breuil2017smoothness}. \par
The following proposition is the key new step to achieve the existence of companion points, where Theorem \ref{theoremcyclepartialderham} is used. In the proof of Theorem \ref{theoremmaincrystalline}, we will use some induction and deformation arguments to reduce the existence of more general companion points on $X_p(\overline{\rho})$ to the special situation considered in Proposition \ref{propositionmaincycle} below.
\begin{proposition}\label{propositionmaincycle}
	Assume that the points $\left((\rho_p,\underline{\delta}_{\cR,w}),z\right)\in \iota\left(X_{\mathrm{tri}}(\overline{\rho}_p)\right)\times (\mathfrak{X}_{\overline{\rho}^p}\times\mathbb{U}^g)$ are in $X_p(\overline{\rho})(L)$ for any $w>w_{\cR}$ in $W_{G_p}/W_{P_p}$ where $\rho_p$ is generic crystalline and $\cR$ is a refinement of $\rho_p$ as above. If $w_{\cR}W_{P_p}\neq w_0W_{P_p}$, then $\left((\rho_p,\underline{\delta}_{\cR,w_{\cR}}),z\right)\in X_p(\overline{\rho})(L)$.	
\end{proposition}
\begin{proof}
	We will prove 
	\[\Hom_{U(\fg)}\left(L(w_{\cR}w_0\cdot\lambda),\Pi_{\infty}^{\mathrm{an}}\right)^{U_0}[\fm_{r_x}^{\infty}][\fm_{\underline{\delta}_{\cR,\mathrm{sm}}}^{\infty}]\neq 0\]
	where $r_x:=(\rho_p,z)\in\fX_{\infty}$. This will imply (by (\ref{equationadjointliealgebra}))
	\[\Hom_{G_p}\left(\mathcal{F}_{\overline{B}_p}^{G_p}\left(\overline{L}(-w_{\cR}w_0\cdot\lambda),\underline{\delta}_{\cR,\mathrm{sm}}\delta_{B_p}^{-1}\right),\Pi_{\infty}^{\mathrm{an}}[\fm_{r_x}]\right)\neq 0\] 
	and by (\ref{equationorlikstrauchadjunction}) and (\ref{equationsocleinjection}), imply furthermore that $\left((\rho_p,\underline{\delta}_{\cR,w_{\cR}}),z\right)\in X_p(\overline{\rho})$.\par
	For each $v\in S_p$, we still write $w_{\cR_{\widetilde{v}}}$ for the shortest representative of $w_{\cR_{\widetilde{v}}}$ in $(\cS_n)^{\Sigma_v}$ and write $w_{\cR}=(w_{\cR_{\widetilde{v}}})_{v\in S_p}\in W_{G_p}$. We may assume that there exists and fix a place $v$ such that $w_{\cR_{\widetilde{v}}}$ is not in the coset $w_{v,0}W_{P_v}$ since $w_{\cR}W_{P_p}\neq w_0W_{P_p}$. Then by Lemma \ref{lemmakeyinductionweylgroup}, there exists a simple root $\alpha$ of $(\mathrm{Res}_{F_{\widetilde{v}}/\Q_p}\GL_{n/F_{\widetilde{v}}})\times_{\Q_p}L$ and a standard parabolic subgroup $Q=\prod_{v\in S_p}Q_v=\prod_{\tau\in \Sigma_p}Q_{\tau}$ of $\prod_{v\in S_p}(\mathrm{Res}_{F_{\widetilde{v}}/\Q_p}\GL_{n/F_{\widetilde{v}}})\times_{\Q_p}L$ such that the element $w':=s_{\alpha}w_{\cR}=(w_{v'}')_{v'\in S_p}\in W_{G_p}$ where $w_{v'}'=w_{\cR_{\widetilde{v}'}}$ if $v'\neq v$ and $w_v'=s_{\alpha}w_{\cR_{\widetilde{v}}}$ satisfies that $\lg_{P_p}(w')=\lg_{P_p}(w_{\cR})+1, w'_v(\mathbf{h}_{\widetilde{v}})$ is strictly $Q_v$-dominant and $w_{\cR_{\widetilde{v}}}(\mathbf{h}_{\widetilde{v}})$ is \emph{not} strictly $Q_v$-dominant (Definition \ref{definitionstrictlyQdominant}). By our assumption, the point $x_{w'}:=\left((\rho_p,\underline{\delta}_{\cR,w'}),z\right)$ is in $X_p(\overline{\rho})$. \par
	In particular, ${\cM_{\infty}\otimes_{\cO_{X_p(\overline{\rho})}}\widehat{\cO}_{X_p(\overline{\rho})_{w'w_0\cdot\lambda},x_{w'}}\neq 0}$. Equivalently, since $\cM_{\infty}$ is defined using $J_{B_p}(\Pi^{\mathrm{an}}_{\infty})=((\Pi^{\mathrm{an}}_{\infty})^{U_0})_{\mathrm{fs}}$, by taking dual and arguing as in (\ref{equationsoclesheafstalklimit}), we get
	\[\Hom_{U(\fg)}\left(U(\fg)\otimes_{U(\fb)}w'w_0\cdot\lambda,\Pi_{\infty}^{\mathrm{an}}\right)^{U_0}[\fm_{r_x}^{\infty}][\fm_{\underline{\delta}_{\cR,\mathrm{sm}}}^{\infty}]\neq 0,\]
	(see also \cite[(5.16), (5.18)]{breuil2019local}).\par 
	If \[\Hom_{G_p}\left(\mathcal{F}_{\overline{B}_p}^{G_p}\left(\overline{L}(-ww_0\cdot\lambda),\underline{\delta}_{\cR,\mathrm{sm}}\delta_{B_p}^{-1}\right),\Pi_{\infty}^{\mathrm{an}}[\fm_{r_x}]\right)\neq 0,\] 
	then by (\ref{equationorlikstrauchadjunction}) and (\ref{equationsocleinjection}), the point $\left((\rho_p,\underline{\delta}_{\cR,w}),z\right)$ will appear in $X_p(\overline{\rho})$. Then $(\rho_p,\iota^{-1}(\underline{\delta}_{\cR,w}))\in X_{\mathrm{tri}}(\overline{\rho}_p)$. This is possible only if $w\geq w_{\cR}$ by Theorem \ref{sectionlocalcompanionpoints}.
	Hence the irreducible constituents of $\mathcal{F}_{\overline{B}_p}^{G_p}(\underline{\delta}_{\cR,w'})$ that may appear as subrepresentations in $\Pi_{\infty}^{\mathrm{an}}[\fm_{r_x}]$ are 
	\[\mathcal{F}_{\overline{B}_p}^{G_p}\left(\overline{L}(-ww_0\cdot\lambda),\underline{\delta}_{\cR,\mathrm{sm}}\delta_{B_p}^{-1}\right)
	\]
	where $w'\geq w\geq w_{\cR}$ in $W_{G_p}/W_{P_p}$, i.e. $w=w'$ or $w=w_{\cR}$.
	And for any $w\leq w',w\notin \{ w',w_{\cR}\}$, 
	\[\Hom_{U(\fg)}\left(L(ww_0\cdot\lambda),\Pi_{\infty}^{\mathrm{an}}\right)^{U_0}[\fm_{r_x}^{\infty}][\fm_{\underline{\delta}_{\cR,\mathrm{sm}}}^{\infty}]=0.\] 
	The functor $M\mapsto \Hom_{U(\fg)}\left(M,\Pi_{\infty}^{\mathrm{an}}\right)^{U_0}[\fm_{r_x}^{\infty}][\fm_{\underline{\delta}_{\cR,\mathrm{sm}}}^{\infty}]$ is exact on the category $\cO$. This fact comes essentially from that the dual of $\Pi_{\infty}$ is a finite projective $S_{\infty}[[K_p]][\frac{1}{p}]$-module and that the functors of taking generalized eigenspace of compact operators are exact. See discussions before \cite[(5.21)]{breuil2019local} (and the arguments in \cite[Thm. 5.5]{breuil2017smoothness} of proving firstly similar results for ideals of $S_{\infty}[\frac{1}{p}]$). \par
	Hence we get an exact sequence
	\begin{align*}
		0\rightarrow \Hom_{U(\fg)}\left(L(w'w_0\cdot\lambda),\Pi_{\infty}^{\mathrm{an}}\right)^{U_0}[\fm_{r_x}^{\infty}][\fm_{\underline{\delta}_{\cR,\mathrm{sm}}}^{\infty}] \\\rightarrow
		\Hom_{U(\fg)}\left(U(\fg)\otimes_{U(\fb)}w'w_0\cdot\lambda,\Pi_{\infty}^{\mathrm{an}}\right)^{U_0}&[\fm_{r_x}^{\infty}][\fm_{\underline{\delta}_{\cR,\mathrm{sm}}}^{\infty}]\\
		\rightarrow \Hom_{U(\fg)}\left(L(w_{\cR}w_0\cdot\lambda),\Pi_{\infty}^{\mathrm{an}}\right)^{U_0}[\fm_{r_x}^{\infty}][\fm_{\underline{\delta}_{\cR,\mathrm{sm}}}^{\infty}]&\rightarrow 0.
	\end{align*}
	To prove $\Hom_{U(\fg)}\left(L(w_{\cR}w_0\cdot\lambda),\Pi_{\infty}^{\mathrm{an}}\right)^{U_0}[\fm_{r_x}^{\infty}][\fm_{\underline{\delta}_{\cR,\mathrm{sm}}}^{\infty}]\neq 0$, by the above exact sequence, we only need to show that \[\Hom_{U(\fg)}\left(L(w'w_0\cdot\lambda),\Pi_{\infty}^{\mathrm{an}}\right)^{U_0}[\fm_{r_x}^{\infty}][\fm_{\underline{\delta}_{\cR,\mathrm{sm}}}^{\infty}]\neq \Hom_{U(\fg)}\left(U(\fg)\otimes_{U(\fb)}w'w_0\cdot\lambda,\Pi_{\infty}^{\mathrm{an}}\right)^{U_0}[\fm_{r_x}^{\infty}][\fm_{\underline{\delta}_{\cR,\mathrm{sm}}}^{\infty}]\]
	or equivalently by taking dual as in (\ref{equationsoclesheafstalklimit}) to show that the map  
	\begin{equation}\label{equationsurjectionsheavescontradiction}
		\cM_{\infty}\otimes_{\cO_{X_p(\overline{\rho})}}\widehat{\cO}_{X_p(\overline{\rho})_{w'w_0\cdot\lambda},x_{w'}}\twoheadrightarrow \mathcal{L}_{w'w_0\cdot\lambda}\otimes_{\cO_{X_p(\overline{\rho})_{w'w_0\cdot\lambda}}}\widehat{\cO}_{X_p(\overline{\rho})_{w'w_0\cdot\lambda},x_{w'}}
	\end{equation}
	is not an isomorphism. \par
	We prove it by contradiction. Assume that (\ref{equationsurjectionsheavescontradiction}) is an isomorphism. The action of $\widehat{\cO}_{X_p(\overline{\rho})_{w'w_0\cdot\lambda},x_{w'}}$ on the right-hand side factors through 
	\[\widehat{\cO}_{Y_p(\overline{\rho})_{w’w_0\cdot\lambda},x_{w'}}={\cO_{Y_p(\overline{\rho})_{w’w_0\cdot\lambda},x_{w'}}\otimes_{\cO_{X_p(\overline{\rho})_{w'w_0\cdot\lambda},x_{w'}}}\widehat{\cO}_{X_p(\overline{\rho})_{w'w_0\cdot\lambda},x_{w'}}}.\] 
	Thus by Lemma \ref{lemmasupportcompletion} below the $\widehat{\cO}_{X_p(\overline{\rho})_{w'w_0\cdot\lambda},x_{w'}}$-module in the right hand side of (\ref{equationsurjectionsheavescontradiction}) has support whose underlying reduced subspace is $\Spec(\widehat{\cO}_{Y_p(\overline{\rho})_{w’w_0\cdot\lambda}^{\mathrm{red}},x_{w'}})$. Since $\cO_{X_p(\overline{\rho})}$ acts faithfully on $\cM_{\infty}$, the support of $\cM_{\infty}\otimes_{\cO_{X_p(\overline{\rho})}}\cO_{X_p(\overline{\rho})_{w'w_0\cdot\lambda}}$ is set-theoretically equal to $X_p(\overline{\rho})_{w'w_0\cdot\lambda}$ (\cite[\href{https://stacks.math.columbia.edu/tag/00L3}{Tag 00L3}]{stacks-project}) with the underlying reduced subspace $X_p(\overline{\rho})_{w'w_0\cdot\lambda}^{\mathrm{red}}$. By Lemma \ref{lemmasupportcompletion} below, the underlying reduced subscheme of the support of the $\widehat{\cO}_{X_p(\overline{\rho})_{w'w_0\cdot\lambda},x_{w'}}$-module in the left hand side of (\ref{equationsurjectionsheavescontradiction}) is $\Spec(\widehat{\cO}_{X_p(\overline{\rho})_{w'w_0\cdot\lambda}^{\mathrm{red}},x_{w'}})$. Thus we have
	\[\Spec(\widehat{\cO}_{X_p(\overline{\rho})_{w'w_0\cdot\lambda}^{\mathrm{red}},x_{w'}})=\Spec(\widehat{\cO}_{Y_p(\overline{\rho})_{w’w_0\cdot\lambda}^{\mathrm{red}},x_{w'}}).\]
	By Theorem \ref{theoremcyclepartialderham} and the above equality, the map 
	\[R_{\rho_p}\rightarrow \widehat{\cO}_{X_{\mathrm{tri}}(\overline{\rho}_p)_{w'(\mathbf{h})},\iota^{-1}(\rho_p,\underline{\delta}_{\cR,w'})}\widehat{\otimes}_L \widehat{\cO}_{\mathfrak{X}_{\overline{\rho}^p}\times\mathbb{U}^g,z}\rightarrow \widehat{\cO}_{X_p(\overline{\rho})_{w'w_0\cdot\lambda}^{\mathrm{red}},x_{w'}}\] 
	factors through $R_{\rho_p,\cM_{\bullet}}^{Q_p}$. We now prove this is not possible. Assume that $\iota(X)\times \fX^p\times \mathbb{U}^g$ is an irreducible component of $X_p(\overline{\rho})$ passing through $x_{w'}$ where $X$ is the unique irreducible component of $X_{\mathrm{tri}}(\overline{\rho}_p)$ passing through $\iota^{-1}(\rho_p,\underline{\delta}_{\cR,w'})$ and $\fX^{p}$ is an irreducible (reduced) component of $\fX_{\overline{\rho}^p}$. Let $X_{w'(\mathbf{h})}$ be the fiber of $X$ over the weight $w'(\mathbf{h})$ and let $X_{w'(\mathbf{h})}^{\mathrm{red}}$ be the underlying reduced subspace. Then $\iota(X_{w'(\mathbf{h})}^{\mathrm{red}})\times \fX^p\times \mathbb{U}^g$ is a reduced subspace of $X_p(\overline{\rho})_{w'w_0\cdot\lambda}^{\mathrm{red}}$ and the map 
	\[R_{\rho_p}\rightarrow \widehat{\cO}_{X^{\mathrm{red}}_{w'(\mathbf{h})},\iota^{-1}(\rho_p,\underline{\delta}_{\cR,w'})}\widehat{\otimes}_L \widehat{\cO}_{\fX^p\times\mathbb{U}^g,z}\] factors through $R_{\rho_p,\cM_{\bullet}}^{Q_p}$. Hence the map $R_{\rho_p}\rightarrow \widehat{\cO}_{X^{\mathrm{red}}_{w'(\mathbf{h})},\iota^{-1}(\rho_p,\underline{\delta}_{\cR,w'})}$ factors through $R_{\rho_p,\cM_{\bullet}}^{Q_p}$. Denote by $y_{w'_v}:=\left(\rho_{\widetilde{v}},z^{w'_v(\mathbf{h}_{\widetilde{v}})}\mathrm{unr}(\underline{\varphi}_{\widetilde{v}})\right)$ the point on $X_{\mathrm{tri}}(\overline{\rho}_{\widetilde{v}})_{w'_v(\mathbf{h}_{\widetilde{v}})}$ 
	and $X_v$ the irreducible component passing the point with $X_{w'_v(\mathbf{h}_{\widetilde{v}})}$ the fiber of $X_v$ over the weight $w'_v(\mathbf{h}_{\widetilde{v}})$. We get that the morphism $R_{\rho_{\widetilde{v}}}\twoheadrightarrow R_{\rho_{\widetilde{v}},\cM_{\widetilde{v},\bullet}}^{w_v'}\twoheadrightarrow \widehat{\cO}_{X_{w'_v(\mathbf{h}_{\widetilde{v}})}^{\mathrm{red}},y_{w'_v}}$ factors through $R_{\rho_{\widetilde{v}},\cM_{\widetilde{v},\bullet}}^{Q_v}$, i.e.
	\begin{equation}\label{equationinclusioncycle}
		\Spec(\widehat{\cO}_{X_{w'_v(\mathbf{h}_{\widetilde{v}})}^{\mathrm{red}},y_{w'_v}})\subset \Spec(R_{\rho_{\widetilde{v}},\cM_{\widetilde{v},\bullet}}^{Q_v}). 
	\end{equation} 
	By discussions in \S\ref{sectioncycles} and (\ref{formulairreduciblecomponents1}), the underlying topological space $\Spec(\widehat{\cO}_{X_{w'_v(\mathbf{h}_{\widetilde{v}})}^{\mathrm{red}},y_{w'_v}})$, which as a topological space is equal to $\Spec(\overline{R}_{\rho_{\widetilde{v}},\cM_{\widetilde{v},\bullet}}^{w_v'})$, is a union of non-empty cycles denoted by $\mathfrak{Z}_{w_{\cR_{\widetilde{v}}}}$ and $\mathfrak{Z}_{w_v'}$ in \S\ref{sectioncycles}. By our choice of $Q_v$, $w_{\cR_{\widetilde{v}}}(\mathbf{h}_{\widetilde{v}})$ is not strictly $Q_v$-dominant. Then by Lemma \ref{lemmakeyinductionweylgroup1} or the discussion in the end of \S\ref{sectioncycles}, $\mathfrak{Z}_{w_{\cR_{\widetilde{v}}}}$ is not contained in $\Spec(R_{\rho_{\widetilde{v}},\cM_{\widetilde{v},\bullet}}^{Q_v})$, this contradicts (\ref{equationinclusioncycle})!
\end{proof}
\begin{lemma}\label{lemmasupportcompletion}
	Let $A$ be an excellent Noetherian ring, $\fm$ be a maximal ideal of $A$, $J$ be the nilradical of $A$, and $M$ be a finite $A$-module with a faithful action of $A$. Let $\widehat{A}$ be the $\fm$-adic completion of $A$.
	\begin{enumerate}
		\item $\widehat{A/J}=\widehat{A}/\widehat{J}$ is the nilreduction of $\widehat{A}$.
		\item $\widehat{A}$ acts on $\widehat{M}:=\widehat{A}\otimes_AM$ faithfully and the underlying reduced scheme of the support of $\widehat{M}$ is $\Spec(\widehat{A/J})$.
	\end{enumerate}
\end{lemma}
\begin{proof}
	(1) The sequence $0\rightarrow \widehat{J}\rightarrow\widehat{A}\rightarrow \widehat{A/J}\rightarrow 0$ is exact (\cite[Prop. 10.12]{atiyah1969introduction}). Moreover, $\widehat{A/J}$ is reduced (\cite[\href{https://stacks.math.columbia.edu/tag/07NZ}{Tag 07NZ}]{stacks-project}) and $\widehat{J}$ is nilpotent. Hence $\widehat{J}$ is the nilradical of $\widehat{A}$.\par
	(2) We have a natural injection $A\hookrightarrow \Hom_A(M,M)$ of finite $A$-modules. Tensoring with $\widehat{A}$, we get an injection $\widehat{A}\hookrightarrow \Hom_A(M,M)\otimes_A\widehat{A}=\Hom_{\widehat{A}}(\widehat{M},\widehat{M})$ (\cite[Cor. 0.7.3.4]{grothendieck1960EGAI}) of $\widehat{A}$-modules by the flatness of $\widehat{A}$ over $A$ (\cite[Prop. 10.14]{atiyah1969introduction}). The injection means that $\widehat{A}$ acts faithfully on $\widehat{M}$. Hence the support of $\widehat{M}$ as a $\widehat{A}$-module is $\Spec(\widehat{A})$ and the underlying reduced subscheme is $\Spec(\widehat{A/J})$ by (1).
\end{proof}
The remaining steps mainly rely on the local irreducibility of the trianguline variety at generic points and the crystalline deformation spaces introduced in the proof of Theorem \ref{theoremlocalcompanionpoint}. The following proposition reduces the existence of companion constituents in the generic crystalline cases to the existence of companion points.
\begin{proposition}\label{propositionsocleappear}
	Let $x=\left((\rho_p,\underline{\delta}_{\cR,w}),z\right)\in X_p(\overline{\rho})(L)\subset\iota\left(X_{\mathrm{tri}}(\overline{\rho}_p)\right)\times (\mathfrak{X}_{\overline{\rho}^p}\times\mathbb{U}^g)$ be a point such that $\rho_p$ is generic crystalline with a refinement $\cR$, $w\in W_{G_p}/W_{P_p}$ and the weight of $\underline{\delta}_{\cR,w}$ (Definition \ref{definitioncharacterrefinementglobal}) is $ww_0\cdot \lambda$. Let $r_x$ be the image of $x$ in $\mathfrak{X}_{\infty}$ and $\fm_{r_x}$ be the corresponding maximal ideal of $R_{\infty}[\frac{1}{p}]$. Then 
	\[\Hom_{G_p}\left(\mathcal{F}_{\overline{B}_p}^{G_p}\left(\overline{L}(-ww_0\cdot\lambda),\underline{\delta}_{\cR,\mathrm{sm}}\delta_{B_p}^{-1}\right),\Pi_{\infty}^{\mathrm{an}}[\fm_{r_x}]\right)\neq 0.\]
\end{proposition}
\begin{proof}
	The method is similar to \textbf{Step 8} and \textbf{Step 9} in the proof of \cite[Thm. 5.3.3]{breuil2019local}. We will firstly prove the result for the case when $w=w_{\cR}$ which will be a consequence of (\ref{equationorlikstrauchadjunction}) and Theorem \ref{theoremlocalcompanionpoint}. For general points, notice that the appearance of the companion constituents is equivalent to that $x$ is in the Zariski closed subspace $Y_p(\overline{\rho})_{ww_0\cdot\lambda}$ of $X_{p}(\overline{\rho})$ constructed in the beginning of \S\ref{sectionthesocleconjection}. Thus, it suffices to prove that $x$ lies in the closure of generic crystalline points in $Y_p(\overline{\rho})_{ww_0\cdot\lambda}$ satisfying $w=w_{\cR}$. This can be achieved using the variants of the crystalline deformation space defined in the proof of Theorem \ref{theoremlocalcompanionpoint}.\par
	First, assume $wW_{P_p}=w_{\cR}W_{P_p}$ and $\Hom_{G_p}\left(\mathcal{F}_{\overline{B}_p}^{G_p}\left(\overline{L}(-ww_0\cdot\lambda),\underline{\delta}_{\cR,\mathrm{sm}}\delta_{B_p}^{-1}\right),\Pi_{\infty}^{\mathrm{an}}[\fm_{r_x}]\right)=0$, then there exists $w'\in W_{G_p}/W_{P_p}$ and $w'<w$ in $W_{G_p}/W_{P_p}$ such that 
	\[\Hom_{G_p}\left(\mathcal{F}_{\overline{B}_p}^{G_p}(\overline{L}(-w'w_0\cdot\lambda),\underline{\delta}_{\cR,\mathrm{sm}}\delta_{B_p}^{-1}),\Pi_{\infty}^{\mathrm{an}}[\fm_{r_x}]\right)\neq 0\] 
	since $\Hom_{G_p}\left(\mathcal{F}_{\overline{B}_p}^{G_p}(\underline{\delta}_{\cR,w}),\Pi_{\infty}^{\mathrm{an}}[\fm_{r_x}]\right)\neq 0$. Therefore $\Hom_{G_p}\left(\mathcal{F}_{\overline{B}_p}^{G_p}(\underline{\delta}_{\cR,w'}),\Pi_{\infty}^{\mathrm{an}}[\fm_{r_x}]\right)\neq 0$. By (\ref{equationorlikstrauchadjunction}), the point 
	$\left((\rho_p,\underline{\delta}_{\cR,w'}),z\right)\in (\mathfrak{X}_{\overline{\rho}_p}\times \widehat{T}_{p,L})\times(\mathfrak{X}_{\overline{\rho}^p}\times\mathbb{U}^g)$ is in $X_p(\overline{\rho})$. Then the point $\left(\rho_p, z^{w'(\mathbf{h})}\mathrm{unr}(\underline{\varphi})\right)$ is in $X_{\mathrm{tri}}(\overline{\rho}_p)$. This is not possible by Theorem \ref{theoremirreducibletriangullinevariety}. Hence the conclusion holds in the case when $w=w_{\cR}$.\par
	In general, by Theorem \ref{theoremirreducibletriangullinevariety} and the generic assumption on $\rho_p$, $\left(\rho_p, z^{w(\mathbf{h})}\mathrm{unr}(\underline{\varphi})\right)$ lies in a unique irreducible component $X$ of $X_{\mathrm{tri}}(\overline{\rho}_p)$. We assume that $x$ lies in an irreducible component of $X_p(\overline{\rho})$ of the form $\iota(X)\times \fX^p\times \mathbb{U}^g$ for an irreducible component $\fX^p$ of $\fX_{\overline{\rho}^p}$. \par
	In the proof of Theorem \ref{theoremlocalcompanionpoint}, we have constructed crystalline deformation spaces $\widetilde{W}_{\overline{\rho}_{\widetilde{v}},w_v}^{\mathbf{h}_{\widetilde{v}}-\mathrm{cr}}$ (resp. the closure $\overline{\widetilde{W}_{\overline{\rho}_{\widetilde{v}},w_v}^{\mathbf{h}_{\widetilde{v}}-\mathrm{cr}}}$), which roughly parameterizing the pairs $(\rho_v,\cR_v)$ of generic crystalline deformations with refinements satisfying that $w_{\cR_v}=w_v$ (resp. $w_{\cR_v}\leq w_v$). We have also defined morphisms $\iota_{\mathbf{h}_{\widetilde{v}},w_v}:\overline{\widetilde{W}_{\overline{\rho}_{\widetilde{v}},w_v}^{\mathbf{h}_{\widetilde{v}}-\mathrm{cr}}}\rightarrow X_{\mathrm{tri}}(\overline{\rho}_{\widetilde{v}})$ in the end of the proof of Theorem \ref{theoremlocalcompanionpoint} sending $(\rho_v,\cR_v=\underline{\varphi}_v)$ to $(\rho_v,z^{w_v(\mathbf{h}_{\widetilde{v}})}\mathrm{unr}(\underline{\varphi}_v))$. Let $\iota_{\mathbf{h},w}:=\prod_{v\in S_p}\iota_{\mathbf{h}_{\widetilde{v}},w_v}:\prod_{v\in S_p}\overline{\widetilde{W}_{\overline{\rho}_{\widetilde{v}},w_v}^{\mathbf{h}_{\widetilde{v}}-\mathrm{cr}}}\rightarrow \prod_{v\in S_p}X_{\mathrm{tri}}(\overline{\rho}_{\widetilde{v}})=X_{\mathrm{tri}}(\overline{\rho}_p)$. Then the point $\left(\rho_p, z^{w(\mathbf{h})}\mathrm{unr}(\underline{\varphi})\right)$ is in the image of $\iota_{\mathbf{h},w}$ since $w\geq w_{\cR}$. Denote by $\overline{\widetilde{W}_{\overline{\rho}_p,w}^{\mathbf{h}-\mathrm{cr}}}=\prod_{v\in S_p}\overline{\widetilde{W}_{\overline{\rho}_{\widetilde{v}},w_v}^{\mathbf{h}_{\widetilde{v}}-\mathrm{cr}}}$ and $\widetilde{W}_{\overline{\rho}_p,w}^{\mathbf{h}-\mathrm{cr}}=\prod_{v\in S_p}\widetilde{W}_{\overline{\rho}_{\widetilde{v}},w_v}^{\mathbf{h}_{\widetilde{v}}-\mathrm{cr}}$. \par
	We take an affinoid neighbourhood $U$ of $\left(\rho_p, z^{w(\mathbf{h})}\mathrm{unr}(\underline{\varphi})\right)$ in $X$ and pick a small open affinoid $V\subset \iota_{\mathbf{h},w}^{-1}(U)$ such that $\left(\rho_p,z^{w(\mathbf{h})}\mathrm{unr}(\underline{\varphi})\right)\in \iota_{\mathbf{h},w}(V)$. Then $\widetilde{W}_{\overline{\rho}_p,w}^{\mathbf{h}-\mathrm{cr}}\cap V$ is Zariski open dense in $V$. Points in $\iota\circ\iota_{\mathbf{h},w}(\widetilde{W}_{\overline{\rho}_p,w}^{\mathbf{h}-\mathrm{cr}}\cap V)\times \fX^p\times \mathbb{U}^g\subset X_p(\overline{\rho})_{ww_0\cdot\lambda}\cap (\iota(X)\times \fX^p\times\mathbb{U}^g)$ are generic crystalline. Hence by the discussion above for the case $w=w_{\cR}$, we have $\iota\circ\iota_{\mathbf{h},w}(\widetilde{W}_{\overline{\rho}_p,w}^{\mathbf{h}-\mathrm{cr}}\cap V)\times \fX^p\times \mathbb{U}^g\subset Y_p(\overline{\rho})_{ww_0\cdot\lambda}$, i.e. 
	\[(\widetilde{W}_{\overline{\rho}_p,w}^{\mathbf{h}-\mathrm{cr}}\cap V)\times \fX^p\times \mathbb{U}^g\subset \left(\left(\iota\circ\iota_{\mathbf{h},w}\right)\times \mathrm{id}\times \mathrm{id}\right)^{-1}(Y_p(\overline{\rho})_{ww_0\cdot\lambda}).\] 
	Since $Y_p(\overline{\rho})_{ww_0\cdot\lambda}$ is Zariski closed in $X_p(\overline{\rho})$, $\left(\left(\iota\circ\iota_{\mathbf{h},w}\right)\times \mathrm{id}\times \mathrm{id}\right)^{-1}(Y_p(\overline{\rho})_{ww_0\cdot\lambda})\cap  (V\times \fX^p\times \mathbb{U}^g)$ is Zariski closed in $V\times \fX^p\times \mathbb{U}^g$. Hence $V\times \fX^p\times \mathbb{U}^g\subset ((\iota\circ\iota_{\mathbf{h},w})\times \mathrm{id}\times \mathrm{id})^{-1}(Y_p(\overline{\rho})_{ww_0\cdot\lambda})$. This implies $\left((\rho_p,\underline{\delta}_{\cR,w}),z\right)\in Y_p(\overline{\rho})_{ww_0\cdot\lambda}$. 
\end{proof}
Now we can prove our main theorem.
\begin{theorem}\label{theoremmaincrystalline}
	Assume that there exists $w'\in W_{G_p}/W_{P_p}$ such that the point 
	\[\left((\rho_p,\underline{\delta}_{\cR,w'}),z\right)\in\iota\left(X_{\mathrm{tri}}(\overline{\rho}_p)\right)\times (\mathfrak{X}_{\overline{\rho}^p}\times\mathbb{U}^g) \]
	is in $X_p(\overline{\rho})(L)$ where $\rho_p$ is generic crystalline and $\cR$ is a refinement of $\rho_p$. Then $\left((\rho_p,\underline{\delta}_{\cR,w}),z\right)\in X_p(\overline{\rho})(L)$ if and only if $w\geq w_{\cR}$ in $W_{G_p}/W_{P_p}$.
\end{theorem}
\begin{proof}
	The ``only if'' part follows from Theorem \ref{theoremlocalcompanionpoint}. \par
	By \cite[Thm. 5.5]{breuil2017smoothness} and the assumption $\left((\rho_p,\underline{\delta}_{\cR,w'}),z\right)\in X_p(\overline{\rho})$, we can assume 
	\[\left((\rho_p,\underline{\delta}_{\cR,w_0}),z\right)\in X_p(\overline{\rho}).\]
	We prove the ``if'' part by descending induction on the integers $\ell\leq \lg_{P_p}(w_0)$ for the following induction hypothesis: \\
	$\mathcal{H}_{\ell}:$ if $y_{w_0}= \left((\rho_p,\underline{\delta}_{\cR,w_0}),z\right)\in X_p(\overline{\rho})$ is a generic crystalline point of Hodge-Tate weights $\mathbf{h}$ with a refinement $\cR$, then for any $w$ such that $w \geq w_{\cR}$ and $\lg_{P_p}(w)\geq \ell$, the point $y_{w}:=\left((\rho_p,\underline{\delta}_{\cR,w}),z\right)\in \iota\left(X_{\mathrm{tri}}(\overline{\rho}_p)\right)\times \mathfrak{X}_{\overline{\rho}^p}\times\mathbb{U}^g$ is in $X_p(\overline{\rho})$. \par
	For $\ell=\lg_{P_p}(w_0)$, there is nothing to prove. We now assume $\mathcal{H}_{\ell}$ and prove $\mathcal{H}_{\ell-1}$.\par
	Firstly, the assertion of $\mathcal{H}_{\ell-1}$ holds under $\mathcal{H}_{\ell}$ automatically for any generic crystalline $y_{w_0}$ and $w$ if $\lg_{P_p}(w_{\cR})\geq \ell$. By Proposition \ref{propositionmaincycle}, $\mathcal{H}_{\ell}$ implies that if $y_{w_0}$ as above is in $X_p(\overline{\rho})$ and $\lg_{P_p}(w_{\cR})= \ell-1$, then $y_{w_{\cR}}\in X_p(\overline{\rho})$. By Proposition \ref{propositionsocleappear}, $y_{w_{\cR}}\in Y_p(\overline{\rho})_{w_{\cR}w_0\cdot\lambda}$ (Proposition \ref{propositionmaincycle} and \ref{propositionsocleappear} are proved for $L$-points, but the equivalent statements for companion constituents can be proved in the same way after enlarging the coefficient field $L$).\par
	Thus, the assertion of $\mathcal{H}_{\ell-1}$ holds at least for generic crystalline points $y_{w_0}$ such that $\lg_{P_p}(w_{\cR})= \ell-1$. For more general crystalline generic points $y_{w_0},w\geq w_{\cR}$ and $\mathrm{lg}_{P_p}(w)\geq \ell-1$, we will show that $y_w$ lies in the closure of points $y'_w=\left((\rho_p',\underline{\delta}_{\cR',w}),z\right)$, the companion points of generic crystalline points of the form $y'_{w_0}=\left((\rho_p',\underline{\delta}_{\cR',w}),z\right)\in X_p(\overline{\rho})$ such that $\mathrm{lg}_{P_p}(w_{\cR'})=\ell-1$. Since $y'_w$ are in $X_p(\overline{\rho})$ by the previous discussions, $y_w$ will be also in $X_p(\overline{\rho})$.\par
	We take an arbitrary $y_{w_0}:=\left((\rho_p,\underline{\delta}_{\cR,w_0}),z\right)\in X_p(\overline{\rho})$ as in the hypothesis $\mathcal{H}_{\ell-1}$, we need to prove that for any $w$ such that $w\geq w_{\cR}$, $\lg_{P_p}(w)=\ell-1$, we have $y_{w}\in X_p(\overline{\rho})$. We may assume $\lg_{P_p}(w_{\cR})\leq \ell-1$ and by proving the equivalent statement on companion constituents, we may assume $y_{w_0}$ is an $L$-point. Recall as in the proof of Theorem \ref{theoremlocalcompanionpoint} or Proposition \ref{propositionsocleappear}, for each $v\in S_p$, there is a variant of crystalline deformation space $\iota_{\mathbf{h}_{\widetilde{v}},w_{v,0}}:\widetilde{W}_{\overline{\rho}_{\widetilde{v}}}^{\mathbf{h}_{\widetilde{v}}-\mathrm{cr}}\hookrightarrow X_{\mathrm{tri}}(\overline{\rho}_{\widetilde{v}})$ with the image consisting of generic crystalline points with weights $w_{v,0}(\mathbf{h}_{\widetilde{v}})$ and for each $w\in (\mathcal{S}_n)^{\Sigma_v}/W_{P_v}$, there exists a Zariski locally closed subset $\widetilde{W}_{\overline{\rho}_{\widetilde{v}},w_v}^{\mathbf{h}_{\widetilde{v}}-\mathrm{cr}}$ and its closure $\overline{\widetilde{W}_{\overline{\rho}_{\widetilde{v}},w_v}^{\mathbf{h}_{\widetilde{v}}-\mathrm{cr}}}=\cup_{w'_v\leq w_v} \widetilde{W}_{\overline{\rho}_{\widetilde{v}},w'_v}^{\mathbf{h}_{\widetilde{v}}-\mathrm{cr}}$ with an injection $\iota_{\mathbf{h}_{\widetilde{v}},w_v}: \overline{\widetilde{W}_{\overline{\rho}_{\widetilde{v}},w_v}^{\mathbf{h}_{\widetilde{v}}-\mathrm{cr}}}\hookrightarrow X_{\mathrm{tri}}(\overline{\rho}_{\widetilde{v}})$ sending $(\rho_{\widetilde{v}},\underline{\varphi}_{\widetilde{v}})$ to $\left(\rho_{\widetilde{v}}, z^{w_v(\mathbf{h}_{\widetilde{v}})}\mathrm{unr}(\underline{\varphi}_{\widetilde{v}})\right)$. The image of $\iota_{\mathbf{h}_{\widetilde{v}},w_v}$ consists of generic crystalline points of the weight $w_v(\mathbf{h}_{\widetilde{v}})$ such that the relative positions of the trianguline filtrations and the Hodge filtrations are parameterized by some $w_v'\leq w_v$ in $(\mathcal{S}_n)^{\Sigma_v}/W_{P_v}$. We let $\iota_{\mathbf{h},w_0}:=\prod_{v\in S_p}\iota_{\mathbf{h}_{\widetilde{v}},w_{v,0}}$ (resp. $\iota_{\mathbf{h},w}:=\prod_{v\in S_p}\iota_{\mathbf{h}_{\widetilde{v}},w_{v}}$) be the embedding of $\widetilde{W}_{\overline{\rho}_p}^{\mathbf{h}-\mathrm{cr}}:=\prod_{v\in S_p}\widetilde{W}_{\overline{\rho}_{\widetilde{v}}}^{\mathbf{h}_{\widetilde{v}}-\mathrm{cr}}$ (resp. $\overline{\widetilde{W}_{\overline{\rho}_p,w}^{\mathbf{h}-\mathrm{cr}}}:=\prod_{v\in S_p}\overline{\widetilde{W}_{\overline{\rho}_{\widetilde{v}},w_v}^{\mathbf{h}_{\widetilde{v}}-\mathrm{cr}}}$) into $X_{\mathrm{tri}}(\overline{\rho}_p)$. \par
	Take $w\in W_{G_p}/W_{P_p}$ such that $w\geq w_{\cR}$ and $\lg_{P_p}(w)=\ell-1$. Let $\iota(X)\times \fX^p\times \mathbb{U}^g$ be an irreducible component of $X_p(\overline{\rho})$ passing through $y_{w_0}$ where $X$ is the irreducible component of $X_{\mathrm{tri}}(\overline{\rho}_p)$ passing through $\iota^{-1}(\rho_p,\underline{\delta}_{\cR,w_0})$ and $\fX^{p}$ is an irreducible component of $\fX_{\overline{\rho}^p}$. We take an affinoid open neighbourhood $U$ of the point $\iota^{-1}(\rho_p,\underline{\delta}_{\cR,w_0})$ in $X$. Since $w_{\cR}\leq w$, the point $\iota^{-1}(\rho_p,\underline{\delta}_{\cR,w_0})=\left(\rho_p,z^{w_0(\mathbf{h})}\mathrm{unr}(\underline{\varphi})\right)$ is in $\iota_{\mathbf{h},w_0}(\overline{\widetilde{W}_{\overline{\rho}_p,w}^{\mathbf{h}-\mathrm{cr}}})$. Hence the intersection $V:=\widetilde{W}_{\overline{\rho}_p,w}^{\mathbf{h}-\mathrm{cr}}\cap \iota_{\mathbf{h},w_0}^{-1}(U)$ is Zariski open dense in the affinoid $\overline{V}:=\overline{\widetilde{W}_{\overline{\rho}_p,w}^{\mathbf{h}-\mathrm{cr}}}\cap \iota_{\mathbf{h},w_0}^{-1}(U)$ (we take $U$ small enough so that $\iota_{\mathbf{h},w_0}^{-1}(U)\cap\widetilde{W}_{\overline{\rho}_p}^{\mathbf{h}-\mathrm{cr}}=\iota_{\mathbf{h},w_0}^{-1}(U)\cap \widetilde{\mathfrak{X}}_{\overline{\rho}_p}^{\mathbf{h}-\mathrm{cr}}$ where $\widetilde{\mathfrak{X}}_{\overline{\rho}_p}^{\mathbf{h}-\mathrm{cr}}:=\prod_{v\in S_p}\widetilde{\mathfrak{X}}_{\overline{\rho}_{\widetilde{v}}}^{\mathbf{h}_{\widetilde{v}}-\mathrm{cr}}$) and $(\rho_p,\underline{\varphi})\in \overline{V}$. Then $\left(\iota\circ\iota_{\mathbf{h},w_0}(V)\right)\times \fX^p\times\mathbb{U}^g$ is a subset in $X_{p}(\overline{\rho})$. Any point in the subset satisfies the condition in $\mathcal{H}_{\ell-1}$ and for these points, $w_{\cR}=w,\lg_{P_p}(w_{\cR})=\ell-1$. Hence their companion points $\left(\iota\circ\iota_{\mathbf{h},w}(V)\right)\times \fX^p\times\mathbb{U}^g$ is contained in $X_p(\overline{\rho})$ by the discussion in the beginning of the proof where we have used Proposition \ref{propositionmaincycle} and $\mathcal{H}_{\ell}$. We get that the preimage of $X_p(\overline{\rho})$ under the map $(\iota\circ\iota_{\mathbf{h},w})\times \mathrm{id}\times\mathrm{id}: \overline{\widetilde{W}_{\overline{\rho}_p,w}^{\mathbf{h}-\mathrm{cr}}}\times \fX^p\times\mathbb{U}^g\rightarrow \iota\left(X_{\mathrm{tri}}(\overline{\rho_p})\right)\times \fX^p\times\mathbb{U}^g$ contains the Zariski closure of $V\times \fX^p\times\mathbb{U}^g$ inside $\overline{V}\times \fX^p\times\mathbb{U}^g$ which is equal to $\overline{V}\times \fX^p\times\mathbb{U}^g$. This means that the point $y_w=\left(\iota\left(\rho_p,z^{w(\mathbf{h})}\mathrm{unr}(\underline{\varphi})\right),z \right)$ is in $X_p(\overline{\rho})$. Hence $\mathcal{H}_{\ell-1}$ holds.
\end{proof}
\begin{remark}
	In the proof of \cite[Thm. 5.3.3]{breuil2019local}, results in Theorem \ref{theoremmaincrystalline} were obtained under the assumption, in addition to regular weights, that $z$ is in the smooth locus of $\fX_{\overline{\rho}^p}\times \mathbb{U}^g$ which is certainly expected to be not necessary. Our proof realizes this expectation!
\end{remark}
Finally, we state the theorem for $p$-adic automorphic forms. Recall that there is a closed embedding $Y(U^p,\overline{\rho}){}\hookrightarrow X_p(\overline{\rho})$ from the eigenvariety to the patched one and remark that we always assume all the hypotheses in \S\ref{sectionglobalsettings}. Proposition \ref{propositionsocleappear} and Theorem \ref{theoremmaincrystalline} immediately imply the following theorem.
\begin{theorem}\label{theoremmainautomorphicforms}
	Let $\rho:\Gal(\overline{F}/F)\rightarrow \GL_n(L)$ be a continuous representation such that $\rho_p=(\rho|_{\cG_{F_{\widetilde{v}}}})_{v\in S_p}$ is generic crystalline. Assume that $\rho$ corresponds to a point $(\rho,\underline{\delta}_{\cR,w'})\in Y(U^p,\overline{\rho})(L)\subset \Spf(R_{\overline{\rho},S})^{\mathrm{rig}}\times_{L}\widehat{T}_{p,L}$ where $\cR$ is a refinement of $\rho_p$ and $w'\in W_{G_p}/W_{P_p}$. Let $\fm_{\rho}$ be the maximal ideal of $R_{\overline{\rho},S}[\frac{1}{p}]$ corresponding to $\rho$. Let $\lambda$ be the weight of $\underline{\delta}_{\cR,w_0}$. Then $(\rho,\underline{\delta}_{\cR,w})\in Y(U^p,\overline{\rho})$ if and only if $w\geq w_{\cR}$ in $W_{G_p}/W_{P_p}$, and for all $w\geq w_{\cR}$,
	\[\Hom_{G_p}\left(\mathcal{F}_{\overline{B}_p}^{G_p}\left(\overline{L}(-ww_0\cdot\lambda),\underline{\delta}_{\cR,\mathrm{sm}}\delta_{B_p}^{-1}\right),\widehat{S}(U^p,L)_{\mathfrak{m}^S}^{\mathrm{an}}[\fm_{\rho}]\right)\neq 0.\]
\end{theorem}
\begin{proof}
	Recall the action of $R_{\overline{\rho},S}$ on $\widehat{S}(U^p,L)_{\mathfrak{m}^S}$ factors through a quotient $R_{\overline{\rho},\mathcal{S}}$. And there is an ideal $\mathfrak{a}$ of $R_{\infty}$, a surjection $R_{\infty}/\mathfrak{a}R_{\infty}\twoheadrightarrow R_{\overline{\rho},\mathcal{S}}$ and an isomorphism  
	\[\widehat{S}(U^p,L)_{\mathfrak{m}^S}\simeq \Pi_{\infty}[\mathfrak{a}]\]
	that is compatible with the action of $R_{\infty}$ and $R_{\overline{\rho},S}$ on the two sides.\par 
	Suppose that under the closed embedding $Y(U^p,\overline{\rho})\hookrightarrow X_p(\overline{\rho})$, $(\rho,\underline{\delta}_{\cR,w'})$ is sent to the point $x=((\rho_p,\underline{\delta}_{\cR,w'}),z)\in X_p(\overline{\rho})$. Let $r_x=(\rho_p,z)$ and let $\fm_{r_x}$ be the maximal ideal of $R_{\infty}[\frac{1}{p}]$ corresponding to $r_x$. Then $\fm_{r_x}$ contains $\mathfrak{a}$. Hence there is an isomorphism of $G_p$-representations
	\[\widehat{S}(U^p,L)_{\mathfrak{m}^S}^{\mathrm{an}}[\fm_{\rho}]\simeq \Pi^{\mathrm{an}}_{\infty}[\fm_{r_x}].\]
	Note that $(\rho,\underline{\delta}_{\cR,w})\in Y(U^p,\overline{\rho})$ is equivalent to $\Hom_{T_p}(\underline{\delta}_{\cR,w}, J_{B_p}(\widehat{S}(U^p,L)_{\mathfrak{m}^S}^{\mathrm{an}}[\fm_{\rho}]))\neq 0$. Hence the assertions of the theorem follows from similar statements replacing $\widehat{S}(U^p,L)_{\mathfrak{m}^S}^{\mathrm{an}}[\fm_{\rho}]$ by $\Pi^{\mathrm{an}}_{\infty}[\fm_{r_x}]$, which are true by Proposition \ref{propositionsocleappear} and Theorem \ref{theoremmaincrystalline}.
\end{proof}

\section{The partial eigenvariety}\label{sectionthepartialeigenvariety}
In this section, we use Ding's partial eigenvariety (\cite{ding2019some}) to prove some general result on the relationship between partially classical finite slope locally analytic representations and partially de Rham properties of trianguline $(\varphi,\Gamma)$-modules (Theorem \ref{theoremQderham}) which has been used for Theorem \ref{theoremcyclepartialderham}. Most results around the partial eigenvariety in this section except for those in \S\ref{sectionpartiallyderhamtrianguline} and \S\ref{sectionconjectures} are essentially due to Ding, and we adapt his results for the patching module.
\subsection{Notation}\label{sectionthepartialeigenvarietynotation}
We keep the notation and assumptions in \S\ref{sectionglobalsettings} and \S\ref{sectionOrlikStrauch}.\par
For each $v\in S_p$, let $Q_v$ be a standard parabolic subgroup of $\GL_{n/F_{\widetilde{v}}}$ (not $\mathrm{Res}_{F_{\widetilde{v}}/\Q_p}(\GL_{n/F_{\widetilde{v}}})\times_{\Q_p}L$!) containing the Borel subgroup of upper-triangular matrices. Write $Q_v=M_{Q_v}N_{Q_v}$ for the standard Levi decomposition where $M_{Q_v}$ is the standard Levi subgroup containing the diagonal torus and $N_{Q_v}$ is the unipotent radical. For an algebraic reductive group $H$, we use the notation $H'$ to denote the derived subgroup of $H$. We also use the same notation $Q_v, M_{Q_v}, N_{Q_v}, M_{Q_v}'$ to denote the groups of $F_{\widetilde{v}}$-points which are identified with subgroups of the $p$-adic Lie group $G_v$ via $i_{\widetilde{v}}:G(F_v^+)\simeq \GL_n(F_{\widetilde{v}})$. We have $p$-adic Lie groups $Q_p:=\prod_{v\in S_p}Q_v,M_{Q_p}:=\prod_{v\in S_p}M_{Q_v},M_{Q_p}':=\prod_{v\in S_p}M_{Q_v}',$ etc.. Assume for each $v\in S_p$, the standard Levi $M_{Q_v}$ is the group of diagonal block matrices of $\GL_{n/F_{\widetilde{v}}}$ of the form $\GL_{q_{v,1}/F_{\widetilde{v}}}\times \cdots \times \GL_{q_{v,t_{v}}/F_{\widetilde{v}}}$ where $n=q_{v,1}+\cdots+q_{v,t_v}$. We let $\widetilde{q}_{v,i}=\sum_{j=1}^iq_{v,j}$ for any $v,1\leq i\leq t_v$ and let $\widetilde{q}_{v,0}=0$. Let $B_{Q_p}=B_p\cap M_{Q_p},\overline{B}_{Q_p}=\overline{B}_p\cap M_{Q_p}$ and write $B_{Q_p}=T_pU_{Q_p}$ (resp. $\overline{B}_{Q_p}=T_p\overline{U}_{Q_p}$) for the Levi decomposition of $B_{Q_p}$ (resp. $\overline{B}_{Q_p}$).\par 
We have $p$-adic Lie subgroups $T'_{Q_p}:=T_p\cap M_{Q_p}', B_{Q_p}':=T'_{Q_p}U_{Q_p}$ and $Z_{M_{Q_p}}$ the center of $M_{Q_p}$. \par
Recall $\fg=\prod_{v\in S_p}\prod_{\tau\in \Sigma_v}\fg_{\tau}$ and similarly we let $\fm_{Q_p}=\prod_{v\in S_p}\prod_{\tau\in \Sigma_v}\fm_{Q_v,\tau}$ (resp. $\fm_{Q_p}'=\prod_{v\in S_p}\prod_{\tau\in \Sigma_v}\fm_{Q_v,\tau}'$, resp. $\ft'_{Q_p}=\prod_{v\in S_p}\prod_{\tau\in \Sigma_v}\ft_{Q_v,\tau}'$, resp. $\fb_{Q_p}'=\prod_{v\in S_p}\prod_{\tau\in \Sigma_v}\fb_{Q_v,\tau}'$, resp. $\mathfrak{z}_{M_{Q_p}}=\prod_{v\in S_p}\prod_{\tau\in \Sigma_v}\mathfrak{z}_{M_{Q_v},\tau}$, etc.) be the base change to $L$ of the $\Q_p$-Lie algebra of the $p$-adic Lie group $M_{Q_p}$ (resp. $M_{Q_p}'$, resp. $T'_{Q_p}$, resp. $B_{Q_p}'$, resp. $Z_{M_{Q_p}}$, etc.). We have $\ft=\ft'_{Q_p}\times \mathfrak{z}_{M_{Q_p}}, \fm_{Q_p}=\fm_{Q_p}'\times \mathfrak{z}_{M_{Q_p}}$ and the morphism $Z_{M_{Q_p}}\times M_{Q_p}'\rightarrow M_{Q_p}$ is locally an isomorphism. \par
We pick an arbitrary nonempty subset $J$ of $\Sigma_p$ and set $J_v=J\cap \Sigma_v$ for $v\in S_p$. We let $\fm_{Q_p,J}':=\prod_{v\in S_p}\prod_{\tau\in J_v}\fm_{Q_v,\tau}'$, $\ft_{Q_p,J}':=\prod_{v\in S_p}\prod_{\tau\in J}\ft_{Q_v,\tau}'$, etc.. We will only need the case when $|J|=1$ but adding this extra assumption will not simplify the notation.\par
We fix a uniform pro-$p$ normal subgroup $H_p=\prod_{v\in S_p}H_v$ of the maximal compact subgroup $K_p=\prod_{v\in S_p} K_v=\prod_{v\in S_p}i_{\widetilde{v}}^{-1}(\GL_n(\cO_{F_{\widetilde{v}}}))$ of $G_p$ where each $H_v$ is good $F_{\widetilde{v}}$-analytic with an Iwahori decomposition as in \cite[Def. 4.1.3]{emerton2006jacquet}. Let $U_{Q_p,0}=U_{Q_p}\cap H_p,M_{Q_p,0}=M_{Q_p}\cap H_p, N_{Q_p,0}=N_{Q_p}\cap H_p$, $T_{Q_p,0}'=T_{Q_p}'\cap H_p$, etc. and define $U_{Q_v,0}$, etc similarly. Let $Z_{M_{Q_p}}^{+}=\prod_{v\in S_p}Z_{M_{Q_v}}^{+}$ (resp. $T_{p}^{+}=\prod_{v\in S_p}T_{v}^{+}$, resp. $T_{M_{Q_p}}^+=\prod_{v\in S_p}T_{M_{Q_v}}^{+}$) be the submonoid of $Z_{M_{Q_p}}$ (resp. $T_p$, resp. $T_p$) consisting of elements $t$ such that $tN_{Q_p,0}t^{-1}\subset N_{Q_p,0}$ (resp. $tN_{B_p,0}t^{-1}\subset N_{B_p,0}$, resp. $tU_{Q_p,0}t^{-1}\subset tU_{Q_p,0}$). We use the notation $(-)_{\mathrm{fs}}$ to denote Emerton's finite slope part functor \cite[Def. 3.2.1]{emerton2006jacquet} with respect to one of the submonoids $Z_{M_{Q_p}}^+$, $T_p^{+}$ or $T_{M_{Q_p}}^+$ of $T_p$ where the exact meaning will be clear from the context.\par
Recall that $\Pi_{\infty}$ is the patched representation of $G_p$ and $\Pi_{\infty}^{\mathrm{an}}$ denotes the subspace of locally $R_{\infty}$-analytic vectors of $\Pi_{\infty}$. We have an integer $q$, a ring $S_{\infty}$ in \cite[\S3.2]{breuil2017interpretation} and we fix an isomorphism $S_{\infty}\simeq \cO_L[[\Z_p^q]]$. If $H$ is a group, we denote by $\widetilde{H}:=H\times \Z_p^q$. Then $\Pi_{\infty}$ is equipped with an action of $\widetilde{G}_p$ from the action of $S_{\infty}\rightarrow R_{\infty}$ (\cite[\S3.1]{breuil2017interpretation}). Since the patching module $M_{\infty}$ is finite projective over $S_{\infty}[[K_p]]$ (\cite[Thm. 3.5]{breuil2017interpretation}), it is finite free over $S_{\infty}[[H_p]]$ as the ring $S_{\infty}[[H_p]]$ is local. Hence $\Pi_{\infty}|_{\widetilde{H}_p}\simeq \cC(\widetilde{H}_p,L)^m$ for some integer $m$ where $\cC(\widetilde{H}_p,L)$ denotes the space of continuous functions over $\widetilde{H}_p$ with coefficients in $L$. By \cite[Prop. 3.4]{breuil2017interpretation}, $J_{Q_p}\left(\Pi_{\infty}^{\mathrm{an}}\right)$ is an essentially admissible locally $\Q_p$-analytic representation (\cite[Def. 6.4.9]{emerton2017locally}) of $\Z_p^s\times M_{Q_p}$ for some surjection $\cO_L[[\Z_p^s]]\twoheadrightarrow R_{\infty}$ where $J_{Q_p}$ is the Emerton's Jacquet module functor with respect to the parabolic subgroup $Q_p$. By definition (\cite[Def. 3.4.5]{emerton2006jacquet}), $J_{Q_p}(\Pi^{\mathrm{an}}_{\infty})$ is the finite slope part of $\Pi_{\infty}^{\mathrm{an},N_{Q_p,0}}$ with respect to the action of the submonoid $Z_{M_{Q_p}}^+$ of $Z_{M_{Q_p}}$.
\subsection{The partial Emerton-Jacquet module functor}\label{sectionthepartialemertonjacquetmodule}
We recall the notion of locally $\Sigma_p\setminus J$-analytic representations introduced in \cite[\S2]{schraen2010representations} (also see \cite[\S6.1]{ding2017formes} or \cite[Appendix B]{ding2019companion}) and the partial Emerton-Jacquet module defined in \cite[\S2.2.2]{ding2019some}.\par
Suppose that $V$ is a locally $\Q_p$-analytic representation of $M_{Q_p}=\prod_{v\in S_p}M_{Q_v}$ over $L$. A vector $v\in V$ is called locally $\Sigma_p\setminus J$-analytic with respect to the derived subgroup $M'_{Q_p}$ if the differential of the locally analytic function on $M'_{Q_p}: g\mapsto gv$ at the identity $e\in M_{Q_p}'$, which \textit{a priori} lies in $\Hom_{\Q_p}(T_eM_{Q_p}',V)=\Hom_L(\fm_{Q_p}', V)$, vanishes on $\fm_{Q_p,J}'=\prod_{v\in S_p}\prod_{\tau\in J_v}\fm_{Q_v,\tau}'$. Remark that since the adjoint action of $M_{Q_v}'$ on $\fm_{Q_v}'$ is $F_{\widetilde{v}}$-linear, $\Ad(M_{Q_p}')\fm_{Q_p,J}'=\fm_{Q_p,J}'$. \par
We fix a tuple $\widetilde{\lambda}_{J}=(\widetilde{\lambda}_{\tau})_{\tau\in J}=(\widetilde{\lambda}_{\tau,1},\cdots,\widetilde{\lambda}_{\tau,n})_{\tau\in J}\in (\Z^n)^{J}$ such that $\widetilde{\lambda}_{\tau,i}\geq \widetilde{\lambda}_{\tau,j}$ for all $i\geq j,\tau\in J$. We identify $\widetilde{\lambda}_J$ with an element in $\ft^*:=\Hom_L(\ft, L)$ that vanishes on $\prod_{\tau \in \Sigma_p\setminus J}\ft_{\tau}$. Denote by $\lambda_J'$ the image of $\widetilde{\lambda}_J$ in $(\ft'_{Q_p})^*:=\Hom_L(\ft'_{Q_p}, L)$. Then there exists a unique algebraic representation of $\prod_{v\in S_p}\mathrm{Res}_{F_{\widetilde{v}}/\Q_p}\left(M_{Q_v}'\right)\otimes_{\Q_p}L$ over $L$ with the highest weight $\lambda_J'$. Let $L_{M_{Q_p}'}(\lambda_J')$ be the associated $\Q_p$-algebraic representation of the $p$-adic Lie group $M_{Q_p}'$ over $L$ via the embedding $M_{Q_p}'=\prod_{v\in S_p}M_{Q_v}'\left(F_{\widetilde{v}}\right)\hookrightarrow \prod_{v\in S_p}\mathrm{Res}_{F_{\widetilde{v}}/\Q_p}\left(M_{Q_v}'\right)_L(L)$. The $U(\fm_{Q_p}')$-module $L_{M_{Q_p}'}(\lambda_J')$ is the unique irreducible quotient of $U(\fm_{Q_p}')\otimes_{U(\fb_{Q_p}')}\lambda_J'$ or $U(\fm_{Q_p,J}')\otimes_{U(\fb_{Q_p,J}')}\lambda_J'$ (elements in $\fm_{Q_{v},\tau}'$ act as zero on the module if $\tau\notin J_v$). We equip $L_{M_{Q_p}'}(\lambda_J')$ with an action of $\fm_{Q_p}=\fm_{Q_p}'\oplus \mathfrak{z}_{M_{Q_p}}$ where the action of $\mathfrak{z}_{M_{Q_p}}$ is given by $\widetilde{\lambda}_J$ and denote by $L_{M_{Q_p}}(\widetilde{\lambda}_J)$ for the $\Q_p$-algebraic $M_{Q_p}$-representation on $L_{M_{Q_p}'}(\lambda'_J)$. Let $L_{M_{Q_p}}(\widetilde{\lambda}_J)':=\Hom_{L}\left(L_{M_{Q_p}}(\widetilde{\lambda}_J),L\right)$ be the usual dual representation of $M_{Q_p}$. \par
If $V$ is a locally $\Q_p$-analytic representation of $M_{Q_p}$ over $L$, let $\left(V\otimes_L L_{M_{Q_p}}(\widetilde{\lambda}_J)'\right)^{\Sigma_p\setminus J-\mathrm{an}}$ be the closed $L$-subspace of $V\otimes_L L_{M_{Q_p}}(\widetilde{\lambda}_J)'$ generated by locally $\Sigma_p\setminus J$-analytic vectors with respect to the diagonal action of $M_{Q_p}'$. Then $\left(V\otimes_L L_{M_{Q_p}}(\widetilde{\lambda}_J)'\right)^{\Sigma_p\setminus J-\mathrm{an}}$ is a locally $\Sigma_p\setminus J$-analytic representation of $M_{Q_p}'$ in the sense of \cite[Def. 2.4]{schraen2010representations} and is stable under the action of $M_{Q_p}$. We have 
\[\Hom_{U(\fm_{Q_p,J}')}\left(L_{M_{Q_p}}(\widetilde{\lambda}_J), V\right)\simeq \left(L_{M_{Q_p}}(\widetilde{\lambda}_J)'\otimes_L V\right)^{\Sigma_p\setminus J-\mathrm{an}}\]
as topological representations of $M_{Q_p}$ (cf. \cite[Rem. 6.1.5]{ding2017formes}) where the action on the left hand side is the natural one as in \cite[\S5.2]{breuil2019local}.\par
Now assume that $V$ is an essentially admissible locally $\Q_p$-analytic representation of $\Z_p^s\times G_p$. We will take $V=\Pi_{\infty}^{\mathrm{an}}=\Pi_{\infty}^{R_{\infty}-\mathrm{an}}$ and the action of $\Z_p^s$ is given by $\cO_L[[\Z_p^s]]\twoheadrightarrow R_{\infty}$ or $s=q$ and $\cO_L[[\Z_p^q]]\simrightarrow S_{\infty}\rightarrow R_{\infty}$. We let $\Z_p^s$ act trivially on $L_{M_{Q_p}}(\widetilde{\lambda}_J)$. Then $J_{Q_p}(V)$ is an essentially admissible representation of $\Z_p^s\times M_{Q_p}$ by \cite[Thm. 4.2.32]{emerton2006jacquet}. We define 
\[J_{Q_p}(V)_{\lambda_J'}:=\Hom_{U(\fm_{Q_p,J}')}\left(L_{M_{Q_p}}(\widetilde{\lambda}_J), J_{Q_p}(V)\right)\otimes_L L_{M_{Q_p}}(\widetilde{\lambda}_J)\]
equipped with the diagonal action of $\Z_p^s\times M_{Q_p}$. There is a natural $\Z_p^s\times M_{Q_p}$-equivariant morphism:
\begin{align}\label{morphismclosedembedding}
    \Hom_{U(\fm_{Q_p,J}')}\left(L_{M_{Q_p}}(\widetilde{\lambda}_J), J_{Q_p}(V)\right) \otimes_L L_{M_{Q_p}}(\widetilde{\lambda}_J) \to J_{Q_p}(V):
    f\otimes v\mapsto f(v).
\end{align}
\begin{lemma}\label{lemmamorphismclosedembedding}
    The representation $J_{Q_p}(V)_{\lambda_J'}$ of $\Z_p^s\times M_{Q_p}$ depends only on $\lambda_J'$ (in particular on $J$, but not on the lift $\widetilde{\lambda}_J$). The morphism (\ref{morphismclosedembedding}) is a closed embedding and identifies $J_{Q_p}(V)_{\lambda_J'}$ with a closed $\Z_p^s\times M_{Q_p}$-sub-representation of $J_{Q_p}(V)$.
\end{lemma}
\begin{proof}
    For the injection, we can apply \cite[Prop. 6.1.3]{ding2017formes} with respect to $M_{Q_p}'$. The last assertion follows from the same arguments in \cite[Cor. B.2]{ding2017partiallyderham} and we prove it now.  Since $J_{Q_p}(V)$ is an essentially admissible representation of $\Z_p^s\times M_{Q_p}$, $J_{Q_p}(V)\otimes_L L_{M_{Q_p}}(\widetilde{\lambda}_J)'\otimes_L L_{M_{Q_p}}(\widetilde{\lambda}_J)$ is also an essentially admissible representation of $\Z_p^s\times M_{Q_p}$ by Lemma \ref{lemmaessentiallyadmissible} below. By \cite[Prop. 6.4.11]{emerton2017locally}, the closed sub-representation $J_{Q_p}(V)_{\lambda_J'}$ is an essentially admissible representation of $\Z_p^s\times M_{Q_p}$. Thus the injection $J_{Q_p}(V)_{\lambda_J'}\hookrightarrow V$ of essentially admissible representations of $\Z_p^s\times M_{Q_p}$ is a closed embedding by \textit{loc. cit.}.
\end{proof}
\begin{lemma}\label{lemmaessentiallyadmissible}
    Assume that $V$ is an essentially admissible locally analytic representation over $L$ of a locally analytic group $G$ where the center $Z$ of $G$ is topologically finitely generated. If $W$ is a finite-dimensional locally analytic representation over $L$ of $G$ on which the action of $Z$ is trivial, then $V\otimes_LW$ is also essentially admissible.
\end{lemma}
\begin{proof}
    We have $(V\otimes_L W)'\simeq V'\otimes_L W'$ as topological vector spaces. Since $V$ is essentially admissible, by definition, there exists a covering of $\widehat{Z}$ by open affinoids $\widehat{Z}_{1}\subset \widehat{Z}_{2}\subset\cdots $ and a sequence $\mathbb{H}_0\supset\mathbb{H}_1\supset \mathbb{H}_2\cdots$ of rigid analytic open subgroups with respect to a compact open subgroup $H=\mathbb{H}_0(\Q_p)$ of $G$ as in \cite[\S5.2]{emerton2017locally} such that the strong dual $V'$, as a coadmissible module over the Fr\'echet-Stein algebra $\cC^{\mathrm{an}}(\widehat{Z},L)\widehat{\otimes}_LD(H,L)$, is isomorphic to $\varprojlim_n \cC^{\mathrm{an}}(\widehat{Z}_{n},L)\widehat{\otimes}_LD(\mathbb{H}^{\circ}_n,H)\widehat{\otimes}_{\cC^{\mathrm{an}}(\widehat{Z},L)\widehat{\otimes}_LD(H,L)}V'$ (\cite[Thm. 1.2.11, Def. 6.4.9]{emerton2017locally}). Here $L$ denotes the coefficient field, $\cC^{\mathrm{an}}(\widehat{Z},L)$ (resp. $\cC^{\mathrm{an}}(\widehat{Z}_n,L)$) is the algebra of rigid analytic functions on $\widehat{Z}$ (resp. $\widehat{Z}_n$) (\cite[Def. 2.1.18]{emerton2017locally}), $D(H,L)$ is the algebra of locally $\Q_p$-analytic distributions on $H$ and $D(\mathbb{H}^{\circ}_n,H)$ is the strong dual of $\mathbb{H}_n^{\circ}$-analytic functions on $H$ as \cite[(4.1.2)]{emerton2006jacquet}. The isomorphism $\cC^{\mathrm{an}}(\widehat{Z},L)\widehat{\otimes}_LD(H,L)\simeq\varprojlim_n\cC^{\mathrm{an}}(\widehat{Z}_{n},L)\widehat{\otimes}_LD(\mathbb{H}^{\circ}_n,H)$ defines a weak Fr\'echet-Stein structure on $\cC^{\mathrm{an}}(\widehat{Z},L)\widehat{\otimes}_LD(H,L)$ (\cite[Def. 1.2.6, Lem. 1.1.29]{emerton2017locally}). We write $A=\cC^{\mathrm{an}}(\widehat{Z},L)\widehat{\otimes}_LD(H,L)$ and $A_n= \cC^{\mathrm{an}}(\widehat{Z}_{n},L)\widehat{\otimes}_LD(\mathbb{H}^{\circ}_n,H)$.\par
    The Dirac distribution $L[H](\subset D(H,K))$ (\cite[\S2, \S3]{schneider2002locally}), as well as $L[Z](\subset D(Z,L) \subset \cC^{\mathrm{an}}(\widehat{Z},L))$ (\cite[Prop. 6.4.6]{emerton2017locally}), acts on $(V\otimes_L W)'=V'\otimes_L W'$ diagonally where $L[Z]$ acts trivially on the second factor $W'$ since $Z$ acts trivially on $W'$. \par    
    Denote by $\rho$ the action of $H$ on $W'$. There are two ring homomorphisms $\alpha: L[H]\rightarrow L[H]\otimes_L \mathrm{End}_L(W'), h\mapsto h\otimes 1$, $\beta: L[H]\rightarrow L[H]\otimes_L \mathrm{End}_L(W'), h\mapsto h\otimes \rho(h)$ and a map $\gamma: L[H]\otimes_L \mathrm{End}_L(W')\rightarrow L[H]\otimes_L \mathrm{End}_L(W'): h\otimes m\mapsto h\otimes \rho(h)m$ such that $\beta=\gamma\circ \alpha$. By \cite[Prop. 4.4]{hill2011emerton}, for each (large enough) $n$, $\alpha,\beta, \gamma$ can be extended uniquely to continuous maps $\alpha_n,\beta_n:D(\mathbb{H}^{\circ}_n,H)\rightarrow D(\mathbb{H}^{\circ}_n,H)\otimes_L \mathrm{End}_L(W')$ and $\gamma_n:D(\mathbb{H}^{\circ}_n,H)\otimes_L \mathrm{End}_L(W')\rightarrow D(\mathbb{H}^{\circ}_n,H)\otimes_L \mathrm{End}_L(W')$ such that $\beta_n=\gamma_n\circ \alpha_n$. Taking the complete tensor product with $\cC^{\mathrm{an}}(\widehat{Z}_{n},L)$, we get similar maps $\alpha_n',\beta_n':A_n\rightarrow A_n\otimes_L \mathrm{End}_L(W')$ and $\gamma_n': A_n\otimes_L \mathrm{End}_L(W')\rightarrow A_n\otimes_L \mathrm{End}_L(W')$ extending the maps $\alpha'=\mathrm{id}\otimes_L \alpha,\beta'=\mathrm{id}\otimes_L \beta:L[Z]\otimes_L L[H]\rightarrow L[Z]\otimes_LL[H] \otimes_L \mathrm{End}_L(W')$ and $\gamma'=\mathrm{id}\otimes_L\gamma$.\par
    The tensor product $U_n:=\left(A_n\widehat{\otimes}_{A}V'\right)\otimes_L W'$ is naturally an $A_n\otimes_L\mathrm{End}_L(W')$-module where $A_n$, as well as $L[Z]\otimes_LL[H]$, acts on the second factor $W'$ trivially. Then as in \cite[Prop. 4.6]{hill2011emerton}, the map $\beta_n': A_n\rightarrow A_n\otimes_L\mathrm{End}_L(W')$ equips $U_n$ a twisted action of $A_n$ extending the diagonal action of $L[Z]\otimes_LL[H]$ on $U_n$. Since $U_n$ is a finitely generated $A_n\otimes_L\mathrm{End}_L(W)$-module and $A_n\otimes_L\mathrm{End}_L(W)$ is a finitely generated $A_n$-module with respect to the action of $A_n$ via both $\alpha_n'$ and $\beta_n'$ by \cite[Cor. 4.5]{hill2011emerton}, we get that $U_n$ is a finitely generated $A_n$-module and we have an isomorphism $A_n\widehat{\otimes}_{A_{n+1}}U_{n+1}\simeq U_n$ with the twisted actions. Hence $\varprojlim_n U_n$ is a coadmissible module over the Fr\'echet-Stein algebra $A$ (\cite[Def. 1.2.8]{emerton2017locally}). The action of $A$ extends the action of $L[Z]\otimes_LL[H]$ on $(V\otimes W)'$ via the isomorphism $V'\otimes_L W'\simeq\varprojlim_n U_n$ of topological vector spaces. Such extension of the action of $L[Z]\otimes_LL[H]$ to $A$ on $(V\otimes_L W)'$ is unique by \cite[Prop. 6.4.7(ii)]{emerton2017locally} and by that $L[H]$ is dense in $D(H,L)$ which acts continuously on $(V\otimes W)'$ by \cite[Lem. 3.1, Cor. 3.4]{schneider2002locally}, thus coincides with the usual action of $A$ on $(V\otimes_LW)'$. We get that $V\otimes_LW$ is an essentially admissible representation of $G$.
\end{proof}
\begin{lemma}\label{lemmatwistedJacquetmodule}
    There is an isomorphism 
    \[J_{B_{Q_p}}\left(J_{Q_p}(V)_{\lambda_J'}\right)\simeq J_{B_{Q_p}}\left(\Hom_{U(\fm_{Q_p,J}')}\left(L_{M_{Q_p}}(\widetilde{\lambda}_J), J_{Q_p}(V)\right)\right)\otimes_L \widetilde{\lambda}_J\]
    of essentially admissible locally $\Q_p$-analytic representations of $\Z_p^s\times T_p$.
\end{lemma}
\begin{proof}
    This can be proved by arguments in \cite[Cor. 2.11]{ding2019some} together with \cite[Lem. 7.2.12]{ding2017formes}. Remark that $L_{M_{Q_p}}(\widetilde{\lambda}_J)$ is $J$-algebraic in the sense of \cite[\S6.1.1]{ding2017formes} and $L_{M_{Q_p}}(\widetilde{\lambda}_J)^{U_{Q_p,0}}=L_{M_{Q_p}}(\widetilde{\lambda}_J)^{\mathfrak{u}_{Q_p,J}}$. 
\end{proof}
\begin{lemma}\label{lemmafiniteslopeJacquet}
    We have
    \[\left(\Hom_{U(\fm'_{Q_p,J})}\left(L_{M_{Q_p}}(\widetilde{\lambda}_J),J_{Q_p}(V)\right)^{U_{Q_p,0}}\right)_{\mathrm{fs}}\simeq \left(\left(V^{N_{Q_p,0}}\otimes L_{M_{Q_p}}(\widetilde{\lambda}_J)'\right)^{\Sigma_p\setminus J-\mathrm{an},U_{Q_p,0}}\right)_{\mathrm{fs}},\]
    where the finite slope part is taken with respect to the Hecke action of $T_{M_{Q_p}}^+$ for the left hand side and that of $T_p^{+}$ for the right hand side. 
\end{lemma}
\begin{proof}
    By \cite[Prop. 3.2.9]{emerton2006jacquet}, $\left(V^{N_{Q_p,0}}\right)_{\mathrm{fs}}\otimes L_{M_{Q_p}}(\widetilde{\lambda}_J)'\simeq \left(V^{N_{Q_p,0}}\otimes L_{M_{Q_p}}(\widetilde{\lambda}_J)'\right)_{\mathrm{fs}}$. Since the action of $Z_{M_{Q_p}}^{+}$ commutes with $\fm_{Q_p,J}'$, by \cite[Prop. 3.2.11]{emerton2006jacquet}, we get 
    \begin{equation}\label{equationlemmafiniteslopeJacquet}
        \Hom_{U(\fm'_{Q_p,J})}\left(L_{M_{Q_p}}(\widetilde{\lambda}_J),J_{Q_p}\left(V\right)\right)=\left(\left(V^{N_{Q_p,0}}\otimes L_{M_{Q_p}}(\widetilde{\lambda}_J)'\right)^{\Sigma_p\setminus J-\mathrm{an}}\right)_{\mathrm{fs}}.
    \end{equation} 
    Now the arguments of \cite[Thm. 5.3(2)]{hill2011emerton} or \cite[Lem. 2.18]{ding2019some} using \cite[Prop. 3.2.4(ii)]{emerton2006jacquet} shows that 
    \[\left(\left(\left(V^{N_{Q_p,0}}\otimes L_{M_{Q_p}}(\widetilde{\lambda}_J)'\right)^{\Sigma_p\setminus J-\mathrm{an}}\right)^{U_{Q_p,0}}_{\mathrm{fs}}\right)_{\mathrm{fs}}\simeq \left(\left(V^{N_{Q_p,0}}\otimes L_{M_{Q_p}}(\widetilde{\lambda}_J)'\right)^{\Sigma_p\setminus J-\mathrm{an},U_{Q_p,0}}\right)_{\mathrm{fs}}.\]
    Combining the isomorphism above with (\ref{equationlemmafiniteslopeJacquet}), we get the desired isomorphism.
\end{proof}
\subsection{An adjunction formula}\label{sectionadjunctionformula}
We prove an adjunction formula for the partial Emerton-Jacquet module functor based on \cite[Lem. 5.2.1]{breuil2019local}. \par
Suppose that $\Pi^{\mathrm{an}}$ is a very strongly admissible locally $\Q_p$-analytic representation over $L$ of $G_p$ and $\underline{\delta}=z^{\lambda}\underline{\delta}_{\mathrm{sm}}$ where $z^{\lambda}$ is the $\Q_p$-algebraic character of $T_p$ of weight $\lambda=(\lambda_{\tau})_{\tau\in \Sigma_p}$ and $\underline{\delta}_{\mathrm{sm}}$ is a smooth character of $T_p$. We write $\lambda=\lambda_J+\lambda_{\Sigma_{p}\setminus J}$ according to the decomposition $\ft^{*}=\ft^{*}_{J}\oplus \ft_{\Sigma_p\setminus J}^*$ which means that $\lambda_J$ (resp. $\lambda_{\Sigma_p\setminus J}$) vanishes on $\ft_{\Sigma_p\setminus J}$ (resp. $\ft_{J}$). We assume that the image of $\lambda$ in $(\ft'_{Q_p,J})^{*}$ is equal to $\lambda_J'$ in \S\ref{sectionthepartialemertonjacquetmodule}.
Then (see \S\ref{sectionOrlikStrauch})
\begin{align*}
    \mathrm{Hom}_{T_p}\left(\underline{\delta}, J_{B_p}(\Pi^{\mathrm{an}})\right)=\Hom_{G_p}\left(\cF_{\overline{B}_p}^{G_p}\left(\left(U(\fg)\otimes_{U(\overline{\fb})}(-\lambda)\right)^{\vee},\underline{\delta}_{\mathrm{sm}}\delta_{B_p}^{-1}\right), \Pi^{\mathrm{an}}\right)
\end{align*}
The $U(\fq)$-module $U(\fq)\otimes_{U(\fb)}\lambda$ admits a quotient
\begin{align}\label{equationdefinitionLJ}
    L_J(\lambda):&=\left(\otimes_{v\in S_v,\tau \in J_v}L_{\fm_{Q_v,\tau}}(\lambda_{\tau})\right)\otimes \left(\otimes_{v\in S_v,\tau \notin J_v}U(\fm_{Q_v,\tau})\otimes_{U(\fb_{Q_v,\tau})}\lambda_{\tau}\right)\\
    &= L_{M_{Q_p}}(\lambda_J)\otimes_L M_{\fm_{Q_p}}(\lambda_{\Sigma_p\setminus J}).\nonumber
\end{align}
where $M_{\fm_{Q_p}}(\lambda_{\Sigma_p\setminus J}):=U(\fm_{Q_p,\Sigma_p\setminus J})\otimes_{U(\fb_{Q_p,\Sigma_p\setminus J})}\lambda_{\Sigma_p\setminus J}$ and $L_{M_{Q_p}}(\lambda_J)$ is defined as $L_{M_{Q_p}}(\widetilde{\lambda}_J)$ (only with a possibly different action of $\mathfrak{z}_{M_{Q_p},J}$). Thus there is an injection (see the beginning of \S\ref{sectionthesocleconjection} for $(-)^{\overline{\mathfrak{u}}^{\infty}}$)
\begin{align}\label{equationinjectionvermamodules}
    \Hom_{G_p}\left(\cF_{\overline{B}_p}^{G_p}\left(\Hom\left(U(\fg)\otimes_{U(\fq)}L_J(\lambda),L\right)^{\overline{\mathfrak{u}}^{\infty}},\underline{\delta}_{\mathrm{sm}}\delta_{B_p}^{-1}\right), \Pi^{\mathrm{an}}\right)\\
    \hookrightarrow \Hom_{G_p}\left(\cF_{\overline{B}_p}^{G_p}\left(\left(U(\fg)\otimes_{U(\overline{\fb})}(-\lambda)\right)^{\vee},\underline{\delta}_{\mathrm{sm}}\delta_{B_p}^{-1}\right), \Pi^{\mathrm{an}}\right).\nonumber
\end{align} 
Recall there is a closed immersion $J_{Q_p}(\Pi^{\mathrm{an}})_{\lambda_J'}\hookrightarrow J_{Q_p}(\Pi^{\mathrm{an}})$ in Lemma \ref{lemmamorphismclosedembedding}, which induces a closed embedding 
\begin{equation*}
    J_{B_{Q_p}}\left(J_{Q_p}(\Pi^{\mathrm{an}})_{\lambda_J'}\right)\hookrightarrow J_{B_p}(\Pi^{\mathrm{an}})
\end{equation*}
by \cite[Lem. 3.4.7(iii)]{emerton2006jacquet} and by that $J_{B_p}(\Pi^{\mathrm{an}})\simeq J_{B_{Q_p}}\left(J_{Q_p}(\Pi^{\mathrm{an}})\right)$ (\cite[Thm. 5.3]{hill2011emerton}).
\begin{proposition}\label{propositionadjunctionJacquetmodule}
    There exists an isomorphism
    \begin{align*}
        \Hom_{G_p}\left(\cF_{\overline{B}_p}^{G_p}\left(\Hom\left(U(\fg)\otimes_{U(\fq)}L_J(\lambda),L\right)^{\overline{\mathfrak{u}}^{\infty}},\underline{\delta}_{\mathrm{sm}}\delta_{B_p}^{-1}\right), \Pi^{\mathrm{an}}\right)\simrightarrow \Hom_{T_{p}}\left(\underline{\delta}, J_{B_{Q_p}}\left(J_{Q_p}(\Pi^{\mathrm{an}})_{\lambda_J'}\right)\right)
    \end{align*}
    such that the following diagram commutes
    \begin{center}
        \begin{tikzpicture}[scale=1.0, font=\small]
            \node (A1) at (0,2) {$\Hom_{G_p}\left(\cF_{\overline{B}_p}^{G_p}\left(\Hom\left(U(\fg)\otimes_{U(\fq)}L_J(\lambda),L\right)^{\overline{\mathfrak{u}}^{\infty}},\underline{\delta}_{\mathrm{sm}}\delta_{B_p}^{-1}\right), \Pi^{\mathrm{an}}\right)$};
            \node (B1) at (8,2) {$\Hom_{T_{p}}\left(\underline{\delta}, J_{B_{Q_p}}\left(J_{Q_p}(\Pi^{\mathrm{an}})_{\lambda_J'}\right)\right)$};
            \node (A2) at (0,0) {$\Hom_{G_p}\left(\cF_{\overline{B}_p}^{G_p}\left(\left(U(\fg)\otimes_{U(\overline{\fb})}(-\lambda)\right)^{\vee},\underline{\delta}_{\mathrm{sm}}\delta_{B_p}^{-1}\right), \Pi^{\mathrm{an}}\right)$};
            \node (B2) at (8,0) {$\Hom_{T_{p}}(\underline{\delta},J_{B_p}(\Pi^{\mathrm{an}}))$};
            \path[->,font=\scriptsize,>=angle 90]
            (A1) edge node[above]{$\sim$} (B1)
            (A2) edge node[above]{$\sim$} (B2)
            ;
            \path[right hook ->,font=\scriptsize,>=angle 90]
            (A1) edge node[right]{(\ref{equationinjectionvermamodules})} (A2)
            (B1) edge node[above]{} (B2)
            ;
            \end{tikzpicture}
     \end{center}
     where the right vertical arrow is induced by (\ref{morphismclosedembedding}).
\end{proposition}
\begin{proof}
    By \cite[Lem. 5.2.1]{breuil2019local}, \cite[Prop. 3.4.9]{emerton2006jacquet} and that $U(\fn_{Q_p})$ acts trivially on $L_J(\lambda)$, we have
    \begin{align*}
        &\Hom_{G_p}\left(\cF_{\overline{B}_p}^{G_p}\left(\Hom\left(U(\fg)\otimes_{U(\fq)}L_J(\lambda),L\right)^{\overline{\mathfrak{u}}^{\infty}},\underline{\delta}_{\mathrm{sm}}\delta_{B_p}^{-1}\right), \Pi^{\mathrm{an}}\right)\\
        =&\Hom_{T_{p}}\left(\underline{\delta}_{\mathrm{sm}}, \Hom_{U(\fg)}\left(U(\fg)\otimes_{U(\fq)}L_J(\lambda),\Pi^{\mathrm{an}}\right)^{N_{B_p,0}}\right)\\
        =&\Hom_{T_{p}}\left(\underline{\delta}_{\mathrm{sm}}, \Hom_{U(\fm_{Q_p})}\left(L_J(\lambda),\Pi^{\mathrm{an}}[\fn_{Q_p}]\right)^{N_{B_p,0}}\right)\\
        =&\Hom_{T_{p}}\left(\underline{\delta}_{\mathrm{sm}}, \Hom_{U(\fm_{Q_p})}\left(L_J(\lambda),\Pi^{\mathrm{an},N_{Q_p,0}}\right)^{U_{Q_p,0}}\right)\\
        =&\Hom_{T_{p}}\left(\underline{\delta}_{\mathrm{sm}}, \Hom_{U(\fm_{Q_p})}\left(L_J(\lambda),J_{Q_p}(\Pi^{\mathrm{an}})\right)^{U_{Q_p,0}}\right)
    \end{align*}
    where the last equations are given by similar arguments as in the proof of Lemma \ref{lemmafiniteslopeJacquet}. Similarly using $U(\fg)\otimes_{U(\fb)}\lambda=U(\fg)\otimes_{U(\fq)}U(\fq)\otimes_{U(\fb)}\lambda$ and that $U(\fn_{Q_p})$ acts trivially on $U(\fq)\otimes_{U(\fb)}\lambda=U(\fm_{Q_p})\otimes_{U(\fb_{Q_p})}\lambda$, we get 
    \begin{align*}
        &\Hom_{G_p}\left(\cF_{\overline{B}_p}^{G_p}\left(\left(U(\fg)\otimes_{U(\overline{\fb})}(-\lambda)\right)^{\vee},\underline{\delta}_{\mathrm{sm}}\delta_{B_p}^{-1}\right), \Pi^{\mathrm{an}}\right)\\
        =&\Hom_{T_{p}}\left(\underline{\delta}_{\mathrm{sm}}, \Hom_{U(\fm_{Q_p})}\left(U(\fm_{Q_p})\otimes_ {U(\fb_{Q_p})}\lambda,J_{Q_p}(\Pi^{\mathrm{an}})\right)^{U_{Q_p,0}}\right)
    \end{align*}
    Thus the injection (\ref{equationinjectionvermamodules}) corresponds to the injection
    \begin{align*}
        \Hom_{T_{p}}\left(\underline{\delta}_{\mathrm{sm}}, \Hom_{U(\fm_{Q_p})}\left(L_J(\lambda),J_{Q_p}(\Pi^{\mathrm{an}})\right)^{U_{Q_p,0}}\right)\\
        \hookrightarrow \Hom_{T_{p}}\left(\underline{\delta}_{\mathrm{sm}}, \Hom_{U(\fm_{Q_p})}\left(U(\fm_{Q_p})\otimes_ {U(\fb_{Q_p})}\lambda,J_{Q_p}(\Pi^{\mathrm{an}})\right)^{U_{Q_p,0}}\right).
    \end{align*}
    Firstly, we have
    \begin{align*}
        &\Hom_{U(\fm_{Q_p,J})}\left(L_{M_{Q_p}}(\lambda_J),J_{Q_p}(\Pi^{\mathrm{an}})\right)^{U_{Q_p,0}}\\
        =&\Hom_{U(\fm_{Q_p,J}')\otimes U(\mathfrak{z}_{M_{Q_p},J})}\left(L_{M_{Q_p}}(\lambda_J),J_{Q_p}(\Pi^{\mathrm{an}})\right)^{U_{Q_p,0}}\\
        =&\Hom_{U(\mathfrak{z}_{M_{Q_p},J})}\left(1,\Hom_{U(\fm_{Q_p,J}')}\left(L_{M_{Q_p}}(\lambda_J), J_{Q_p}(\Pi^{\mathrm{an}})\right)\right)^{U_{Q_p,0}}\\
        =&\Hom_{U(\mathfrak{z}_{M_{Q_p},J})}\left(1,\Hom_{U(\fm_{Q_p,J}')}\left(L_{M_{Q_p}}(\lambda_J), J_{Q_p}(\Pi^{\mathrm{an}})\right)^{U_{Q_p,0}}\right)\\
        =&\Hom_{U(\mathfrak{t}_{Q_p,J})}\left(1,\Hom_{U(\fm_{Q_p,J}')}\left(L_{M_{Q_p}}(\lambda_J), J_{Q_p}(\Pi^{\mathrm{an}})\right)^{U_{Q_p,0}}\right)
    \end{align*}
    where $1$ denotes the trivial module of the universal envelope algebras and the last equality comes from that the action of $\ft_{Q_p,J}'$ on $\Hom_{U(\fm_{Q_p,J}')}\left(L_{M_{Q_p}}(\lambda_J), J_{Q_p}(\Pi^{\mathrm{an}})\right)^{U_{Q_p,0}}$ is already trivial. Similar arguments as in Lemma \ref{lemmatwistedJacquetmodule} replacing $\widetilde{\lambda}_J$ there by $\lambda_J$ gives
    \begin{align}\label{equationadjunction3}
        \Hom_{U(\mathfrak{t}_{Q_p,J})}\left(1,\Hom_{U(\fm_{Q_p,J}')}\left(L_{M_{Q_p}}(\lambda_J), J_{Q_p}(\Pi^{\mathrm{an}})\right)^{U_{Q_p,0}}\right) \nonumber \\=\Hom_{U(\mathfrak{t}_{Q_p,J})}\left(\lambda_J, \left(J_{Q_p}(\Pi^{\mathrm{an}})_{\lambda_J'}\right)^{U_{Q_p,0}}\right).    
    \end{align}
    We get
    \begin{equation}\label{equationadjunction1}
        \Hom_{U(\fm_{Q_p,J})}\left(L_{M_{Q_p}}(\lambda_J),J_{Q_p}(\Pi^{\mathrm{an}})\right)^{U_{Q_p,0}}\simeq \Hom_{U(\mathfrak{t}_{Q_p,J})}\left(\lambda_J, \left(J_{Q_p}(\Pi^{\mathrm{an}})_{\lambda_J'}\right)^{U_{Q_p,0}}\right).
    \end{equation}
    Hence
    \begin{align}\label{equationadjunction2}
        &\Hom_{U(\fm_{Q_p})}\left(L_J(\lambda),J_{Q_p}(\Pi^{\mathrm{an}})\right)^{U_{Q_p,0}}\nonumber\\
        =&\Hom_{U(\fm_{Q_p})}\left(L_{M_{Q_p}}(\lambda_J)\otimes_L M_{\fm_{Q_p}}(\lambda_{\Sigma_p\setminus J}),J_{Q_p}(\Pi^{\mathrm{an}})\right)^{U_{Q_p,0}}\nonumber\\
        =&\Hom_{U(\fm_{Q_p})}\left(M_{\fm_{Q_p}}(\lambda_{\Sigma_p\setminus J}),L_{M_{Q_p}}(\lambda_J)'\otimes J_{Q_p}(\Pi^{\mathrm{an}})\right)^{U_{Q_p,0}}\nonumber\\
        =&\Hom_{U(\ft)}\left(\lambda_{\Sigma_p\setminus J},\Hom_{U(\fm_{Q_p,J})}\left(L_{M_{Q_p}}(\lambda_J),J_{Q_p}(\Pi^{\mathrm{an}})\right)[\mathfrak{u}_{Q_p}]\right)^{U_{Q_p,0}}\nonumber\\
        =&\Hom_{U(\ft)}\left(\lambda_{\Sigma_p\setminus J},\Hom_{U(\fm_{Q_p,J})}\left(L_{M_{Q_p}}(\lambda_J), J_{Q_p}(\Pi^{\mathrm{an}})\right)^{U_{Q_p,0}}\right)\nonumber\\
        \stackrel{(\ref{equationadjunction1})}{=}&\Hom_{U(\ft)}\left(\lambda_{\Sigma_p\setminus J},\Hom_{U(\mathfrak{t}_{Q_p,J})}\left(\lambda_J, \left(J_{Q_p}(\Pi^{\mathrm{an}})_{\lambda_J'}\right)^{U_{Q_p,0}}\right)\right)\nonumber\\
        =&\Hom_{U(\ft)}\left(\lambda, (J_{Q_p}(\Pi^{\mathrm{an}})_{\lambda_J'})^{U_{Q_p,0}}\right).
    \end{align}
    Finally, we get
    \begin{align*}
        &\Hom_{T_{p}}\left(\underline{\delta}_{\mathrm{sm}}, \Hom_{U(\fm_{Q_p})}\left(L_J(\lambda),J_{Q_p}(\Pi^{\mathrm{an}})\right)^{U_{Q_p,0}}\right)\\
        \stackrel{(\ref{equationadjunction2})}{=}&\Hom_{T_{p}}\left(\underline{\delta}_{\mathrm{sm}}, \Hom_{U(\ft)}\left(\lambda, \left(J_{Q_p}(\Pi^{\mathrm{an}})_{\lambda_J'}\right)^{U_{Q_p,0}}\right)\right)\\ 
        =&\Hom_{T_{p}}\left(\underline{\delta}, \left(J_{Q_p}(\Pi^{\mathrm{an}})_{\lambda_J'}\right)^{U_{Q_p,0}}\right)\\
        =&\Hom_{T_{p}}\left(\underline{\delta}, J_{B_{Q_p}}\left(J_{Q_p}(\Pi^{\mathrm{an}})_{\lambda_J'}\right)\right)
    \end{align*}
    and similarly
    \begin{align*}
        &\Hom_{T_{p}}\left(\underline{\delta}_{\mathrm{sm}}, \Hom_{U(\fm_{Q_p})}\left(U(\fm_{Q_p})\otimes_ {U(\fb_{Q_p})}\lambda,J_{Q_p}(\Pi^{\mathrm{an}})\right)^{U_{Q_p,0}}\right)\\
        =&\Hom_{T_{p}}\left(\underline{\delta}_{\mathrm{sm}}, \Hom_{U(\ft)}\left(\lambda,J_{Q_p}(\Pi^{\mathrm{an}})[\mathfrak{u}_{Q_p}]\right)^{U_{Q_p,0}}\right)\\
        =&\Hom_{T_{p}}\left(\underline{\delta}_{\mathrm{sm}}, \Hom_{U(\ft)}\left(\lambda,J_{Q_p}(\Pi^{\mathrm{an}})^{U_{Q_p,0}}\right)\right)\\
        =&\Hom_{T_{p}}\left(\underline{\delta},J_{Q_p}(\Pi^{\mathrm{an}})^{U_{Q_p,0}}\right)\\
        =&\Hom_{T_{p}}\left(\underline{\delta},J_{B_p}(\Pi^{\mathrm{an}})\right).
    \end{align*}
    The commutativity of the diagram in the statement of the proposition can be checked by the following commutative diagram and by comparing with (\ref{morphismclosedembedding}) and (\ref{equationadjunction3})
    \begin{center}
        \begin{tikzpicture}[scale=1.0,font=\small]
            \node (A1) at (0,2) {$\Hom\left(\lambda_J,L_{M_{Q_p}}(\lambda_J)'\otimes J_{Q_p}(\Pi^{\mathrm{an}})\otimes L_{M_{Q_p}}(\lambda_J)^{U_{Q_p,0}}\right)$};
            \node (A2) at (8,2) {$\Hom\left(\lambda_J,J_{Q_p}(\Pi^{\mathrm{an}})\right)$};
            \node (B1) at (0,0) {$\Hom_{U(\fm_{Q_p,J})}\left(L_{M_{Q_p}}(\lambda_J),J_{Q_p}(\Pi^{\mathrm{an}})\right)$};
            \node (B2) at (8,0) {$\Hom_{U(\fm_{Q_p,J})}\left(U(\fm_{Q_p,J})\otimes_ {U(\fb_{Q_p,J})}\lambda_{J},J_{Q_p}(\Pi^{\mathrm{an}})\right)$};
            \path[right hook ->,font=\scriptsize,>=angle 90]
            (B1) edge node[above]{} (A1)
            (B2) edge node[above]{} (A2)
            ;
            \path[->,font=\scriptsize,>=angle 90]
            (A1) edge node[right]{} (A2)
            (B1) edge node[above]{} (B2)
            ;
            \end{tikzpicture}
     \end{center}
     where we identify $L_{M_{Q_p}}(\lambda_J)^{U_{Q_p,0}}=\lambda_J$.
\end{proof}
\subsection{The partial eigenvariety}\label{sectionpartialeigenvarietygeometry}
We use the partial Emerton-Jacquet functor to define the partial eigenvariety and use the usual eigenvariety machinery to obtain its basic properties.\par
Since $J_{B_{Q_p}}\left(J_{Q_p}(\Pi_{\infty}^{\mathrm{an}})_{\lambda_J'}\right)$ is an essentially admissible locally analytic representation of $\Z_p^s\times T_p$, the continuous dual $J_{B_{Q_p}}\left(J_{Q_p}(\Pi_{\infty}^{\mathrm{an}})_{\lambda_J'}\right)'$ is the global section of a coherent sheaf over the quasi-Stein space $\Spf(R_{\infty})^{\mathrm{rig}}\times \widehat{T}_{p,L}$. We define the \emph{partial eigenvariety} $X_{p}(\overline{\rho})(\lambda'_J)$ to be the scheme-theoretic support of the coherent sheaf associated with $J_{B_{Q_p}}\left(J_{Q_p}(\Pi_{\infty}^{\mathrm{an}})_{\lambda_J'}\right)'$ on $\Spf(R_{\infty})^{\mathrm{rig}}\times \widehat{T}_{p,L}$. \par
The closed embedding in Lemma \ref{lemmamorphismclosedembedding} induces a closed embedding $J_{B_{Q_p}}\left(J_{Q_p}(\Pi_{\infty}^{\mathrm{an}})_{\lambda_J'}\right)\hookrightarrow J_{B_p}(\Pi_{\infty}^{\mathrm{an}})$ by \cite[Lem. 3.4.7(iii)]{emerton2006jacquet} and \cite[Thm. 5.3]{hill2011emerton}. Taking the continuous dual and then taking supports, we get a closed embedding $X_{p}(\overline{\rho})(\lambda_J')\hookrightarrow X_p(\overline{\rho})$ of rigid analytic spaces over $L$. Let $X_p(\overline{\rho})_{\lambda_J'}$ (resp. $(\widehat{T}_{p,L})_{\lambda_J'}$) be the fiber of $X_p(\overline{\rho})$ (resp. $\widehat{T}_{p,L}$) over $\lambda_J'$ via the map $X_p(\rho)\rightarrow \widehat{T}_{p,L}\stackrel{\mathrm{wt}}{\rightarrow}\ft^*\rightarrow (\ft'_{Q_p,J})^*$ where $\ft^*\rightarrow (\ft'_{Q_p,J})^*$ is the restriction map. Since the action of $\ft'_{Q_p,J}$ on $\Hom_{U(\fm_{Q_p,J}')}\left(L_{M_{Q_p}}(\widetilde{\lambda}_J), J_{Q_p}(\Pi_{\infty}^{\mathrm{an}})\right)^{U_{Q_p,0}}$ as well as on 
\[J_{B_{Q_p}}\left(\Hom_{U(\fm_{Q_p,J}')}\left(L_{M_{Q_p}}(\widetilde{\lambda}_J), J_{Q_p}(\Pi_{\infty}^{\mathrm{an}})\right)\right)\] 
is zero by \cite[Prop. 3.2.11]{emerton2006jacquet}, by Lemma \ref{lemmatwistedJacquetmodule}, we have a commutative diagram:
\begin{center}
    \begin{tikzpicture}[scale=1.3]
        \node (A1) at (0,1) {$X_{p}(\overline{\rho})(\lambda_J')$};
        \node (B1) at (2,1) {$X_p(\overline{\rho})_{\lambda_J'}$};
        \node (C1) at (4,1) {$ X_p(\rho)$};
        \node (B2) at (2,0) {$(\widehat{T}_{p,L})_{\lambda_J'}$};
        \node (C2) at (4,0) {$ \widehat{T}_{p,L}$};
        \path[right hook ->,font=\scriptsize,>=angle 90]
        (A1) edge node[above]{} (B1)
        (B1) edge node[above]{} (C1)
        (B2) edge node[above]{} (C2)
        ;
        \path[->,font=\scriptsize,>=angle 90]
        (B1) edge node[above]{} (B2)
        (C1) edge node[above]{} (C2)
        (A1) edge node[above]{} (B2)
        ;
        \end{tikzpicture}
 \end{center}
where horizontal arrows are closed immersions. \par
If $x=(y,\underline{\delta})\in \mathrm{Spf}(R_{\infty})^{\mathrm{rig}}\times \widehat{T}_{p,L}$ is a point in $X_p(\overline{\rho})_{\lambda_J'}$ with $\fm_y$ the maximal ideal of $R_{\infty}[\frac{1}{p}]$ corresponding to $y$, then equivalently $\mathrm{Hom}_{T_p}\left(\underline{\delta}, J_{B_p}(\Pi_{\infty}^{\mathrm{an}}[\mathfrak{m}_y]\otimes_{k(y)}k(x))\right)\neq 0$ and $\mathrm{wt}(\underline{\delta})_J'$, the image of the weight $\mathrm{wt}(\underline{\delta})$ of $\underline{\delta}$ in $(\ft_{Q_p,J}')^*$, is equal to $\lambda_J'$. Assume that $\lambda:=\mathrm{wt}(\underline{\delta})\in (\Z^{n})^{\Sigma_p}$ is integral and let $\underline{\delta}_{\mathrm{sm}}:=z^{-\lambda}\underline{\delta}$. Then by (\ref{equationinjectionvermamodules}), there is an injection
\begin{align*}
    \Hom_{G_p}\left(\cF_{\overline{B}_p}^{G_p}\left(\Hom\left(U(\fg)\otimes_{U(\fq)}L_J(\lambda),L\right)^{\overline{\mathfrak{u}}^{\infty}},\underline{\delta}_{\mathrm{sm}}\delta_{B_p}^{-1}\right),\Pi_{\infty}[\fm_y]^{\mathrm{an}}\otimes_{k(y)}k(x)\right)\\
    \hookrightarrow \Hom_{G_p}\left(\cF_{\overline{B}_p}^{G_p}\left(\left(U(\fg)\otimes_{U(\overline{\fb})}(-\lambda)\right)^{\vee},\underline{\delta}_{\mathrm{sm}}\delta_{B_p}^{-1}\right), \Pi_{\infty}[\fm_y]^{\mathrm{an}}\otimes_{k(y)}k(x)\right)\neq 0.
\end{align*}
\begin{proposition}\label{propositionpointspartialeigenvariety}
    Let $x=(y,\underline{\delta})\in X_p(\overline{\rho})_{\lambda_J'}\subset \mathrm{Spf}(R_{\infty})^{\mathrm{rig}}\times \widehat{T}_{p,L}$ be a point with $\lambda=\mathrm{wt}(\underline{\delta})\in (\Z^{n})^{\Sigma_p}$. Then $x$ is in $X_{p}(\overline{\rho})(\lambda_J')$ if and only if 
    \[\Hom_{G_p}\left(\cF_{\overline{B}_p}^{G_p}\left(\Hom\left(U(\fg)\otimes_{U(\fq)}L_J(\lambda),L\right)^{\overline{\mathfrak{u}}^{\infty}},\underline{\delta}_{\mathrm{sm}}\delta_{B_p}^{-1}\right),\Pi_{\infty}[\fm_y]^{\mathrm{an}}\otimes_{k(y)}k(x)\right)\neq 0.\]
\end{proposition}
\begin{proof}
    The point $x$ is in $X_{p}(\overline{\rho})(\lambda_J')$ if and only if $\Hom_{T_p}\left(\underline{\delta}, J_{B_{Q_p}}\left(J_{Q_p}(\Pi_{\infty}^{\mathrm{an}})_{\lambda_J'}\right)[\fm_{y}]\otimes_{k(y)} k(x)\right)\neq 0$. By the left exactness of Jacquet module functors (\cite[Lem. 3.4.7(iii)]{emerton2006jacquet}) and the definition of $J_{Q_p}(\Pi_{\infty}^{\mathrm{an}})_{\lambda_J'}$, we have $J_{B_{Q_p}}\left(J_{Q_p}(\Pi_{\infty}^{\mathrm{an}})_{\lambda_J'}\right)[\fm_{y}]\otimes_{k(y)} k(x)=J_{B_{Q_p}}\left(J_{Q_p}(\Pi_{\infty}^{\mathrm{an}}[\fm_{y}])_{\lambda_J'}\right)\otimes_{k(y)} k(x)$. Since $\Pi_{\infty}[\fm_{y}]$ is an admissible Banach representation of $G_p$ over $k(y)$, we get $\Pi_{\infty}[\fm_{y}]^{\mathrm{an}}=\Pi_{\infty}[\fm_{y}]^{R_{\infty}-\mathrm{an}}=\Pi_{\infty}^{\mathrm{an}}[\fm_{y}]$ by \cite[(3.3), Prop. 3.8]{breuil2017interpretation} (cf. \cite[Prop. 3.7]{breuil2017interpretation}). Hence $x\in X_{p}(\overline{\rho})(\lambda_J')$ if and only if \[\Hom_{T_p}\left(\underline{\delta}, J_{B_{Q_p}}\left(J_{Q_p}(\Pi_{\infty}[\fm_{y}]^{\mathrm{an}})_{\lambda_J'}\right)\otimes_{k(y)} k(x)\right)\neq 0.\]
    Now the result follows by applying Proposition \ref{propositionadjunctionJacquetmodule}.
\end{proof}
We now study the eigenvariety $X_p(\overline{\rho})(\lambda_J')$ in a standard way as in \cite[\S3.3]{breuil2017interpretation} or \cite[\S2.4]{ding2019some}.\par
For a uniform pro-$p$ Lie group $H$ which is a product of locally $F_{\widetilde{v}}$-analytic groups such as $H_p$, we let $\cC^{\mathrm{la}}(H,L)$ be the space of locally $\Q_p$-analytic functions on $H$ with coefficients in $L$. For every integer $h\geq 1$, let $r_h=\frac{1}{p^{h-1}(p-1)}$. Let $\cC^{(h)}(H,L)$ be the subspace of $\cC^{\mathrm{la}}(H,L)$ defined in \cite[Def. IV.1]{colmez2014completes}. Then $\cC^{(h)}(H,L)$ is a Banach space over $L$ and $\cC^{\mathrm{la}}(H,L)=\varinjlim_{h}\cC^{(h)}(H,L)$. Let $\cC^{\Sigma_p\setminus J-\mathrm{an}}(H,L):=\cC^{\mathrm{la}}(H,L)^{\Sigma_p\setminus J-\mathrm{an}}$ (resp. $\cC^{\Sigma_p\setminus J-\mathrm{an}, (h)}(H,L):=\cC^{(h)}(H,L)^{\Sigma_p\setminus J-\mathrm{an}}$) be the space of locally $\Sigma_p\setminus J$-analytic functions in $\cC^{\mathrm{la}}(H,L)$ (resp. $\cC^{(h)}(H,L)$) with respect to the action of $H$. Remark that the notion of locally $\Sigma_p\setminus J$-analytic functions can be defined for general products of locally $F_{\widetilde{v}}$-analytic manifolds without refering to group actions, see \cite[Def. 2.1]{schraen2010representations}. Then $\cC^{\Sigma_p\setminus J-\mathrm{an}}(H,L)=\varinjlim_{h}\cC^{\Sigma_p\setminus J-\mathrm{an}, (h)}(H,L)$ (\cite[Prop. 1.1.41]{emerton2017locally}). Let $D(H,L)=\cC^{\mathrm{an}}(H,L)'$ be the strong dual. Recall for any $h$, there is a closed immersion of Banach algebras $D_{p^{-r_h}}(H,L)\hookrightarrow D_{<p^{-r_h}}(H,L)$ defined in \cite[\S4]{schneider2003algebras} and note $D_{<p^{-r_h}}(H,L)=\cC^{(h)}(H,L)'$ is the strong dual. Set $D^{\Sigma_p\setminus J-\mathrm{an}}(H,L):=\cC^{\Sigma_p\setminus J-\mathrm{an}}(H,L)'=\varprojlim_{h}D_{<p^{-r_h}}^{\Sigma_p\setminus J-\mathrm{an}}(H,L)$ where 
\[D_{<p^{-r_h}}^{\Sigma_p\setminus J-\mathrm{an}}(H,L):=\left(\cC^{\Sigma_p\setminus J-\mathrm{an}, (h)}(H,L) \right)'.\]
Let $D_{p^{-r_h}}^{\Sigma_p\setminus J-\mathrm{an}}(H,L)$ be the completion of $D^{\Sigma_p\setminus J-\mathrm{an}}(H,L)$ with respect to the quotient norm from the norm on $D(H,L)$ induced via $D(H,L)\hookrightarrow D_{p^{-r_h}}(H,L)$ (cf. \cite[\S2.5]{schraen2010representations}). Hence we have a quotient $D_{p^{-r_{h}}}(H,L)\twoheadrightarrow D_{p^{-r_{h}}}^{\Sigma_p\setminus J-\mathrm{an}}(H,L)$ . As $\cC^{\Sigma_p\setminus J-\mathrm{an}, (h)}(H,L)$ is dense in $\cC^{\Sigma_p\setminus J-\mathrm{an}}(H,L)$, we get an injection $D^{\Sigma_p\setminus J-\mathrm{an}}(H,L)\hookrightarrow D_{<p^{-r_{h}}}^{\Sigma_p\setminus J-\mathrm{an}}(H,L)$. The norm on the target is the quotient norm from the norm on $D_{<p^{-r_h}}(H,L)$ where the surjectivity of $D_{<p^{-r_{h}}}(H,L)\rightarrow D_{<p^{-r_{h}}}^{\Sigma_p\setminus J-\mathrm{an}}(H,L)$ follows from the Hahn-Banach theorem (\cite[Prop. 9.2]{schneider2013nonarchimedean}) and the norm on $D_{<p^{-r_h}}(H,L)$ is the same as that on $D_{p^{-r_h}}(H,L)$ when restricted to $D(H,L)$ (\cite[P162]{schneider2003algebras}). Hence we get a closed embedding 
\[D_{p^{-r_{h}}}^{\Sigma_p\setminus J-\mathrm{an}}(H,L)\hookrightarrow D_{<p^{-r_{h}}}^{\Sigma_p\setminus J-\mathrm{an}}(H,L).\] 
In conclusion, we have the following diagram of morphisms of Banach algebras over $L$
\begin{center}
    \begin{tikzpicture}[scale=1.3]
        \node (A) at (0,1) {$D_{<p^{-r_{h+1}}}(H,L)$};
        \node (B) at (3,1) {$D_{p^{-r_h}}(H,L)$};
        \node (C) at (6,1) {$ D_{<p^{-r_{h}}}(H,L)$};
        \node (D) at (0,0) {$D_{<p^{-r_{h+1}}}^{\Sigma_p\setminus J-\mathrm{an}}(H,L)$};
        \node (E) at (3,0) {$D_{p^{-r_h}}^{\Sigma_p\setminus J-\mathrm{an}}(H,L)$};
        \node (F) at (6,0) {$ D_{<p^{-r_{h}}}^{\Sigma_p\setminus J-\mathrm{an}}(H,L)$};
        \path[right hook->,font=\scriptsize,>=angle 90]
        (A) edge node[above]{} (B)
        (B) edge node[above]{} (C)
        (D) edge node[above]{} (E)
        (E) edge node[above]{} (F)
        ;
        \path[->>,font=\scriptsize,>=angle 90]
        (A) edge node[above]{} (D)
        (B) edge node[above]{} (E)
        (C) edge node[above]{} (F)
        ;
        \end{tikzpicture}
 \end{center}
where each horizontal arrow is an injection and each vertical arrow is a surjection. \par
If $H$ is moreover abelian, we let $\widehat{H}_L$ be the rigid analytic space over $L$ parameterizing locally $\Q_p$-analytic characters of $H$. The rigid analytic space $\widehat{T}_{Q_p,0,\Sigma_p\setminus J,L}':=\varinjlim_h\mathrm{Sp}(D_{p^{-r_{h}}}^{\Sigma_p\setminus J-\mathrm{an}}(T_{Q_p,0}',L))$ over $L$ parameterizing locally $\Sigma_p\setminus J$-analytic characters of $T_{Q_p,0}'$ is strictly quasi-Stein, smooth and equidimensional (\cite[Prop. 6.1.13, Prop. 6.1.14]{ding2017formes}) and is a closed analytic subspace of $\widehat{T}_{Q_p,0,L}'$. By \cite[Prop. 2.18]{schraen2010representations} (cf. \cite[\S6.1.4]{ding2017formes}), the following commutative diagram is Cartesian
\begin{center}
    \begin{tikzpicture}[scale=1.3]
        \node (A1) at (0,1) {$\widehat{T}_{Q_p,0,\Sigma_p\setminus J,L}'$};
        \node (B1) at (3,1) {$\widehat{T}_{Q_p,0,L}'$};
        \node (A2) at (0,0) {$(\ft'_{Q_p,\Sigma_p\setminus J})^*$};
        \node (B2) at (3,0) {$(\ft'_{Q_p})^*$};
        \path[right hook ->,font=\scriptsize,>=angle 90]
        (A1) edge node[above]{} (B1)
        (A2) edge node[above]{} (B2)
        ;
        \path[->,font=\scriptsize,>=angle 90]
        (A1) edge node[above]{} (B1)
        (A2) edge node[above]{} (B2)
        ;
        \path[->,font=\scriptsize,>=angle 90]
        (A1) edge node[right]{$\mathrm{wt}$} (A2)
        (B1) edge node[right]{$\mathrm{wt}$} (B2)
        ;
        \end{tikzpicture}
 \end{center}
where $\ft'_{Q_p,\Sigma_p\setminus J}\simeq \ft'_{Q_p}/\ft_{Q_p,J}'$ (thus there is a closed embedding $(\ft'_{Q_p,\Sigma_p\setminus J})^*\hookrightarrow (\ft'_{Q_p})^*$). \par
We pick an element $z\in T_p^{+}$ such that $\cap_{i\geq 1}z^iN_{B_p,0}z^{-i}$ only consists of the identity element. Assume that $zH_pz^{-1}$ is normalized by $N_{B_p,0}$. Denote by $\mathcal{W}_0:= \widehat{\widetilde{Z}}_{M_{Q_p},0,L} \times \widehat{T}_{Q_p,0,\Sigma_p\setminus J,L}'=\Spf(S_{\infty})^{\mathrm{rig}}\times \widehat{Z}_{M_{Q_p},0,L}\times \widehat{T}_{Q_p,0,\Sigma_p\setminus J,L}'$.
\begin{lemma}\label{lemmaorthonormalmodules}
There exists an admissible covering of $\mathcal{W}_0$ by open affinoids $\mathrm{Sp}(B_1)\subset \mathrm{Sp}(B_2)\subset\cdots\subset\mathrm{Sp}(B_h)\subset\cdots$ and for any $h\geq 1$, there exist
\begin{itemize}
    \item[--] an orthonormalizable Banach $B_h$-module $V_h$,
    \item[--] a compact $B_h$-map $z_h:V_h\rightarrow V_h$, continuous $B_h$-maps $\alpha_h: V_h\rightarrow V_{h+1}\widehat{\otimes}_{B_{h+1}}B_h$ and $\beta_h:V_{h+1}\widehat{\otimes}_{B_{h+1}}B_h\rightarrow V_h$ such that $z_h=\beta_{h}\circ \alpha_h$ and $\alpha_h\circ\beta_h=z_{h+1}\otimes 1_{B_h}$, and
    \item[--] an isomorphism of topological $D(\widetilde{Z}_{M_{Q_p},0},L)\widehat{\otimes}_L D^{\Sigma_p\setminus J-\mathrm{an}}(T'_{Q_p,0},L)$-modules 
    \[\left(\left(\Pi_{\infty}^{\mathrm{an},N_{Q_p,0}}\otimes L_{M_{Q_p}}(\widetilde{\lambda}_J)'\right)^{\Sigma_p\setminus J-\mathrm{an},U_{Q_p,0}}\right)'\simeq\varprojlim_{h}V_h\]
\end{itemize}
such that the action of the operator $\pi_z$ induced by the Hecke action of $z$ on the left hand side coincides with the action of $(z_h)_{h\in\N}$ on the right hand side.\\
We can summarize the datum in the following commutative diagram:
\begin{center}
    \begin{tikzpicture}[scale=1.1]
        \node (A1) at (0,2) {$\left(\left(\Pi_{\infty}^{\mathrm{an},N_{Q_p,0}}\otimes L_{M_{Q_p}}(\widetilde{\lambda}_J)'\right)^{\Sigma_p\setminus J-\mathrm{an},U_{Q_p,0}}\right)'$};
        \node (B1) at (4,2) {$\cdots$};
        \node (C1) at (6,2) {$ V_{h+1}$};
        \node (D1) at (8,2) {$ V_{h+1}\widehat{\otimes}_{B_{h+1}}B_h$};
        \node (E1) at (10,2) {$ V_{h}$};
        \node (A2) at (0,0) {$\left(\left(\Pi_{\infty}^{\mathrm{an},N_{Q_p,0}}\otimes L_{M_{Q_p}}(\widetilde{\lambda}_J)'\right)^{\Sigma_p\setminus J-\mathrm{an},U_{Q_p,0}}\right)'$};
        \node (B2) at (4,0) {$\cdots$};
        \node (C2) at (6,0) {$ V_{h+1}$};
        \node (D2) at (8,0) {$ V_{h+1}\widehat{\otimes}_{B_{h+1}}B_h$};
        \node (E2) at (10,0) {$ V_{h}$};
        \path[->,font=\scriptsize,>=angle 90]
        (A1) edge node[above]{} (B1)
        (B1) edge node[above]{} (C1)
        (C1) edge node[above]{} (D1)
        (D1) edge node[above]{$\beta_h$} (E1)
        (A2) edge node[above]{} (B2)
        (B2) edge node[above]{} (C2)
        (C2) edge node[above]{} (D2)
        (D2) edge node[above]{$\beta_h$} (E2)
        (A1) edge node[left]{$\pi_z$} (A2)
        (C1) edge node[left]{$z_{h+1}$} (C2)
        (D1) edge node[left]{$z_{h+1}\otimes 1_{B_h}$} (D2)
        (E1) edge node[right]{$z_h$} (E2)
        (E1) edge node[below]{$\alpha_h$} (D2)
        ;
        \end{tikzpicture}
 \end{center}
\end{lemma}
\begin{proof}
    We have $\Pi_{\infty}|_{\widetilde{H}_p}\simeq \cC(\widetilde{H}_p,L)^m$ and $\Pi_{\infty}^{\mathrm{an}}=\Pi_{\infty}^{S_{\infty}-\mathrm{an}}\simeq \cC^{\mathrm{la}}(\widetilde{H}_p,L)^m$ by \cite[Prop. 3.8]{breuil2017interpretation}. Since $\widetilde{H}_p=\overline{N}_{Q_p,0}\times \widetilde{M}_{Q_p,0}\times N_{Q_p,0}$, we have 
    \begin{equation}\label{equationCanPiinfty}
        \Pi_{\infty}^{\mathrm{an}}\simeq \left(\cC^{\mathrm{la}}(\overline{N}_{Q_p,0},L)\widehat{\otimes}_L\cC^{\mathrm{la}}(\widetilde{M}_{Q_p,0},L)\widehat{\otimes}_L \cC^{\mathrm{la}}(N_{Q_p,0},L)\right)^m.
    \end{equation}
    Thus 
    \[\Pi_{\infty}^{\mathrm{an},N_{Q_p,0}}\simeq \left(\cC^{\mathrm{la}}(\overline{N}_{Q_p,0},L)\widehat{\otimes}_L\cC^{\mathrm{la}}(\widetilde{M}_{Q_p,0},L)\right)^m.\]
    By the twisting lemma (\cite[Lem. 2.19]{ding2019some}), we have an isomorphism of $\widetilde{M}_{Q_p,0}$-representations
    \begin{align}\label{twistinglemma}
        \cC^{\mathrm{la}}(\widetilde{M}_{Q_p,0},L)\otimes L_{M_{Q_p}}(\widetilde{\lambda}_J)'|_{\widetilde{M}_{Q_p,0}}&\simrightarrow  \cC^{\mathrm{la}}(\widetilde{M}_{Q_p,0},L)^{m'} \\\nonumber
      f\otimes v&\mapsto (g\mapsto f(g)gv)  
    \end{align}
    for $m'=\dim_L L_{M_{Q_p}}(\widetilde{\lambda}_J)'$. Hence 
    \begin{align*}
        \Pi_{\infty}^{\mathrm{an},N_{Q_p,0}}\otimes L_{M_{Q_p}}(\widetilde{\lambda}_J)' &\simeq \left(\cC^{\mathrm{la}}(\overline{N}_{Q_p,0},L)\widehat{\otimes}_L\cC^{\mathrm{la}}(\widetilde{M}_{Q_p,0},L)\otimes_L L_{M_{Q_p}}(\widetilde{\lambda}_J)'\right)^m\\
        &\simeq \left(\cC^{\mathrm{la}}(\overline{N}_{Q_p,0},L)\widehat{\otimes}_L\left(\cC^{\mathrm{la}}(\widetilde{M}_{Q_p,0},L)\right)^{m'}\right)^m\\
        &\simeq \left(\cC^{\mathrm{la}}(\overline{N}_{Q_p,0},L)\widehat{\otimes}_L\cC^{\mathrm{la}}(\widetilde{Z}_{M_{Q_p},0},L)\widehat{\otimes}_L\cC^{\mathrm{la}}(M_{Q_p,0}',L) \right)^r,
    \end{align*}
    where $r=mm'$. 
    Then 
    \[ \Pi_{\infty}^{\mathrm{an},N_{Q_p,0}}\otimes L_{M_{Q_p}}(\widetilde{\lambda}_J)'\simeq \varinjlim_{h} \left(\cC^{(h)}(\overline{N}_{Q_p,0},L)\widehat{\otimes}_L\cC^{(h)}(\widetilde{Z}_{M_{Q_p},0},L)\widehat{\otimes}_L\cC^{(h)}(M_{Q_p,0}',L) \right)^r.\]
    Hence by \cite[Prop. 1.1.41]{emerton2017locally}, we get
    \begin{align*}
        &\left(\Pi_{\infty}^{\mathrm{an},N_{Q_p,0}}\otimes L_{M_{Q_p}}(\widetilde{\lambda}_J)'\right)^{\Sigma_p\setminus J-\mathrm{an}}\\
        =&\varinjlim_{h} \left(\cC^{(h)}(\overline{N}_{Q_p,0},L)\widehat{\otimes}_L\cC^{(h)}(\widetilde{Z}_{M_{Q_p},0},L)\widehat{\otimes}_L\cC^{\Sigma_p\setminus J-\mathrm{an}, (h)}(M_{Q_p,0}',L) \right)^r\\
        =&\varinjlim_{h} \cC^{(h)}(\overline{N}_{Q_p,0},L)^r\widehat{\otimes}_L\cC^{(h)}(\widetilde{Z}_{M_{Q_p},0},L)\widehat{\otimes}_L\cC^{\Sigma_p\setminus J-\mathrm{an}, (h)}(\overline{U}_{Q_p,0},L) \\
        &\qquad\widehat{\otimes}_L\cC^{\Sigma_p\setminus J-\mathrm{an}, (h)}(T_{Q_p,0}',L) \widehat{\otimes}_L\cC^{\Sigma_p\setminus J-\mathrm{an}, (h)}(U_{Q_p,0},L). 
    \end{align*}
    Thus 
    \begin{align*}
    &\left(\Pi_{\infty}^{\mathrm{an},N_{Q_p,0}}\otimes L_{M_{Q_p}}(\widetilde{\lambda}_J)'\right)^{\Sigma_p\setminus J-\mathrm{an},U_{Q_p,0}} \\
    &=\varinjlim_{h} \cC^{(h)}(\overline{N}_{Q_p,0},L)^r\widehat{\otimes}_L\cC^{(h)}(\widetilde{Z}_{M_{Q_p},0},L)\widehat{\otimes}_L\cC^{\Sigma_p\setminus J-\mathrm{an}, (h)}(\overline{U}_{Q_p,0},L) \widehat{\otimes}_L\cC^{\Sigma_p\setminus J-\mathrm{an}, (h)}(T_{Q_p,0}',L).
    \end{align*}
    We let $W_h=\left(\cC^{(h)}(\overline{N}_{Q_p,0},L)^r\widehat{\otimes}_L\cC^{\Sigma_p\setminus J-\mathrm{an}, (h)}(\overline{U}_{Q_p,0},L) \right)'$. By \cite[Prop. 1.1.22, Prop. 1.1.32]{emerton2017locally}, we get
    \begin{align}
        \left(\left(\Pi_{\infty}^{\mathrm{an},N_{Q_p,0}}\otimes L_{M_{Q_p}}(\widetilde{\lambda}_J)'\right)^{\Sigma_p\setminus J-\mathrm{an},U_{Q_p,0}}\right)'&=\varprojlim_{h} W_h\widehat{\otimes}_L\left(\cC^{(h)}(\widetilde{Z}_{M_{Q_p},0},L)\widehat{\otimes}_L\cC^{\Sigma_p\setminus J-\mathrm{an}, (h)}(T_{Q_p,0}',L)\right)'\nonumber\\
        &=\varprojlim_{h} W_h\widehat{\otimes}_LD_{<p^{-r_h}}(\widetilde{Z}_{M_{Q_p},0},L)\widehat{\otimes}_LD_{<p^{-r_h}}^{\Sigma_p\setminus J-\mathrm{an}}(T_{Q_p,0}',L).
    \end{align}
    Set $B_h=D_{p^{-r_h}}(\widetilde{Z}_{M_{Q_p},0},L)\widehat{\otimes}_L D_{p^{-r_h}}^{\Sigma_p\setminus J-\mathrm{an}}(T'_{Q_p,0},L)$ and $V_h:=W_h\widehat{\otimes}_LB_h$ which is an orthonormalizable Banach $B_h$-module by definition \cite[\S2]{buzzard2007eigenvarieties}. Then 
    \begin{align*}
        V_h=&\left(\left(\cC^{(h)}(\overline{N}_{Q_p,0},L)\widehat{\otimes}_L\cC^{\Sigma_p\setminus J-\mathrm{an}, (h)}(\overline{U}_{Q_p,0},L) \right)^r
        \widehat{\otimes}_L\cC^{(h+1)}(\widetilde{Z}_{M_{Q_p},0},L)\widehat{\otimes}_L\cC^{\Sigma_p\setminus J-\mathrm{an}, (h+1)}(T_{Q_p,0}',L)\right)'\\
        &\widehat{\otimes}_{D_{<p^{-r_{h+1}}}(\widetilde{Z}_{M_{Q_p},0},L)\widehat{\otimes}_LD_{<p^{-r_{h+1}}}^{\Sigma_p\setminus J-\mathrm{an}}(T_{Q_p,0}',L)}B_h.
    \end{align*}
    As the map 
    \[ W_{h+1}\widehat{\otimes}_LD_{<p^{-r_{h+1}}}(\widetilde{Z}_{M_{Q_p},0},L)\widehat{\otimes}_LD_{<p^{-r_{h+1}}}^{\Sigma_p\setminus J-\mathrm{an}}(T_{Q_p,0}',L)\rightarrow W_h\widehat{\otimes}_LD_{<p^{-r_h}}(\widetilde{Z}_{M_{Q_p},0},L)\widehat{\otimes}_LD_{<p^{-r_{h}}}^{\Sigma_p\setminus J-\mathrm{an}}(T_{Q_p,0}',L)\] 
    factors through $V_h$, we conclude that $\left(\left(\Pi_{\infty}^{\mathrm{an},N_{Q_p,0}}\otimes L_{M_{Q_p}}(\widetilde{\lambda}_J)'\right)^{\Sigma_p\setminus J-\mathrm{an},U_{Q_p,0}}\right)'=\varprojlim_{h}V_h$.\par
    We now track the action of the element $z$. The action of $z$ sends $(\Pi_{\infty})^{(h)}_{\widetilde{H}_p}$ to $(\Pi_{\infty})^{(h)}_{z\widetilde{H}_pz^{-1}}$ where $(\Pi_{\infty})^{(h)}_{\widetilde{H}_p}$ is the subspace of $\Pi_{\infty}^{\mathrm{an}}$ defined in \cite[IV.D]{colmez2014completes}. By \cite[Lem. 5.2 \& proof of Lem. 5.3]{breuil2017smoothness}, the action of $z$ on $\Pi_{\infty}^{\mathrm{an}}$ sends $(\Pi_{\infty})^{(h)}_{\widetilde{H}_p}$ to 
    \[\cC^{(h-1)}(\overline{N}_{Q_p,0},L)^m\widehat{\otimes}_L \cC^{(h-1)}(\overline{U}_{Q_p,0},L)\widehat{\otimes}_L\cC^{(h)}(\widetilde{Z}_{M_{Q_p},0},L) \widehat{\otimes}_L\cC^{(h)}(T_{Q_p,0}',L) \widehat{\otimes}_L\cC^{\mathrm{la}}(U_{Q_p,0},L)\otimes \cC^{\mathrm{la}}(N_{Q_p,0},L)\]
    in term of the isomorphism (\ref{equationCanPiinfty}).\par
    We assume that an homeomorphism $\Z_p^u\simeq M_{Q_p,0}$ is chosen so that the matrix coefficients of elements in $M_{Q_p,0}$ are overconvergent analytic functions on $\Z_p^u$ (for example, we can choose coordinates of $U_{Q_p,0}$ and $\overline{U}_{Q_p, 0}$ as a product of root groups where each is identified with some $\cO_{F_{\widetilde{v}}}$ and coordinates of $T_{p,0}$ as products of some $1+\varpi_{F_{\widetilde{v}}}^t\cO_{F_{\widetilde{v}}}$ with $t$ large enough. Then the matrix coefficients are in the ring generated by polynomials and functions of the form $\exp(\varpi_{F_{\widetilde{v}}}^tx),x\in \cO_{F_{\widetilde{v}}}$). Since the action of $g$ and $g^{-1}$ on $L_{M_{Q_p}}(\widetilde{\lambda}_J)'$, for $g\in M_{Q_p,0}$, are given by polynomial functions on the matrix coefficients of $g,g^{-1}$ compositing with embeddings in $\Sigma_p$, the matrix coefficients of the twisting isomorphism (\ref{twistinglemma}) in the basis given by one of that of $L_{M_{Q_p}}(\widetilde{\lambda}_J)'$ are overconvergent analytic functions on $\Z_p^u$. By a similar argument as in Lemma \ref{lemmahanalyticfunctions} below, for any $h$ large enough (which we may assume from now on), there is an isomorphism 
    \[\left((\Pi_{\infty})^{(h)}_{\widetilde{H}_p}\right)^{N_{Q_p,0}}\otimes L_{M_{Q_p}}(\widetilde{\lambda}_J)' \simeq \cC^{(h)}(\overline{N}_{Q_p,0},L)^r\widehat{\otimes}_L\cC^{(h)}(\widetilde{Z}_{M_{Q_p},0},L)\widehat{\otimes}_L\cC^{(h)}(M_{Q_p,0}',L)\]
    and similarly the isomorphism (\ref{twistinglemma}) sends 
    \[\cC^{(h-1)}(\overline{N}_{Q_p,0},L)^m\widehat{\otimes}_L \cC^{(h-1)}(\overline{U}_{Q_p,0},L)\widehat{\otimes}_L\cC^{(h)}(\widetilde{Z}_{M_{Q_p},0},L) \widehat{\otimes}_L\cC^{(h)}(T_{Q_p,0}',L)\widehat{\otimes}_L \cC^{\mathrm{la}}(U_{Q_p,0},L)\widehat{\otimes} L_{M_{Q_p}}(\widetilde{\lambda}_J)'\] 
    to 
    \[\cC^{(h-1)}(\overline{N}_{Q_p,0},L)^r\widehat{\otimes}_L \cC^{(h-1)}(\overline{U}_{Q_p,0},L)\widehat{\otimes}_L\cC^{(h)}(\widetilde{Z}_{M_{Q_p},0},L) \widehat{\otimes}_L\cC^{(h)}(T_{Q_p,0}',L)\widehat{\otimes }\cC^{\mathrm{la}}(U_{Q_p,0},L).\]
    Hence the action of $z$ on $\left((\Pi_{\infty})^{(h)}_{\widetilde{H}_p}\right)^{N_{Q_p,0}}\otimes L_{M_{Q_p}}(\widetilde{\lambda}_J)'$ sends, after (\ref{twistinglemma}),
    \[\cC^{(h)}(\overline{N}_{Q_p,0},L)^r\widehat{\otimes}_L\cC^{(h)}(\widetilde{Z}_{M_{Q_p},0},L)\widehat{\otimes}_L\cC^{(h)}(M_{Q_p,0}',L)\] 
    to 
    \[\cC^{(h-1)}(\overline{N}_{Q_p,0},L)^r\widehat{\otimes}_L \cC^{(h-1)}(\overline{U}_{Q_p,0},L)\widehat{\otimes}_L\cC^{(h)}(\widetilde{Z}_{M_{Q_p},0},L) \widehat{\otimes}_L\cC^{(h)}(T_{Q_p,0}',L)\widehat{\otimes }\cC^{\mathrm{la}}(U_{Q_p,0},L).\] 
    Finally, we conclude that the Hecke action of $z$, denoted by $\pi_z$, on $\left(\Pi_{\infty}^{\mathrm{an},N_{Q_p,0}}\otimes L_{M_{Q_p}}(\widetilde{\lambda}_J)'\right)^{\Sigma_p\setminus J-\mathrm{an},U_{Q_p,0}}$ induces a map sending 
    \[\left(\cC^{(h)}(\overline{N}_{Q_p,0},L)\widehat{\otimes}_L\cC^{\Sigma_p\setminus J-\mathrm{an}, (h)}(\overline{U}_{Q_p,0},L) \right)^r\widehat{\otimes}_L\cC^{(h)}(\widetilde{Z}_{M_{Q_p},0},L)\widehat{\otimes}_L\cC^{\Sigma_p\setminus J-\mathrm{an}, (h)}(T_{Q_p,0}',L)\]
    to
    \[\left(\cC^{(h-1)}(\overline{N}_{Q_p,0},L)\widehat{\otimes}_L\cC^{\Sigma_p\setminus J-\mathrm{an}, (h-1)}(\overline{U}_{Q_p,0},L)\right)^r\widehat{\otimes}_L\cC^{(h)}(\widetilde{Z}_{M_{Q_p},0},L) \widehat{\otimes}_L\cC^{\Sigma_p\setminus J-\mathrm{an}, (h)}(T_{Q_p,0}',L).\]
    Taking the dual of the above map we see that $\pi_z$ induces a morphism $\alpha_{h-1}: V_{h-1}\rightarrow V_{h}\widehat{\otimes}_{B_{h}}B_{h-1}$. We also have a morphism $\beta_h=\beta_h'\otimes 1_{B_h}:V_{h+1}\widehat{\otimes}_{B_{h+1}}B_h\rightarrow V_h$ where $\beta_h'$ is induced by the dual of the compact map
    \[\cC^{(h)}(\overline{N}_{Q_p,0},L)^r\widehat{\otimes}_L\cC^{\Sigma_p\setminus J-\mathrm{an}, (h)}(\overline{U}_{Q_p,0},L)\rightarrow \cC^{(h+1)}(\overline{N}_{Q_p,0},L)^r\widehat{\otimes}_L\cC^{\Sigma_p\setminus J-\mathrm{an}, (h+1)}(\overline{U}_{Q_p,0},L).\]
    Hence $\beta_h$ is a compact map of Banach $B_h$-modules. We put $z_h=\beta_h\circ \alpha_h$. Then $z_{h+1}\otimes 1_{B_h}=\alpha_h\circ \beta_h$. 
\end{proof}
\begin{lemma}\label{lemmahanalyticfunctions}
    Assume that $g$ is an overconvergent analytic function over $\Z_p$. Then there exists $C$ such that for any $h>C$ and $f\in\cC^{(h)}(\Z_p,L)$, we have $gf\in \cC^{(h)}(\Z_p,L)$. 
\end{lemma}
\begin{proof}
    By definition (\cite[Def. IV.1]{colmez2014completes}), 
    \[\cC^{(h)}(\Z_p,L)=\left\{\sum_{n\in\N}a_n\binom{x}{n}\mid \lim_{n\to+\infty}\left(v_p(a_n)-r_hn\right)=+\infty\right\}\] 
    is a Banach space with valuation $v^{(h)}\left(\sum_{n\in\N}a_n\binom{x}{n}\right)=\inf_n\left(v_p(a_n)-r_hn\right)$. Let $f=\sum_{n\in\N}a_n\binom{x}{n}\in \cC^{(h)}(\Z_p,L)$. We compute $x\left(\sum_{n\in\N}a_n\binom{x}{n}\right)=\sum_{n\in\N}a_nx\binom{x}{n}=\sum_{n\in\N}a_n\left((n+1)\binom{x}{n+1}+n\binom{x}{n}\right)=\sum_{n\geq 1}n(a_n+a_{n-1})\binom{x}{n}$. We have 
    \[\lim_{n\to \infty}\left(v_p\left(n(a_n+a_{n-1})\right)-r_hn\right)\geq \lim_{n\to+\infty}\min\left\{\left(v_p(a_n)-r_hn\right), \left(v_p(a_{n-1})-r_h(n-1)-r_h\right)\right\}=+\infty\] 
    and 
    \begin{align*}
        v^{(h)}(x\sum_{n\in\N}a_n\binom{x}{n})=&\inf_n\left(v_p\left(n(a_n+a_{n-1}\right)\right)-r_hn)\\
        \geq &\inf_n\min\left\{(v_p(a_n)-r_hn), (v_p(a_{n-1})-r_h(n-1)-r_h)\right\}\\
        \geq &\inf_n (v_p(a_n)-r_hn)-r_h=v^{(h)}(f)-r_h.
    \end{align*}
    We get $x^{n}f$ lies in $\cC^{(h)}(\Z_p,L)$ and $v^{(h)}(x^{n}f)\geq v^{(h)}(f)-r_hn$ for any $n$. Since $g$ is analytic, we can assume $g=\sum_{n\in\N}b_nx^n$. Then $v^{(h)}(b_nx^{n}f)\geq v^{(h)}(f)+(v_p(b_n)-r_hn)$. Since $g$ is overconvergent and $r_h\to 0$ if $h\to \infty$, for $h$ large enough, we have $\lim_{n\to+\infty}  v^{(h)}(f)+(v_p(b_n)-r_hn)=+\infty$. Hence $gf=\sum_{n\in\N}b_nx^nf$ converges in the Banach space $\cC^{(h)}(\Z_p,L)$.
\end{proof}
We denote by $\mathcal{W}_{\lambda_J'}:=\Spf(S_{\infty})^{\mathrm{rig}}\times \widehat{Z}_{M_{Q_p},0,L}\times (\widehat{T}_{Q_p,0,L}')_{\lambda_J'}$ where $(\widehat{T}_{Q_p,0,L}')_{\lambda_J'}$ is the fiber of $\widehat{T}_{Q_p,0,L}'$ over $\lambda_J'$ via the map $\widehat{T}_{Q_p,0,L}'\stackrel{\mathrm{wt}}{\rightarrow}(\ft')^*\stackrel{\mathrm{res}}{\rightarrow} (\ft'_{Q_p,J})^*$. There is an isomorphism $\mathcal{W}_0\rightarrow \mathcal{W}_{\lambda_J'}:x\mapsto xz^{\widetilde{\lambda}_J}$ where $z^{\widetilde{\lambda}_J}$ is the character of $T_{p,0}$ on $L_{M_{Q_p}}(\widetilde{\lambda}_J)^{U_{Q_p,0}}$ in Lemma \ref{lemmatwistedJacquetmodule}. The restriction of characters from $T_p$ to $T_{p,0}$ induces a morphism $\omega_{X(\lambda_J')}:X_{p}(\overline{\rho})(\lambda_J')\rightarrow \mathcal{W}_{\lambda_J'}$. The eigenvariety machinery in \cite{buzzard2007eigenvarieties} and \cite{chenevier2004familles} leads to the following basic result on the partial eigenvariety.
\begin{proposition}\label{propositionequidimensional}
    The partial eigenvariety $X_{p}(\overline{\rho})(\lambda_J')$ is equidimensional and for any point in $X_{p}(\overline{\rho})(\lambda_J')$, there exists an open affinoid neighborhood $U$ such that there exists an affinoid open subset $W$ of $\mathcal{W}_{\lambda_J'}$ satisfying that the restriction of $\omega_{X(\lambda_J')}$ to any irreducible component of $U$ is finite and surjective onto $W$. Moreover, the image of an irreducible component of $X_{p}(\overline{\rho})(\lambda_J')$ is a Zariski open subset of $\mathcal{W}_{\lambda_J'}$.
\end{proposition}
\begin{proof}
    The result can be proved by a slight modification of the proofs in \cite[\S3.3]{breuil2017interpretation} replacing $J_{B_p}(\Pi_{\infty}^{\mathrm{an}})$ (resp. $\mathcal{W}_{\infty}$) in \textit{loc. cit.} with the module $J_{B_{Q_p}}\left(\Hom_{U(\fm_{Q_p,J}')}\left(L_{M_{Q_p}}(\widetilde{\lambda}_J), J_{Q_p}(\Pi_{\infty}^{\mathrm{an}})\right)\right)$ (resp. $\mathcal{W}_{0}$) using Lemma \ref{lemmafiniteslopeJacquet} and Lemma \ref{lemmaorthonormalmodules} and then applying Lemma \ref{lemmatwistedJacquetmodule} to obtain the results for $X_{p}(\overline{\rho})(\lambda_J')$ and $\mathcal{W}_{\lambda_J'}$.
\end{proof}
\subsection{Density of classical points}\label{sectiondesityofclassicalpoints}
We prove the density of de Rham points on the partial eigenvariety which will be the input for the application of the partial eigenvariety in next subsection.\par
Suppose $x=(y,\underline{\delta})\in X_{p}(\overline{\rho})\subset \mathrm{Spf}(R_{\infty})^{\mathrm{rig}}\times \widehat{T}_{p,L}$ is a point such that $\lambda=\mathrm{wt}(\underline{\delta})$ is in $(\Z^n)^{\Sigma_p}$ and is dominant with respect to $\fb$. The $U(\fg)$-module $U(\fg)\otimes_{U(\fb)}\lambda $ (and its quotient $U(\fg)\otimes_{U(\fq)}L_J(\lambda)$ in \S\ref{sectionadjunctionformula}) admits a unique irreducible quotient $L(\lambda)$, and hence 
\[\cF_{\overline{B}_p}^{G_p}\left(\Hom\left(U(\fg)\otimes_{U(\fq)}L_J(\lambda),L\right)^{\overline{\mathfrak{u}}^{\infty}},\underline{\delta}_{\mathrm{sm}}\delta_{B_p}^{-1}\right),\] 
as well as $\cF_{\overline{B}_p}^{G_p}(\underline{\delta})$, admits a locally $\Q_p$-algebraic quotient $\cF_{\overline{B}_p}^{G_p}\left(L(\lambda)',\underline{\delta}_{\mathrm{sm}}\delta_{B_p}^{-1}\right)$ which is isomorphic to $\cF_{G_p}^{G_p}\left(L(\lambda)',(\mathrm{Ind}_{\overline{B}_p}^{G_p}\underline{\delta}_{\mathrm{sm}}\delta_{B_p}^{-1})^{\mathrm{sm}}\right)\simeq L(\lambda)\otimes (\mathrm{Ind}_{\overline{B}_p}^{G_p}\underline{\delta}_{\mathrm{sm}}\delta_{B_p}^{-1})^{\mathrm{sm}}$ (cf. \cite[\S3.5]{breuil2017interpretation}). We say that $x$ is a \emph{classical point} if 
\[\Hom_{G_p}\left(L(\lambda)\otimes (\mathrm{Ind}_{\overline{B}_p}^{G_p}\underline{\delta}_{\mathrm{sm}}\delta_{B_p}^{-1})^{\mathrm{sm}},\Pi_{\infty}[\fm_y]^{\mathrm{an}}\otimes_{k(y)}k(x)\right)\neq 0.\]
\begin{remark}
    Our definition of classical points differs from \cite[Def. 3.15]{breuil2017interpretation} because we will consider points that are crystabelline rather than only crystalline. 
\end{remark}
\begin{proposition}\label{propositiondenseclassicalpoints}
    The subset of classical points in $X_{p}(\overline{\rho})(\lambda_J')$ is Zariski dense. Moreover, for any point $x=(y,\underline{\delta})\in X_{p}(\overline{\rho})(\lambda_J')$ such that $\underline{\delta}$ is locally algebraic and any irreducible component $X$ of $X_{p}(\overline{\rho})(\lambda_J')$ such that $x\in X$, there is an affinoid open subset $U$ of $X$ containing $x$ such that the subset of classical points is Zariski-dense in $U$.  
\end{proposition}
\begin{proof}
    The proof follows that of \cite[Thm. 3.19]{breuil2017interpretation} and \cite[Thm. 3.12]{ding2019some}. For each $v\in S_p, 1\leq i\leq n$, we let $\gamma_{\widetilde{v},i}$ be the element $\mathrm{diag}(\underbrace{\varpi_{\widetilde{v}},\cdots,\varpi_{\widetilde{v}}}_{i},1,\cdots,1)\in\GL_{n}(F_{\widetilde{v}})$ where $\varpi_{\widetilde{v}}$ is a fixed uniformizer of $F_{\widetilde{v}}$.\par
    By Proposition \ref{propositionequidimensional}, we can pick a covering by open affinoids of any irreducible component of $X_p(\overline{\rho})(\lambda_J')$ such that $\omega_{X(\lambda_J')}$ sends such open affinoids surjectively and finitely onto open affinoids of $\mathcal{W}_{\lambda_J'}$. For any $x=(y,\underline{\delta})\in X_{p}(\overline{\rho})(\lambda_J')$ with $\underline{\delta}$ locally algebraic, we pick a such affinoids $U\subset X_p(\overline{\rho})(\lambda_J')$ and let $W=\omega_{X(\lambda_J')}(U)$. We prove that the subset of classical points is Zariski dense in $U$.\par
    Since $U$ is affinoid, the rigid analytic functions $(y,\underline{\delta})\mapsto \underline{\delta}_v\delta_{B_{v}}^{-1}(\gamma_{\widetilde{v},i})$ are bounded on $U$ and thus there exists a constant $C>0$ such that $C\leq |\underline{\delta}_v\delta_{B_{v}}^{-1}(\gamma_{\widetilde{v},i})|_p$ over $U$ for any $v\in S_p,i=1,\cdots,n$ where $|\cdot|_p$ is the $p$-adic absolute value such that $|p|_p=\frac{1}{p}$.\par
    If $\underline{\delta}\in (\widehat{T}_{Q_p,0,L})_{\lambda_J'}=\widehat{Z}_{M_{Q_p},0,L}\times (\widehat{T}_{Q_p,0,L}')_{\lambda_J'}$ and $\lambda=\mathrm{wt}(\underline{\delta})$, then for any $v\in S_p,\tau\in J_v$, we have $\lambda_{\tau,\widetilde{q}_{v,i}+j}-\lambda_{\tau, \widetilde{q}_{v,i}+j+1}=\widetilde{\lambda}_{\tau,\widetilde{q}_{v,i}+j}-\widetilde{\lambda}_{\tau,\widetilde{q}_{v,i}+j+1}\geq 0$ for $1\leq j \leq q_{v,i+1}-1$ and $0\leq i\leq t_v-1$. We pick a constant $C'>0$ such that $C |\tau(\varpi_{\widetilde{v}})|_p^{-1-C'}>1$ for any $v\in S_p, \tau\in \Sigma_v $. Then the set of points in $(\widehat{T}_{Q_p,0,L})_{\lambda_J'}$ with integral dominant weights (i.e. $\lambda\in (\Z^n)^{\Sigma_p}$ and $\lambda_{\tau,i}\geq \lambda_{\tau,j}, \forall \tau\in\Sigma_p,i\leq j$) satisfying the following conditions
    \begin{align}\label{equationconditionclassical1}
        &\lambda_{\tau,i}-\lambda_{\tau,i+1}> C', \forall i\in\left\{1,\cdots, n-1\right\} \quad \qquad\text{if }v\in S_p, \tau\notin J_v\\
        &\lambda_{\tau,i}-\lambda_{\tau,i+1}> C', \forall i\in \left\{\widetilde{q}_{v,1},\cdots,\widetilde{q}_{v,t_v-1}\right\} \quad \qquad\text{if }v\in S_p, \tau\in J_v\label{equationconditionclassical2}
    \end{align}
    accumulates (\cite[Def. 2.2]{breuil2017interpretation}) at locally algebraic characters in $(\widehat{T}_{Q_p,0,L})_{\lambda_J'}$. \par
    Let $Z$ be the subset of $W$ consisting of points such that the images in $(\widehat{T}_{Q_p,0,L})_{\lambda_J'}$ via the map $W\subset \mathcal{W}_{\lambda_J'}=\Spf(S_{\infty})^{\mathrm{rig}}\times (\widehat{T}_{Q_p,0,L})_{\lambda_J'}\rightarrow (\widehat{T}_{Q_p,0,L})_{\lambda_J'}$ are locally algebraic and satisfy (\ref{equationconditionclassical1}) and (\ref{equationconditionclassical2}). Note that the last map is smooth, hence open. Since $U$ contains a point with locally algebraic character, so is $W$. Hence $Z$ is Zariski dense in $W$.\par
    We claim that the set of dominant points in $U$ satisfying the conditions (\ref{equationconditionclassical1}) and (\ref{equationconditionclassical2}) (i.e. $\omega_{X(\lambda_J')}^{-1}(Z)$) is Zariski dense in $U$. Otherwise, by the irreducibility of $U$, the Zariski closure of $\omega_{X(\lambda_J')}^{-1}(Z)$ in $U$ has dimension strictly less than that of $U$. But the image of the closure of $\omega_{X(\lambda_J')}^{-1}(Z)$ in $W$ is a closed subset containing $Z$, which must equal to $W$ and hence shares the same dimension with $U$. This contradicts the assertion on the dimension of the Zariski closure of $\omega_{X(\lambda_J')}^{-1}(Z)$. Hence the claim holds (this is the argument in the proof of \cite[Thm. 3.19]{breuil2017interpretation}). \par
    We will show that any point in $\omega_{X(\lambda_J')}^{-1}(Z)$ is classical which allows us to conclude that the set of classical points is Zariski-dense in $U$. The Zariski density of the classical points in the whole $X_{p}(\overline{\rho})(\lambda_J')$ then follows from the fact that the set of locally algebraic characters is Zariski dense in $(\widehat{T}_{Q_p,0,L})_{\lambda_J'}$ and the last assertion in Proposition \ref{propositionequidimensional}.\par
    Now we assume that $x=(y,\underline{\delta})\in U\subset X_{p}(\overline{\rho})(\lambda_J')\subset \mathrm{Spf}(R_{\infty})^{\mathrm{rig}}\times \widehat{T}_{p,L}$ is a point such that the weight $\lambda$ of $\underline{\delta}$ is integral dominant and satisfies (\ref{equationconditionclassical1}) and (\ref{equationconditionclassical2}). We prove that $x$ is a classical point. \par
    Remark that the conditions (\ref{equationconditionclassical1}) and (\ref{equationconditionclassical2}) are some ``small slope'' conditions and the proof of the classicality will be essentially the same with the usual case (i.e., for points on $X_p(\overline{\rho})$ as in \cite[Thm. 3.19]{breuil2017interpretation}). However, we cannot directly cite the proof of \cite[Thm. 3.19]{breuil2017interpretation}. This is because that under the restriction of the weights on the partial eigenvariety $X_{p}(\overline{\rho})(\lambda_J')$, the points that satisfy the full ``small slope'' condition as in \cite[(3.11)]{breuil2017interpretation} can not be Zariski dense. The condition (\ref{equationconditionclassical1}) and (\ref{equationconditionclassical2}) here is only some \emph{weaker} ``small slope'' condition. To prove the classicality, one will need furthermore to use the fact that $x$ is ``partially classical'', i.e., $x$ lies in the partial eigenvariety $X_{p}(\overline{\rho})(\lambda_J')$. Ding has already proved such results in special cases in \cite{ding2017formes} and \cite{ding2019some}. Since a direct reference is not available for our situation, we would like to write down the details of the proof below.\par
    Without loss of generality, we assume that the residue field of $x$ is $L$. By Proposition \ref{propositionpointspartialeigenvariety}, there is a non-zero morphism 
    \[\cF_{\overline{B}_p}^{G_p}\left(\Hom\left(U(\fg)\otimes_{U(\fq)}L_J(\lambda),L\right)^{\overline{\mathfrak{u}}^{\infty}}, \underline{\delta}_{\mathrm{sm}}\delta_{B_p}^{-1}\right)\hookrightarrow\Pi_{\infty}[\fm_y]^{\mathrm{an}}.\]
    By definition (see \S\ref{sectionadjunctionformula}),
    \[U(\fg)\otimes_{U(\fq)}L_J(\lambda)=\left(\otimes_{v\in S_v, \tau\in \Sigma_v\setminus J_v} U(\fg_{\tau})\otimes_{U(\fb_{\tau})} \lambda_{\tau} \right)\otimes \left(\otimes_{v\in S_v, \tau\in J_v} U(\fg_{\tau})\otimes_{U(\fq_{\tau})} L_{\fm_{Q_v,\tau}}(\lambda_{\tau})\right).\] 
    Hence the irreducible subquotients of $U(\fg)\otimes_{U(\fq)}L_J(\lambda)$ are $L(ww_0\cdot \lambda)=\otimes_{\tau\in\Sigma_p}L(w_{\tau}w_{\tau,0}\cdot\lambda_{\tau})$ where $w=(w_{\tau})_{\tau\in\Sigma_p}, w_0=(w_{\tau,0})_{\tau\in \Sigma_p}$ such that if $\tau\in J_v$, then $L(w_{\tau}w_{\tau,0}\cdot\lambda_{\tau})$ is a subquotient of $U(\fg_{\tau})\otimes_{U(\fq_{\tau})} L_{\fm_{Q_v,\tau}}(\lambda_{\tau})$. In particular, the weight $w_{\tau}w_{\tau,0}\cdot\lambda_{\tau}$ is $\fm_{Q_v,\tau}$-dominant (with respect to $\fb_{Q_v,\tau}$) by \cite[Prop. 9.3(e)]{humphreys2008representations}. We fix one such $w$ as above and assume $\Hom_{G_p}\left(\cF_{\overline{B}_p}^{G_p}\left(\overline{L}(-ww_0\cdot\lambda), \underline{\delta}_{\mathrm{sm}}\delta_{B_p}^{-1}\right),\Pi_{\infty}[\fm_y]^{\mathrm{an}}\right)\neq 0$. We need prove $w=w_0$.\par 
    For each $v\in S_p$, let $P_{w,v}=M_{P_{w,v}}N_{P_{w,v}}$ be the standard parabolic subgroup of upper-triangular block matrices in $\GL_{n/F_{\widetilde{v}}}$ with the standard Levi decomposition such that $P_{w,v}$ is maximal for $w_vw_{v,0}\cdot\lambda$ (i.e the opposite $\overline{P}_{w,v}$ is the maximal parabolic subgroup such that $\overline{L}(-w_vw_{v,0}\cdot\lambda_v)\in\cO^{\overline{\fp}_{w,v}}$). We also use the same notation $P_{w,v}=M_{P_{w,v}}N_{P_{w,v}}$ for the associated $p$-adic Lie groups and let $P_{w,p}=\prod_{v\in S_p}P_{w,v}$, etc.. Any irreducible constituent of 
    \[\cF_{\overline{B}_p}^{G_p}\left(\overline{L}(-ww_0\cdot\lambda), \underline{\delta}_{\mathrm{sm}}\delta_{B_p}^{-1}\right)\simeq \cF_{\overline{P}_{w,p}}^{G_p}\left(\overline{L}(-ww_0\cdot\lambda), (\mathrm{Ind}_{\overline{B}_p\cap M_{w,p}}^{M_{w,p}}\underline{\delta}_{\mathrm{sm}}\delta_{B_p}^{-1})^{\mathrm{sm}}\right)\] 
    has the form $\cF_{\overline{P}_{w,p}}^{G_p}\left(\overline{L}(-ww_0\cdot\lambda), \pi_{M_{P_{w,p}}}\right)$ for some irreducible constituent $\pi_{M_{P_{w,p}}}$ of the smooth induction $(\mathrm{Ind}_{\overline{B}_p\cap M_{w,p}}^{M_{w,p}}\underline{\delta}_{\mathrm{sm}}\delta_{B_p}^{-1})^{\mathrm{sm}}$ (cf. \cite[Thm 2.3(ii)(iii), (2.6)]{breuil2016versI}). The central character of $\pi_{M_{P_{w,p}}}$ is $\underline{\delta}_{\mathrm{sm}}\delta_{B_p}^{-1}$. Since at least one of such irreducible constituents appears in $\Pi_{\infty}[\fm_y]$, by \cite[Cor. 3.5]{breuil2016versI}, we get that $z^{ww_0\cdot\lambda}\underline{\delta}_{\mathrm{sm}}\delta_{B_p}^{-1}(z)\in \cO_L$ for any $z$ lies $Z_{M_{P_{w,p}}}^+$ where $Z_{M_{P_{w,p}}}$ is the center of $M_{P_{w,p}}$.  Equivalently, for any $v\in S_p$, 
    \begin{equation}\label{equationinvariantnorm}
        z^{w_vw_{v,0}\cdot\lambda_v}\underline{\delta}_{v,\mathrm{sm}}\delta_{B_v}^{-1}(z)\in \cO_L
    \end{equation} 
    for any $z\in Z_{M_{P_{w,v}}}^+$.\par
    Firstly assume that there exist $v\in S_p,\tau\in \Sigma_v\setminus J_v$ such that $w_{\tau}\neq w_{\tau,0}$. Then $P_{w,v}\neq G_v$. By (the proof of) \cite[Prop. 5.4]{breuil2017interpretation}, there exists $i_v\in \left\{1,\cdots,n\right\}$ such that $\gamma_{\widetilde{v},i_v}\in Z_{M_{P_{w,v}}}^+$ and $|\gamma_{\widetilde{v},i_v}^{w_{\tau}w_{\tau,0}\cdot \lambda_{\tau}- \lambda_{\tau}}|_p\geq |\tau(\varpi_{\widetilde{v}})|_p^{-1-\mathrm{min}_{i}(\lambda_{\tau,i}-\lambda_{\tau,i+1})}$. Since for any $\tau'\in \Sigma_v,\lambda_{\tau'}+\rho_{\tau'}$ is strictly dominant and $w_{\tau'}w_{\tau',0}\cdot \lambda_{\tau'}- \lambda_{\tau'}=w_{\tau'}w_{\tau',0}(\lambda_{\tau'}+\rho_{\tau'})- (\lambda_{\tau'}+\rho_{\tau'})$, we have $|\gamma_{\widetilde{v},i_v}^{w_{\tau'}w_{\tau',0}\cdot \lambda_{\tau'}- \lambda_{\tau'}}|_p\geq 1$. Hence by (\ref{equationconditionclassical1}) we get
    \[|\gamma_{\widetilde{v},i_v}^{w_vw_{v,0}\cdot\lambda_v}\underline{\delta}_{v,\mathrm{sm}}\delta_{B_v}^{-1}(\gamma_{\widetilde{v},i_v})|_p=|\gamma_{\widetilde{v},i_v}^{w_vw_{v,0}\cdot\lambda_v-\lambda_v}\underline{\delta}_v\delta_{B_v}^{-1}(\gamma_{\widetilde{v},i_v})|_p\geq C |\tau(\varpi_{\widetilde{v}})|_p^{-1-\mathrm{min}_{i}(\lambda_{\tau,i}-\lambda_{\tau,i+1})}>1\]
    which contradicts (\ref{equationinvariantnorm}).\par
    Now we assume $w_{\tau}=w_{\tau,0}$ for every $\tau\notin J$. Then for any $v\in S_p$, $P_{w,v}\supset Q_v$. Assume $w_{\tau}\neq w_{\tau,0}$ for some $v\in S_p, \tau\in J_v$. By (the proof of) \cite[Lem. 3.18]{ding2019some}, there exists $i_v\in \left\{\widetilde{q}_{v,1},\cdots,\widetilde{q}_{v,t_v-1}\right\}$ such that $\gamma_{\widetilde{v},i_v}\in Z_{M_{P_{w,v}}}^+$ and $|\gamma_{\widetilde{v},i_v}^{w_{\tau}w_{\tau,0}\cdot \lambda_{\tau}- \lambda_{\tau}}|_p\geq |\tau(\varpi_{\widetilde{v}})|_p^{-1-\mathrm{min}_{i}(\lambda_{\tau,\widetilde{q}_{v,i}}-\lambda_{\tau,\widetilde{q}_{v,i}+1})}$. As in the previous step, by (\ref{equationconditionclassical2}), we have 
    \[|\gamma_{\widetilde{v},i_v}^{w_vw_{v,0}\cdot\lambda_v}\underline{\delta}_{v,\mathrm{sm}}\delta_{B_v}^{-1}(\gamma_{\widetilde{v},i_v})|_p=|\gamma_{\widetilde{v},i_v}^{w_vw_{v,0}\cdot\lambda_v-\lambda_v}\underline{\delta}_v\delta_{B_v}^{-1}(\gamma_{\widetilde{v},i_v})|_p\geq C |\tau(\varpi_{\widetilde{v}})|_p^{-1-\mathrm{min}_{i}(\lambda_{\tau,\widetilde{q}_{v,i}}-\lambda_{\tau,\widetilde{q}_{v,i}+1})}>1\]
    which also contradicts (\ref{equationinvariantnorm}).\par
    Therefore we conclude $w=w_0$ and the point is classical.
\end{proof}
\begin{proposition}\label{propositionclassicalpointsderham}
    If $(y,\underline{\delta})=\left((\rho_{\widetilde{v}})_{v\in S_p},z,\underline{\delta}\right)\in X_{p}(\overline{\rho})\subset \mathfrak{X}_{\overline{\rho}_p}\times (\mathfrak{X}_{\overline{\rho}^p}\times\mathbb{U}_g)\times \widehat{T}_{p,L}$ is a classical point, then $\rho_{\widetilde{v}}$ is de Rham for any $v\in S_p$.
\end{proposition}
\begin{proof}
    This is the local-global compatibility result of \cite[Lemma. 4.31]{caraiani2016patching}. Since $x$ is classical, after possibly enlarging $L$, there is an injection $L(\lambda)\otimes \pi_{\mathrm{sm}}\hookrightarrow \Pi_{\infty}[\fm_y]$, where $\fm_y$ is the corresponding maximal ideal of $R_{\infty}[\frac{1}{p}]$ and $\pi_{\mathrm{sm}}$ is some smooth representation of $G_p$. Let $\Omega$ be the Bernstein component containing $\pi_{\mathrm{sm}}$, $(J,\lambda_{\mathrm{sm}})$ be a semisimple Bushnell-Kutzko type for $\Omega$, where $J$ is compact open subgroup of $G_p$ and $\lambda_{\mathrm{sm}}$ is an irreducible smooth representation of $J$, and $\tau=(\tau_{\widetilde{v}})_{v\in S_p}$ be the corresponding inertia type (\cite[\S3.3, \S3.4]{caraiani2016patching}). After possibly enlarging $L$, we assume that $\lambda_{\mathrm{sm}}$ is defined over $L$ and then by the type theory, there is a $J$-injection $L(\lambda)\otimes \lambda_{\mathrm{sm}}\hookrightarrow L(\lambda)\otimes \pi_{\mathrm{sm}}\hookrightarrow \Pi_{\infty}[\fm_y]$ (for the rationality problem, cf. \cite[\S3.13]{caraiani2016patching}, especially (3.15) in \textit{loc. cit.}). We fix a $J$-stable $\cO_L$-lattice $\lambda^{\circ}$ of $L(\lambda)\otimes \lambda_{\mathrm{sm}}$, then $\Hom_J(\lambda^{\circ},\Pi_{\infty}[\fm_y])\neq 0$. Thus $\Hom_J(\lambda^{\circ},\Pi_{\infty}^{\circ}[\fm_y])\neq 0$ here $\fm_y$ is viewed as an ideal of $R_{\infty}$ that doesn't contain $p$ and $\Pi_{\infty}^{\circ}$ is the unit ball of $\Pi_{\infty}$. By the Schikhof duality (\cite[\S1.8]{caraiani2016patching}), we get that $\fm_y$ is in the support of $M_{\infty}(\lambda^{\circ})$ where $M_{\infty}(\lambda^{\circ}):=\Hom_{\cO_L[[J]]}^{\mathrm{cont}}(M_{\infty}, (\lambda^{\circ})')'$ is finite free over $S_{\infty}$ (cf. \cite[Lem. 4.30]{caraiani2016patching}), thus finite over $R_{\infty}$. The methods in the proof of \cite[Lem. 4.17 (1)]{caraiani2016patching} together with the classical local-global compatibility when $\ell=p$ show that the action of $R_{\infty}$ on $M_{\infty}(\lambda^{\circ})$ factors through $R_{\infty}\otimes_{\widehat{\otimes}_{v\in S_p}R_{\overline{\rho}_{\widetilde{v}}}'}\widehat{\otimes}_{v\in S_p}R_{\overline{\rho}_{\widetilde{v}}}^{\mathbf{h}_{\widetilde{v}}, \tau_{\widetilde{v}}}$ where $R_{\overline{\rho}_{\widetilde{v}}}^{\mathbf{h}_{\widetilde{v}}, \tau_{\widetilde{v}}}$ is the framed potentially semi-stable deformation ring of Kisin (\cite[Thm. 2.7.6]{kisin2008potentially}) of inertia type $\tau_{\widetilde{v}}$ and Hodge-Tate weights $\mathbf{h}_{\widetilde{v}}$ associated with $\lambda_v$. This implies that $\rho_{\widetilde{v}}$ is potentially semi-stable for any $v\in S_p$.
\end{proof}
\subsection{Partially de Rham trianguline $(\varphi,\Gamma)$-modules}\label{sectionpartiallyderhamtrianguline}
In this subsection, we use the density of de Rham points on the partial eigenvariety and the global triangulation to prove that points on the partial eigenvariety are ``partially de Rham''.\par 
Let $x=\left((\rho_{x,p},\iota (\underline{\delta}_x),z_x)\right)=\left(\left(\rho_{x,\widetilde{v}},\iota_v(\underline{\delta}_{x,v})\right)_{v\in S_p},z_x\right)\in \iota\left(\prod_{v\in S_p}X_{\mathrm{tri}}(\overline{\rho}_{\widetilde{v}})\right)\times (\mathfrak{X}_{\overline{\rho}^p}\times\mathbb{U}_g)$ be a point in $X_{p}(\overline{\rho})$ (for notation see \S\ref{sectionglobalsettings}), and assume that $\underline{\delta}_x$ is locally algebraic and $\delta_{x,v}\in \cT_{v,0}^n,\forall v\in S_p$ (Definition \ref{definitiongenericglobal}, or $\iota(\underline{\delta}_x)$ is generic). Then $\rho_{x, \widetilde{v}}$ is an almost de Rham representation of $\cG_{F_{\widetilde{v}}}$ (cf. \S\ref{sectionalmostderham}). The global triangulation (on $X_{\mathrm{tri}}(\overline{\rho}_{v})$) implies that the $(\varphi,\Gamma_{F_{\widetilde{v}}})$-module $\cM_{x,\widetilde{v}}:=D_{\mathrm{rig}}(\rho_{x, \widetilde{v}})[\frac{1}{t}]$ over $\cR_{k(x), F_{\widetilde{v}}}[\frac{1}{t}]$ is equipped with a unique triangulation of parameter $\underline{\delta}_{x,v}$ (\cite[Prop. 3.7.1]{breuil2019local}), denoted by $\{0\}=\cM_{x,\widetilde{v},0}\subsetneq \cdots \subsetneq \cM_{x,\widetilde{v},n}=\cM_{x,\widetilde{v}}$. Then for $\tau\in \Sigma_v$, $D_{\mathrm{pdR},\tau}(W_{\mathrm{dR}}\left(\cM_{x,\widetilde{v}})\right)$ is equipped with a filtration $D_{\mathrm{pdR},\tau}\left(W_{\mathrm{dR}}(\cM_{x,\widetilde{v},\bullet})\right)$ of vector spaces over $k(x)$ together with a nilpotent linear operator $\nu_{x,\tau}$ which keeps the filtration. Recall that $M_{Q_{v}}$ is the group of diagonal block matrices of the form $\GL_{q_{v,1}/F_{\widetilde{v}}}\times \cdots \times \GL_{q_{v,t_{v}}/F_{\widetilde{v}}}$. For $i\in\left\{1,\cdots, t_{v}\right\}$, we let $\nu_{x,\tau,i}$ be the action of $\nu_{x,\tau}$ on 
\[D_{\mathrm{pdR},\tau}\left(W_{\mathrm{dR}}(\cM_{x,\widetilde{v},\widetilde{q}_{v,i}})\right)/D_{\mathrm{pdR},\tau}\left(W_{\mathrm{dR}}(\cM_{x,\widetilde{v},\widetilde{q}_{v,i-1}})\right).\] 
Recall $Q_v$ is a standard parabolic subgroup of $\GL_{n/F_{\widetilde{v}}}$. We denote by $Q_{\tau}:=Q_v\otimes_{F_{\widetilde{v}},\tau}L$. Then $\nu_{x,\tau,i}=0$ for all $i=\left\{1,\cdots, t_{v}\right\}$ if and only if $\rho_{x,\widetilde{v}}$ with the filtration $\cM_{x,\widetilde{v},\bullet}$ is $Q_{\tau}$-de Rham (Definition \ref{definitionQderham}).
\begin{proposition}\label{propositionpartialvarietypartialderham}
    Let $x$ be a point as above. If $x$ is in $X_{p}(\overline{\rho})(\lambda_J')\subset X_p(\overline{\rho})$, then $\rho_{x,\widetilde{v}}$ with the filtration $\cM_{x,\widetilde{v},\bullet}$ is $Q_{\tau}$-de Rham for any $\tau\in J_v,v\in S_p$.
\end{proposition}
\begin{proof}
    We fix $i\in \left\{1,\cdots,t_v\right\}, v\in S_p,\tau\in J_v$. Consider the closed immersion $X_{p}(\overline{\rho})(\lambda_J')\hookrightarrow X_{p}(\overline{\rho})$ in \S\ref{sectionpartialeigenvarietygeometry}. Let $D_{\mathrm{rig}}(\rho_{\widetilde{v}}^{\mathrm{univ}})$ be the $(\varphi,\Gamma_{F_{\widetilde{v}}})$-module over $\cR_{X_p(\overline{\rho}),F_{\overline{v}}}$ associated with the universal Galois representation $\rho_{\widetilde{v}}^{\mathrm{univ}}$ of $\cG_{F_{\widetilde{v}}}$ (\cite[Def. 2.12]{liu2015triangulation}). By Lemma \ref{lemmaglobaltrianguline} below and \cite[Thm. 3.19]{breuil2017interpretation}, there is a birational proper morphism $f:X'\rightarrow X_p(\overline{\rho})$, a $(\varphi,\Gamma_{F_{\widetilde{v}}})$-module $M'$ over $\cR_{X',F_{\widetilde{v}}}$ such that for any $x'\in X'$ with $\delta_{x',v}\in \cT_{v,0}^n$ (we use the same notation $\delta_v$, etc. with different subscripts to denote the pull back of the character $\delta_v$, etc. from $X_{\mathrm{tri}}(\overline{\rho}_p)$), an isomorphism of $(\varphi,\Gamma_{F_{\widetilde{v}}})$-module over $\cR_{k(x'),F_{\widetilde{v}}}[\frac{1}{t}]:M'_{x'}[\frac{1}{t}]\simeq \cM_{x',\widetilde{v},\widetilde{q}_{v,i}}/\cM_{x',\widetilde{v},\widetilde{q}_{v,i-1}}$ where $\cM_{x',\widetilde{v},\bullet}$ denotes the unique filtration of $D_{\mathrm{rig}}(\rho_{x',\widetilde{v}})[\frac{1}{t}]$ of parameter $\delta_{x',v,1},\cdots,\delta_{x',v,n}$. We know $X'$ is reduced. Since the Sen polynomial varies analytically and the set of strictly trianguline points is Zariski dense, we get that the $\tau$-Sen polynomial of $M'_{x'}$ is equal to $\prod_{j=\widetilde{q}_{v,i-1}+1}^{\widetilde{q}_{v,i}}(T-\mathrm{wt}_{\tau}(\delta_{x',v,j}))$ for any $x'\in X'$. Let $X''$ be the preimage of $X_p(\overline{\rho})(\lambda_J')$ under $f$. Then $f|_{X''}$ is still proper (\cite[\S9.6.2]{bosch1984non}). Let $M_{X''}'$ be the pullback of $M'$ to $X''$ and let $M_{X''}'(\delta_{X'',v,\widetilde{q}_v}^{-1}):=M_{X''}'\otimes_{\cR_{X'',F_{\widetilde{v}}}}\cR_{X'',F_{\widetilde{v}}}(\delta_{X'',v,\widetilde{q}_v}^{-1})$. Thus for any $x'\in X''$, the $\tau$-Sen weights of $M_{x'}'(\delta_{x',v,\widetilde{q}_v}^{-1})$: 
    \begin{align*}
        &\left(\wt_{\tau}(\delta_{x',v,\widetilde{q}_{v,i-1}+1})-\wt_{\tau}(\delta_{x',v,\widetilde{q}_{v}}),\cdots, \wt_{\tau}(\delta_{x',v,\widetilde{q}_{v,i-1}+q_{v,i}})-\wt_{\tau}(\delta_{x',v,\widetilde{q}_{v}})\right)\\
        =&\left((\widetilde{\lambda}_{\tau,\widetilde{q}_{v,i-1}+1}-\widetilde{q}_{v,i-1}) -(\widetilde{\lambda}_{\tau,\widetilde{q}_{v}}-\widetilde{q}_{v,i}+1),\cdots, 0\right)
    \end{align*}
    are all certain fixed integers (see \S\ref{sectionpartialeigenvarietygeometry}). Applying Proposition \ref{propositionfamilydpdr} in Appendix \ref{sectionfamiliesofalmostdeRhamrepresetnations}, we conclude that the subset of points $x'\in X''$ such that $M_{x'}'[\frac{1}{t}](\delta_{x',v,\widetilde{q}_v}^{-1})$ is $\tau$-de Rham, the points such that the nilpotent operator vanishes on $D_{\mathrm{pdR},\tau}\left(W_{\mathrm{dR}}(M_{x'}'(\delta_{x',v,\widetilde{q}_v}^{-1})[\frac{1}{t}])\right)$, is Zariski closed in $X''$. We denote this subset by $Y$. Then $f(Y)$ is an analytic closed subset of $X_{p}(\overline{\rho})(\lambda_J')$ (\cite[Prop. 9.6.3/3]{bosch1984non}).\par
    We pick an affinoid $U$ of $X_{p}(\overline{\rho})(\lambda_J')$ containing $x$ as in Proposition \ref{propositiondenseclassicalpoints} so that the classical points is Zariski dense in $U$. By shrinking $U$ and its image in $\mathcal{W}_{\lambda_J'}$ suitably, we can assume that for any point $x'\in U$, $\delta_{x',v}\in \mathcal{T}_{v,0}^n$ (this is possible since $\mathcal{T}_{v,0}^n$ is Zariski open in the space of characters of $(F_{\widetilde{v}}^{\times})^n$). Suppose that $x'$ is a point in $U$ such that $\delta_{x',v}$ is locally algebraic and $x''\in f^{-1}(x')$. We have $\cM_{x',\widetilde{v},\widetilde{q}_{v,i}}/\cM_{x',\widetilde{v},\widetilde{q}_{v,i-1}}\otimes_{k(x')}k(x'')\simeq \cM_{x'',\widetilde{v},\widetilde{q}_{v,i}}/\cM_{x'',\widetilde{v},\widetilde{q}_{v,i-1}}$. Since $W_{\mathrm{dR}}\left(\cR_{k(x'),F_{\widetilde{v}}}(\delta_{x',v,\widetilde{q}_v})[\frac{1}{t}]\right)$ is trivial (\cite[Lem. 3.3.7]{breuil2019local}) and the functor $W_{\mathrm{dR}}(-)$ is tensor functorial, we get that $\cM_{x',\widetilde{v},\widetilde{q}_{v,i}}/\cM_{x',\widetilde{v},\widetilde{q}_{v,i-1}}$ is $\tau$-de Rham if and only if $\left(\cM_{x',\widetilde{v},\widetilde{q}_{v,i}}/\cM_{x',\widetilde{v},\widetilde{q}_{v,i-1}}\right)(\delta_{x',v,\widetilde{q}_v}^{-1})$ is $\tau$-de Rham if and only if $x'\in f(Y)$. Then by Proposition \ref{propositionclassicalpointsderham}, $f(Y)\cap U$ contains a Zariski dense subset of $U$ (classical points in $U$). Furthermore, $f(Y)\cap U$ is Zariski closed in $U$ (\cite[Prop. 9.5.3/2]{bosch1984non}). Thus $U\subset f(Y)$. Hence $\cM_{x,\widetilde{v},\widetilde{q}_{v,i}}/\cM_{x,\widetilde{v},\widetilde{q}_{v,i-1}}$ is $\tau$-de Rham.
\end{proof}

\begin{lemma}\label{lemmaglobaltrianguline}
    Let $X$ be a reduced analytic rigid space over $L$ and $M$ is a $(\varphi,\Gamma_K)$-module over $\cR_{X,K}$ of rank $n$ where $K$ is a local field over $\Q_p$. We assume that there exists a Zariski dense subset $X_{\mathrm{alg}}$ of $X$ such that $M$ is densely pointwise strictly trianguline (\cite[Def. 6.3.2]{kedlaya2014cohomology}) with respect to a parameter $\delta_{X,1},\cdots,\delta_{X,n}: K^{\times}\rightarrow \Gamma(X,\cO_X)^{\times}$ and the subset $X_{\mathrm{alg}}$. We assume furthermore that if $x\in X_{\mathrm{alg}}$, then $\underline{\delta}_x \in \cT_{\mathrm{reg}}^n$ (see \S\ref{sectiontriangulinevariety}). Then for any $0\leq a< b\leq n$, there exists a birational proper map $f:X'\rightarrow X$ and a $(\varphi,\Gamma_K)$-module $M'$ over $\cR_{X',K}$ such that, let $\delta_{X',1},\cdots,\delta_{X',n}$ be the pull back of characters $\delta_{X,1},\cdots,\delta_{X,n}$, the following statements hold.\par
    (1) The set of points $x\in X'$ such that $M'_x\simeq \mathrm{fil}_{b}((f^*M)_x)/\mathrm{fil}_{a}((f^*M)_x)$, where $\mathrm{fil}_{\bullet}((f^*M)_x)$ is the unique strictly trianguline filtration on $(f^*M)_x$ of parameter $\delta_{x,1},\cdots,\delta_{x,n}$, contains $f^{-1}(X_{\mathrm{alg}})$ and is Zariski open dense in $X'$. \par
    (2) Suppose $x\in X'$ such that $\underline{\delta}_x\in \cT_{\mathrm{0}}^n$. Assume that $A$ is a finite-dimensional local $L$-algebra with residue field $k(x)$ and a map $\Sp(A)\rightarrow X'$ with image $x$, then the pull back $(f^*M)_A[\frac{1}{t}]$ is trianguline with a unique triangulation $\mathrm{fil}_{\bullet}((f^*M)_A[\frac{1}{t}])$ of parameter $\delta_{A,1},\cdots,\delta_{A,n}$ in the sense of \S\ref{sectionalmostderham}. Moreover there is an isomorphism $M'_A[\frac{1}{t}]\simeq \mathrm{fil}_{b}((f^*M[\frac{1}{t}])_A)/\mathrm{fil}_{a}((f^*M[\frac{1}{t}])_A)$ and $M'_A[\frac{1}{t}]$ has a parameter $\delta_{A,a+1},\cdots,\delta_{A,b}$.
\end{lemma}
\begin{proof}
    By \cite[Cor. 6.3.10]{kedlaya2014cohomology}, after replacing $X$ (resp. $X_{\mathrm{alg}}$) by some $X'$ (resp. $f^{-1}(X_{\mathrm{alg}})$), we may assume that $M$ admits a filtration $\mathrm{fil}_{\bullet}M$ of $(\varphi,\Gamma_K)$-modules over $\cR_{X,K}$ satisfying the requirement (1) and (2) in \textit{loc. cit.}. The uniqueness of the triangulation of given parameter in (2) is by \cite[Lem. 3.4.3 \& Prop. 3.4.6]{breuil2019local}.\par
    If $a=0$, we can take the submodule $M':=\mathrm{fil}_{b}M$ of $M$. Then (1) of the Lemma is satisfied by the choice. The existence of a triangulation of $M_A[\frac{1}{t}]$ in (2) follows from (2) in \cite[Cor. 6.3.10]{kedlaya2014cohomology} which is of parameter $\delta_{A,1},\cdots,\delta_{A,n}$ and it is by our choice that $M'_A[\frac{1}{t}]\simeq \mathrm{fil}_{b}(M_A[\frac{1}{t}])$.\par
    By replacing $M$ with $\mathrm{fil}_{b}M$, we may assume $b=n$. We need roughly pick a ``quotient'' $(\varphi,\Gamma_K)$-modules of $M$ (on some $X'$) rather than a submodule as in the previous step. For a $(\varphi,\Gamma_K)$-module $M_Y$ over $\cR_{Y,K}$ for a rigid space $Y$, we let $M_Y^{\vee}:=\Hom_{\cR_{Y,K}}(M,\cR_{Y,K})$ be the dual $(\varphi,\Gamma_K)$-module of $M$. The dual functor $(\cdot)\mapsto (\cdot)^{\vee}$ of $(\varphi,\Gamma_K)$-modules commutes with base change, is exact on short exact sequences of $(\varphi,\Gamma_K)$-modules and sends $\cR_{Y,K}(\delta)$ to $\cR_{Y,K}(\delta^{-1})$ for any continuous character $\delta:K^{\times}\rightarrow \Gamma(Y,\cO_Y)^{\times}$ (cf. \cite[Con. 6.2.4]{kedlaya2014cohomology}). Then the $(\varphi,\Gamma_K)$-module $M^{\vee}$ over $\cR_{X,K}$ is densely pointwise strictly trianguline with respect to parameters $\delta_{X,n}^{-1},\cdots,\delta_{X,1}^{-1}$ by the assumption and \cite[Prop. 6.2.8]{kedlaya2014cohomology}. By \cite[Cor. 6.3.10]{kedlaya2014cohomology}, after replacing $X$ (resp. $X_{\mathrm{alg}}$) by some $X'$ (resp. $f^{-1}(X_{\mathrm{alg}})$), we may assume that $M^{\vee}$ admits a filtration $\mathrm{fil}_{\bullet}M^{\vee}$ of $(\varphi,\Gamma_K)$-modules over $\cR_{X,K}$ such that for any point $x$ in a Zariski open dense subset $Z$ containing $X_{\mathrm{alg}}$, $(\mathrm{fil}_{\bullet}M^{\vee})_x$ is a strictly triangulation of $M^{\vee}_x$ with parameters $\delta_{x,n}^{-1},\cdots,\delta_{x,1}^{-1}$ and each $(\mathrm{fil}_{i}M^{\vee}/\mathrm{fil}_{i-1}M^{\vee})[\frac{1}{t}]$ is isomorphic to $\cR_{X,K}(\delta_{n-i+1}^{-1})[\frac{1}{t}]$ up to a line bundle. We set $M'= (\mathrm{fil}_{n-a}M^{\vee})^{\vee}$. Then for $x\in Z$, $M'_x\simeq M_x/\mathrm{fil}_{a}M_x$ is trianguline of parameter $\delta_{x,a+1}\cdots,\delta_{x,n}$. For $x\in X$ and $A$ satisfying the condition in (2) of the lemma, we get by the construction that $(\mathrm{fil}_{n-a}M^{\vee})_A[\frac{1}{t}]$ is trianguline of parameter $\delta_{A,n}^{-1},\cdots,\delta_{A,a+1}^{-1}$ and $M^{\vee}_A[\frac{1}{t}]$ is trianguline of parameter $\delta_{A,n}^{-1},\cdots,\delta_{A,1}^{-1}$. Taking dual, we get that $M_A[\frac{1}{t}]$ is trianguline of parameters $\delta_{A,1},\cdots,\delta_{A,n}$ and $M'_A[\frac{1}{t}]$ is trianguline of parameters $\delta_{A,a+1},\cdots,\delta_{A,n}$. 
\end{proof}
Then we immediately get the main result of this section.
\begin{theorem}\label{theoremQderham}
    Let $x=\left(\left(\rho_p,\underline{\delta}\right),z\right)\in X_{p}(\overline{\rho})\subset \iota\left(X_{\mathrm{tri}}\left(\overline{\rho}_{p}\right)\right)\times (\mathfrak{X}_{\overline{\rho}^p}\times\mathbb{U}_g)$ be a point such that $\underline{\delta}$ is locally algebraic and generic. Let $\lambda=\wt(\underline{\delta})$ and $\underline{\delta}_{\mathrm{sm}}$ be the smooth part of $\underline{\delta}$. Let $y$ be the image of $x$ in $\fX_{\infty}$. If for some $v\in S_p,\tau\in\Sigma_v$, $\lambda_{\tau}$ is $\fm_{Q_v,\tau}$-dominant (with respect to $\fb_{Q_v,\tau}$) and we have
    \[\Hom_{G_p}\left(\mathcal{F}_{\overline{B}_p}^{G_p}\left(\overline{L}(-\lambda),\underline{\delta}_{\mathrm{sm}}\delta_{B_p}^{-1}\right),\Pi_{\infty}^{\mathrm{an}}[\fm_{y}]\otimes_{k(y)}k(x)\right)\neq 0,\]
    then $\rho_{\widetilde{v}}$ with the unique triangulation on $D_{\mathrm{rig}}(\rho_{\widetilde{v}})[\frac{1}{t}]$ of parameter $\delta_{v,1},\cdots,\delta_{v,n}$ is $Q_{\tau}$-de Rham.
\end{theorem}
\begin{proof}
    We take $J=\{\tau\}\subset \Sigma_p$. The irreducible $U(\fg)$-module $L(\lambda)$ of the highest weight $\lambda$ is the unique quotient of $U(\fg)\otimes_{U(\fq)}L_J(\lambda)$ (for the notation, see \S\ref{sectionadjunctionformula}) by \cite[\S9.4]{humphreys2008representations}. Thus by the functoriality of $\mathcal{F}_{\overline{B}_p}^{G_p}(-,-)$, we get 
    \[\Hom_{G_p}\left(\cF_{\overline{B}_p}^{G_p}\left(\Hom\left(U(\fg)\otimes_{U(\fq)}L_J(\lambda),L\right)^{\overline{\mathfrak{u}}^{\infty}},\underline{\delta}_{\mathrm{sm}}\delta_{B_p}^{-1}\right),\Pi_{\infty}[\fm_y]^{\mathrm{an}}\otimes_{k(y)}k(x)\right)\neq 0.\]
    Hence by Proposition \ref{propositionpointspartialeigenvariety}, $x\in X_{p}(\overline{\rho})(\lambda_J')$. By Proposition \ref{propositionpartialvarietypartialderham}, $\rho_{\widetilde{v}}$ with the triangulation is $Q_{\tau}$-de Rham.
\end{proof}
\subsection{Conjectures on partial classicality and locally analytic socle}\label{sectionconjectures}
We state a conjecture on partial classicality of almost de Rham Galois representations and discuss its relationship with the locally analytic socle conjecture. We only state the conjecture for the patched eigenvariety where the local model is available since in this special case the conjecture is more accessible and the converse of the conjecture is known to some extent (Theorem \ref{theoremQderham}). We also give some partial results.\par
We use the notation before Theorem \ref{theoremcyclepartialderham}. Let $x=\left((\rho_p,\underline{\delta}),z\right)\in X_p(\overline{\rho})(L)\subset\iota\left(X_{\mathrm{tri}}(\overline{\rho}_p)\right)\times (\mathfrak{X}_{\overline{\rho}^p}\times\mathbb{U}^g)$ be a generic point with integral weights. Let $\mathbf{h}$ be the Hodge-Tate weights of $\rho_p$ and $\lambda$ be the weight of $\underline{\delta}$. We assume that $\lambda+\rho$ is dominant (with respect to $\fb$). There are the companion points $x_{w}=\left((\rho_p,\underline{\delta}_w),z\right)\in (\fX_{\overline{\rho}_p}\times\widehat{T}_{p,L})\times (\mathfrak{X}_{\overline{\rho}^p}\times\mathbb{U}^g)$ for $w\in W_{G_p}$ where $\underline{\delta}_w$ is defined in the end of \S\ref{sectionOrlikStrauch}. Let $r_x$ be the image of $x$ in $\mathfrak{X}_{\infty}$ and $\fm_{r_x}$ be the corresponding maximal ideal of $R_{\infty}[\frac{1}{p}]$. For $v\in S_p$, as $\iota_{v}^{-1}(\underline{\delta}_{v})\in \cT_{v,0}^n$, $\cM_{\widetilde{v}}:=D_{\mathrm{rig}}(\rho_{\widetilde{v}})[\frac{1}{t}]$ is trianguline with a unique triangulation $\cM_{\widetilde{v},\bullet}$ of parameter $\iota^{-1}_v(\underline{\delta}_{v})$. Let $Q_p=\prod Q_v$ be a standard parabolic subgroup of $G_p$ as in \S\ref{sectionthepartialeigenvarietynotation} and let $Q_{\tau}$ be the base change to $L$ of the standard parabolic subgroup of $\GL_{n/F_{\widetilde{v}}}$ via $\tau:F_{\widetilde{v}}\rightarrow L$ for all $\tau\in\Sigma_v,v\in S_p$. Recall that $J_{Q_p}(\Pi_{\infty}[\fm_{r_x}]^{\mathrm{an}})$ is a locally analytic representation of $M_{Q_p}$. Take a non-empty subset $J\subset \Sigma_p$ and let $J_v=J\cap \Sigma_v$ for all $v\in S_p$. Following \cite[\S6.1]{ding2017formes}, a vector $v\in J_{Q_p}(\Pi_{\infty}[\fm_{r_x}]^{\mathrm{an}})$ is called $J$-classical if there is a finite-dimensional algebraic $\fm_{Q_p,J}$-module $V$ and a $\fm_{Q_p,J}$-equivariant map $V\rightarrow J_{Q_p}(\Pi_{\infty}[\fm_{r_x}]^{\mathrm{an}})$ such that the image of the map contains $v$.\par
We say that the Hodge-Tate weights $\mathbf{h}$ are regular if $h_{\tau,i}\neq h_{\tau,j}$ for all $\tau\in \Sigma_p,i\neq j$. We now state the conjecture on partial classicality.
\begin{conjecture}\label{conjecturepartialclassicality}
	Assume that $\mathbf{h}$ is regular and the pair $(\rho_{\widetilde{v}}, \cM_{\widetilde{v},\bullet})$ as above is $Q_{\tau}$-de Rham for all $\tau\in J_v,v\in S_p$ (Definition \ref{definitionQderham}), then $J_{Q_p}(\Pi_{\infty}[\fm_{r_x}]^{\mathrm{an}})$ contains non-zero $J$-classical vectors.
\end{conjecture}
\begin{remark}
	The assumption that $\mathbf{h}$ is regular is necessary: if $\rho_p$ is de Rham with non-regular Hodge-Tate weights, $Q_p=G_p$ and $J=\Sigma_p$, then there exists no non-zero locally algebraic vector in $\Pi_{\infty}[\fm_{r_x}]^{\mathrm{an}}$ by the local-global compatibility or \cite{dospinescu2020infinitesimal}.
\end{remark}
Next, we formulate a weak version of the locally analytic socle conjecture for the point $x$. Recall that $\Spf(R_{\rho_p,\cM_{\bullet}}) $ is isomorphic to $\widehat{X}_{P_p, x_{\pdR}}$ up to formally smooth morphisms where $X_{P_p}$ is the variety defined in \S\ref{sectionunibranchness} with respect to the standard parabolic subgroup $P_p$ determined by $\mathbf{h}$ or $\lambda$ of the algebraic group $\prod_{v\in S_p}\mathrm{Res}_{F_{\widetilde{v}}/\Q_p}(\GL_{n/F_{\widetilde{v}}})\times_{\Q_p}L$ as in \S\ref{sectionOrlikStrauch} and $x_{\pdR}=(x_{\pdR,\widetilde{v}})_{v\in S_p}$ is a point on $X_{P_p}$ associated with $(\rho_{p},\cM_{\bullet})=(\rho_{\widetilde{v}},\cM_{\widetilde{v},\bullet})_{v\in S_p}$ as in \S\ref{sectiontriangulinevariety}. By \S\ref{sectioncycles}, $\Spec(\overline{R}_{\rho_p,\cM_{\bullet}})\subset \Spec(R_{\rho_p})$ is a union of cycles of the form $\mathfrak{Z}_{w}$ for $w\in W_{G_p}/W_{P_p}$. Moreover, $\mathfrak{Z}_{w}\neq \emptyset$ if and only if $x_{\mathrm{pdR}}\in Z_{P_p,w}$. Let $w_x\in W_{G_p}/W_{P_p}$ be the element as before parameterizing the relative position between the Hodge filtration and the trianguline filtration. Then by discussions in \S\ref{sectionsteinbergvarieties}, $\mathfrak{Z}_{w}\neq \emptyset$ only if $w\geq w_x$ in $W_{G_p}/W_{P_p}$. In general, there exists $w\geq w_x$ such that $\mathfrak{Z}_{w}=\emptyset$ (for example, if $\mathbf{h}$ is regular and $\rho_p$ is not de Rham, then $\mathfrak{Z}_{w_0}=\emptyset$).
\begin{conjecture}\label{conjecturesocles}
	If $\mathfrak{Z}_{w}\neq \emptyset$, then
	\[\Hom_{G_p}\left(\mathcal{F}_{\overline{B}_p}^{G_p}\left(\overline{L}(-ww_0\cdot\lambda),\underline{\delta}_{\mathrm{sm}}\delta_{B_p}^{-1}\right),\Pi_{\infty}[\fm_{r_x}]^{\mathrm{an}}\right)\neq 0.\]
\end{conjecture}
\begin{remark}\label{remarkweakconjecture}
	This is only a weak form of the locally analytic socle conjecture as we explain below. \par
	For simplicity, we assume that $\mathbf{h}$ is regular and $z$ is in the smooth locus of $\fX_{\overline{\rho}^p}\times \mathbb{U}^g$. As in the proof of \cite[Thm. 5.3.3]{breuil2019local} or Proposition \ref{propositionmaincycle}, the dual of
	\[\Hom_{U(\fg)}\left(L(ww_0\cdot\lambda),\Pi_{\infty}^{\mathrm{an}}\right)^{U_0}[\fm_{r_x}^{\infty}][\fm_{\underline{\delta}_{\mathrm{sm}}}^{\infty}]\]
	is a module over $\widehat{\cO}_{\mathfrak{X}_{\infty},r_x}$ via the map $\Spec(\widehat{\cO}_{X_p(\overline{\rho})_{ww_0\cdot \lambda},x_w})\hookrightarrow\Spec(\widehat{\cO}_{\mathfrak{X}_{\infty},r_x})$ (remark that $\Spec(\widehat{\cO}_{X_p(\overline{\rho})_{ww_0\cdot \lambda},x_w})$ can be empty), denoted by $\mathcal{L}(ww_0\cdot \lambda)$ in \cite[(5.22)]{breuil2019local}. Abusing the notation, let $[\mathcal{L}(ww_0\cdot \lambda)]$ be the set-theoretic support of $\mathcal{L}(ww_0\cdot \lambda)$ in $\Spec(\widehat{\cO}_{\mathfrak{X}_{\infty},r_x})$. For $w'\in W_{G_p}$, we let $\mathfrak{Z}_{w'}'\subset \Spec(\widehat{\cO}_{\mathfrak{X}_{\infty},r_x})$ be the preimage of the closed subset $\mathfrak{Z}_{w'}$ in $\Spec(R_{\rho_{p}})$ via the formally smooth morphism $\Spec(\widehat{\cO}_{\mathfrak{X}_{\infty},r_x})\rightarrow \Spec(R_{\rho_p})$. Now let $\mathfrak{C}_{w'}$ be the union of $\mathfrak{Z}_{w''}',w''\in W_{G_p}$ such that $a_{w',w''}\geq 1$ where $a_{w',w''}$ is the coefficient appeared in \cite[(2.16)]{breuil2019local} (with respect to the algebraic group $\prod_{v\in S_p}\mathrm{Res}_{F_{\widetilde{v}}/\Q_p}(\GL_{n/F_{\widetilde{v}}})\times_{\Q_p}L$) so that $\mathfrak{C}_{w'}$ is the underlying set of the Kazhdan-Lusztig cycle appeared in \cite[(5.24)]{breuil2019local} with the same notation. Then $\mathfrak{Z}_{w'}'\subset \mathfrak{C}_{w'}$ and $\mathfrak{Z}_{w''}'\subset \mathfrak{C}_{w'}$ only if $w'\geq w''$ in $W_{G_p}$ (\cite[Thm. 2.4.7(iii)]{breuil2019local}). \par
	In the spirit of the analogue between the subsets (or even real cycles) $[\mathcal{L}(ww_0\cdot \lambda)]$ and $\mathfrak{C}_{w}$, Conjecture \ref{conjecturesocles} expects $\mathfrak{Z}_w'\subset [\mathcal{L}(ww_0\cdot \lambda)]$. In general, it is possible that there exists some $w'>w$ such that $\mathfrak{Z}_{w}'\subset \mathfrak{C}_{w'}$ (\cite[Rem. 2.4.5]{breuil2019local}). Hence we also expect that $\mathfrak{Z}_{w}'\subset [\mathcal{L}(w'w_0\cdot \lambda)]$ and thus if $\mathfrak{Z}_{w}'\neq \emptyset$, we expect that
	\[\Hom_{G_p}\left(\mathcal{F}_{\overline{B}_p}^{G_p}\left(\overline{L}(-w'w_0\cdot\lambda),\underline{\delta}_{\mathrm{sm}}\delta_{B_p}^{-1}\right),\Pi_{\infty}[\fm_{r_x}]^{\mathrm{an}}\right)\neq 0.\]
	But now it may happen that $\mathfrak{Z}_{w'}'=\emptyset$ (this will not happen if $\rho_p$ is de Rham, see (\ref{formulairreduciblecomponents1})). Hence Conjecture \ref{conjecturesocles} should not predict all the possible companion constituents in general.\par
	However, Theorem \ref{theoremcyclepartialderham} imposes some restriction for elements $w'\in W_{G_p}$ such that $\mathfrak{Z}_{w}'\subset [\mathcal{L}(w'w_0\cdot \lambda)]$ just as the restriction for the characteristic cycles in Proposition \ref{propositioncharacteristiccycle}.
\end{remark}
\begin{remark}
    The weak conjecture on locally analytic socles (Conjecture \ref{conjecturesocles}) still implies the existence of the companion points on the eigenvariety: since we have $\mathfrak{Z}_{w_x}\neq \emptyset$ by definition, we get $x_{w_x}\in X_p(\overline{\rho})$ if Conjecture \ref{conjecturesocles} is true, which implies that $x_{w}\in X_p(\overline{\rho})$ for all $w\geq w_x$ in $W_{G_p}/W_{P_p}$ by \cite[Thm. 5.5]{breuil2017smoothness}. 
\end{remark}
\begin{proposition}\label{propositiontwoconjecture}
	Conjecture \ref{conjecturesocles} implies Conjecture \ref{conjecturepartialclassicality}.
\end{proposition}
\begin{proof}
	Now $\mathbf{h}$ is regular. Let $B_p=\prod_{\tau\in\Sigma_p}B_{\tau}$ be the standard Borel subgroup of upper-triangular matrices in $\prod_{v\in S_p}\mathrm{Res}_{F_{\widetilde{v}}/\Q_p}(\GL_{n/F_{\widetilde{v}}})\times_{\Q_p}L$. Since $(\rho_{\widetilde{v}},\cM_{\widetilde{v},\bullet})$ is $Q_{\mathrm{\tau}}$-de Rham for $\tau\in J_v,v\in S_p$, by definition, we have $x_{\mathrm{pdR}}\in Z_{\widetilde{Q}_{p},P_p}$ where $\widetilde{Q}_{p}:=\prod_{\tau\in J}Q_{\tau}\prod_{\tau\notin J}B_{\tau}$    
    and $Z_{\widetilde{Q}_{p},P_p}$ is defined in \S\ref{sectionsteinbergvarieties}. By Corollary \ref{corollarysteinbergvarietypartialclassicality}, there exists $w\in W_{G_p}$ such that $x_{\mathrm{pdR}}\in Z_{B_p,w}$ and $Z_{B_p,w}\subset Z_{\widetilde{Q}_{p},B_p}$. Take one such $w$, then $\mathfrak{Z}_w\neq \emptyset$. By Theorem \ref{theoremvarietyweight}, $ww_0\cdot \lambda$ is a dominant weight for the Lie algebra of the standard Levi subgroup of $\widetilde{Q}_{p}$. In particular, $L(ww_0\cdot \lambda)$ is the irreducible quotient of $U(\fg)\otimes_{U(\fq)}L_{J}(ww_0\cdot \lambda)$ where $L_J(ww_0\cdot \lambda)$ is defined via (\ref{equationdefinitionLJ}). The statement of Conjecture \ref{conjecturesocles} and the exactness of the functor $\cF_{\overline{B}_p}^{G_p}$ imply that
	\[\Hom_{G_p}\left(\cF_{\overline{B}_p}^{G_p}\left(\Hom\left(U(\fg)\otimes_{U(\fq)}L_J(ww_0\cdot \lambda),L\right)^{\overline{\mathfrak{u}}^{\infty}},\underline{\delta}_{\mathrm{sm}}\delta_{B_p}^{-1}\right), \Pi_{\infty}[\fm_{r_x}]^{\mathrm{an}}\right)\neq 0.\]
	By the beginning part of the proof of Proposition \ref{propositionadjunctionJacquetmodule}, we get 
	\[\Hom_{U(\fm_{Q_p})}\left(L_J(ww_0\cdot \lambda),J_{Q_p}(\Pi_{\infty}[\fm_{r_x}]^{\mathrm{an}})\right)\neq 0.\]
	Then the non-zero image of any non-zero $U(\fm_{Q_p})$-equivariant map $L_J(ww_0\cdot \lambda)\rightarrow J_{Q_p}(\Pi_{\infty}[\fm_{r_x}]^{\mathrm{an}})$ gives rise to non-zero $J$-classical vectors since $L_{M_{Q_p}}(\lambda_J)$ in (\ref{equationdefinitionLJ}) is a finite-dimensional representation of $\fm_{Q_p,J}$.
\end{proof}
\begin{corollary}
	Conjecture \ref{conjecturepartialclassicality} is true if $\rho_{\widetilde{v}}$ is crystalline for every $v\in S_p$.
\end{corollary}
\begin{proof}
	By \cite[Thm. 5.3.3]{breuil2019local}, Conjecture \ref{conjecturesocles} is true for crystalline points. Hence Conjecture \ref{conjecturepartialclassicality} is true for such points by Proposition \ref{propositiontwoconjecture}.
\end{proof}
\begin{remark}
    The method in the proof of Proposition \ref{propositionmaincycle} based on Theorem \ref{theoremcyclepartialderham} can be used to give non-trivial evidence for Conjecture \ref{conjecturesocles} beyond de Rham cases. To convince the reader, we sketch an example here. Let's keep the notation and assumptions in Conjecture \ref{conjecturesocles} and Remark \ref{remarkweakconjecture} and assume $w_x=e$ is the identity in $W_{G_p}$. Hence $\mathfrak{Z}_{e}'\neq \emptyset$. By the (dual of) exact sequences as \cite[(5.21)]{breuil2019local} writing $M(w_0w_0\cdot\lambda)$ as a successive extension of subquotients $L(ww_0\cdot\lambda),w\in W_{G_p}$, we get an equality of topological spaces inside $\Spec(\widehat{\cO}_{\mathfrak{X}_{\infty},r_x})$:
    \[\Spec(\widehat{\cO}_{X_p(\overline{\rho})_{w_0w_0\cdot\lambda},x_{w_0}})=\mathrm{Supp}(\cM_{\infty}\otimes_{\cO_{X_p(\overline{\rho})}}\widehat{\cO}_{X_p(\overline{\rho})_{w_0w_0\cdot\lambda},x_{w_0}})=\bigcup_{w\in W_{G_p}}[\mathcal{L}(ww_0\cdot \lambda)].\]
    On the other hand, by the theory of the local model (\ref{formulairreduciblecomponents}), we have 
    \[\Spec(\widehat{\cO}_{X_p(\overline{\rho})_{w_0w_0\cdot\lambda},x_{w_0}})=\bigcup_{w\in W_{G_p}}\mathfrak{Z}_{w}'\]
    Hence we get
    \[\bigcup_{w\in W_{G_p}}[\mathcal{L}(ww_0\cdot \lambda)]=\bigcup_{w\in W_{G_p}}\mathfrak{Z}_{w}'.\]
    A priori, in the above equality, we know only $\mathfrak{Z}_{e}'\neq \emptyset$. As $\mathbf{h}$ is antidominant, we see that any irreducible component of $\mathfrak{Z}_{e}'$ is not contained in $\Spec(R_{\rho_{p},\cM_{\bullet}}^{\widetilde{Q}_{p}})$ for any standard parabolic subgroup $\widetilde{Q}_{p}$ of $\prod_{v\in S_p}\mathrm{Res}_{F_{\widetilde{v}}/\Q_p}(\GL_{n/F_{\widetilde{v}}})\times_{\Q_p}L$ unless $\widetilde{Q}_{p}$ is the Borel subgroup by Theorem \ref{theoremvarietyweight}. But for any $w\neq e$, the cycle $[\cL(ww_0\cdot \lambda)]$ is contained in $\Spec(R_{\rho_{p},\cM_{\bullet}}^{\widetilde{Q}_{p}})$ for some non-Borel $\widetilde{Q}_{p}$ by Theorem \ref{theoremcyclepartialderham}. Hence any component of $\mathfrak{Z}_{e}'$ is not contained in $[\mathcal{L}(ww_0\cdot \lambda)]$ for any $w\neq e$. Since $\mathfrak{Z}_{e}'$ is equidimensional with the same dimension as $\Spec(\widehat{\cO}_{X_p(\overline{\rho})_{w_0w_0\cdot\lambda},x_{w_0}})$, we must have $\mathfrak{Z}_{e}'\subset [\mathcal{L}(w_0\cdot \lambda)]$. Hence $[\mathcal{L}(w_0\cdot \lambda)]\neq \emptyset$ which implies that \[\Hom_{G_p}\left(\mathcal{F}_{\overline{B}_p}^{G_p}\left(\overline{L}(-w_0\cdot\lambda),\underline{\delta}_{\mathrm{sm}}\delta_{B_p}^{-1}\right),\Pi_{\infty}[\fm_{r_x}]^{\mathrm{an}}\right)\neq 0.\]
	Then $x_e\in X_p(\overline{\rho})$. By \cite[Thm. 5.5]{breuil2017smoothness}, $x_w\in X_p(\overline{\rho})$ for all $w\in W_{G_p}$.
\end{remark}
\begin{remark}\label{remarksmoothcompanionpoints}
	Theorem \ref{theoremmaincrystalline} only concerns the existence of all companion constituents that can be seen via the local models. In general, there exist companion points $x'=\left((\rho_p,\underline{\delta}'),z\right)\in X_p(\overline{\rho})\subset\iota\left(X_{\mathrm{tri}}(\overline{\rho}_p)\right)\times (\mathfrak{X}_{\overline{\rho}^p}\times\mathbb{U}^g)$ of $x$ such that $\underline{\delta}^{-1}\underline{\delta}'$ is not an algebraic character. In the situation of Theorem \ref{theoremmaincrystalline}, $x'$ is one of points $\left((\rho_p,\underline{\delta}_{\cR',w'}),z\right)$ for some other choice of refinements $\cR'$. The existence of all companion points of all refinements on $X_p(\overline{\rho})$ for regular crystalline points is obtained by the existence of locally algebraic constituents of the form $L(\lambda)\otimes (\mathrm{Ind}_{\overline{B}_p}^{G_p}\underline{\delta}_{\mathrm{sm}}\delta_{B_p}^{-1})^{\mathrm{sm}}$ in $\Pi_{\infty}[\fm_{r_x}]^{\mathrm{an}}$ and the intertwining between smooth principal series (see the beginning of \cite[\S5.3]{breuil2019local} for details). In the non-regular crystalline cases, there exist no such locally algebraic constituents in the socle of $\Pi_{\infty}[\fm_{r_x}]^{\mathrm{an}}$. Still, in some special cases, one can obtain the existence of some companion points of other refinements using the locally algebraic representations of $M_{Q_p}$ that appear in $J_{Q_p}(\Pi_{\infty}[\fm_{r_x}]^{\mathrm{an}})$ (cf. \cite[Prop. 2.14]{ding2019some}).
\end{remark}
\appendix
\section{Families of almost de Rham $(\varphi,\Gamma_K)$-modules}\label{sectionfamiliesofalmostdeRhamrepresetnations}
We show that the method of Berger-Colmez in \cite{berger2008familles} to construct de Rham families for Galois representations can be generalized to construct partially almost de Rham families for $(\varphi,\Gamma_K)$-modules.
\subsection{Preliminary}
Let $K$ be a finite extension of $\Q_p$ with a uniformizer $\varpi_K$, $\overline{K}$ be an algebraic closure of $K$ and $C$ be the completion of $\overline{K}$. Let $K_0$ be the maximal unramified (over $\Q_p$) subfield of $K$. Let $K_{\infty}$ be the extension of $K$ by adding all $p$-th power roots of unity, $K_m=K(\mu_{p^m})$ for $m\geq 1$. We set $\Gamma_K:=\Gal(K_{\infty}/K), \Gamma_{K_m}:=\Gal(K_{\infty}/K_m)$ and $H_K:=\Gal(\overline{K}/K_{\infty})$. Let $\epsilon:\Gamma_K\rightarrow \Z_p^{\times}$ be the cyclotomic character.\par
Let $A$ be an affinoid algebra over $\Q_p$. We keep the notation of the Robba rings and $(\varphi,\Gamma_K)$-modules in \S\ref{sectionintroductionnotation}. Recall we have Robba rings $\cR_{A,K}^{r}$ which is denoted by $\cR_{A}^{r}(\pi_K)$ in \cite[Def. 2.2.2]{kedlaya2014cohomology} for any $r>0$ less than certain constant. \par
Recall a $\varphi$-module over $\cR^{r}_{A,K}$ is a finite projective module $M^{r}$ over $\cR_{A,K}^{r}$ equipped with an isomorphism $\varphi^*M^{r}\simrightarrow M^{r/p}=M^{r}\otimes_{\cR_{A,K}^{r}}\cR_{A,K}^{r/p}$ where $\varphi: \cR_{A,K}^{r}\rightarrow \cR_{A,K}^{r/p}$ and a $\varphi$-module over $\cR_{A,K}$ is the base change of a $\varphi$-module over some $\cR_{A,K}^{r_0}$ (\cite[Def. 2.2.6]{kedlaya2014cohomology}). A $(\varphi,\Gamma_K)$-module over $\cR_{A,K}$ is the base change of a $\varphi$-module equipped with a commuting semi-linear continuous action of $\Gamma_K$ over some $\cR_{A,K}^{r_0}$ (\cite[Def. 2.2.12]{kedlaya2014cohomology}). \par
We recall some constructions from $(\varphi,\Gamma_K)$-modules.\par
Suppose that $D_A$ is a $(\varphi,\Gamma)$-module over $\cR_{A,K}$ of rank $n$. By definition, $D_A$ is the base change of a $(\varphi,\Gamma)$-module $D^{r_0}_A$ over $\cR_{A,K}^{r_0}=\cR_{\Q_p,K}^{r_0}\widehat{\otimes}_{\Q_p}A$ (\cite[Prop. 1.1.29]{emerton2017locally}) for some $r_0$ small enough. For any $r$, let $m(r)$ be the minimal integer such that $p^{m(r)-1}[K(\mu_{\infty}):K_0(\mu_\infty)]r\geq 1$. Then there exists a continuous $\Gamma_K$-equivariant injection $\iota_{m(r)}:\cR_{\Q_p,K}^r\hookrightarrow K_{m(r)}[[t]]$ (e.g. \cite[\S1.2]{berger2008construction}). As in \cite[\S4.3]{berger2008familles}, we define $A\widehat{\otimes}_{\Q_p}K_{m}[[t]]:=\varprojlim_k A\widehat{\otimes}_{\Q_p}(K_m[[t]]/t^k)$ for any $m\geq m(r_0)$ and we will always assume $m\geq m(r_0)$ in the remaining part of this appendix. \par
Let $D_{\mathrm{dif}}^{K_m,+}(D_A):=A\widehat{\otimes}K_m[[t]]\otimes_{\iota_m,\cR_{A,K}^{r}}D_A^{r}$ be the ``localization'' if $m=m(r)$ and $r\leq r_0$ where $D^r_A=D^{r_0}_A\otimes_{\cR_{A,K}^{r_0}}\cR_{A,K}^{r}$. Then $D_{\mathrm{dif}}^{K_m,+}(D_A)$ is a finite projective $A\widehat{\otimes}K_m[[t]]$-module with a semi-linear continuous action of $\Gamma_K$. Here the continuity means that for each $s\geq 1$ the action of $\Gamma_K$ on $D_{\mathrm{dif}}^{K_m,+}(D_A)/t^sD_{\mathrm{dif}}^{K_m,+}(D_A)$ is continuous. In particular, for any $x\in D_{\mathrm{dif}}^{K_m,+}(D_A)/t^sD_{\mathrm{dif}}^{K_m,+}(D_A)$, $\lim_{\gamma\in \Gamma_K,\gamma\to 1} \gamma(x)=x$ in the finite $A$-Banach module 
\[D_{\mathrm{dif}}^{K_m,+}(D_A)/t^sD_{\mathrm{dif}}^{K_m,+}(D_A).\] 
Write $D_{\mathrm{dif}}^{K_m}(D_A):=D_{\mathrm{dif}}^{K_m,+}(D_A)[\frac{1}{t}]$.\par
Let $D_{\mathrm{Sen}}^{K_m}(D_A):=D_{\mathrm{dif}}^{K_m,+}(D_A)/tD_{\mathrm{dif}}^{K_m,+}(D_A)$. Then $\varphi$ induces $\Gamma_K$-equivariant isomorphisms  (cf. \cite[\S2.3]{liu2015triangulation})
\begin{align*}
    D_{\mathrm{Sen}}^{K_m}(D_A)\otimes_{K_m}K_{m+1}&\simrightarrow D_{\mathrm{Sen}}^{K_{m+1}}(D_A),\\
    D_{\mathrm{dif}}^{K_m,+}(D_A)\otimes_{K_m[[t]]}K_{m+1}[[t]]&\simrightarrow D_{\mathrm{dif}}^{+,K_{m+1}}(D_A).
\end{align*} 
If $B$ is an affinoid algebra with a morphism $A\rightarrow B$, denote by $D_B:= M\widehat{\otimes}_AB$ the $(\varphi,\Gamma_K)$-module over $\cR_{B,K}$. Then by the construction, there is a natural isomorphism $D_{\mathrm{dif}}^{K_m,+}(D_B)\simeq D_{\mathrm{dif}}^{K_m,+}(D_A)\widehat{\otimes}_{A}B$. If $x$ is a point on $\mathrm{Sp}(A)$, we let $D_x=D_A\otimes_Ak(x)$ be the $(\varphi,\Gamma_K)$-module over $\cR_{k(x),K}$. The following lemma is not essential.
\begin{lemma}\label{lemmalocallyfreeddiff}
    There exists a finite admissible covering $\{\mathrm{Sp}(A_i)\}$ of $\mathrm{Sp}(A)$ such that each $D_{\mathrm{dif}}^{K_m,+}(D_{A_i})$ (resp. $D_{\mathrm{Sen}}^{K_m}(D_{A_i})$) is a free $A_i\widehat{\otimes}_{\Q_p}K_m[[t]]$-module (resp. $A_i\otimes_{\Q_p}K_m$-module).
\end{lemma}
\begin{proof}
    We have $A\widehat{\otimes}_{\Q_p}K_m[[t]]=\varprojlim_k A\widehat{\otimes}_{\Q_p}(K_m[[t]]/t^k)=\varprojlim_k (A\widehat{\otimes}_{\Q_p}K_m)[[t]]/t^k=K_m\otimes_{\Q_p} A[[t]]$ (cf. \cite[Prop. 1.2.5]{emerton2017locally}) is $t$-adically complete. Thus the finite projective module $D_{\mathrm{dif}}^{K_m,+}(D_A)\widehat{\otimes}_AA_i$ over $A_i\widehat{\otimes}_{\Q_p}K_m[[t]]$ is free if and only if $D_{\mathrm{Sen}}^{K_m}(D_A)\widehat{\otimes}_AA_i$ over $A_i\otimes_{\Q_p}K_m$ is free. Therefore we only need work for $D_{\mathrm{Sen}}^{K_m}(D_A)$. \par
    Let $\fm$ be a maximal ideal of $A$ corresponding to a point $x\in\mathrm{Sp}(A)$. Since $D_x^r$ is free over $\cR_{k(x),K}^r$ (\cite[Prop. 1.1.1]{berger2008construction}) , $D_{\mathrm{Sen}}^{K_m}(D_x)$ is free over $k(x)\otimes_{\Q_p} K_m$. As $D_{\mathrm{Sen}}^{K_m}(D_A)\otimes_AA_{\fm}$ is finite projective over $A_{\fm}\otimes_{\Q_p}K_m$ and $\fm\otimes_{\Q_p} K_m$ is contained in the Jacobson radical of $A_{\fm}\otimes_{\Q_p}K_m$, $D_{\mathrm{Sen}}^{K_m}(D_A)\otimes_AA_{\fm}$ is free over $A_{\fm}\otimes_{\Q_p} K_m$ by Nakayama's lemma. We choose a morphism $(A\otimes K_m)^{n}\rightarrow D_{\mathrm{Sen}}^{K_m}(D_A)$ of $A\otimes_{\Q_p}K_m$-modules (after possibly shrinking $\Spec(A)$) such that it induces an isomorphism $(A_{\fm}\otimes_{\Q_p} K_m)^{n}\rightarrow D_{\mathrm{Sen}}^{K_m}(D_A)\otimes_AA_{\fm}$ of $A_{\fm}\otimes_{\Q_p}K_m$-modules after tensoring $A_{\fm}$. Let $C_1$ (resp. $C_2$) be the kernel (resp. cokernel) of this morphism. Then $C_1\otimes_AA_{\fm}=C_2\otimes_AA_{\fm}=0$. Since $A$ is Noetherian, we can pick a finite set $\{x_i\}$ of generators of $C_1$ and $C_2$ over $A$. For any $x_i$. there exists an element $a_i\in A$ such that $a_ix_i=0$ and $a_i\notin \fm$. Let $A_S$ be the localization of $A$ by inverting all $a_i$. Then $C_1\otimes_{A}A_{S}=C_2\otimes_{A}A_{S}=0$. Thus the morphism $(A_{S}\otimes_{\Q_p}K_m)^n\rightarrow D_{\mathrm{Sen}}^{K_m}(D_A)\otimes_{A}A_{S}$ of $A_S\otimes_{\Q_p} K_m$-modules is an isomorphism. Now since Zariski open subsets are admissible open and every Zariski covering is admissible (\cite[\S5.1 Cor. 9]{bosch2014lectures}) and $\Spec(A)$ is quasi-compact, we can find a finite covering of $\mathrm{Sp}(A)$ by affinoid subdomains $\mathrm{Sp}(A_i)$ such that $D_{\mathrm{Sen}}^{K_m}(D_A)\otimes_AA_i$ is free. 
\end{proof}
Thus from now on we assume that $D_{\mathrm{Sen}}^{K_{m(r_0)}}(D_A)$ is a free $A\otimes_{\Q_p}K_{m(r_0)}$-module. Since the action of $\Gamma_K$ on $D_{\mathrm{Sen}}^{K_{m}}(D_A)$ is continuous, there exists $m_1>m(r_0)$ such that $||M_{\gamma}-\mathrm{I}||<1$ for any $\gamma\in\Gamma_{K_{m_1}}$ where $M_{\gamma}$ denotes the matrix of $\gamma$ on $D_{\mathrm{Sen}}^{K_{m(r_0)}}(D_A)$ with coefficients in $A\otimes_{\Q_p} K_{m(r_0)}$ with respect to some basis and the operator norm $||-||$ of operators on the finite Banach $A$-module $D_{\mathrm{Sen}}^{K_{m(r_0)}}(D_A)$ is equivalent to the matrix norm given by norms of coefficients (cf. \cite[Prop. 2.2.14]{kedlaya2014cohomology}, \cite[Prop. 1.2.4]{emerton2017locally}). Then for any $\gamma\in \Gamma_{K_{m_1}}$, the Sen operator $\nabla:\log(\gamma)/\log(\epsilon(\gamma))=-\frac{1}{\log(\epsilon(\gamma))}\sum_{i=1}^{\infty}\frac{(1-\gamma)^i}{i}$ on $D_{\mathrm{Sen}}^{K_{m_1}}(D_A)$ converges and acts $A\otimes_{\Q_p}K_{m_1}$-linearly on $D_{\mathrm{Sen}}^{K_{m_1}}(D_A)$. The characteristic polynomial of the Sen operator lies in $(A\otimes_{\Q_p}K_{m_1}[T])^{\Gamma_K}=A\otimes_{\Q_p}K[T]$ (cf. \cite[Prop. 2.5]{fontaine2004arithmetique} or \cite[Def. 6.2.11]{kedlaya2014cohomology}). We take $m_0\geq m_1$ such that for any $\gamma\in \Gamma_{K_{m_0}}, ||\log(\epsilon(\gamma))\nabla||<||p||^{\frac{1}{p-1}}$. Then for any $m>m_0,\gamma\in\Gamma_{K_{m}}$, the action of the $A\otimes_{\Q_p}K_m$-linear operator $\gamma$ on $D_{\mathrm{Sen}}^{K_{m}}(D_A)=K_{m}\otimes_{K_{m(r_0)}}D_{\mathrm{Sen}}^{K_{m(r_0)}}(D_A)$ can be recovered by $\gamma=\exp\left(\log(\epsilon(\gamma))\nabla\right)=\sum_{i=0}^{\infty}\frac{\left(\log(\epsilon(\gamma))\nabla\right)^i}{i!}$. Later, we will need to take $m>m_0$ to deduce the final result on $K$. We remark that the same formula for $\gamma\in\Gamma_{K_m}$ may not hold on $D_{\text{dif}}^{K_m,+}(D_A)$ in general.\par
From now on, we assume that $A$ is moreover an $L$-algebra where $L$ is a finite extension of $\Q_p$ that splits $K$. We denote by $\Sigma$ the set of all embeddings of $K$ in $L$. For $\tau\in\Sigma$, we set $D_{\mathrm{Sen},\tau}^{K_m}(D_A):=D_{\mathrm{Sen}}^{K_m}(D_A)\otimes_{A\otimes_{\Q_p}K, 1\otimes \tau} A$ (resp. $D_{\mathrm{dif},\tau}^{K_m,+}(D_A):=D_{\mathrm{dif}}^{K_m,+}(D_A)\otimes_{A\otimes_{\Q_p}K, 1\otimes \tau} A$) which is a free $A\otimes_KK_m$-module (resp. $A\widehat{\otimes}_KK_m[[t]]$-module) and is stable under the action of $\Gamma_K$ since $\Gamma_K$ commutes with $A\otimes_{\Q_p}K$.
\subsection{Almost de Rham representations}\label{sectionalmostderhamfamilyfieldcase}
We start with the cases when $A$ is a finite $L$-algebra. We will give certain characteristic properties of almost de Rham $\BdR$-representations on the level of $D_{\mathrm{dif}}^{K_m,+}(D_A)$ (Proposition \ref{propositiondifferentialequation}) as is done for de Rham representations in \cite[Prop. 5.2.1]{berger2008familles}. \par
Recall that the ring $B_{\mathrm{pdR}}=\BdR[\log(t)]$ is equipped with a $\BdR$-derivation $\nu_{\mathrm{pdR}}(\log(t))=-1$. We pick a topological generator $\gamma_K\in\Gamma_K$ (and $\gamma_{K_m}\in\Gamma_{K_m}$). Then $\gamma_K(\log(t))=\log(t)+\log(\epsilon(\gamma_K))$. Let $\nabla$ be the operator $\log(\gamma_K^{p^s})/\log(\epsilon(\gamma_K^{p^s}))$ for $s$ large enough acting on $D_{\mathrm{dif}}^{K_m,+}(D_A)$ in the sense of \cite[Prop. 3.7]{fontaine2004arithmetique}. Then $\nabla$ acts $K_m$-linearly, $\nabla(t^a)=at^{a}$ for any $a\in \Z$ and $\nabla (ax)=a\nabla(x)+\nabla(a)x$ for any $a\in K_m((t))$. Formally let $\nabla (\log(t)^a)=a\log(t)^{a-1},\forall a\in\N$ as in \cite[\S3.8]{fontaine2004arithmetique}. The following lemma will be useful.
\begin{lemma}\label{lemmagammanilpotent}
    Let $B\in\{K_m((t)), K((t))\}$ and $W$ be a continuous semi-linear $B$-representation of $\Gamma_K$. Then $(W\otimes_B B[\log(t)])^{\nabla=0}$ can be identified with $W^{\nabla-\mathrm{nil}}$, the space of elements $x\in W$ such that $\nabla$ acts nilpotently on $x$. The identification restricts to a bijection between $(W\otimes_B B[\log(t)])^{\Gamma_K}$ and $W^{(\gamma_K-1)-\mathrm{nil}}$, the space of elements $x\in W$ such that $\gamma_K-1$ acts nilpotently on $x$. Moreover, under the identification, the action of the $B$-derivation $\nu_{\mathrm{pdR}}$ on $(W\otimes_B B[\log(t)])^{\Gamma_K}$ is the action of $\nabla$ on $W^{(\gamma_K-1)-\mathrm{nil}}$.
\end{lemma}
\begin{proof}
    Assume $\sum_{i=0}^Na_i\log(t)^i\in (W\otimes_B B[\log(t)])^{\nabla=0}$ where $a_i\in W$. Then 
    \begin{align*}
        0&=\nabla\left(\sum_{i=0}^Na_i\log(t)^i\right)\\
        &=\nabla(a_0)+\sum_{i=1}^{N}\left(\nabla(a_i)\log(t)^i+ia_i\log(t)^{i-1}\right)\\
        &=\nabla(a_N)\log(t)^N+\sum_{i=0}^{N-1}\left((i+1)a_{i+1}+\nabla(a_i)\right)\log(t)^i.    
    \end{align*} 
    Hence $a_i=(-1)^{i}\frac{1}{i!}\nabla^i(a_0)$ and $\nabla^{N+1}(a_0)=0$. The map $(W\otimes_B B[\log(t)])^{\nabla=0}\rightarrow W^{\nabla-\mathrm{nil}}:\sum_{i=0}^Na_i\log(t)^i \mapsto a_0$ is then a bijection with the inverse map $a\mapsto \sum_{i=0}^{N}\frac{(-1)^{i}}{i!}\nabla^i(a_0)\log(t)^i$ if $\nabla^{N+1}(a)=0$. \par
    Now assume $\sum_{i=0}^Na_i\log(t)^i\in (W\otimes_B B[\log(t)])^{\gamma_K=1}$ where $a_i\in W$ and $a_N\neq 0$. Then 
    \begin{align*}
        \sum_{i=0}^Na_i\log(t)^i&=\sum_{i=0}^{N}\gamma_K(a_i)\left(\log(t)+\log(\epsilon(\gamma_K))\right)^{i}\\
        &=\sum_{i=0}^N\left(\sum_{j=i}^N\binom{j}{i}\gamma_K(a_j)\log(\epsilon(\gamma_K))^{j-i}\right)\log(t)^i.   
    \end{align*}
    Thus $a_i=\sum_{j=i}^N\binom{j}{i}\log(\epsilon(\gamma_K))^{j-i}\gamma_K(a_j)$ for each $i$. We get $a_N=\gamma_K(a_N)$. We now prove $(\gamma_K-1)^{i}a_{N-i+1}=0$ by descending  induction. Assume this is true for all $i>i_0$. Then 
    \begin{align*}
        (\gamma_K-1)a_{i_0}=\gamma_K(a_{i_0})-a_{i_0}=&\gamma_K(a_{i_0})-\sum_{j=i_0}^N\binom{j}{i_0}\log(\epsilon(\gamma_K))^{j-i_0}\gamma_K(a_j)\\
        =&-\sum_{j=i_0+1}^N\binom{j}{i_0}\log(\epsilon(\gamma_K))^{j-i_0}\gamma_K(a_j).
    \end{align*} 
    By the induction hypothesis, $(\gamma_K-1)^{N-i_0}\gamma_K(a_j)=\gamma_K(\gamma_K-1)^{N-i_0}(a_j)=0$ for any $j>i_0$. Thus $(\gamma_K-1)^{N-i_0+1}a_{i_0}=0$ which finishes the induction step. Hence $\gamma_K-1$ acts nilpotently on $a_{0}$. \par
    Conversely, assume we are given $a_0$ such that $(\gamma_K-1)^{N+1}a_0=0$. Then $\nabla^{N+1}(a_0)=0$ since $\nabla=\log(\gamma_K^{p^s})/\log(\epsilon(\gamma_K^{p^s}))=\log(\gamma_K)/\log(\epsilon(\gamma_K))=-\sum_{i=1}^{N}\frac{(1-\gamma_K)^i}{i\log(\epsilon(\gamma_K))}$ on $a_0$ where $s$ is any large integer (since $\nabla$ is defined using \cite[Prop. 3.7]{fontaine2004arithmetique}, the equality should be verified modulo $t^s$ for all $s$ and notice that $N$ is independent of $s$). We now verify that $x=\sum_{i=0}^{N}\frac{(-1)^{i}}{i!}\nabla^i(a_0)\log(t)^i$ is fixed by $ \gamma_K$. Since $\gamma_K$ commutes with $\nabla$, $\gamma_K-1$ and $\nabla$ act nilpotently on each $\nabla^i(a_0)$ as well as $\log(t)^i$ and $\nabla^i(a_0)\log(t)^i$ (since $(\gamma_K-1)(a\log(t)^i)=(\gamma_K-1)(a)\log(t)^i+\sum_{j=0}^{i-1}\binom{i}{j}\log(\epsilon(\gamma_K))^{i-j}\gamma_K(a)\log(t)^j$). Thus $\gamma_K(x)=\exp(\log(\epsilon(\gamma_K))\nabla)x=\left(1+\sum_{i=1}^{\infty}\frac{\log(\epsilon(\gamma_K))^i\nabla^i}{i!}\right)x=x$ where the first two identities come from the formal calculation using $\nabla=\log(\gamma_K)/\log(\epsilon(\gamma_K))$ on $x$ and the last identity comes from that $\nabla(x)=0$ by $\nabla^{N+1}(a_0)=0$ and the computation in first step.\par
    The last assertion follows from that 
    \[\nu_{\mathrm{pdR}}\left(\sum_{i=0}^{N}\frac{(-1)^{i}}{i!}\nabla^i(x)\log(t)^i\right)=\sum_{i=0}^{N-1}\frac{(-1)^i}{i!}\nabla^{i}(\nabla(x))\log(t)^i\]
    if $\nabla^{N+1}(x)=0.$
\end{proof}
The following proposition generalizes \cite[Prop. 5.2.1]{berger2008familles} using essentially the same technique in \textit{loc. cit.} and will be applied for $m>m_0$.
\begin{proposition}\label{propositiondifferentialequation}
    Assume that $D_{\mathrm{dif}}^{K_m,+}$ is a continuous semilinear $K_m[[t]]$-representation of $\Gamma_{K_m}$ of rank $n$ and $D_{\mathrm{dif}}^{K_m}=D_{\mathrm{dif}}^{K_m,+}[\frac{1}{t}]$. Assume that each eigenvalue of $\gamma_{K_m}$ on $D_{\mathrm{Sen}}^{K_m}:=D_{\mathrm{dif}}^{K_m,+}/tD_{\mathrm{dif}}^{K_m,+}$ is equal to $\epsilon(\gamma_{K_m})^c$ for some integer $c\in [a,b]$ where $a,b\in\Z$. Then the following statements hold.
    \begin{enumerate}
        \item There exists a $\Gamma_{K_m}$-invariant $K_m[[t]]$-lattice $N$ inside $D_{\mathrm{dif}}^{K_m}$ such that $(\gamma_{K_m}-1)^lN\subset tN$ for some $l\geq 1$. And we have $t^{-a}D_{\mathrm{dif}}^{K_m,+} \subset N\subset t^{-b}D_{\mathrm{dif}}^{K_m,+}$.
        \item $D_{\mathrm{pdR}}^{K_m}:=(D_{\mathrm{dif}}^{K_m}\otimes_{K_m}K_m[\log(t)])^{\Gamma_{K_m}}$ is an $n$-dimensional $K_m$-space. 
    \end{enumerate}
    Moreover, the number $l$ can be taken to be $(b-a+1)n$.
\end{proposition}
\begin{proof}
We firstly prove that (1) implies (2). Consider operators $\alpha_k:=\prod_{i=1}^kf_i(\gamma_{K_m})$ for $k\geq 1$ where we take polynomials $f_i(T)=(T-1)^lh_i(T)+1=(\epsilon(\gamma_{K_m})^{-i}T-1)^lg_i(T)$ for $i\geq 1$ thanks to Lemma \ref{lemmapolynomialfk} below. Since $(\gamma_{K_m}-1)^l$ is trivial on $N/tN$, $\alpha_k$ is the identity on $N/tN$. We set $\alpha_0(x)=x$. We now prove by induction that $(\gamma_{K_m}-1)^l\alpha_k(x)\in t^{k+1}N$ and $\alpha_{k}(x)-\alpha_{k+1}(x)\in t^{k+1}N$ for any $x\in N$ and $k\geq 0$. If $k=0$, the result is easy. Assume that the result is true for $k$. We have $(\frac{\gamma_{K_m}}{\epsilon(\gamma_{K_m})^{k+1}}-1)(t^{k+1}x)=t^{k+1}\gamma_{K_m}(x)-t^{k+1}x=t^{k+1}(\gamma_{K_m}-1)(x)$ for any $x\in D_{\mathrm{dif}}^{K_m}$.
Then $(\gamma_{K_m}-1)^{l}\alpha_{k+1}(x)=f_{k+1}(\gamma_{K_m})(\gamma_{K_m}-1)^{l}\alpha_k(x)=g_{k+1}(\gamma_{K_m})(\frac{\gamma_{K_m}}{\epsilon(\gamma_{K_m})^{k+1}}-1)^{l}(t^{k+1}\frac{(\gamma_{K_m}-1)^{l}\alpha_k(x)}{t^{k+1}})=g_{k+1}(\gamma_{K_m})t^{k+1}(\gamma_{K_m}-1)^{l}(\frac{(\gamma_{K_m}-1)^{l}\alpha_k(x)}{t^{k+1}})\in t^{k+2}N$ since $\gamma_{K_m}(t^{k+2}N)\subset t^{k+2}N$ and we have $(\gamma_{K_m}-1)^{l}(\frac{(\gamma_{K_m}-1)^{l}\alpha_k(x)}{t^{k+1}})\in tN$ by the induction hypothesis and the assumption. Then $\alpha_{k+1}(x)-\alpha_{k+2}(x)=h_{k+2}(\gamma_{K_m})(\gamma_{K_m}-1)^l\alpha_{k+1}(x)\in t^{k+2}N$ which finishes the induction. Hence when $k\to +\infty$, the sequence $\alpha_k(x)$ converges to an element in $N$ denoted by $\alpha(x)$. Moreover $ (\gamma_{K_m}-1)^l\alpha(x)=0$ and the $K_m$-map $\alpha:N\rightarrow N$ is the identity modulo $t$. Hence the image of $\alpha$, denoted by $N_{\mathrm{pdR}}$, has $K_m$-dimension at least $\dim N/tN=n$. Applying Lemma \ref{lemmagammanilpotent} for $\Gamma_{K_m}$, the space $N_{\mathrm{pdR}}$ can be viewed as a subspace of $D_{\mathrm{pdR}}^{K_m}$. Assume that $N'_{\mathrm{pdR}}$ is a $\Gamma_{K_m}$-invariant finite-dimensional $K_m$-subspace of $(D_{\mathrm{dif}}^{K_m})^{(\gamma_{K_m}-1)-\mathrm{nil}}$ that is stable under $\Gamma_{K_m}$ and contains $N_{\mathrm{pdR}}$. Then the $K_m[[t]]$-lattice $N'$ generated by $N'_{\mathrm{pdR}}$ has full rank $n$, stable under the action of $\Gamma_{K_m}$ and satisfies $(\gamma_{K_m}-1)^{l'}N'\subset tN'$ for some $l'\geq 1$. Repeat the previous argument of $N,l$ for $N',l'$, we find there exists a $K_m$-map $\alpha':N'\rightarrow N'$ which is the identity over $N'/tN'$ and $N_{\textrm{pdR}}'$. One can also prove that $\alpha'(tN')=0$ by showing that $\alpha'_{k}(tN')\subset t^{k+1}N'$ for any $k$ (since $f_{k+1}(\gamma_{K_m})(t^{k+1}N')=g_{k+1}(\gamma_{K_m})(\frac{\gamma_{K_m}}{\epsilon(\gamma_{K_m})^{k+1}}-1)^{l'}(t^{k+1}N')\subset t^{k+1}(\gamma_{K_m}-1)^{l'}N'\subset t^{k+2}N'$). Thus $N_{\textrm{pdR}}'=\alpha'(N_{\textrm{pdR}}')=\alpha'(N'/tN')$ has $K_m$-dimension $n$. Hence the dimension of $D_{\mathrm{pdR}}^{K_m}$ is no more than $n$ (this also follows from general theory, cf. \cite[Thm. 3.22(ii) \& Prop. 2.1]{fontaine2004arithmetique}). Thus $D_{\mathrm{pdR}}^{K_m}$ is identified with $N_{\mathrm{pdR}}$ via Lemma \ref{lemmagammanilpotent}. This finishes the proof of (2).\par
To prove (1), we have the following claim.
\begin{claim}\label{claimcharacteristicpolynomials}
    Under our assumption, there exists $l$ such that $\prod_{i=a}^{2b-a}(\gamma_{K_m}-\epsilon(\gamma_{K_m})^i)^{l}D_{\mathrm{dif}}^{K_m,+}\subset t^{b-a+1}D_{\mathrm{dif}}^{K_m,+}$. 
\end{claim}
\begin{proof}[Proof of Claim \ref{claimcharacteristicpolynomials}]
    Let $F_k(T)=\prod_{i=a+k}^{b+k}(T-\epsilon(\gamma_{K_m})^i)^n$. The characteristic polynomial of $\gamma_{K_m}$ on the $K_m$-space $t^kD_{\mathrm{dif}}^{K_m,+}/t^{k+1}D_{\mathrm{dif}}^{K_m,+}=t^kD_{\mathrm{Sen}}^{K_m,+}$ divides $F_k(T)$ for any $k\in\N$ by the assumption. In particular, $F_k(\gamma_{K_m})t^kD_{\mathrm{dif}}^{K_m,+}\subset t^{k+1}D_{\mathrm{dif}}^{K_m,+}$. Hence $F_{b-a}(\gamma_{K_m})\cdots F_0(\gamma_{K_m})D_{\mathrm{dif}}^{K_m,+}\subset t^{b-a+1}D_{\mathrm{dif}}^{K_m,+}$. We take $l=(b-a+1)n$. Then $F_{b-a}(T)\cdots F_0(T)$ divides $\prod_{i=a}^{2b-a}(T-\epsilon(\gamma_{K_m})^i)^{l}$. Hence $\prod_{i=a}^{2b-a}(\gamma_{K_m}-\epsilon(\gamma_{K_m})^i)^{l}D_{\mathrm{dif}}^{K_m,+}\subset t^{b-a+1}D_{\mathrm{dif}}^{K_m,+}$.
\end{proof}
Now we prove (1). We let $N$ be the sub-$K_m[[t]]$-lattice generated by $ t^{-k}\prod_{i\in [a,2b-a]\setminus\{k\}}(\gamma_{K_m}-\epsilon(\gamma_{K_m})^i)^{l}x$ for all $x\in D_{\mathrm{dif}}^{K_m,+}$ and $k\in[a,b]$. Then $N\subset t^{-b}D_{\mathrm{dif}}^{K_m,+}$ by definition and $\gamma_{K_m}(N)\subset N$. We have $t^aN\supset \{\prod_{i\in [a,2b-a]\setminus\{k\}}(\gamma_{K_m}-\epsilon(\gamma_{K_m})^i)^{l}x\mid x\in D_{\mathrm{dif}}^{K_m,+},k\in [a,b]\}$. The image of $\{\prod_{i\in [a,2b-a]\setminus\{k\}}(\gamma_{K_m}-\epsilon(\gamma_{K_m})^i)^{l}x\mid x\in D_{\mathrm{dif}}^{K_m,+}, k\in [a,b]\}$ in $D_{\mathrm{Sen}}^{K_m}$ is equal to $\{\prod_{i\in [a,2b-a]\setminus\{k\}}(\gamma_{K_m}-\epsilon(\gamma_{K_m})^i)^{l}x \mid x\in D_{\mathrm{Sen}}^{K_m},k\in [a,b] \}$. By our assumption the characteristic polynomial of $\gamma_{K_m}$ on $D_{\mathrm{Sen}}^{K_m}$ is of the form $\prod_{i\in [a,b]} (x-\epsilon(\gamma_{K_m})^i)^{n_i}$ for some $n_i\geq 0$. The polynomials $\prod_{i\in [a,b]\setminus \{k\}} (T-\epsilon(\gamma_{K_m})^i)^l, k\in [a,b]$ share no common zero and generate the constant polynomial $1$ by Hilbert's Nullstellensatz. Hence $\{\prod_{i\in [a,b]\setminus\{k\}}(\gamma_{K_m}-\epsilon(\gamma_{K_m})^i)^{l}x \mid x\in D_{\mathrm{Sen}}^{K_m},k\in [a,b] \}$ generates $D_{\mathrm{Sen}}^{K_m}$ as a $K_m$-space. Moreover, we have assumed that $(T-\epsilon(\gamma_{K_m})^i)$ is prime to the characteristic polynomial of $\gamma_{K_m}$ on $D_{\mathrm{Sen}}^{K_m}$ if $i\notin [a,b]$. Thus $\prod_{i\in [b+1,2b-a]} (\gamma_{K_m}-\epsilon(\gamma_{K_m})^i)^{l}$ is a bijection on $D_{\mathrm{Sen}}^{K_m}$. Hence the image of $\{\prod_{i\in [a,2b-a]\setminus\{k\}}(\gamma_{K_m}-\epsilon(\gamma_{K_m})^i)^{l}x\mid x\in D_{\mathrm{dif}}^{K_m,+},k\in [a,b]\}$ in $D_{\mathrm{Sen}}^{K_m}$ generates $D_{\mathrm{Sen}}^{K_m}$ which implies that $t^{a}N\supset D_{\mathrm{dif}}^{K_m,+}$ by Nakayama's lemma. Finally, 
\begin{align*}
    (\gamma_{K_m}-1)^{l}t^{-k}\prod_{i\in [a,2b-a]\setminus\{k\}}(\gamma_{K_m}-\epsilon(\gamma_{K_m})^i)^{l}D_{\mathrm{dif}}^{K_m,+}
    &=t^{-k}\prod_{i\in [a,2b-a]}(\gamma_{K_m}-\epsilon(\gamma_{K_m})^i)^{l}D_{\mathrm{dif}}^{K_m,+}\\
    &\subset t^{b-a+1-k}D_{\mathrm{dif}}^{K_m,+} \quad \text{(Claim \ref{claimcharacteristicpolynomials})}\\
    &\subset t^{b+1-k}N\\
    &\subset tN 
\end{align*}
for any $k\in [a,b]$. Since $(\gamma_{K_m}-1)^{l}tN\subset tN$ and $N/tN$ is generated under $K_m$ by the image of $t^{-k}\prod_{i\in [a,2b-a]\setminus\{k\}}(\gamma_{K_m}-\epsilon(\gamma_{K_m})^i)^{l}D_{\mathrm{dif}}^{K_m,+},k\in [a,b]$, we conclude that $(\gamma_{K_m}-1)^{l}N\subset tN$.
\end{proof}
\begin{lemma}\label{lemmapolynomialfk}
    For any $k\in\Z\setminus\{0\}$, there exists a polynomial $f_k(T)\in K[T]$ such that $f_k(T)=(T-1)^lh_k(T)+1$ and $f_k(T)=(\epsilon(\gamma_{K_m})^{-k}T-1)^lg_k(T)$ for some polynomials $h_k(T),g_k(T)\in K[T]$. 
\end{lemma}
\begin{proof}
    The ideal $((T-1)^l,(\epsilon(\gamma_{K_m})^{-k}T-1)^l)=(1)$ by Hilbert's Nullstellensatz since $\epsilon(\gamma_{K_m})^k\neq 1$ for $k\neq 0$. Hence we can find $h_k(T),g_k(T)\in K[T]$ such that $(\epsilon(\gamma_{K_m})^{-k}T-1)^lg_k(T)-(T-1)^lh_k(T)=1$.
\end{proof}
\begin{lemma}\label{lemmadpdRgammaKinvariant}
    Assume that $A$ is a finite extension of $L$ and all the Sen weights (the roots of the Sen polynomial) of $D_A$ are in $\Z$ and $m>m_0$. Then we have
    \begin{align*}
    \left(D_{\mathrm{dif}}^{K_{m}}(D_A)\otimes_{K_{m}}K_{m}[\log(t)]\right)^{\Gamma_K}&=D_{\pdR}(W_{\mathrm{dR}}(D_A)),\\
    \left(D_{\mathrm{dif}}^{K_{m}}(D_A)\otimes_{K_{m}}K_{m}[\log(t)]\right)^{\Gamma_{K_m}}&=K_m\otimes_KD_{\pdR}(W_{\mathrm{dR}}(D_A)).
    \end{align*}
\end{lemma}
\begin{proof}
    By definition, $W_{\mathrm{dR}}^+(D_A)=D_A^{r}\otimes_{\cR_{\Q_p,K}^r,\iota_m}\BdR^+$ (the $m$ here is large enough since $\varphi^{*}D_A^{r}\simeq D_A^{r/p}$, cf. \cite[Prop. 2.2.6(2)]{berger2008construction}) is a free $A\otimes_{\Q_p}\BdR^+$-module of rank $n$. The map $\BdR^+\otimes_{K_m[[t]]}D_{\mathrm{dif}}^{K_m,+}(D_A)\rightarrow W_{\mathrm{dR}}^+(D_A)$ induced by the injection $K_m[[t]]\hookrightarrow \BdR^+$ is a $\cG_K$-equivariant isomorphism and we have a similar statement for $D_{\mathrm{dif}}^{K_{\infty},+}(D_A)=K_{\infty}\widehat{\otimes}_{K_m}D_{\mathrm{dif}}^{+,K_{m}}(D_A)$. Now let $X_f$ be the $K_{\infty}[[t]]$-module associated with the $\BdR^+$-representation $W_{\mathrm{dR}}^+(D_A)$ in \cite[Thm. 3.6]{fontaine2004arithmetique}. Then $X_f$ contains $D_{\mathrm{dif}}^{K_{\infty},+}(D_A)$ by \textit{loc. cit.} (our convention on $\Gamma_K$ differs from that in \cite{fontaine2004arithmetique}, however the results in \textit{loc. cit.} apply if we firstly replace $K$ by some finite extension in $K_{\infty}$ and then descent if needed). Modulo $t$ and by definition, $X_f/t= ((W_{\mathrm{dR}}^+(D_A)/t)^{H_K})_f=((C\otimes_{K_{\infty}}D_{\mathrm{Sen}}^{K_{\infty}}(D_A))^{H_K})_f$, where $C\otimes_{K_{\infty}}D_{\mathrm{Sen}}^{K_{\infty}}(D_A)$ is a semi-linear $C$-representation of $\cG_K$ on which $\cG_K$ acts on $D_{\mathrm{Sen}}^{K_{\infty}}(D_A)$ via $\Gamma_K$ and $((C\otimes_{K_{\infty}}D_{\mathrm{Sen}}^{K_{\infty}}(D_A))^{H_K})_f$ denotes the union of all finite-dimensional $K$-subspace that is stable under $\Gamma_K$ of $C\otimes_{K_{\infty}}D_{\mathrm{Sen}}^{K_{\infty}}(D_A)$ as in \cite[Thm. 2.4]{fontaine2004arithmetique}. Thus $((C\otimes_{K_{\infty}}D_{\mathrm{Sen}}^{K_{\infty}}(D_A))^{H_K})_f$ contains $D_{\mathrm{Sen}}^{K_{\infty}}(D_A)$ and has the same $K_{\infty}$-rank with $D_{\mathrm{Sen}}^{K_{\infty}}(D_A)$. Hence $X_f/t$ is identified with $D_{\mathrm{Sen}}^{K_{\infty}}(D_A)$ in $W_{\mathrm{dR}}^+(D_A)/t$. Then the inclusion $D_{\mathrm{dif}}^{K_{\infty},+}(D_A)\hookrightarrow X_f$ is a surjection. Hence $X_f=D_{\mathrm{dif}}^{K_{\infty},+}(D_A)$ in $W_{\mathrm{dR}}^+(D_A)$ and $\Delta_{\mathrm{dR}}(W_{\mathrm{dR}}(D_A))=D_{\mathrm{dif}}^{K_{\infty}}(D_A)=D_{\mathrm{dif}}^{K_{\infty},+}(D_A)[\frac{1}{t}]$ in the notation of \cite[Thm. 3.12]{fontaine2004arithmetique}. \par
    Now by the assumption on Sen weights, the $A\otimes_{\Q_p}\BdR$-representation $W_{\mathrm{dR}}(D_A)$ is almost de Rham. We have 
    \begin{align*}
        D_{\pdR}(W_{\mathrm{dR}}(D_A))=\left(W_{\mathrm{dR}}(D_A)\otimes_{\BdR}\BdR[\log(t)]\right)^{\cG_K}=&\left(\left(W_{\mathrm{dR}}(D_A)\otimes_{\BdR}\BdR[\log(t)]\right)^{H_K}\right)^{\Gamma_K}.
    \end{align*}
    If $a\in D_{\mathrm{dif}}^{K_{\infty}}(D_A)$ such that $\nabla^N(a)=0$ for some $N$, then the sub-$K_{\infty}$-space spanned by $\nabla^i(a)$ is finite-dimensional and stable under $\nabla$. Thus $(D_{\mathrm{dif}}^{K_{\infty}}(D_A)\otimes_{K_{\infty}}K_{\infty}[\log(t)])^{\nabla=0}=D_{\mathrm{dR},\infty}(W_{\mathrm{dR}}(D_A))$ by \cite[Prop. 3.25]{fontaine2004arithmetique}. Since $W_{\mathrm{dR}}(D_A)$ is almost de Rham, $D_{\mathrm{dR},\infty}(W_{\mathrm{dR}}(D_A))=K_{\infty}\otimes_{K}D_{\pdR}(W_{\mathrm{dR}}(D_A))$ by \cite[\S3.6]{fontaine2004arithmetique}. The identification 
    \begin{equation}\label{equationKinftydpdr}
        K_{\infty}\otimes_{K}D_{\pdR}(W_{\mathrm{dR}}(D_A))=\left(D_{\mathrm{dif}}^{K_{\infty}}(D_A)\otimes_{K_{\infty}}K_{\infty}[\log(t)]\right)^{\nabla=0}
    \end{equation}
    of $K_{\infty}$-subspaces in $\left(W_{\mathrm{dR}}(D_A)\otimes_{\BdR}\BdR[\log(t)]\right)^{H_K}$ is $\Gamma_K$-equivariant. Since $m\geq m_0$, eigenvalues of $\gamma_{K_m}=\exp(\log(\epsilon(\gamma_{K_m}))\nabla)$ on $D_{\mathrm{Sen}}^{K_m}(D_A)$ are of the form $\epsilon(\gamma_{K_m})^i$ for $i\in\Z$. By Proposition \ref{propositiondifferentialequation}, the dimension of $\left(D_{\mathrm{dif}}^{K_{m}}(D_A)\otimes_{K_{m}}K_{m}[\log(t)]\right)^{\Gamma_{K_m}}$ over $K_m$ is equal to $n[A:\Q_p]$. By Hilbert 90, the dimension of the $K$-space $\left(D_{\mathrm{dif}}^{K_{m}}(D_A)\otimes_{K_{m}}K_{m}[\log(t)]\right)^{\Gamma_{K}}$ is equal to $n[A:\Q_p]$. Taking $\Gamma_K$-invariants on the two sides of (\ref{equationKinftydpdr}) and counting dimensions, we get $D_{\pdR}(W_{\mathrm{dR}}(D_A))=\left(D_{\mathrm{dif}}^{K_{\infty}}(D_A)\otimes_{K_{\infty}}K_{\infty}[\log(t)]\right)^{\Gamma_K}=\left(D_{\mathrm{dif}}^{K_{m}}(D_A)\otimes_{K_{m}}K_{m}[\log(t)]\right)^{\Gamma_K}.$    
\end{proof} 
We may consider some more general cases. Suppose that $A$ is a local Artinian $L$-algebra with residue field $L'$ finite over $L$. We assume that all the $\tau$-Sen weights of $D_{L'}$ are integers for a fixed $\tau\in\Sigma$ which means that we do not require other Sen weights to be integers. As a $(\varphi,\Gamma_K)$-module over $\cR_{L,K}$, the $\tau$-Sen weights of $D_A$ are all integers since $D_A$ is a successive extension of $D_{L'}$. Recall 
\[D_{\mathrm{pdR},\tau}(W_{\mathrm{dR}}(D_A)):={D_{\pdR}(W_{\mathrm{dR}}(D_A))\otimes_{A\otimes_{\Q_p}K, 1\otimes\tau}A}.\] 
Thus 
\begin{align*}
    D_{\mathrm{pdR},\tau}(W_{\mathrm{dR}}(D_A))=&\left(\left(D_{\mathrm{dif}}^{K_{\infty}}(D_A)\otimes_{A\otimes_{\Q_p}K, 1\otimes\tau}A\right)\otimes_{K_{\infty}}K_{\infty}[\log(t)]\right)^{\Gamma_K}\\=&\left(D_{\mathrm{dif},\tau}^{K_{\infty}}(D_A)\otimes_{K_{\infty}}K_{\infty}[\log(t)]\right)^{\Gamma_K}
\end{align*}
as both the actions of $A$ and $K$ commute with $\Gamma_K$. By Proposition \ref{propositiondifferentialequation} and Hilbert 90, we get 
\[D_{\mathrm{pdR},\tau}(W_{\mathrm{dR}}(D_A))=\left(D_{\mathrm{dif},\tau}^{K_{m}}(D_A)\otimes_{K_{m}}K_{m}[\log(t)]\right)^{\Gamma_K}\]
has $L$-rank $n\dim_{L}(A)$ if $m>m_0$. An argument of \cite[Lem. 3.1.4]{breuil2019local} shows that $D_{\mathrm{pdR},\tau}(W_{\mathrm{dR}}(D_A))$ is flat over $A$ and thus free of rank $n$ over $A$.
\subsection{A family of almost de Rham $(\varphi,\Gamma_K)$-modules}
We assume that $A$ is an affinoid algebra over $L$ and show that $D_{\mathrm{pdR}}(W_{\mathrm{dR}}(D_x)), x\in \mathrm{Sp}(A)$ form a family under certain condition. \par
We take $m>m_0$ and fix $\tau\in\Sigma$. After possibly shrinking $A$, we assume that $D_{\mathrm{dif},\tau}^{K_{m},+}(D_A)$ is a free $A\widehat{\otimes}_KK_m[[t]]$-module of rank $n$ with a continuous semilinear action of $\Gamma_K$. We assume that for any point $x\in \mathrm{Sp}(A)$, the $\tau$-Sen weights, by definition the eigenvalues of $\nabla$ on $D_{\mathrm{Sen},\tau}^{K_{m}}(D_x)$, are integers and lie in $[a,b]$ where $a,b\in\Z$ is independent of $x$. If $A_0$ is a finite-dimensional local $L$-algebra with a morphism $A\rightarrow A_0$, then $D_{\mathrm{pdR},\tau}(W_{\mathrm{dR}}(D_{A_0}))$ is a finite free $A_0$-module of rank $n$ equipped with an $A_0$-linear nilpotent operator $\nu_{A_0}$ by discussions in the end of last subsection. 
\begin{lemma}\label{lemmacharacteristicpolynomialgamma}
    If the $\tau$-Sen polynomial of $D_{\mathrm{Sen},\tau}^{K_{m}}(D_A)$ is $\prod_{i\in I}(T-a_i)^{n_i}\in \Z[T]$ where $a_i$ are integers and $a_i\neq a_j$ if $i\neq j$, then we have $\prod_{i\in I}(\gamma_{K_m}-\epsilon(\gamma_{K_m})^{a_i})^{n_i}D_{\mathrm{dif},\tau}^{K_m,+}(D_A)\subset tD_{\mathrm{dif},\tau}^{K_m,+}(D_A)$.    
\end{lemma}
\begin{proof}
    By the Cayley-Hamilton theorem, $\prod_{i}(\nabla-a_i)^{n_i}=0$ on $D_{\mathrm{Sen},\tau}^{K_m}(D_A)$. Since the polynomials $(T-a_i)^{n_i},i\in I$ are prime to each other in $\Q[T]$, we can focus on each generalized eigenspace after possibly shrinking $\mathrm{Spec}(A)$ and reduce to the case when $\nabla$ is nilpotent on $D_{\mathrm{Sen},\tau}^{K_m}(D_A)$, i.e. $\nabla^n=0$ on $D_{\mathrm{Sen},\tau}^{K_m}(D_A)$. Since $m>m_0, \gamma_{K_m}=\exp\left(\log(\epsilon(\gamma_{K_m}))\nabla\right)=\sum_{i=0}^{\infty}\frac{\left(\log(\epsilon(\gamma_{K_m}))\nabla\right)^i}{i!}$ on $D_{\mathrm{Sen},\tau}^{K_m}(D_A)$. We get that $ (\gamma_{K_m}-1)^n=0$ on $D_{\mathrm{Sen},\tau}^{K_m}(D_A)$.
\end{proof}
\begin{lemma}\label{lemmagamma}
    There exists an integer $l_0$ such that $\prod_{i\in [a,b]}(\gamma_{K_m}-\epsilon(\gamma_{K_m})^{i})^{l_0}D_{\mathrm{dif},\tau}^{K_m,+}(D_A)\subset tD_{\mathrm{dif},\tau}^{K_m,+}(D_A)$.
\end{lemma}
\begin{proof}
    Let $J$ be the nilradical of $A$. We have $D_{\mathrm{dif},\tau}^{K_m,+}(D_{A/J})=D_{\mathrm{dif},\tau}^{K_m,+}(D_{A})/JD_{\mathrm{dif},\tau}^{K_m,+}(D_{A})$. The $(\varphi,\Gamma_K)$-module $D_{A/J}$ satisfies the condition in Lemma \ref{lemmacharacteristicpolynomialgamma} at least on each connected component of $\Spec(A/J)$. Thus there exists an integer $l'$ such that \[\prod_{i\in [a,b]}(\gamma_{K_m}-\epsilon(\gamma_{K_m})^{i})^{l'}D_{\mathrm{dif},\tau}^{K_m,+}(D_A)\subset tD_{\mathrm{dif},\tau}^{K_m,+}(D_A)+JD_{\mathrm{dif},\tau}^{K_m,+}(D_A).\] 
    Since $A$ is Noetherian, $J$ is finitely generated and there is an integer $N$ such that $J^N=0$. Then we can take $l_0=Nl'$.
\end{proof}
\begin{lemma}\label{lemmainvertiblepolynomial}
    Let $A$ be a Noetherian ring over $\Q_p$ and $J$ be its nilradical. If $P(T)\in A[T]$ is a polynomial such that the image of $P(T)$ in $A/J[T]$ is in $\Q_p[T]$ and has no factor $(T-1)$ in $\Q_p[T]$, then for any $l\geq 1$, there exist $G_1(T),G_2(T)\in A[T]$ such that $1=G_1(T)P(T)+G_2(T)(T-1)^l $ in $A[T]$.
\end{lemma}
\begin{proof}
    Denote by $\overline{P}(T)$ the image of $P(T)$ in $\Q_p(T)\subset A/J[T]$. By  Hilbert's Nullstellensatz, there exist $G_1'(T),G_2'(T)\in \Q_p[T]$ such that $1=G_1'(T)\overline{P}(T)+G_2'(T)(T-1)^l$. Thus $G_1'(T)P(T)+G_2'(T)(T-1)^l=1-H(T)$ where $H(T)\in J[T]$. There exists $N$ such that $H(T)^N=0$. Hence $1-H(T)$ is invertible in $A[T]$. We let $G_i(T)=G_i'(T)(1-H(T))^{-1}$ for $i=1,2$.
\end{proof}
The proof of the following proposition follows that of \cite[Thm. 5.3.2]{berger2008familles}.
\begin{proposition}\label{propositionfamilydpdr}
    Assume that for every $x\in\Sp(A)$, all the roots of the $\tau$-Sen polynomial of $D_x$ are in $\Z\cap [a,b]$. Then there exists a finite projective $A$-module $D_{\mathrm{pdR},\tau}(W_{\mathrm{dR}}(D_A))$ of rank $n$ equipped with a nilpotent $A$-linear operator $\nu_A$ such that there is a natural isomorphism of pairs
    \[\left(D_{\mathrm{pdR},\tau}(W_{\mathrm{dR}}(D_A))\otimes_AA_0,\nu_A\otimes_AA_0\right)\simeq \left(D_{\mathrm{pdR},\tau}(W_{\mathrm{dR}}(D_{A_0})),\nu_{A_0}\right)\] 
    for any finite-dimensional local $L$-algebra and map $\Sp(A_0)\rightarrow \Sp(A)$.
\end{proposition}
\begin{proof}
    For $k\geq 1$, let 
    \[\beta_k=\left(\prod_{j=b-a+1}^{b-a+k}f_j(\gamma_{K_m})\right)\prod_{i=a-b,i\neq 0}^{b-a}\left(\frac{\gamma_{K_m}-\epsilon(\gamma_{K_m})^i}{1-\epsilon(\gamma_{K_m})^i}\right)^{l}\]
    where $f_j(T)$ is chosen by Lemma \ref{lemmapolynomialfk} and $l=(b-a+1)l_0$ is determined by Lemma \ref{lemmagamma}. Since $m>m_0$, by Lemma \ref{lemmagamma} and the argument for Claim \ref{claimcharacteristicpolynomials}, 
    \[\prod_{i=a-b}^{b-a+k}(\gamma_{K_m}-\epsilon(\gamma_{K_m})^i)^{l}t^{-b}D_{\mathrm{dif},\tau}^{K_{m},+}(D_A)\subset t^{1+k-a}D_{\mathrm{dif},\tau}^{K_{m},+}(D_A)\]
    for $k\geq 0$. Since $(T-\epsilon(\gamma_{K_m})^{b-a+i})^{l}\mid f_{b-a+i}(T)$ for any $i\geq 1$, we get that if $x\in t^{-b}D_{\mathrm{dif},\tau}^{K_{m},+}(D_A)$, then $(\gamma_{K_m}-1)^l\beta_k(x) \in t^{k+1-a}D_{\mathrm{dif},\tau}^{K_{m},+}(D_A)$. Hence if $x\in t^{-b}D_{\mathrm{dif},\tau}^{K_{m},+}(D_A)$, $\beta_{k}(x)-\beta_{k+1}(x)\in t^{k+1-a}D_{\mathrm{dif},\tau}^{K_{m},+}(D_A)$ by the condition in Lemma \ref{lemmapolynomialfk}. Thus $\beta:=\varinjlim_{k\to+\infty}\beta_k$ defines an $A\otimes_KK_m$-linear map $t^{-b}D_{\mathrm{dif},\tau}^{K_{m},+}(D_A)\rightarrow t^{-b}D_{\mathrm{dif},\tau}^{K_{m},+}(D_A)$. Let $M_A^{K_m}$ be the image of $\beta$. Then $M_A^{K_m}$ is stable under the action of $\Gamma_K$ and $(\gamma_{K_m}-1)^lM_A^{K_m}=0$. Since the map $\left(\prod_{i=a-b,i\neq 0}^{b-a}\frac{\gamma_{K_m}-\epsilon(\gamma_{K_m})^i}{1-\epsilon(\gamma_{K_m})^i}\right)^{l}$ is an isomorphism on $M_A^{K_m}$ and $f_{b-a+k}(\gamma_{K_m})=(\gamma_{K_m}-1)^{l}h_{b-a+k}(\gamma_{K_m})+1$ is the identity on $M_A^{K_m}$ for $k\geq 1$, $\beta$ induces an automorphism of $M_A^{K_m}$. The image of the characteristic polynomial of $\gamma_{K_m}=\exp\left(\log(\epsilon(\gamma_{K_m}))\nabla\right)$ on $t^{-k}D_{\mathrm{dif},\tau}^{K_{m},+}(D_A)/t^{-k+1}D_{\mathrm{dif},\tau}^{K_{m},+}(D_A)$ in $A^{\mathrm{red}}[T]$ is prime to $(T-1)^l$ if $k\notin [a,b]$ (the roots have the form $\epsilon(\gamma_{K_m})^{i-k}$ for some $i\in [a,b]$). Thus by Lemma \ref{lemmainvertiblepolynomial}, we know there exists no non-zero $(\gamma_{K_m}-1)$-nilpotent element in $t^{-k}D_{\mathrm{dif},\tau}^{K_{m},+}(D_A)/t^{-k+1}D_{\mathrm{dif},\tau}^{K_{m},+}(D_A)$ for any $k\notin [a,b]$. Hence $M_A^{K_m}\cap t^{1-a}D_{\mathrm{dif},\tau}^{K_{m},+}(D_A)=\{0 \}$. 
    We get that the natural $A\otimes_KK_m$-module morphism $M_A^{K_m}\rightarrow t^{-b}D_{\mathrm{dif},\tau}^{K_{m},+}(D_A)/t^{-a+1}D_{\mathrm{dif},\tau}^{K_{m},+}(D_A)$ is an injection. The decomposition 
    \[\beta:M_{A}^{K_m}\hookrightarrow t^{-b}D_{\mathrm{dif},\tau}^{K_{m},+}(D_A)/t^{-a+1}D_{\mathrm{dif},\tau}^{K_{m},+}(D_A)\stackrel{\beta}{\rightarrow}M_A^{K_m}\] 
    implies that $M_A^{K_m}$ is a finite projective $A\otimes_KK_m$-module as a direct summand of a finite free $A\otimes_KK_m$-module. For any finite-dimensional local Artinian $L$-algebra $A_0$ with $\Sp(A_0)\rightarrow\Sp(A)$, $\beta$ also acts on $t^{-b}D_{\mathrm{dif},\tau}^{K_{m},+}(D_{A_0})=t^{-b}D_{\mathrm{dif},\tau}^{K_{m},+}(D_{A})\otimes_AA_0$ by extending the scalars since $D_{\mathrm{dif},\tau}^{K_m,+}(D_{A_0})=D_{\mathrm{dif},\tau}^{K_m,+}(D_A)\otimes_AA_0$ satisfies the same result as $A$ for the same $l_0$ in Lemma \ref{lemmagamma}. 
    We have that $\beta$ is also an automorphism on $M_{A_0}^{K_m}$ which is defined to be the $(\gamma_{K_m}-1)$-nilpotent elements in $D_{\mathrm{dif},\tau}^{K_{m}}(D_{A_0})$ (which is contained in $t^{-b}D_{\mathrm{dif},\tau}^{K_{m},+}(D_{A_0})$ by Proposition \ref{propositiondifferentialequation}). The image of $\beta$ on $t^{-b}D_{\mathrm{dif},\tau}^{K_{m}}(D_{A_0})$ is $M_A^{K_m}\otimes_AA_0$. Hence $M_A^{K_m}\otimes_AA_0$ contains and thus is equal to $M_{A_0}^{K_m}$ since $M_A^{K_m}\otimes_AA_0$ is $(\gamma_{K_m}-1)$-nilpotent. \par
    The $A\otimes_KK_m$-linear action of $\nabla=\log(\gamma_{K_m})/\log(\epsilon(\gamma_{K_m}))$ on $M_A^{K_m}$ is defined since $\gamma_{K_m}-1$ is nilpotent on $M_A^{K_m}$. We now set $D_{\mathrm{pdR},\tau}^{K_m}(W_{\mathrm{dR}}(D_A)):= \left(M_A^{K_m}\otimes_{K_m}K_m[\log(t)]\right)^{\Gamma_{K_m}}$ which is equipped with an $A\otimes_KK_m$-linear nilpotent operator $\nu_A^{K_m}$ via $\nu_{\mathrm{pdR}}$ in the usual way. By the same method in the proof of Lemma \ref{lemmagammanilpotent}, there is an isomorphism \[\left(D_{\mathrm{pdR},\tau}^{K_m}(W_{\mathrm{dR}}(D_A)), \nu_{A}^{K_m}\right)\simeq \left(M_A^{K_m},\nabla\right)\]
    of $A\otimes_K K_m$-modules with nilpotent operators which is \emph{not} compatible with the action of $\Gamma_K/\Gamma_{K_m}$. For each $A_0$ as before, via the identification in Lemma \ref{lemmagammanilpotent}, $M_{A_0}^{K_m}$ is identified with 
    \[\left(M_{A_0}^{K_m}\otimes_{K_m}K_m[\log(t)]\right)^{\Gamma_{K_m}}=\left(D_{\mathrm{dif},\tau}^{K_m}(D_{A_0})\otimes_{K_m}K_m[\log(t)]\right)^{\Gamma_{K_m}}=D_{\mathrm{pdR},\tau}(W_{\mathrm{dR}}(D_{A_0}))\otimes_KK_m\] 
    by Lemma \ref{lemmadpdRgammaKinvariant} since $m>m_0$. Hence
    \[(D_{\mathrm{pdR},\tau}^{K_m}(W_{\mathrm{dR}}(D_A))\otimes_A A_0,\nu_{A}^{K_m}\otimes_A A_0)\simeq \left(D_{\mathrm{pdR},\tau}(W_{\mathrm{dR}}(D_{A_0}))\otimes_KK_m, \nu_{A_0}\otimes_K K_m\right)\] 
    which is compatible with the action of $\Gamma_K$.\par
    Now the result follows from the descent in \cite[Prop. 2.2.1]{berger2008familles} by setting 
    \[\left(D_{\mathrm{pdR},\tau}(W_{\mathrm{dR}}(D_A)),\nu_A\right):=\left(D_{\mathrm{pdR},\tau}^{K_m}(W_{\mathrm{dR}}(D_A))^{\Gamma_K/\Gamma_{K_m}}, \nu_{A}^{K_m}|_{D_{\mathrm{pdR},\tau}(W_{\mathrm{dR}}(D_A))}\right)\]
    which is also isomorphic to $\left((M_A^{K_m})^{(\gamma_K-1)-\mathrm{nil}}, \nabla\right)$ by the same arguments in Lemma \ref{lemmagammanilpotent}.
\end{proof}
\bibliography{bibfile.bib} 

\begin{thebibliography}{10}

\bibitem{atiyah1969introduction}
Michael~Francis Atiyah and Ian~G MacDonald.
\newblock {\em Introduction to commutative algebra}.
\newblock Addison-Wesley Publishing Company, 1969.

\bibitem{backelin2015singular}
Erik Backelin and Kobi Kremnitzer.
\newblock Singular localization of $\mathfrak{g}$-modules and applications to
  representation theory.
\newblock {\em Journal of the European Mathematical Society},
  17(11):2763--2787, 2015.

\bibitem{bergdall2017paraboline}
John Bergdall.
\newblock {Paraboline variation over $ p $-adic families of $(\varphi,\Gamma)
  $-modules}.
\newblock {\em Compositio Mathematica}, 153(1):132--174, 2017.

\bibitem{berger2008construction}
Laurent Berger.
\newblock Construction de {($\varphi$, $\Gamma$)}-modules: repr{\'e}sentations
  $p$-adiques et {$B$}-paires.
\newblock {\em Algebra \& Number Theory}, 2(1):91--120, 2008.

\bibitem{berger2008equations}
Laurent Berger.
\newblock {\'E}quations diff{\'e}rentielles $p$-adiques et {$(\varphi,
  N)$}-modules filtr{\'e}s.
\newblock {\em Ast{\'e}risque}, 319:13--38, 2008.

\bibitem{berger2008familles}
Laurent Berger and Pierre Colmez.
\newblock {Familles de repr\'esentations de de Rham et monodromie $p$-adique}.
\newblock In Berger Laurent, Breuil Christophe, and Colmez Pierre, editors,
  {\em Repr\'esentations $p$-adiques de groupes $p$-adiques I :
  repr\'esentations galoisiennes et $(\varphi, \Gamma)$-modules}, number 319 in
  Ast\'erisque, pages 303--337. Soci\'et\'e math\'ematique de France, 2008.

\bibitem{bezrukavnikov2012affine}
Roman Bezrukavnikov and Simon Riche.
\newblock {Affine braid group actions on derived categories of Springer
  resolutions}.
\newblock In {\em Annales scientifiques de l'Ecole normale sup{\'e}rieure},
  volume~45, pages 535--599, 2012.

\bibitem{bjorner2006combinatorics}
Anders Bjorner and Francesco Brenti.
\newblock {\em Combinatorics of Coxeter groups}, volume 231.
\newblock Springer Science \& Business Media, 2006.

\bibitem{borel2012linear}
Armand Borel.
\newblock {\em Linear algebraic groups}, volume 126.
\newblock Springer Science \& Business Media, 2012.

\bibitem{borel1972complements}
Armand Borel and Jacques Tits.
\newblock Compl{\'e}ments {\`a} l'article {\guillemotleft}groupes
  r{\'e}ductifs{\guillemotright}.
\newblock {\em Publications Math{\'e}matiques de l'IH{\'E}S}, 41:253--276,
  1972.

\bibitem{borho1982differential}
Walter Borho and Jean-Luc Brylinski.
\newblock Differential operators on homogeneous spaces. {I}.
\newblock {\em Inventiones mathematicae}, 69(3):437--476, 1982.

\bibitem{bosch2014lectures}
Siegfried Bosch.
\newblock {\em Lectures on formal and rigid geometry}, volume 2105.
\newblock Springer, 2014.

\bibitem{bosch1984non}
Siegfried Bosch, Ulrich G{\"u}ntzer, and Reinhold Remmert.
\newblock {\em {Non-Archimedean analysis}}, volume 261.
\newblock Springer Berlin, 1984.

\bibitem{breuil2015versII}
Christophe Breuil.
\newblock Vers le socle localement analytique pour {$\text{GL}_n$} {II}.
\newblock {\em Mathematische Annalen}, 361(3-4):741--785, 2015.

\bibitem{breuil2016versI}
Christophe Breuil.
\newblock Vers le socle localement analytique pour {$\text{GL}_n$} {I}.
\newblock {\em Annales de l'Institut Fourier}, 66(2):633--685, 2016.

\bibitem{breuil2017smoothness}
Christophe Breuil, Eugen Hellmann, and Benjamin Schraen.
\newblock Smoothness and classicality on eigenvarieties.
\newblock {\em Inventiones mathematicae}, 209(1):197--274, 2017.

\bibitem{breuil2017interpretation}
Christophe Breuil, Eugen Hellmann, and Benjamin Schraen.
\newblock Une interpr{\'e}tation modulaire de la vari{\'e}t{\'e} trianguline.
\newblock {\em Mathematische Annalen}, 367(3-4):1587--1645, 2017.

\bibitem{breuil2019local}
Christophe Breuil, Eugen Hellmann, and Benjamin Schraen.
\newblock A local model for the trianguline variety and applications.
\newblock {\em Publications math{\'e}matiques de l'IH{\'E}S}, 130(1):299--412,
  2019.

\bibitem{buzzard2007eigenvarieties}
Kevin Buzzard.
\newblock Eigenvarieties.
\newblock {\em London Mathematical Society Lecture Note Series},
  1(320):59--120, 2007.

\bibitem{caraiani2016patching}
Ana Caraiani, Matthew Emerton, Toby Gee, David Geraghty, Vytautas
  Pa{\v{s}}k{\=u}nas, and Sug~Woo Shin.
\newblock Patching and the $p$-adic local {Langlands} correspondence.
\newblock {\em Cambridge Journal of Mathematics}, 4(2):197--287, 2016.

\bibitem{chenevier2004familles}
Ga{\"e}tan Chenevier.
\newblock Familles $p$-adiques de formes automorphes pour $\mathrm{GL}_n$.
\newblock {\em Journal fur die reine und angewandte Mathematik}, 570:143--218,
  2004.

\bibitem{chenevier2011infinite}
Ga{\"e}tan Chenevier.
\newblock {On the infinite fern of Galois representations of unitary type}.
\newblock {\em Annales scientifiques de l'Ecole normale sup{\'e}rieure},
  44(6):963--1019, 2011.

\bibitem{Chriss1997RepresentationTA}
Neil~A. Chriss and Victor Ginzburg.
\newblock {\em Representation theory and complex geometry}.
\newblock Springer Science \& Business Media, 1997.

\bibitem{colmez2008trianguline}
Pierre Colmez.
\newblock Repr{\'e}sentations triangulines de dimension $2$.
\newblock {\em Ast{\'e}risque}, 319(213-258):83, 2008.

\bibitem{colmez2010representations}
Pierre Colmez.
\newblock {Repr{\'e}sentations de $\mathrm{GL}_2(\mathbb{Q}_p)$ et $(\varphi,
  \Gamma)$-modules}.
\newblock {\em Ast{\'e}risque}, 330(281):509, 2010.

\bibitem{colmez2014completes}
Pierre Colmez and Gabriel Dospinescu.
\newblock Compl{\'e}t{\'e}s universels de repr{\'e}sentations de
  $\mathrm{GL}_2(\mathbb{Q}_p)$.
\newblock {\em Algebra \& Number Theory}, 8(6):1447--1519, 2014.

\bibitem{de1995crystalline}
A.~J. de~Jong.
\newblock Crystalline {D}ieudonn{\'e} module theory via formal and rigid
  geometry.
\newblock {\em Publications Math{\'e}matiques de l'Institut des Hautes
  {\'E}tudes Scientifiques}, 82(1):5--96, 1995.

\bibitem{ding2017formes}
Yiwen Ding.
\newblock Formes modulaires $p$-adiques sur les courbes de shimura unitaires et
  compatibilit\'e local-global.
\newblock {\em M\'emoires de la {Soci\'et\'e Math\'ematique de France}},
  155:viii+245, 2017.

\bibitem{ding2017partiallyderham}
Yiwen Ding.
\newblock {$\protect \mathcal{L}$-invariants, partially de Rham families, and
  local-global compatibility}.
\newblock {\em Annales de l'Institut Fourier}, 67(4):1457--1519, 2017.

\bibitem{ding2019companion}
Yiwen Ding.
\newblock {Companion points and locally analytic socle for $\mathrm{GL}_2(L)$}.
\newblock {\em Israel Journal of Mathematics}, 231(1):47--122, 2019.

\bibitem{ding2019ordinary}
Yiwen Ding.
\newblock $\mathrm{GL}_2(\mathbb{Q}_p)$-ordinary families and automorphy
  lifting.
\newblock {\em arXiv:1907.07372}, 2019.

\bibitem{ding2019some}
Yiwen Ding.
\newblock Some results on the locally analytic socle for
  {$\mathrm{GL}_n(\mathbb{Q}_p)$}.
\newblock {\em International Mathematics Research Notices},
  2019(19):5989--6035, 2019.

\bibitem{dospinescu2020infinitesimal}
Gabriel Dospinescu, Vytautas Pa{\v{s}}k{\=u}nas, and Benjamin Schraen.
\newblock Infinitesimal characters in arithmetic families.
\newblock {\em arXiv preprint arXiv:2012.01041}, 2020.

\bibitem{douglass2014equivariant}
J~Matthew Douglass and Gerhard Roehrle.
\newblock {Equivariant $K$-theory of generalized Steinberg varieties}.
\newblock {\em Transformation Groups}, 19(1):27--56, 2014.

\bibitem{douglass2004geometry}
J~Matthew Douglass and Gerhard R{\"o}hrle.
\newblock {The geometry of generalized Steinberg varieties}.
\newblock {\em Advances in Mathematics}, 187(2):396--416, 2004.

\bibitem{emerton2006jacquet}
Matthew Emerton.
\newblock Jacquet modules of locally analytic representations of $p$-adic
  reductive groups {I}. {Construction and first properties}.
\newblock In {\em Annales scientifiques de l’Ecole normale sup{\'e}rieure},
  volume~39, pages 775--839. Elsevier, 2006.

\bibitem{emerton2006interpolation}
Matthew Emerton.
\newblock On the interpolation of systems of eigenvalues attached to
  automorphic {H}ecke eigenforms.
\newblock {\em Inventiones mathematicae}, 164(1):1--84, 2006.

\bibitem{emerton2007jacquetII}
Matthew Emerton.
\newblock Jacquet modules of locally analytic representations of $p$-adic
  reductive groups {II}. {The relation to parabolic induction}.
\newblock {\em J. Institut Math. Jussieu}, 2007.

\bibitem{emerton2011local}
Matthew Emerton.
\newblock Local-global compatibility in the $p$-adic {Langlands} programme for
  $\mathrm{GL}_{2/\mathbb{Q}}$.
\newblock {\em preprint}, 3, 2011.

\bibitem{emerton2017locally}
Matthew Emerton.
\newblock {\em Locally analytic vectors in representations of locally $ p
  $-adic analytic groups}, volume 248.
\newblock American Mathematical Soc., 2017.

\bibitem{fontaine2004arithmetique}
Jean-Marc Fontaine.
\newblock Arithm{\'e}tique des repr{\'e}sentations galoisiennes $p$-adiques.
\newblock {\em Ast{\'e}risque}, 295:1--115, 2004.

\bibitem{ginsburg1986g}
V~Ginsburg.
\newblock {$\mathfrak{G}$-modules, Springer's representations and bivariant
  Chern classes}.
\newblock {\em Advances in Mathematics}, 61(1):1--48, 1986.

\bibitem{grothendieck1960EGAIV1}
Alexander Grothendieck.
\newblock {\'El\'ements de g\'eom\'etrie alg\'ebrique : IV. \'Etude locale des
  sch\'emas et des morphismes de sch\'emas, Premi\`ere partie}.
\newblock {\em Publications Math\'ematiques de l'IH\'ES}, 20:5--259, 1964.

\bibitem{grothendieck1965EGAIV2}
Alexander Grothendieck.
\newblock {\'El\'ements de g\'eom\'etrie alg\'ebrique : IV. \'Etude locale des
  sch\'emas et des morphismes de sch\'emas, Seconde partie}.
\newblock {\em Publications Math\'ematiques de l'IH\'ES}, 24:5--231, 1965.

\bibitem{grothendieck1960EGAI}
Alexander Grothendieck and Jean Dieudonn{\'e}.
\newblock {El{\'e}ments de g{\'e}om{\'e}trie alg{\'e}brique : I. Le langage des
  sch{\'e}mas}.
\newblock {\em Publications Math{\'e}matiques de l'Institut des Hautes
  {\'E}tudes Scientifiques}, 4(1):5--214, 1960.

\bibitem{hansen2017universal}
David Hansen and James Newton.
\newblock {Universal eigenvarieties, trianguline Galois representations, and
  $p$-adic Langlands functoriality}.
\newblock {\em Journal f{\"u}r die reine und angewandte Mathematik},
  2017(730):1--64, 2017.

\bibitem{hartl2020universal}
Urs Hartl and Eugen Hellmann.
\newblock {The universal family of semistable $p$-adic Galois representations}.
\newblock {\em Algebra \& Number Theory}, 14(5):1055--1121, 2020.

\bibitem{hill2011emerton}
Richard Hill and David Loeffler.
\newblock Emerton’s {J}acquet functors for non-{B}orel parabolic subgroups.
\newblock {\em Documenta Mathematica}, 16:1--31, 2011.

\bibitem{hotta2007dmodules}
Ryoshi Hotta and Toshiyuki Tanisaki.
\newblock {\em $D$-modules, perverse sheaves, and representation theory},
  volume 236.
\newblock Springer Science \& Business Media, 2007.

\bibitem{humphreys2008representations}
James~E Humphreys.
\newblock {\em {Representations of Semisimple Lie Algebras in the BGG Category
  $\mathcal{O}$}}, volume~94.
\newblock American Mathematical Soc., 2008.

\bibitem{humphreys2012introduction}
James~E Humphreys.
\newblock {\em Introduction to {L}ie algebras and representation theory},
  volume~9.
\newblock Springer Science \& Business Media, 2012.

\bibitem{irving1990singular}
Ronald~S Irving.
\newblock Singular blocks of the category {$\mathcal{O}$}.
\newblock {\em Mathematische Zeitschrift}, 204(1):209--224, 1990.

\bibitem{jantzen2007representations}
Jens~Carsten Jantzen.
\newblock {\em Representations of algebraic groups}.
\newblock Number 107. American Mathematical Soc., 2007.

\bibitem{kedlaya2014cohomology}
Kiran Kedlaya, Jonathan Pottharst, and Liang Xiao.
\newblock Cohomology of arithmetic families of {$(\varphi, \Gamma)$}-modules.
\newblock {\em Journal of the American Mathematical Society}, 27(4):1043--1115,
  2014.

\bibitem{kiehl2013weil}
Reinhardt Kiehl and Rainer Weissauer.
\newblock {\em Weil conjectures, perverse sheaves and l’adic Fourier
  transform}, volume~42.
\newblock Springer Science \& Business Media, 2013.

\bibitem{kisin2003overconvergent}
Mark Kisin.
\newblock {Overconvergent modular forms and the Fontaine-Mazur conjecture}.
\newblock {\em Inventiones mathematicae}, 153(2):373--454, 2003.

\bibitem{kisin2008potentially}
Mark Kisin.
\newblock Potentially semi-stable deformation rings.
\newblock {\em Journal of the American Mathematical Society}, 21(2):513--546,
  2008.

\bibitem{kisin2009moduli}
Mark Kisin.
\newblock Moduli of finite flat group schemes, and modularity.
\newblock {\em Annals of Mathematics}, pages 1085--1180, 2009.

\bibitem{knapp1988lie}
Anthony~W Knapp.
\newblock {\em {Lie groups, Lie algebras, and cohomology}}, volume~34.
\newblock Princeton University Press, 1988.

\bibitem{liu2015triangulation}
Ruochuan Liu.
\newblock Triangulation of refined families.
\newblock {\em Commentarii Mathematici Helvetici}, 90(4):831--904, 2015.

\bibitem{orlik2015jordan}
Sascha Orlik and Matthias Strauch.
\newblock On {Jordan-H{\"o}lder} series of some locally analytic
  representations.
\newblock {\em Journal of the American Mathematical Society}, 28(1):99--157,
  2015.

\bibitem{riche2008geometric}
Simon Riche.
\newblock Geometric braid group action on derived categories of coherent
  sheaves.
\newblock {\em Representation Theory of the American Mathematical Society},
  12(5):131--169, 2008.

\bibitem{schneider2013nonarchimedean}
Peter Schneider.
\newblock {\em Nonarchimedean functional analysis}.
\newblock Springer Science \& Business Media, 2013.

\bibitem{schneider2002locally}
Peter Schneider and Jeremy Teitelbaum.
\newblock Locally analytic distributions and $p$-adic representation theory,
  with applications to {$GL_2$}.
\newblock {\em Journal of the American Mathematical Society}, 15(2):443--468,
  2002.

\bibitem{schneider2003algebras}
Peter Schneider and Jeremy Teitelbaum.
\newblock Algebras of $p$-adic distributions and admissible representations.
\newblock {\em Inventiones mathematicae}, 153(1):145--196, 2003.

\bibitem{schraen2010representations}
Benjamin Schraen.
\newblock Repr{\'e}sentations $p$-adiques de {$\mathrm{GL}_2 (L)$} et
  cat{\'e}gories d{\'e}riv{\'e}es.
\newblock {\em Israel Journal of Mathematics}, 176(1):307--361, 2010.

\bibitem{stacks-project}
The {Stacks Project Authors}.
\newblock \textit{Stacks Project}.
\newblock \url{https://stacks.math.columbia.edu}, 2022.

\bibitem{tenner2020parabolic}
Bridget Tenner, William Slofstra, T~Kyle Petersen, Matjaz Konvalinka, and Sara
  Billey.
\newblock {Parabolic double cosets in Coxeter groups}.
\newblock {\em Discrete Mathematics \& Theoretical Computer Science}, 2020.

\bibitem{yun2016lectures}
Zhiwei Yun.
\newblock {Lectures on Springer theories and orbital integrals}.
\newblock {\em arXiv:1602.01451}, 2016.

\end{thebibliography}
\bibliographystyle{plain}
\end{document}